\documentclass[10pt]{amsart}

\pdfoutput=1

\usepackage{mathrsfs, amsmath, amsthm, amssymb, mathtools}%
\usepackage{tikz}%
\usepackage{hyperref}%
\usepackage[all]{xy}%
\xyoption{2cell}%
\UseAllTwocells%

\hypersetup{
 pdftitle={The formal theory of Tannaka duality},
 pdfauthor={Daniel Schaeppi},
 pdfkeywords={Tannaka duality} {pseudomonoids} {Hopf monoidal comonads}
}

\DeclareMathOperator{\id}{id}
\DeclareMathOperator{\el}{el}
\DeclareMathOperator{\ev}{ev}
\DeclareMathOperator{\op}{op}
\DeclareMathOperator{\rev}{rev}

\DeclareMathOperator{\Nat}{Nat}
\DeclareMathOperator{\Lan}{Lan}

\DeclareMathOperator{\Mod}{\mathbf{Mod}}
\DeclareMathOperator{\Rep}{Rep}
\DeclareMathOperator{\Vect}{\mathbf{Vect}}
\DeclareMathOperator{\Coalg}{\mathbf{Coalg}}
\DeclareMathOperator{\fgp}{fgp}
\DeclareMathOperator{\Hom}{Hom}
\DeclareMathOperator{\Spec}{Spec}

\DeclareMathOperator{\Map}{Map}

\DeclareMathOperator{\cospan}{Cospan}
\DeclareMathOperator{\CommAlg}{\mathbf{CommAlg}}

\DeclareMathOperator{\Coact}{Coact}
\DeclareMathOperator{\Gray}{\mathbf{Gray}}
\DeclareMathOperator{\PsMon}{\mathbf{PsMon}}
\DeclareMathOperator{\BrPsMon}{\mathbf{BrPsMon}}
\DeclareMathOperator{\SymPsMon}{\mathbf{SymPsMon}}
\DeclareMathOperator{\MonComon}{\mathbf{MonComon}}
\DeclareMathOperator{\BrMonComon}{\mathbf{BrMonComon}}
\DeclareMathOperator{\SymMonComon}{\mathbf{SymMonComon}}
\DeclareMathOperator{\Psa}{\mathbf{Psa}}

\DeclareMathOperator{\proj}{proj}
\DeclareMathOperator{\Fil}{Fil}
\DeclareMathOperator{\MF}{MF}
\DeclareMathOperator{\colim}{colim}

\DeclareMathOperator{\Cocts}{\mathbf{Cocts}}

\DeclareMathOperator{\Cat}{\mathbf{Cat}}
\DeclareMathOperator{\CAT}{\mathbf{CAT}}
\DeclareMathOperator{\Set}{\mathbf{Set}}
\DeclareMathOperator{\Ab}{\mathbf{Ab}}

\DeclareMathOperator{\Comon}{\mathbf{Comon}}
\DeclareMathOperator{\Comod}{\mathbf{Comod}}

\newcommand{\ca}[1]{\mathscr{#1}}
\newcommand{\VNat}{\ca{V}\mbox{-}\Nat}
\newcommand{\Vcat}{\ca{V}\mbox{-}\Cat}
\newcommand{\VCAT}{\ca{V}\mbox{-}\CAT}
\newcommand{\Prs}[1]{\mathcal{P}\ca{#1}}

\newcommand{\dual}[1]{{#1}^{\circ}}
\newcommand{\ldual}[1]{{#1}^{\vee}}

\newcommand{\defl}{\mathrel{\mathop:}=}

\theoremstyle{plain}
\newtheorem{thm}{Theorem}[subsection]
\newtheorem{prop}[thm]{Proposition}
\newtheorem{lemma}[thm]{Lemma}
\newtheorem{cor}[thm]{Corollary}

\theoremstyle{definition}
\newtheorem{example}[thm]{Example}
\newtheorem{rmk}[thm]{Remark}
\newtheorem{dfn}[thm]{Definition}
\newtheorem{notation}[thm]{Notation}

\newtheoremstyle{citing}{}{}{\itshape}{}{\bfseries}{.}{ }{\thmnote{#3}}
\theoremstyle{citing}
\newtheorem{cit}{}

\newtheoremstyle{citingdfn}{}{}{}{}{\bfseries}{.}{ }{\thmnote{#3}}
\theoremstyle{citingdfn}

\makeatletter
\def\slashedarrowfill@#1#2#3#4#5{%
  $\m@th\thickmuskip0mu\medmuskip\thickmuskip\thinmuskip\thickmuskip
   \relax#5#1\mkern-7mu%
   \cleaders\hbox{$#5\mkern-2mu#2\mkern-2mu$}\hfill
   \mathclap{#3}\mathclap{#2}%
   \cleaders\hbox{$#5\mkern-2mu#2\mkern-2mu$}\hfill
   \mkern-7mu#4$%
}
\def\rightslashedarrowfill@{%
  \slashedarrowfill@\relbar\relbar\mapstochar\rightarrow}
\newcommand\xslashedrightarrow[2][]{%
  \ext@arrow 0055{\rightslashedarrowfill@}{#1}{#2}}
\makeatother

\keywords{Tannaka duality, pseudomonoids, Hopf monoidal comonads}
\subjclass[2000]{16T05, 18D20}

\title{The formal theory of Tannaka duality}
\author{Daniel Sch\"appi}
\address{University of Chicago, Department of Mathematics, 5734 S University Avenue, 60637 Chicago}

\begin{document}

\begin{abstract}
 A Tannakian category is an abelian tensor category equipped with a fiber functor and additional structures which ensure that it is equivalent to the category of representations of some affine groupoid scheme acting on the spectrum of a field extension. If we are working over an arbitrary commutative ring rather than a field, the categories of representations cease to be abelian. We provide a list of sufficient conditions which ensure that an additive tensor category is equivalent to the category of representations of an affine groupoid scheme acting on an affine scheme, or, more generally, to the category of representations of a Hopf algebroid in a symmetric monoidal category. In order to do this we develop a ``formal theory of Tannaka duality'' inspired by Ross Street's ``formal theory of monads.'' We apply our results to certain categories of filtered modules which are used to study $p$-adic Galois representations.
\end{abstract}

\maketitle

\section{Introduction}\label{INTRODUCTION}

 Tannaka duality is a duality between affine group schemes over a field $k$ and their categories of representations. It was developed by Saavedra Rivano, Deligne and Milne \cite{SAAVEDRA, DELIGNE_MILNE, DELIGNE} and answers the following questions:

\begin{enumerate}
 \item[(1)] The reconstruction problem: can an affine group scheme be reconstructed from its category of representations?
 \item[(2)] The recognition problem: which $k$-linear functors $w \colon \ca{A} \rightarrow \Vect_k$ are equivalent to a forgetful functor $V \colon \Rep(G) \rightarrow \Vect_k$ for an affine group scheme $G$ over $k$?
\end{enumerate}

 Tannaka duality can also be extended to affine groupoid schemes acting on $\Spec(K)$ for some field extension $K \supseteq k$. In \cite[Th\'eor\`eme~1.12(iii)]{DELIGNE} it was shown that any affine groupoid scheme $G$ acting on $\Spec(K)$ can be reconstructed from the forgetful functor $\Rep(G) \rightarrow \Vect_K$.
 
 Recall that a monoidal category is called \emph{autonomous} if every object has a dual. An autonomous symmetric monoidal abelian $k$-linear category is called a \emph{rigid tensor category}. In \cite[Th\'eor\`eme~1.12(ii)]{DELIGNE}, Deligne proved that for a rigid tensor category $\ca{A}$, a $k$-linear tensor functor $w\colon \ca{A} \rightarrow \Vect_K$ is equivalent to a forgetful functor $\Rep(G) \rightarrow \Vect_K$ for an affine groupoid scheme $G$ acting on $\Spec(K)$ if and only if $w$ is faithful and exact. Rigid tensor categories for which such a functor exists are called \emph{Tannakian categories}, and strong monoidal exact $k$-linear functors are called \emph{fiber functors}. If a Tannakian category admits a fiber functor for $K=k$, it is called \emph{neutral}.

 We can ask the corresponding questions for affine group schemes over arbitrary rings, and affine groupoid schemes acting on an arbitrary affine scheme. However, if we want to study \emph{autonomous} categories of representations, we have to restrict our attention to representations whose underlying modules are finitely generated and projective. Then the resulting category is no longer abelian, and a lot of the techniques used in \cite{DELIGNE} are no longer applicable. Using an alternative, more categorical approach, we will prove a generalized version of the reconstruction and recognition results of Saavedro Rivano and Deligne. The motivation for this project was the following question posed by Richard Pink.

 \subsection{Motivating example}\label{F_MODULES_INTRO_SECTION}
 Let $k$ be a perfect field of characteristic $p>0$, and let $W$ be the ring of Witt-vectors with coefficients in $k$. We write $W_n$ for the quotient ring $W\slash p^n W$. In \cite{FONTAINE_LAFFAILLE}, J.-M.\ Fontaine and G.\ Laffaille defined the category $\MF_{fl}$, which consists of filtered $W$-modules of finite length with some additional structure (see Section~\ref{F_MODULES_SECTION} for a precise definition). These categories are used to construct $p$-adic Galois representations. There are faithful functors from certain subcategories of $\MF_{fl}$ to the category of Galois representations on $\mathbb{Z}_p$-modules of finite length (see \cite[Th\'eor\`eme~3.3]{FONTAINE_LAFFAILLE}). By passing to a limit, one obtains continuous Galois representations on free $\mathbb{Z}_p$-modules (see \cite[\S~7.14 and Proposition~7.15]{FONTAINE_LAFFAILLE}), and by inverting $p$ one obtains crystalline Galois representations. The category occurring in this last step is a $\mathbb{Q}_p$-linear Tannakian category whose fiber functor lands in the category of vector spaces over the field of fractions of $W$ (see \cite[Remarques~7.10]{FONTAINE_LAFFAILLE}). Richard Pink asked whether it is possible to apply the Tannakian philosophy at an earlier stage of this process, where the categories involved are $\mathbb{Z}\slash p^n\mathbb{Z}$-linear or $\mathbb{Z}_p$-linear.

We follow J.-P.\ Wintenberger and call the objects of $\MF_{fl}$ \emph{filtered $F$-modules}\footnote{ In \cite{WINTENBERGER}, the objects of $\MF_{fl}$ are called `$F$-modules filtr\'es sur $W$'.}. We let $MF^n_{\proj}$ be the full subcategory of $\MF_{fl}$ consisting of those filtered $F$-modules whose underlying $W$-modules are finitely generated projective $W_n$-modules. This category is $\mathbb{Z}\slash p^n \mathbb{Z}$-linear, and the forgetful functor gives a $\mathbb{Z} \slash p^n \mathbb{Z}$-linear functor $w \colon \MF^n_{\proj} \rightarrow \Mod^{\fgp}_{W_n}$. In Section~\ref{F_MODULES_SECTION} we will prove the following theorem.

\begin{cit}[Theorem~\ref{FILTERED_MODULE_THEOREM}]
 There is a groupoid $G=\Spec(L_n)$ acting on $\Spec(W_n)$ and a symmetric strong monoidal equivalence $\MF^n_{\proj} \simeq \Rep(G)$, where $\Rep(G)$ is the category of dualizable representations of $G$.
 \end{cit}

\subsection{Generalization to arbitrary cosmoi}\label{COSMOS_SECTION}
 An affine groupoid $G=\Spec(H)$ acting on $\Spec(B)$ over $\Spec(R)$ is precisely a Hopf algebroid $(B,H)$ in the monoidal category $\Mod_R$. The notion of a Hopf algebroid makes sense in \emph{any} symmetric monoidal category, and we can equally well study the recognition and reconstruction problems in this context. In order to have the desired categorical techniques available, we need our symmetric monoidal categories to be complete, cocomplete and closed. Following B\'enabou \cite{BENABOU} and Kelly \cite{KELLY_COSMOS}, we call a complete and cocomplete symmetric monoidal closed category a \emph{cosmos}. Hopf algebroids in cosmoi of \emph{graded} modules play an important role in algebraic topology (cf.\ \cite{HOVEY}).
 
 For certain classes of cosmoi these questions have already been studied: T.\ Wedhorn studied the reconstruction problem over Dedekind rings, and the recognition problem for valuation rings (see \cite{WEDHORN}). B.\ Day solved both problems for finitely presentable cosmoi for which the full subcategory of objects with duals is closed under finite limits and colimits (see \cite{DAY}). P.\ McCrudden used a result of B.\ Pareigis to solve the reconstruction problem for \emph{Maschkean} categories, which are certain abelian monoidal categories in which all monomorphisms split (see \cite{PAREIGIS}, \cite{MCCRUDDEN_MASCHKE}). 
 
 All these approaches make the assumption that the category of objects with duals is closed under finite limits. But an $R$-module has a dual if and only if it is finitely generated and projective, and a kernel of a morphism between projective modules is in general not projective; therefore, the above results cannot be applied to the case where $\ca{V}$ is the cosmos $\Mod_R$ of $R$-modules for a general commutative ring $R$, such as the case of the example described in Section~\ref{F_MODULES_INTRO_SECTION}. 
 
 Nori's Tannakian Theorem (unpublished, see \cite{BRUGUIERES}) concerns Hopf algebras in categories of Pro-objects of finitely generated modules, and it is only applicable if the homological dimension is at most one. Our main example of filtered $F$-modules concerns $\mathbb{Z}/p^n \mathbb{Z}$-linear categories, and $\mathbb{Z}/p^n \mathbb{Z}$ has infinite homological dimension.
 
 There are various generalizations of Tannaka duality to quantum algebra, e.g.\, in the work of Ph\`ung H\^o Hai \cite{PHUNG_HO_HAI}, and Korn\'el Szlach\'anyi \cite{SZLACHANYI}. While Hai stays in the world of abelian categories, Szlach\'anyi also encounters the problem that the base category is a category of finitely generated projective modules. His results inspired some of the specialized theorems we prove for the case of cosmoi of (graded) $R$-modules. Note that Szlach\'anyi only studies noncommutative bialgebroids. The forgetful functors of their categories of comodules are strong monoidal for the nonsymmetric tensor product of $R$-$R$-bimodules. Thus his results can't be applied to prove facts about affine groupoids.

\subsection{Discussion of results}\label{DESCRIPTION_OF_RESULTS_SECTION}

 Deligne's proof of the recognition result for the case of fields proceeds in several steps. Under the contravariant equivalence between affine schemes over $\Spec(k)$ and $k$-algebroids, the affine groupoids correspond to Hopf algebroids. A Hopf algebroid $(B,H)$ in the category of $k$-vector spaces consists of two $k$-algebras $B$ and $H$, together with two homomomorphisms of $k$-algebras $B \rightarrow H$ (called the left and right unit, corresponding to the source and target maps). These turn $H$ into a $B$-$B$-bimodule. In addition to this, a Hopf algebroid has a comultiplication $H \rightarrow H\mathop{{\otimes}_B} H$ and a counit $H \rightarrow B$ (corresponding to the composition operation and the map sending an object to its identity in the affine groupoid), and an antipode $H \rightarrow H$ (corresponding to the map which sends a morphism to its inverse). 
 
 A \emph{$B$-$B$-coalgebroid} is a $B$-$B$-bimodule $C$ with a comultiplication $C \rightarrow C \mathop{\otimes_B} C$ and a counit $C \rightarrow B$, which are coassociative and counital. Since we are working over $k$, a $B$-$B$-bimodule is understood to be an abelian group with a left and a right $B$-action such that the two $k$-actions coincide. In particular, if we take $B=k$, then a $B$-$B$-coalgebroid is simply a $k$-coalgebra. Note that every Hopf algebroid is in particular a coalgebroid. Moreover, the structure of a coalgebroid is the bare minimum needed to define a category of comodules. A \emph{$C$-comodule} is a $B$-module $M$ endowed with a coaction $M \rightarrow C\mathop{\otimes_B} M$ which is compatible with the comultiplication and the counit.
 
 It is in fact convenient to think of a Hopf algebroid as a coalgebroid endowed with additional structure. Deligne first studies the relationship between $k$-linear categories equipped with a $k$-linear functor to the category of finite dimensional $K$-vector spaces on the one hand, and $K$-$K$-coalgebroids on the other. In a second step, he shows that symmetric monoidal structures induce a commutative algebra structure on the corresponding coalgebroid, and thirdly that the existence of duals implies the existence of an antipode.
 
 \begin{dfn}
 Let $R$ be a commutative ring, and let $B$ be a commutative $R$-algebra. A \emph{Cauchy comodule} of a $B$-$B$-coalgebroid $C$ is a comodule whose underlying $B$-module is finitely generated and projective. The category of Cauchy comodules of $C$ is denoted by $\Rep(C)$.
 \end{dfn}

 Ross Street observed that the functor which sends a coalgebra to its category of Cauchy comodules, equipped with its forgetful functor, has a left biadjoint. He called this biadjunction the \emph{Tannakian biadjunction} (cf.\ \cite[\S~16]{STREET_QUANTUM_GROUPS}). We give a generalized construction of the Tannakian biadjunction, and show that it is symmetric monoidal. The latter will take up a large part of the second half of the paper, and it provides a conceptual explanation for why the left biadjoint sends (weak) monoids in the domain (monoidal categories) to monoids in the codomain (bialgebroids). Under this interpretation, the reconstruction problem and the recognition problem have a precise mathematical formulation:
 \begin{enumerate}
  \item[(1)] Reconstruction problem: when is the counit of the Tannakian biadjunction an isomorphism?
  \item[(2)] Recognition problem: when is the unit of the Tannakian biadjunction an equivalence?
 \end{enumerate}

 Note that the asymmetry between these two problems is only apparent: a morphism of coalgebroids is an equivalence if and only if it is an isomorphism, because there is no notion of natural transformation between morphisms of coalgebroids.
 
 We prove a necessary and sufficient condition for the counit to be an isomorphism, and a sufficient condition for the unit to be an equivalence, both for arbitrary cosmoi as a base. However, the conditions simplify considerably for the cosmos $\Mod_R$ of modules of a commutative ring $R$. The Tannakian biadjunction relevant for this case is the biadjunction
 \[
  \xymatrix{B\mbox{-}B\mbox{-}\Coalg \rrtwocell<5>^{L(-)}_{\Rep(-)}{`\perp} && R\mbox{-}\Cat \slash \Mod_B}
 \]
 where $B$ denotes a fixed $R$-algebra, and $R\mbox{-}\Cat \slash \Mod_B$ denotes the 2-category whose objects are $R$-linear functors with codomain $\Mod_B$. The neutral Tannakian biadjunction is obtained by taking $B=R$.
 
 The classical reconstruction result relies on the fact that every comodule of a Hopf algebroid can be written as a union of finite dimensional comodules. A union is a special case of a colimit, and we arrive at a necessary and sufficient condition for reconstruction if we replace inclusions by arbitrary maps. We start by giving a description of the relevant diagram.
 
 Given a subcategory $\ca{A} \subseteq \ca{C}$ and an object $C \in \ca{C}$, we write $\ca{A} \slash C$ for the \emph{category of $\ca{A}$-objects over $C$}. The objects of $\ca{A} \slash C$ are morphisms $ \phi \colon A \rightarrow C$ in $\ca{C}$ whose domain lies in $\ca{A}$, and the morphisms between $\phi \colon A \rightarrow C$ and $\phi^{\prime} \colon A^{\prime} \rightarrow C$ are the morphisms $A \rightarrow A^{\prime}$ which make the evident triangle commute. The \emph{domain functor} $D \colon \ca{A}\slash C \rightarrow \ca{C}$ is the functor which sends an object $\phi$ to its domain. The \emph{tautological cocone} on $D$ is the cocone with vertex $C$ whose component at the object $\phi \colon A \rightarrow C$ is $\phi$ itself, thought of as a morphism from $D(\phi)$ to $C$.

 The following theorem solves the reconstruction problem in the neutral case. It is a consequence of Theorem~\ref{DENSE_RECONSTRUCTION_THM}.
 
 \begin{thm}\label{NEUTRAL_RECONSTRUCTION_PROBLEM}
 Let $R$ be a commutative ring, and let $G=\Spec(H)$ be an affine group scheme over $\Spec(R)$. Then the $H$-component of the counit of the Tannakian adjunction is an isomorphism if and only if the tautological cocone exhibits $H$, considered as a comodule over itself, as the colimit of the diagram
 \[
  D \colon \Rep(H) \slash H \rightarrow \Comod(H)  
 \]
 of Cauchy comodules over $H$.
\end{thm}
 
 For example, sufficient conditions for this to hold are that $H$ is flat and the category of Cauchy comodules forms a generator of the category of all comodules (cf.\ Corollary~\ref{ENOUGH_CAUCHY_COMODULES_COR}). As far as the author knows, it is an open question whether or not there are flat Hopf algebras such that the Cauchy comodules do \emph{not} form a generator of $\Comod(H)$.

 The classical recognition result concerns \emph{exact} $k$-linear functors. It turns out that we have to generalize left and right exactness separately. Right exactness concerns the preservation of cokernels, and the category of finitely generated projective modules is usually not closed under cokernels. We can easily deal with this situation by restricting our attention to those morphisms whose cokernel happens to be finitely generated projective.
 
 The generalization of left exactness is more subtle. Since the category of finitely generated projective modules is not closed under kernels, we can't expect that the domain category of a fiber functor has kernels. A \emph{flat} functor is a generalization of a left exact functor which makes sense for any domain category. In order to define flat functors we need to introduce the \emph{category of elements} of a functor $w \colon \ca{A} \rightarrow \Set$. This category is denoted by $\el(w)$, and its objects are pairs $(A,x)$, where $A \in \ca{A}$ and $x\in w(A)$. The morphisms $(A,x) \rightarrow (A^\prime, x^\prime)$ are given by the morphisms $f\colon A \rightarrow A^\prime$ in $\ca{A}$ with $w(f)(x)=x^\prime$. If $\ca{A}$ has finite limits, then a functor $w \colon \ca{A} \rightarrow \Set$ preserves finite limits if and only if the category $\el(w)$ is \emph{cofiltered}, that is, if and only if
\begin{itemize}
 \item the category $\el(w)$ is nonempty;
 \item for any two objects $(A,x),(A^\prime,x^{\prime}) \in \el(w)$, there is an object $(B,y) \in \el(w)$ together with morphisms $(B,y) \rightarrow (A,x)$, $(B,y) \rightarrow (A^\prime,x^{\prime})$; and
 \item for any two morphisms $f,g \colon (A,x) \rightarrow (A^\prime,x^{\prime})$ in $\el(w)$, there is an object $(B,y)$ and a morphism $h\colon (B,y) \rightarrow (A,x)$ such that $fh=gh$.
\end{itemize}

\begin{dfn}
 A functor $w \colon \ca{A} \rightarrow \Set$ (where $\ca{A}$ does not necessarily have finite limits) is called \emph{flat} if the category of elements of $w$ is cofiltered. If a functor lands in the category of modules of a commutative ring, then we call it flat if the composite with the evident forgetful functor to $\Set$ is flat.
 \end{dfn}

 In order to state our recognition theorem we have to explain one more point of terminology. There are basically two perspectives one can take on enriched categories. One point of view is that an enriched category is first and foremost an ordinary category, endowed with further structure that makes it enriched. This is very natural for $R$-linear or topological categories, for example, where the enrichment consists of additional structure on the hom-sets. The second point of view is that an enriched category consists of a class of objects, together with a \emph{hom-object} for any pair of objects. This hom-object is itself an object of some base category $\ca{V}$. The underlying unenriched category is then \emph{constructed} from this data by applying a canonical forgetful functor $\ca{V}\rightarrow \Set$. This point of view is more natural when the base is the category of differential graded $R$-modules. The canonical forgetful functor sends a differential graded module to the set of cycles of degree zero, which is not an ``underlying set.''
 
 Moreover, there are very natural base change functors which---when applied to each hom-object---change the underlying unenriched category significantly. Therefore we adopt the second point of view. In particular, Given an $R$-linear functor $F \colon \ca{A} \rightarrow \ca{B}$, we write $F_0 \colon \ca{A}_0 \rightarrow \ca{B}_0$ for the underlying functor between the underlying unenriched categories.
 
 The following theorem, which is part of Theorem~\ref{TANNAKA_AFFINE_GROUPOID_THM}, is proved by showing that the unit of the Tannakian adjunction is an equivalence under the stated assumptions.
 
 \begin{thm}
 Let $B$ be a commutative $R$-algebra, let $\ca{A}$ be an additive autonomous symmetric monoidal $R$-linear category, and let $w \colon \ca{A} \rightarrow \Mod_B$ be a symmetric strong monoidal $R$-linear functor. Suppose that
 \begin{enumerate}
\item[i)] 
 the functor $w_0$ is faithful and reflects isomorphisms;
\item[ii)] 
 the functor $w_0$ is flat, that is, the category $\el(w_0)$ of elements of $w_0$ is cofiltered;
\item[iii)]
 if the cokernel of $w_0(f)$ is finitely generated and projective, then the cokernel of $f$ exists and is preserved by $w_0$.
 \end{enumerate}
 Then there exists an affine groupoid $G=\Spec(H)$ acting on $\Spec(B)$ and a symmetric strong monoidal equivalence $\ca{A} \simeq \Rep(G)$. This equivalence is compatible with $w$ and the forgetful functor. Moreover, $H$ is flat as a left and as a right $B$-module.
\end{thm}

 Note that the forgetful functor from the category of representations of any affine groupoid scheme acting on $\Spec(B)$ satisfies i) and iii), that is, i) and iii) are necessary conditions. On the other hand, condition ii) is clearly stronger: it implies that $H$ is flat as a left and as a right $B$-module. It is an open question whether or not the converse is true: if $H$ is flat as a left and as a right $B$-module, is the forgetful functor $\Rep(G)\rightarrow \Mod_B$ flat?

\subsection{Outline}\label{PAPER_OUTLINE_SECTION}

 The proof of our Tannakian theorem is split into three parts. In the first part we will set up the categorical framework for dealing with the reconstruction problem and the recognition problem. More precisely, we give a new construction of the Tannakian biadjunction for cosmoi, which we summarize in Section~\ref{OUTLINE_TK_BIADJUNCTION_SECTION}. Instead of proving the existence of this biadjunction directly, we show that it is a special case of a ``formal'' Tannakian biadjunction for 2-categories. The construction of the latter closely mimics the construction of the semantics-structure adjunction in Street's ``formal theory of monads'' \cite{STREET_FTM}, and we shall later use a comparison between the two to prove our recognition results.
 
 The name ``formal category theory'' is sometimes used when we think of the objects of a 2-category as a generalized notion of category, that is, when we forget about the fact that our (structured) categories have objects and morphisms, and think of them as primitive objects in a surrounding 2-category.
 
 Specifically, the notion of a monad makes sense in any 2-category, and Ross Street observed that the category of Eilenberg-Moore algebras of a monad has a universal property in the 2-category of categories. In a general 2-category, objects with the corresponding universal property are called \emph{Eilenberg-Moore objects}. A large part of the theory of monads can be developed if we assume that Eilenberg-Moore objects exist, for example, the semantics-structure adjunction.
 
 The category of finitely generated comodules of a coalgebra has a similar universal property in the bicategory of \emph{$\ca{V}$-modules}\footnote{Modules are also known as \emph{distributors} (see \cite{BENABOU}) or \emph{profunctors}.}. In a general bicategory, we call objects with this universal property \emph{Tannaka-Krein objects}. In Section~\ref{TK_ADJ_2CAT_SECTION} we will see that Tannaka-Krein objects can be used to construct the formal Tannakian biadjunction for 2-categories. In Section~\ref{TANNAKIAN_BIADJUNCTION_MODV_SECTION} we show that for any cosmos $\ca{V}$, the 2-category of $\ca{V}$-modules has Tannaka-Krein objects and that the Tannakian biadjunction for cosmoi is a special case of the formal Tannakian biadjunction.
 
 In the second part of the paper we study the reconstruction and recognition problems in the bicategory of $\ca{V}$-modules. We give criteria which ensure that the unit is an equivalence, and necessary and sufficient conditions for the counit to be an isomorphism. It is well known when the unit of the semantics-structure adjunction is an equivalence (Beck's monadicity theorem), and from the construction of Tannaka-Krein objects in the bicategory of $\ca{V}$-modules it will be clear that this is crucial for understanding the unit of the Tannakian adjunction. Note that Beck's monadicity theorem was also a key ingredient in the proof of \cite[Th\'eor\`eme~1.12]{DELIGNE}. The general recognition result (proved in Section~\ref{RECOGNITION_SECTION}) allows for considerable simplifications if we make some assumptions on the cosmos $\ca{V}$. We prove some of these specialized recognition results in Sections~\ref{DAG_RECOGNITION_SECTION} and \ref{ABELIAN_RECOGNITION_SECTION}. They can be applied to cosmoi of $R$-modules and differential graded $R$-modules for a commutative ring $R$.

 In the third part of the paper we expand our categorical framework to entail (symmetric) monoidal structures and (commutative) bialgebras and bialgebroids. To do this we investigate the interaction between the Tannakian biadjunction and the monoidal structures on its domain and codomain. More precisely, we will show that the Tannakian adjunction is a monoidal biadjunction. As a consequence we find that it lifts to categories of (weak) monoids on either side. This provides a conceptual explanation of the fact that monoidal structures induce a bialgebra structure on the associated coalgebra. We proceed to show compatibilities with braidings and symmetries, which allow us to lift the Tannakian biadjunction to commutative bialgebras. In particular, the recognition results from Sections~\ref{RECOGNITION_SECTION}, \ref{DAG_RECOGNITION_SECTION} and \ref{ABELIAN_RECOGNITION_SECTION} lift to the setting of (symmetric) bialgebras and bialgebroids.
 
 In order not to lose ourselves in technicalities we provide an overview of the main argument and a summary of the compatibility results in Section~\ref{BIALGEBRAS_HOPF_ALGEBRAS_SECTION}, and defer their proofs to Section~\ref{TANNAKIAN_BIADJUNCTION_GRAY_SECTION}. Each one of them entails checking that a considerable number of axioms hold, each of which is straightforward to check if we use a convenient notation for 2-cells in a monoidal 2-category.
 
 To prove recognition result for categories of representations of affine groupoids in Section~\ref{AFFINE_GROUPOIDS} we also need to study the interaction between duals and antipodes. We do this using the notion of Hopf monoidal comonads in Section~\ref{HOPF_MONOIDAL_COMONADS_SECTION}.

 In Appendix~\ref{QUANTUM_BIALGEBRAS_SECTION} we outline how our theory can be extended to deal with dual quasi-bialgebras and dual quasi-triangular quasi-bialgebras.
 
 Throughout the paper we will talk about categories enriched in a cosmos $\ca{V}$. The standard source for these is \cite{KELLY_BASIC}. We provide the necessary background material whenever it is needed.
 
 We also frequently use 2-categories and bicategories. The former are precisely the $\ca{V}$-categories for the cosmos $\ca{V}=\Cat$ of small categories. For $A$, $B$ objects of a 2-category $\ca{K}$ we call the objects of $\ca{K}(A,B)$ the \emph{1-cells} of $\ca{K}$, and the morphisms of $\ca{K}(A,B)$ are called \emph{2-cells}. The \emph{0-cells} of $\ca{K}$ are by definition the objects of $\ca{K}$. For example, in the 2-category of $R$-linear categories, the 0-cells are small $R$-linear categories, the 1-cells are $R$-linear functors, and the 2-cells are natural transformations.
 
 A \emph{bicategory} is a weakened form of a 2-category, where composition of 1-cells is only associative up to coherent invertible 2-cells. The composition of 1-cells in the examples of bicategories we consider mostly arise from some form of tensor product, which is only associative up to canonical ismorphism. A nice exposition of the theory of 2-categories and bicategories can be found Steve Lack's ``2-categories companion'' \cite{LACK_COMPANION}.

\section*{Acknowledgments}
 This paper contains generalizations of results of my master's thesis, which was written under the advice of Prof.\ Richard Pink at ETH Z\"urich, Switzerland. I thank Richard Garner, Peter May, Richard Pink, Mike Shulman and Ross Street for kindly answering questions and for giving suggestions for improvement. I am especially grateful to Mike Shulman and Richard Garner, who both pointed out significant simplifications of some of my original proofs, and to Peter May, for his help with the organization and exposition. I thank Emily Riehl and Claire Tomesch for their help with editing an earlier draft of this paper.

\tableofcontents
\section{The category of filtered modules}\label{F_MODULES_SECTION}
\subsection{Filtered \texorpdfstring{$F$}{F}-modules}
 We can apply the generalized theory of Tannakian duality to the category of filtered modules introduced by Fontaine and Laffaille in \cite{FONTAINE_LAFFAILLE}. Fix a perfect field $k$ of characteristic $p>0$, and let $W$ be the ring of Witt vectors with coefficients in $k$. For our purposes it suffices to know that $W$ is a discrete valuation ring with residue field $k$ which contains the ring of $p$-adic integers $\mathbb{Z}_p$, and that $p \in \mathbb{Z}_p$ is a uniformizer of $W$. A construction of the ring can be found in \cite[\S~II.6]{SERRE}. There is an automorphism $\sigma \colon W \rightarrow W$ of $\mathbb{Z}_p$-algebras which lifts the Frobenius automorphism on the residue field $k$ of $W$ (see \cite[Th\'eor\`eme~II.7 and Proposition~II.10]{SERRE}). This automorphism $\sigma$ is again called the Frobenius automorphism. For a $W$-module $M$, we write $M_\sigma$ for the $W$-module obtained by base change along $\sigma$. In the following definition we use the same notation and terminology that was introduced in \cite{WINTENBERGER}. We write $W_n$ for the quotient ring $W/p^n W$.

\begin{dfn}
 A \emph{filtered $F$-Module}\footnote{A filtered $F$-module is a filtered $W$-module with additional structure; the $F$ is part of the name and does not stand for a ring.} consists of
\begin{itemize}
\item
 a $W$-module $M$ with a decreasing filtration $(\Fil^i M)_{i \in \mathbb{Z}}$ of submodules $\Fil^i M \subseteq M$. The filtration is \emph{exhaustive}, $\bigcup_{i\in \mathbb{Z}} \Fil^i M = M$, and \emph{separated}, $\bigcap_{i\in \mathbb{Z}} \Fil^i M =0$;
\item
 for each $i\in \mathbb{Z}$, a morphism $\phi^i \colon \Fil^i M \rightarrow M_\sigma$ of $W$-modules such that the restriction of $\phi^i$ to $\Fil^{i+1} M$ is $p\phi^{i+1}$.
\end{itemize}
 A morphism of filtered $F$-modules $M \rightarrow M^\prime$ is a morphism $g\colon M\rightarrow M^\prime$ of $W$-modules such that for all $i\in \mathbb{Z}$, $g(\Fil^i M ) \subseteq \Fil^i M^\prime$ and $\phi^i_{M^{\prime}} \circ g = g \circ \phi^i_M$. We denote the category of filtered $F$-modules by $\MF$, and we write $\MF_{fl}$ for the full subcategory of objects $M$ which satisfy
\begin{itemize}
 \item
 the $W$-module $M$ has finite length;
 \item
 the images of the $\phi^i$ span $M$, that is, $\sum_{i\in \mathbb{Z}} \phi^i(\Fil^i M) =M$.
\end{itemize}
The category $\MF^n_{fl}$ is the full subcategory of $\MF_{fl}$ consisting of those objects whose underlying $W$-module $M$ is annihilated by $p^n$ (equivalently, for which $M$ is a $W_n$-module). We write $\MF^n_{\proj}$ for the full subcategory of $\MF^n_{fl}$ consisting of objects whose underlying module is a finitely generated projective $W_n$-module.
\end{dfn}

 J.-M.\ Fontaine and G.\ Laffaille have shown that $\MF_{fl}$ is an abelian $\mathbb{Z}_p$-linear category, and that the forgetful functor $w \colon \MF_{fl} \rightarrow \Mod_W$ is an exact $\mathbb{Z}_p$-linear functor \cite{FONTAINE_LAFFAILLE}. It follows immediately that $\MF^n_{fl}$ is an abelian $\mathbb{Z}/p^n\mathbb{Z}$-linear category, and that $w$ restricts to a $\mathbb{Z}/p^n \mathbb{Z}$-linear functor $w \colon \MF^n_{fl} \rightarrow \Mod_{W_n}$. Thus, $\MF^n_{\proj}$ is $\mathbb{Z}/p^n \mathbb{Z}$-linear, and we can further restrict $w$ to a functor on $\MF^n_{\proj}$ whose image is contained in the category of finitely generated projective $W_n$-modules.

 In order to prove that $\MF^n_{\proj}$ is the category of of representations of an affine groupoid we need to introduce one more auxiliary category. We write $\MF_{fg}$ for the full subcategory of $\MF$ of filtered $F$-modules $M$ which satisfy
\begin{itemize}
 \item
 the $W$-module $M$ is finitely generated;
 \item
 the modules $\Fil^i M$ are direct summands of $M$;
 \item
 the images of the $\phi^i$ span $M$, that is, $\sum_{i\in \mathbb{Z}} \phi^i(\Fil^i M) =M$.
\end{itemize}
 The following proposition was proved by J.-P.\ Wintenberger in \cite{WINTENBERGER}. It shows in particular that we have a sequence $\MF^n_{\proj} \subseteq \MF^n_{fl} \subseteq \MF_{fl} \subseteq MF_{fg}$ of full subcategories.

\begin{prop} \label{F_MODULE_COFILTERED_PROP}
 The category of filtered $F$-modules has the following properties.
\begin{enumerate}
 \item[i)]
 The category $\MF_{fg}$ is abelian, and the forgetful functor $w \colon \MF_{fg} \rightarrow \Mod_W$ is an exact $\mathbb{Z}_p$-linear functor.
 \item [ii)]
 For any object $M$ of $\MF_{fl}$, the filtration by submodules consists of direct summands. Thus $\MF_{fl}$ can be identified with the full subcategory of $\MF_{fg}$ consisting of objects which are annihilated by some power of $p$.
 \item[iii)]
 For any object $M$ of $\MF_{fg}$ there exists an object $M^\prime$ of $\MF_{fg}$ and an epimorphism $g \colon M^{\prime} \rightarrow M$ such that the underlying $W$-module of $M^\prime$ is free.
\end{enumerate}
\end{prop}

\begin{proof}
 Parts i) and ii) are proved in \cite[Proposition~1.4.1]{WINTENBERGER}, and part iii) is \cite[Proposition~1.6.3]{WINTENBERGER}.
\end{proof}

\subsection{Autonomous symmetric monoidal structure}
 In \cite[\S~1.7]{WINTENBERGER} it was shown that the category $\MF_{fg}$ is endowed with a $\mathbb{Z}_p$-linear tensor product which turns $\MF_{fg}$ into a closed symmetric monoidal $\mathbb{Z}_p$-linear category, and that an object in this category is dualizable if and only if its underlying $W$-module is torsion free. Moreover, this tensor product is lifted from $\Mod_W$ in the sense that the forgetful functor $w \colon \MF_{fg} \rightarrow \Mod_W$ is a strong symmetric monoidal $\mathbb{Z}_p$-linear functor. By taking reduction mod $p^n$ it follows immediately that $\MF^n_{\proj}$ is a symmetric monoidal $\mathbb{Z}\slash p^n\mathbb{Z}$-linear category, that $w \colon \MF^n_{\proj} \rightarrow \Mod_{W_n}$ is a symmetric monoidal $\mathbb{Z}\slash p^n\mathbb{Z}$-linear functor, and that every object of $\MF^n_{\proj}$ has a dual. To see this last fact one can use the observation that every object of $\MF^n_{fl}$ is obtained by reduction mod $p^n$ from a torsion free object of $\MF_{fg}$, which is immediate from part iii) of Proposition~\ref{F_MODULE_COFILTERED_PROP}.

\begin{thm} \label{FILTERED_MODULE_THEOREM}
 There is a groupoid $G=\Spec(L_n)$ acting on $\Spec(W_n)$ and a symmetric strong monoidal equivalence $\MF^n_{\proj} \simeq \Rep(G)$. The Hopf algebroid $L_n$ is given by
\[
 L_n = \int^{M \in \MF^n_{\proj}} w(M) \otimes_{\mathbb{Z}/p^n \mathbb{Z}} w(M)^\vee\rlap{,}
\]
 where the right action on $L_n$ is induced by the $W_n$-actions on $w(M)^\vee$, and the left action is induced by the $W_n$-actions on $w(M)$. The $W_n$-$W_n$-coalgebra $L_n$ is flat as a right and as a left $W_n$-module.
\end{thm}

\begin{proof}
 We show that the conditions of Theorem~\ref{TANNAKA_AFFINE_GROUPOID_THM} are satisfied for $R=\mathbb{Z}/p^n \mathbb{Z}$ and $B=W_n$. Conditions i) and iii) are immediate from the fact that we have an embedding $\MF^n_{\proj} \rightarrow \MF^n_{fl}$ where $\MF^n_{fl}$ is abelian, together with an extension of the forgetful functor to an exact faithful functor $\MF^n_{fl} \rightarrow \Mod_{W_n}$.

 It remains to to check that $\el(w)$ is cofiltered. Since $\MF^n_{\proj}$ has direct sums, it suffices to check that for any pair of morphisms $f,g \colon (M,x) \rightarrow (M^\prime,x^\prime)$ in $\el(w)$, there is an object $(N,y)$ in $\el(w)$ and a morphism $h \colon (N,y) \rightarrow (M,x)$ such that $fh=gh$. Let $K$ be the equalizer of $f,g$ in $\MF^n_{fl}$. We have $f(x)=x^\prime=g(x)$, so $x \in K$. From Proposition~\ref{F_MODULE_COFILTERED_PROP}, part iii) we know that there is an object $L$ of $\MF_{fg}$ with an epimorphism $k \colon L \rightarrow K$ such that the underlying $W$-module of $L$ is free. Multiplication with $p^n$ defines an endomorphism of $L$ in $\MF_{fg}$. The cokernel $N$ of this endomorphism is a free $W_n$-module. Since $K$ is an object of $\MF^n_{fl}$, it is annihilated by $p^n$, so we get a morphism $h\colon N \rightarrow K$ in $\MF_{fg}$ making the diagram
\[
 \xymatrix{L \ar[r]^{p^n} \ar[rd]_{0} & L \ar[r] \ar[d]^{k} & N \ar[ld]^{h} \\
& K }
\]
 commutative. Surjectivity of the morphism $k$ implies that $h$ is surjective. In particular, there is an element $y \in N$ with $h(y)=x$. Since $N$ is a finitely generated free $W_n$-module, this gives the desired morphism $h \colon (N,y) \rightarrow (M,x)$ in the category of elements of $w \colon \MF^n_{\proj} \rightarrow \Mod_{W_n}$.
\end{proof}

\section{Outline of the Tannakian biadjunction}
\label{OUTLINE_TK_BIADJUNCTION_SECTION}

 \subsection{The Tannakian biadjunction} 
 Let $\ca{V}$ be a cosmos. At the heart of our work is the Tannakian biadjunction between coalgebras and coalgebroids on the one hand, and fiber functors on the other. In this section we will recall some standard terminology and definitions from enriched category theory which are used in the following theorem.
 
 \begin{thm}\label{TANNAKIAN_BIADJUNCTION_VMOD_THM}
 Let $\ca{B}$ be a $\ca{V}$-category. Let
 \[
  R \colon \Comon(\ca{B}) \rightarrow \Vcat \slash \overline{\ca{B}}
 \]
 be the 2-functor which on objects sends a comonad to the forgetful functor from its $\ca{V}$-category of Cauchy comodules to the Cauchy completion $\overline{\ca{B}}$ of $\ca{B}$. If the category $\Comon(\ca{B})$ is regarded as a 2-category with only identity 2-cells, then $R$ has a left biadjoint $L$.
 \end{thm}

\subsection{Recollection about enriched category theory}\label{ENRICHED_CATEGORIES_SECTION}
 We denote the tensor product of $\ca{V}$ by $-\otimes - \colon \ca{V}\times \ca{V} \rightarrow \ca{V} $, the unit object by $I$ and the internal hom by $[-,-]$. A category $\ca{A}$ \emph{enriched} in $\ca{V}$ has objects $a, a^\prime, \ldots$ and instead of hom-sets, it has \emph{hom-objects} $\ca{A}(a,a^\prime) \in \ca{V}$, see \cite[\S~1]{KELLY_BASIC}. The basic concepts of category theory can be generalized to this context. For example, for a small $\ca{V}$-category $\ca{A}$, there is a $\ca{V}$-category $\Prs{A}$ of enriched presheaves on $\ca{A}$ (that is, $\ca{V}$-functors $\ca{A}^{\op} \rightarrow \ca{V}$), and a corresponding Yoneda embedding. We denote the category of small $\ca{V}$-categories and $\ca{V}$-functors by $\Vcat$, and we write $\VCAT$ for the (very large) 2-category of all large $\ca{V}$-categories and $\ca{V}$-functors. The reader who is unfamiliar with the general theory of enriched categories should keep in mind the case $\ca{V}=\Mod_R$, $R$ a commutative ring, where $\ca{V}$-category, $\ca{V}$-functor and $\ca{V}$-natural transformation correspond to the notions of $R$-linear category, $R$-linear functor and ordinary natural transformation respectively. Note that we do not require that an $R$-linear category has finite direct sums. Most of the general concepts are self-explanatory in this context.

\begin{dfn}
 The \emph{unit $\ca{V}$-category} is the $\ca{V}$-category $\ca{I}$ with a single object $\ast$ and $\ca{I}(\ast,\ast)=I$.
\end{dfn}

\subsection{The bicategory of modules}
Apart from the notion of $\ca{V}$-functor, there is an alternative choice of morphism between $\ca{V}$-categories. Understanding both of these and the interaction between the two is crucial for our construction of the Tannakian adjunction.

\begin{dfn}\label{MODULE_DFN}
 Let $\ca{V}$ be a cosmos, and let $\ca{A}$ and $\ca{B}$ be $\ca{V}$-categories. A \emph{module}\footnote{Modules are also known as \emph{bimodules}, \emph{distributors}, or \emph{profunctors}.} $M \colon \ca{A} \xslashedrightarrow{} \ca{B}$ is a $\ca{V}$-functor $\ca{B}^{\op}\otimes \ca{A} \rightarrow \ca{V}$ and composition of modules $M \colon \ca{A} \xslashedrightarrow{} \ca{B}$ and $N \colon \ca{B} \xslashedrightarrow{} \ca{C}$ is denoted by $N\odot M$ and specified by the coend
\[
N\odot M(c,a) \defl \int^{b \in \ca{B}} N(c,b)\otimes M(b,a)\rlap{.}
\]
 This composition is associative up to coherent isomorphism and the representable modules are identities up to isomorphism by a form of the Yoneda lemma (see \cite[Formula~(3.71)]{KELLY_BASIC}). We get a bicategory $\Mod(\ca{V})$ with 0-cells the small $\ca{V}$-categories, 1-cells the modules and 2-cells the $\ca{V}$-natural transformations between them.
\end{dfn}

\begin{example}
An algebra $B$ in $\ca{V}$ can be interpreted as a $\ca{V}$-category with precisely one object. A $\ca{V}$-functor $B \rightarrow \ca{V}$ picks out an object of $\ca{V}$ together with an action of $B$, that is, the category of $\ca{V}$-functors $B \rightarrow \ca{V}$ is equivalent to the category of $B$-modules. A module $B \xslashedrightarrow{} B$ in the sense of Definition~\ref{MODULE_DFN} is precisely a $B$-$B$-bimodule. Moreover, composition of modules is given by tensoring the corresponding bimodules over $B$. 
\end{example}

\subsection{The category \texorpdfstring{$\Comon(\ca{B})$}{\textbf{Comon}(B)}}
We are now ready to define one of the categories appearing in the Tannakian biadjunction.

\begin{dfn}
 Let $\ca{B}$ be a $\ca{V}$-category. A comonad on $\ca{B}$ is a coalgebra in the category $\Mod(\ca{V})(\ca{B},\ca{B})$ of endomodules of $\ca{B}$, whose monoidal structure is given by composition of modules. The category of comonads is denoted by $\Comon(\ca{B})$.
\end{dfn}

Above we have seen that an endomodule of a $\ca{V}$-category $B$ with a single object (that is, an algebra in $\ca{V}$) is precisely a $B$-$B$-bimodule. Thus a comonad on $B$ is precisely a $B$-$B$-coalgebroid. In particular, if $B=I$, it is simply a coalgebra in $\ca{V}$.

\subsection{Cauchy completion and fiber functors}\label{CAUCHY_COMPLETION_SECTION}
 We can now define Cauchy objects and Cauchy completions of $\ca{V}$-categories. Let $\ca{B}$ be a small $\ca{V}$-category. An object $F\in \Prs{B}$ is called a \emph{Cauchy object} if the representable functor $\Prs{B}(F,-)$ is cocontinuous (in the enriched sense, see Section~\ref{WEIGHTED_COLIMIT_SECTION}). The \emph{Cauchy completion} $\overline{\ca{B}}$ of $\ca{B}$ is the full subcategory of Cauchy objects in $\Prs{B}$. If $F$ is represented by the object $B\in \ca{B}$, then the Yoneda lemma implies that $\Prs{B}(F,-)$ is isomorphic to the functor which evaluates a presheaf at $B$. Since colimits in presheaf categories are computed pointwise (see \cite[Section~3.3]{KELLY_BASIC}), it follows that $\overline{\ca{B}}$ contains all the representable functors, that is, we have $\ca{B} \subseteq \overline{\ca{B}}$.
 
 Cauchy completions are best explained by giving a few examples. 
 
\begin{example}\label{CAUCHY_COMPLETION_SMALL_PROJECTIVE_EXAMPLE}
 Let $\ca{V}=\Mod_R$ be the cosmos of $R$-modules for some commutative ring $R$, and let $B$ be an $R$-algebra, considered as a one object $\ca{V}$-category $\ca{B}$. The presheaf category $\Prs{B}$ is isomorphic to the $R$-linear category of right $B$-modules. A $B$-module $M$ is a Cauchy object if and only if it is finitely generated and projective. For this reason, Cauchy objects in an arbitrary cosmos are sometimes called \emph{small projective} objects.  
\end{example}
 
\begin{rmk}
 The name ``Cauchy completion'' comes from a different example due to F.\ W.\ Lawvere. Let $\ca{V}$ be the cosmos $[0,\infty]$ of extended nonnegative real numbers. For $x,y$ objects of $[0,\infty]$, there is a unique morphism $x \rightarrow y$ if and only if $x\geq y$, and the tensor product is given by addition of real numbers. A $\ca{V}$-category is a (generalized) metric space. Any ordinary metric space $X$ gives an example of a $[0,\infty]$-category, and the Cauchy completion as a $[0,\infty]$-category coincides with the usual Cauchy completion of $X$ as a metric space (see \cite{LAWVERE_METRIC}).
\end{rmk}
 
\begin{example}
 If $\ca{V}=\Set$, then a small $\ca{V}$-category $\ca{B}$ is just a small ordinary category, and $\overline{\ca{B}}$ is the \emph{Karoubi envelope} of $\ca{B}$, which is the universal category containing $\ca{B}$ in which all idempotents split.
\end{example}
 
\begin{example}
 An important example of a Cauchy completion which works in any cosmos is the following. For $\ca{B}=\ca{I}$, the unit $\ca{V}$-category, we have $\Prs{B}\simeq \ca{V}$. The representable functors $\Prs{B}(X,-)$ correspond to $[X,-]$ under this equivalence. Since $X$ has a dual if and only if the internal hom-functor is cocontinuous we conclude that the Cauchy completion of $\ca{I}$ is equivalent to the full subcategory of $\ca{V}$ consisting of objects with duals.
\end{example}

 We can now give a precise definition of the codomain of the Tannakian biadjunction.

\begin{dfn}
 Let $\ca{V}$ be a small $\ca{V}$-category. The 2-category $\Vcat\slash \overline{\ca{B}}$ has objects the $\ca{V}$-functors with codomain $\overline{\ca{B}}$ (and small domain). A 1-cell $(\ca{A},w)\rightarrow (\ca{A}^{\prime},w^{\prime})$ is a pair $(s,\sigma)$ of a $\ca{V}$-functor $s \colon \ca{A} \rightarrow \ca{A}^{\prime}$ and a $\ca{V}$-natural isomorphism $\sigma \colon w^{\prime} \cdot s \Rightarrow w$. The 2-cells $(s,\sigma) \Rightarrow (t,\tau)$ are $\ca{V}$-natural transformations $s \Rightarrow t$ such that the equation
 \[
  \vcenter{\hbox{
  \def\objectstyle{\scriptstyle}
  \def\labelstyle{\scriptscriptstyle}
\def\twocellstyle{\scriptscriptstyle}
  \xymatrix{ \ca{A} \xtwocell[0,2]{}\omit{<3>\sigma} \ar[rr]^{s} \ar[rd]_{w} & & \ca{A}^{\prime} \ar[ld]^{w^{\prime}} \\
  & \overline{\ca{B}}
  }
  }}=
  \vcenter{\hbox{
  \def\objectstyle{\scriptstyle}
  \def\labelstyle{\scriptscriptstyle}
\def\twocellstyle{\scriptscriptstyle}
  \xymatrix{ \ca{A} \xtwocell[0,2]{}\omit{<4>\beta} \rrtwocell^{s}_{*!<10pt,0pt>+{t}}{\tau} \ar[rd]_{w} & & \ca{A}^{\prime} \ar[ld]^{w^{\prime}} \\
  & \overline{\ca{B}}
  }
  }}
 \]
 holds.
\end{dfn}

 Let $R$ be a commutative ring, and let $\ca{V}=\Mod_R$. From Example~\ref{CAUCHY_COMPLETION_SMALL_PROJECTIVE_EXAMPLE} we know that the Cauchy completion of an $R$-algebra $B$ is the category of finitely generated projective right $B$-modules. An object of $\Vcat \slash \overline{B}$ is precisely an $R$-linear functor $w \colon \ca{A} \rightarrow \Mod^{\fgp}_B$. Thus the fiber functors considered in \cite{DELIGNE} are present in the 2-category $\Vcat \slash \overline{B}$.

\subsection{The right biadjoint}\label{RIGHT_BIADJOINT_SUBSECTION}
 Let $\ca{B}$ be a small $\ca{V}$-category, and let $C$ be a comonad on $\ca{B}$. A \emph{comodule} of $C$ is a module $M \colon \ca{I} \xslashedrightarrow{} \ca{B}$ (that is, a presheaf on $\ca{B}$) together with a coaction
 \[
 \rho \colon M \rightarrow C\odot M
 \]
 which is coassociative and counital. A \emph{Cauchy comodule} is a comodule $(M,\rho)$ such that $M$, considered as a presheaf on $\ca{B}$, lies in the Cauchy completion of $\ca{B}$. The category of Cauchy comodules is denoted by $\Rep(C)$.
 
\begin{dfn}
 Let $\ca{B}$ be a small $\ca{V}$-category. The 2-functor
 \[
  R \colon \Comon(B) \rightarrow \Vcat\slash \overline{\ca{B}}
 \]
 sends a comonad $C$ to the forgetful functor $\Rep(C) \rightarrow \overline{\ca{B}}$.
\end{dfn}

\begin{example}
 Let $R$ be a commutative ring, $\ca{V}=\Mod_R$, let $B$ be an $R$-algebra, and let $C$ be a comonad on $B$. Then $C$ is a coalgebroid acting on $B$, and $\Rep(C)$ is the category of $C$-comodules whose underlying $B$-module is finitely generated and projective.
\end{example}

\subsection{Monoidal structure}
 Let $\ca{B}$ be a monoidal $\ca{V}$-category, for example, a commutative algebra in $\ca{V}$. Then both the category $\Comon(\ca{B})$ and the 2-category $\Vcat \slash \overline{\ca{B}}$ inherit a monoidal structure. A weak monoid in $\Vcat\slash \overline{\ca{B}}$ is a small monoidal $\ca{V}$-category equipped with a strong monoidal $\ca{V}$-functor to $\overline{\ca{B}}$. In Section~\ref{BIALGEBRAS_HOPF_ALGEBRAS_SECTION} we will see that the Tannakian adjunction is compatible with these monoidal structures. In particular, the left adjoint sends (weak) monoids to monoids (generalized bialgebroids). This provides a conceptual explanation for the fact that the coalgebra associated to a strong monoidal fiber functor inherits a bialgebra structure.

\section{The Tannakian biadjunction for general 2-categories}\label{TK_ADJ_2CAT_SECTION}

\subsection{Outline}
We begin by stating a theorem of which Theorem~\ref{TANNAKIAN_BIADJUNCTION_VMOD_THM} is a special case. Let $\ca{M}$ be a 2-category, that is, a category enriched in $\Cat$. We shall gradually define terms to make sense of and prove the following theorem.
 
\begin{thm}[The Tannakian biadjunction]\label{TANNAKIAN_BIADJUNCTION_2CAT_THM}
 Let $B$ be an object of $\ca{M}$. Then there is a canonical 2-functor
 \[
  L \colon \Map(\ca{M},B) \rightarrow \Comon(B) \smash{\rlap{.}}
 \]
 If $\ca{M}$ has Tannaka-Krein objects, then there is a canonical pseudofunctor
 \[
  R=\Rep(-) \colon \Comon(B) \rightarrow \Map(\ca{M},B)
 \]
 such that $L$ is left biadjoint to $R$.
\end{thm}

 Since a left biadjoint and its right biadjoint mutually determine each other, it suffices to describe one of them and define the other in terms of a universal property. In Section~\ref{OUTLINE_TK_BIADJUNCTION_SECTION} we chose to define the right biadjoint, without describing its left biadjoint explicitly. It turns out that for the generalized Tannakian biadjunction, the description of the left biadjoint is easier. Thus we will define the right biadjoint in terms of a universal property. Our construction closely mimics the way the semantics-structure adjunction between monads and their categories of algebras is constructed in Ross Street's ``formal theory of monads'' (see \cite{STREET_FTM}).
 
 In Section~\ref{TK_OBJECTS_IN_VMOD_SECTION} we will then show that any 2-category biequivalent to the bicategory $\Mod(\ca{V})$ for some cosmos $\ca{V}$ satisfies the conditions of the above theorem. To prove Theorem~\ref{TANNAKIAN_BIADJUNCTION_VMOD_THM}, it then suffices to show that the right biadjoint appearing in the statement is equivalent to the right biadjoint of Theorem~\ref{TANNAKIAN_BIADJUNCTION_VMOD_THM} (which is described in Section~\ref{RIGHT_BIADJOINT_SUBSECTION}).
 
 We choose this approach because the bicategory $\Mod(\ca{V})$ is rather difficult to work with: composition of 1-cells is defined in terms of a colimit formula, which makes it hard to compute with 2-cells explicitly. By working with a 2-category equivalent to $\Mod(\ca{V})$, the technicalities involved in dealing with these colimits automatically recede in the background.
 
 \subsection{The 2-category \texorpdfstring{$\Comon(B)$}{\textbf{Comon}(B)} of comonads on \texorpdfstring{$B$}{B}} 
 Let $B$ be an object of $\ca{M}$. A \emph{comonad} on $B$ is a comonoid in the (strict) monoidal category $\ca{M}(B,B)$, with monoidal structure given by composition. The category of comonads on $B$ is denoted by $\Comon(B)$. We think of $\Comon(B)$ as a 2-category with no nonidentity 2-cells.
 
 \subsection{The slice 2-category \texorpdfstring{$\Map(\ca{M},B)$}{Map(M,B)}} 
 The notion of a left adjoint 1-cell makes sense in any 2-category: a 1-cell $f \colon A \rightarrow B$ is left adjoint to $g \colon B \rightarrow A$ if there are 2-cells $\eta \colon \id \Rightarrow gf$ and $\varepsilon \colon fg \Rightarrow \id$ satisfying the triangle identities. Left adjoints in $\ca{M}$ will play a crucial role from now on. In accordance with the categorical literature we make the following definition.

\begin{dfn}
 Let $\ca{M}$ be a 2-category. A \emph{map} between 0-cells $A$ and $B$ of $\ca{M}$ is a left adjoint $f \colon A \rightarrow B$, together with a chosen right adjoint $\overline{f}$, unit $\eta$ and counit $\varepsilon$. A 2-cell between maps $(f,\overline{f},\eta,\varepsilon) \Rightarrow (g,\overline{g},\eta,\varepsilon)$ is simply a 2-cell $f \Rightarrow g$. The 2-category of maps in $\ca{M}$ is denoted by $\Map(\ca{M})$.
\end{dfn}

\begin{rmk}
 If we only insisted on the existence of a right adjoint we would get a 2-category that is biequivalent to $\Map(\ca{M})$. The fact that we have a chosen unit and counit at our disposal is merely a technical convenience.
\end{rmk}

\begin{dfn}\label{LAX_SLICE_DFN}
 Let $B$ be an object of $\ca{M}$. The \emph{lax slice} 2-category, denoted by $\ca{M} \slash_\ell B$, is the 2-category with objects the 1-cells with codomain $B$ and 1-cells from $w \colon A \rightarrow B$ to $w^{\prime} \colon A^{\prime} \rightarrow B$ the pairs $(a,\alpha)$ where $a \colon A \rightarrow A^{\prime}$ is a 1-cell in $\ca{M}$ and $\alpha$ is a 2-cell
 \[
  \vcenter{\hbox{
  \def\objectstyle{\scriptstyle}
  \def\labelstyle{\scriptscriptstyle}
  \def\twocellstyle{\scriptscriptstyle}
  \xymatrix{ A \xtwocell[0,2]{}\omit{<3>\alpha} \ar[rr]^{f} \ar[rd]_{w} & & A^{\prime} \ar[ld]^{w^{\prime}} \\
  & B
  }
  }}
 \]
 in $\ca{M}$. The 2-cells $(a,\alpha) \rightarrow (b,\beta)$ are 2-cells $\phi \colon a \Rightarrow b$ such that the equation
 \[
  \vcenter{\hbox{
  \def\objectstyle{\scriptstyle}
  \def\labelstyle{\scriptscriptstyle}
  \def\twocellstyle{\scriptscriptstyle}
  \xymatrix{ A \xtwocell[0,2]{}\omit{<3>\alpha} \ar[rr]^{f} \ar[rd]_{w} & & A^{\prime} \ar[ld]^{w^{\prime}} \\
  & B
  }
  }}=
  \vcenter{\hbox{
  \def\objectstyle{\scriptstyle}
  \def\labelstyle{\scriptscriptstyle}
  \def\twocellstyle{\scriptscriptstyle}
  \xymatrix{ A \xtwocell[0,2]{}\omit{<4>\beta} \rrtwocell^{f}_{*!<10pt,0pt>+{g}}{\phi} \ar[rd]_{w} & & A^{\prime} \ar[ld]^{w^{\prime}} \\
  & B
  }
  }} 
  \]
 holds.

 The 2-category $\Map(\ca{M},B)$ is the sub-2-category of the lax slice 2-category with objects the maps with codomain $B$ and 1-cells the 1-cells $(a,\alpha)$ of $\ca{M} \slash_{\ell} B$ for which $a$ is a map and $\alpha$ is invertible.
\end{dfn}

\subsection{String diagrams and the calculus of mates}\label{STRING_DIAGRAMS_SECTION}
 In a general 2-category, we can replace pasting diagrams (see \cite{KELLY_STREET}) by their ``Poincar\'e dual'' graphs, that is, in a pasting diagram we replace all the objects by 2-dimensional regions (2-disks), all the 1-cells by strings (1-disks) orthogonal to the original arrows, and all the 2-cells by points (0-disks), as indicated in the example below:
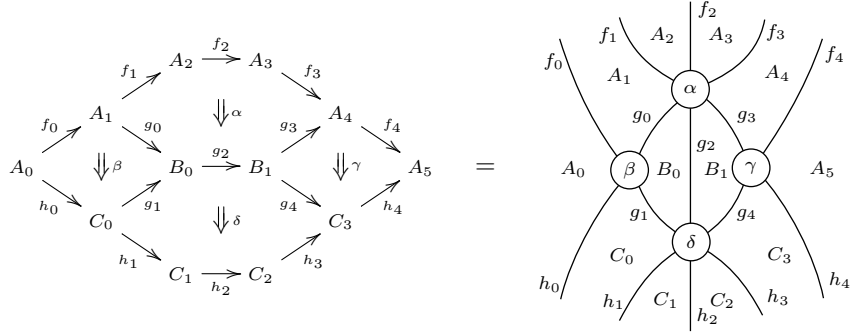
\[
\vcenter{\hbox{
\def\objectstyle{\scriptstyle}
\def\labelstyle{\scriptscriptstyle}
\def\twocellstyle{\scriptscriptstyle}
\xymatrix@C=5mm@R=3mm{
&& A_2 \ar[r]^{f_2} & A_3 \ar[rd]^{f_3}\\
& A_1
\xtwocell[0,3]{}\omit{\alpha}
 \ar[ru]^{f_1} \ar[rd]^{g_0} &&& A_4 \ar[rd]^{f_4}\\
A_0 \ar[ru]^{f_0}
\xtwocell[0,2]{}\omit{\beta}
\ar[rd]_{h_0} && B_0 \ar[r]^{g_2} & B_1
\xtwocell[0,2]{}\omit{\gamma} 
\ar[rd]_{g_4} \ar[ru]^{g_3} && A_5 \\
& C_0 
\xtwocell[0,3]{}\omit{\delta}
\ar[ru]_{g_1} \ar[rd]_{h_1} &&& C_3 \ar[ru]_{h_4}\\
&& C_1 \ar[r]_{h_2} & C_2 \ar[ru]_{h_3}
}}}\quad=\quad
\vcenter{\hbox{


}}
\]
 The same 2-cell can also be represented in symbols by the sequence
 \[
 h_4 \delta h_0 \cdot \gamma g_2 g_1 h_0 \cdot f_4 g_3 g_2 \beta \cdot f_4 \alpha f_0
 \]
 but this notation is very cumbersome and hides a lot of information: frequently a 2-cell can be described by several sequences that look quite different, and it is easy to write down nonsensical sequences where the domains and codomains of the 2-cells don't match up.

 Note that there is no need to add arrows to our strings, the orientation on the page contains enough information: by convention, 1-cells are composed left to right and 2-cells are composed top to bottom. Moreover, we generally omit the labels of objects of our 2-category. Since all 1-cells have a uniquely determined source and target, that information is in fact redundant.
 
 The idea of using string diagrams to describe morphisms is due to Penrose \cite{PENROSE}. For monoidal categories it was made mathematically precise in \cite{JOYAL_STREET_TENSOR}. In the context of 2-categories and bicategories, string diagrams first appeared in \cite{STREET_STRING_DIAGRAMS}. The string notation has several advantages. First, it is more apt to deal with identity morphisms than the pasting diagram notation. For 1-cells $f,g \colon A \rightarrow A$ and 2-cells $\alpha \colon f \Rightarrow \id$, $\beta \colon \id \Rightarrow g$, the three string diagrams
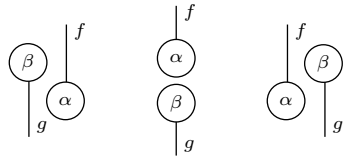
\[
\vcenter{\hbox{


}}
\]
 all represent the same 2-cell. The corresponding pasting diagrams look like
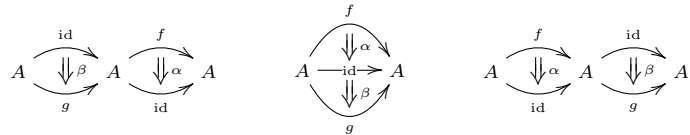
\[
\vcenter{\hbox{
\def\objectstyle{\scriptstyle}
\def\labelstyle{\scriptscriptstyle}
\def\twocellstyle{\scriptscriptstyle}
\xymatrix{ A \rtwocell^{\id}_g{\beta} & A \rtwocell^f_{\id}{\alpha} & A }
}}\quad
\vcenter{\hbox{
\def\objectstyle{\scriptstyle}
\def\labelstyle{\scriptscriptstyle}
\def\twocellstyle{\scriptscriptstyle}
\xymatrix{ A \ruppertwocell^f{\alpha} \rlowertwocell_g{\beta} \ar[r]|{\id} & A }
}}\quad
\vcenter{\hbox{
\def\objectstyle{\scriptstyle}
\def\labelstyle{\scriptscriptstyle}
\def\twocellstyle{\scriptscriptstyle}
\xymatrix{ A \rtwocell^f_{\id}{\alpha} & A \rtwocell^{\id}_g{\beta} & A }
}}
\]
 and it is less evident from the notation that they actually give the same 2-cell when evaluated. Second, the string diagram notation allows us to keep track of the specific order in which a diagram is evaluated. Without loss of generality, we can assume that no two 2-cells appear at the same height. The string diagram is evaluated from top to bottom, that is, one starts with the highest 2-cell, whiskers it appropriately, and then proceeds with the next one. Therefore we can suggest a specific order of (vertical) composition just by giving the string diagram. With the pasting diagram notation, this cannot be done as easily.
 
 String diagrams also simplify dealing with internal adjunctions. An adjunction $f \dashv \overline{f}$ comes with two 2-cells $\eta \colon \id \Rightarrow \overline{f} \cdot f$ and $\varepsilon \colon f \cdot \overline{f} \Rightarrow \id$. We denote these by
 \[
 \vcenter{\hbox{


}}
 \]
 respectively. The triangle identities are then given by the equations
 \[
 \vcenter{\hbox{

%
}}
 \]
 This notation makes the calculus of mates from \cite{KELLY_STREET} much more intuitive. Given two adjunctions $f \dashv \overline{f}$ and $g \dashv \overline{g}$ we sometimes denote the mate of $\alpha \colon f \cdot u \Rightarrow v \cdot g$ by $\overline{\alpha}$, that is, we use the abbreviation
 \[
 \vcenter{\hbox{

%
}}
 \]
 whenever it is convenient.

\subsection{The 2-functor \texorpdfstring{$L$}{L}} The string diagram notation allows us to easily write down the left biadjoint of the Tannakian biadjunction.

\begin{prop} \label{LEFT_BIADJOINT_2FUNCTOR_PROP}
 Let $B$ be an object of $\ca{M}$. Then the assignment
\[
L \colon \Map(\ca{M},B) \rightarrow \Comon(B)
\]
 which sends a map $w \colon A \rightarrow B$ to the comonad $L(w)=w \cdot \overline{w}$, with comultiplication and counit given by
\[
\vcenter{\hbox{

%
}}
\]
 respectively, and which sends a morphism $(a,\alpha) \colon w \rightarrow w^{\prime}$ to
\[
L(a,\alpha) \defl \quad \vcenter{\hbox{

%
}}
\]
 is a 2-functor.
\end{prop}

\begin{proof}
 It is a well-known fact that the composite of a left adjoint with a right adjoint is a comonad. We leave it to the reader to check that $L(-)$ is a 2-functor. This can be done quite easily by using the string diagram formalism from Section~\ref{STRING_DIAGRAMS_SECTION}.
\end{proof}

\subsection{Tannaka-Krein objects}\label{TK_OBJECTS_SECTION}
In \cite{STREET_FTM}, Ross Street showed that the category of Eilenberg-Moore coalgebras of a comonad has a universal property in the 2-category of categories, functors and natural transformations, namely it is universal among coactions on functors. This can be generalized to arbitrary 2-categories: objects which are universal among coactions on 1-cells are called Eilenberg-Moore objects. We introduce the notion of a Tannaka-Krein object of a comonad, which is universal among those coactions whose underlying 1-cells are maps (that is, left adjoints with chosen right adjoints).

\begin{dfn}\label{COACTION_DFN}
 Let $c \colon B \rightarrow B$ be a comonad in $\ca{M}$. A \emph{coaction} of $c$ consists of a 1-cell $v \colon A \rightarrow B$ in $\ca{M}$, together with a 2-cell $\rho \colon v \Rightarrow c \cdot v$ such that the equations
\[
\vcenter{\hbox{


}}=\id_v
\]
hold. A morphism of coactions $(v,\rho) \rightarrow (w,\sigma)$ is a 2-cell $\alpha \colon v \Rightarrow w$ such that the equation
\[
\vcenter{\hbox{

%
}}
\]
holds. A \emph{map coaction} of $c$ is a coaction $(v,\rho)$ where $v$ is a map. For an object $\ca{A} \in \ca{M}$ we write $\Coact_A(c)$ for the category of $c$-coactions with domain $A$, and we write $\Coact_A^m(c)$ for the full subcategory of map coactions of $c$.
\end{dfn}

\begin{dfn}\label{EM_OBJECTS_DFN}
 Let $\ca{M}$ be a 2-category and let $c\colon B \rightarrow B $ be a comonad in $\ca{M} $. An object $B_c$ endowed with a $c$-coaction $(v_c,\rho_c)\colon B_c \rightarrow B$ is called an \emph{Eilenberg-Moore object} of $c$ if it is universal among all $c$-coactions, in the sense that the functor
 \[
 T\colon \ca{M}(A,B_c) \rightarrow \Coact_A(c)
 \]
which sends a morphism $f \colon A \rightarrow B_c$ to the coaction
\[
T(f) \defl \quad \vcenter{\hbox{

\begin{tikzpicture}[y=0.80pt, x=0.8pt,yscale=-1, inner sep=0pt, outer sep=0pt, every text node part/.style={font=\scriptsize} ]
\path[draw=black,line join=miter,line cap=butt,line width=0.500pt]
  (354.3307,885.8267) .. controls (354.3307,896.4567) and (354.3307,900.0000) ..
  (354.3307,910.6299);
\path[draw=black,line join=miter,line cap=butt,line width=0.500pt]
  (336.6142,935.4330) .. controls (354.3307,910.6299) and (336.6142,935.4330) ..
  (354.3307,910.6299);
\path[draw=black,line join=miter,line cap=butt,line width=0.500pt]
  (354.3307,910.6299) .. controls (372.0472,935.4330) and (354.3307,910.6299) ..
  (372.0472,935.4330);
\path[draw=black,line join=miter,line cap=butt,line width=0.500pt]
  (318.8976,885.8267) .. controls (318.8976,914.1732) and (318.8976,914.1732) ..
  (301.1811,935.4330);
\path[fill=black] (351.98203,882.02802) node[above right] (text3208) {$v_c$
  };
\path[fill=black] (354.18192,911.25513) node[circle, draw, line width=0.500pt,
  minimum width=5mm, fill=white, inner sep=0.25mm] (text3212) {$\rho_c$    };
\path[fill=black] (333.07086,946.06299) node[above right] (text3216) {$v_c$
  };
\path[fill=black] (370.50394,945.60626) node[above right] (text3220) {$c$     };
\path[fill=black] (295.63779,948.06299) node[above right] (text3224) {$f$     };
\path[fill=black] (315.35431,882.28345) node[above right] (text3224-8) {$f$   };

\end{tikzpicture}
}}
\]
 and which sends a 2-cell $\phi \colon f \Rightarrow g \colon A \rightarrow B_c$ to $T(\phi) \defl v_c \phi$ is an isomorphism of categories.
\end{dfn}

 The universal property of a Tannaka-Krein object is slightly more complicated to state. Basically we just replace the word ``coaction'' by ``map-coaction'', but there is also a nontrivial interaction between 1-cells and maps that doesn't appear in the definition of Eilenberg-Moore objects. Furthermore, instead of insisting that $T$ be an isomorphism of categories (as in the case of EM-objects), we only want a an equivalence of categories. This allows us to transfer Tannaka-Krein objects along a biequivalence between 2-categories.

\begin{dfn}\label{TK_OBJECTS_DFN}
 Let $c \colon B \rightarrow B$ be a comonad in $\ca{M}$. A \emph{Tannaka-Krein object} for $c$ is an object $\Rep(c)$ together with a map coaction $(v_c,\rho_c) \colon \Rep(c) \rightarrow B$ which is universal among map coactions in the following sense: every map coaction $(v,\rho)$ is isomorphic to a coaction
\[
\rho_c \cdot f \defl \vcenter{\hbox{

\begin{tikzpicture}[y=0.80pt, x=0.8pt,yscale=-1, inner sep=0pt, outer sep=0pt, every text node part/.style={font=\scriptsize} ]
\path[draw=black,line join=miter,line cap=butt,line width=0.500pt]
  (354.3307,885.8267) .. controls (354.3307,896.4567) and (354.3307,900.0000) ..
  (354.3307,910.6299);
\path[draw=black,line join=miter,line cap=butt,line width=0.500pt]
  (336.6142,935.4330) .. controls (354.3307,910.6299) and (336.6142,935.4330) ..
  (354.3307,910.6299);
\path[draw=black,line join=miter,line cap=butt,line width=0.500pt]
  (354.3307,910.6299) .. controls (372.0472,935.4330) and (354.3307,910.6299) ..
  (372.0472,935.4330);
\path[draw=black,line join=miter,line cap=butt,line width=0.500pt]
  (318.8976,885.8267) .. controls (318.8976,914.1732) and (318.8976,914.1732) ..
  (301.1811,935.4330);
\path[fill=black] (351.98203,882.02802) node[above right] (text3208) {$v_c$
  };
\path[fill=black] (354.18192,911.25513) node[circle, draw, line width=0.500pt,
  minimum width=5mm, fill=white, inner sep=0.25mm] (text3212) {$\rho_c$    };
\path[fill=black] (333.07086,946.06299) node[above right] (text3216) {$v_c$
  };
\path[fill=black] (370.50394,945.60626) node[above right] (text3220) {$c$     };
\path[fill=black] (295.63779,948.06299) node[above right] (text3224) {$f$     };
\path[fill=black] (315.35431,882.28345) node[above right] (text3224-8) {$f$   };

\end{tikzpicture}
}}
\]
 for some map $f \colon A \rightarrow \Rep(c)$, and for \emph{any} 1-cell $g \colon A \rightarrow \Rep(c)$, whiskering with $v_c$ induces a bijection between 2-cells $g \Rightarrow f$ and morphisms of coactions $(v_c \cdot g,\rho_c \cdot g) \rightarrow (v_c \cdot f,\rho_c \cdot f)$.
\end{dfn}

\begin{rmk}\label{TK_DOUBLE_LIMIT_RMK}
 If $\Rep(c)$ endowed with the $c$-coaction $(v_c, \rho_c)$ is a Tannaka-Krein object in $\ca{M}$, then the functor
 \[
 T \colon \Map\bigl(A,\Rep(c)\bigr) \rightarrow \Coact_A^m(c)
 \]
 which sends a map $f$ to the coaction $(v_c \cdot f,\rho_c \cdot f)$ and a 2-cell $\phi \colon f \Rightarrow g$ between two maps $A \rightarrow \Rep(c)$ to $v_c \cdot \phi$ is an equivalence of categories. This fact is all we are going to need in order to construct the Tannakian biadjunction.
 
 The stronger universal property in the definition allows us to lift \emph{extraordinary} 2-cells. This is relevant for lifting an autonomous structure (that is, duals with chosen evaluation and coevaluation) along the forgetful functor $v_c$, a problem which we don't study in this paper.

\end{rmk}

The following example was pointed out by Ignacio L\'opez Franco.

\begin{example}
 Let $\ca{K}$ be the 2-category of categories with equalizers, 1-cells the functors which preserve equalizers, and 2-cells the natural transformations. Then the category of comodules of a comonad in $\ca{K}$ has equalizers, so $\ca{K}$ has Eilenberg-Moore objects. Moreover, by the dual of \cite[Theorem~1]{DUBUC_TRIANGLE}, the lift of a left adjoint to the category of comodules is itself a left adjoint. Therefore every Eilenberg-Moore object is also a Tannaka-Krein object in $\ca{K}$. 
 
 The same is true for the 2-category of $\ca{V}$-categories with equalizers (in the enriched sense) and $\ca{V}$-functors which preserve equalizers.
\end{example}

\subsection{The pseudofunctor \texorpdfstring{$\Rep(-)$}{\textbf{Rep}(-)}}
The universal property of Tannaka-Krein objects allows us to construct the desired pseudofunctor from comonads on $B$ to the 2-category of maps into $B$.

\begin{prop}\label{REP_2_FUNCTOR_PROP}
 Let $\ca{M}$ be a 2-category with Tannaka-Krein objects, and let
 \[
 T^{-1} \colon \Coact_A(c) \rightarrow \Map\bigl(A,\Rep(c)\bigr)
 \]
 be an inverse to the functor $T$ from Remark~\ref{TK_DOUBLE_LIMIT_RMK}. Then the assigment which sends a comonad $c$ on $B$ to the Tannaka-Krein object $v_c \colon \Rep(c) \rightarrow B$ and a morphism of comonads $\phi \colon c \rightarrow c^{\prime}$ to the 1-cell 
 \[
 \def\objectstyle{\scriptstyle}
  \def\labelstyle{\scriptscriptstyle}
  \def\twocellstyle{\scriptscriptstyle}
 \xymatrix{\Rep(c) \rrtwocell\omit{<3>*!<-3pt,0pt>+{\lambda_{\phi}}} \ar[rr]^{T^{-1}(\rho_{\phi})} \ar[rd]_{v_c} && \Rep(c^{\prime}) \ar[ld]^{v_{c^{\prime}}} \\ & B }
 \]
 defines a pseudofunctor $\Comon(B) \rightarrow \Map(\ca{M},B)$, where $\rho_{\phi}$ is the coaction
\[
\rho_{\phi} \defl
\vcenter{\hbox{


}}
\]
 and where $\lambda_\phi$ is the $\rho_{\phi}$-component of the natural isomorphism $TT^{-1} \Rightarrow \id$. With an appropriate choice of $T^{-1}$ we can ensure that the resulting pseudofunctor is normal, that is, that it preserves identities strictly.
\end{prop}

\begin{proof}
 It is not hard to see that the 2-cell
 \[
 \alpha \defl \vcenter{\hbox{

%
}}
 \]
 is a morphism of coactions from the coaction $T\bigl(T^{-1}(\rho_{\psi})\cdot T^{-1}(\rho_{\phi})\bigr)$ to the coaction $TT^{-1}(\rho_{\psi \phi})$. Since the functor $T$ is fully faithful there exists a unique 2-cell
 \[
 \mu \colon T^{-1}(\rho_{\psi}) \cdot T^{-1}(\rho_{\phi}) \Rightarrow T^{-1}(\rho_{\psi\phi}) 
 \]
 such that the equation $v_{c^{\prime\prime}} \mu = \alpha $ holds. Similar reasoning shows that there is a unique 2-cell $\mu_0 \colon \id \Rightarrow T^{-1}(\rho_{\id})$ with
 \[
  \vcenter{\hbox{

%
}}
 \]
 By whiskering the equations in question with $v_{c^{\prime}}$ or $v_{c^{\prime\prime\prime}}$ one can now easily check that $\mu$ and $\mu_0$ give the desired pseudofunctor structure.
 
 Choosing $T^{-1}$ amounts to choosing a map $T^{-1}(\rho)$ together with an isomorphism $\lambda_{\rho} \colon TT^{-1}(\rho) \rightarrow \rho$ of coactions. Because $\rho_{\id}=\rho_c$, we can make sure that $T^{-1}(\rho_{\id})=\id$ and $\lambda_{\rho_{\id}}=\id$. With this choice for $T^{-1}$, the resulting pseudofunctor is normal.
\end{proof}

\subsection{Proof of the Tannakian biadjunction} We are now ready to prove that $\Rep(-)$ is a right biadjoint of $L(-)$.

\begin{proof}[Proof of Theorem~\ref{TANNAKIAN_BIADJUNCTION_2CAT_THM}]
 To show that $L$ is a left biadjoint of $\Rep(-)$ we have to give a pseudonatural equivalence
 \[
  \theta_{w,c} \colon \Map(\ca{M},B)\bigr(w,\Rep(c)\bigl) \rightarrow \Comon(B)\bigl(L(w),c\bigr) \smash{\rlap{.}}
 \]
 If such an equivalence exists it has to be strictly 2-natural, because there are no nonidentity 2-cells in the codomain. We define the functor $\theta_{w,c}$ on an object $(s,\sigma) \colon w \rightarrow v_c$ by
 \[
 \theta_{w,c}\bigl((s,\sigma)\bigr) \defl \vcenter{\hbox{


}}
 \]
 and on a morphism $\alpha \colon (s,\sigma) \rightarrow (t,\tau)$ by $\theta_{w,c}(\alpha)=\id$. The latter makes sense because the existence of $\alpha$ implies $L\bigl((s,\sigma)\bigr)=L\bigl((t,\tau)\bigr)$ (cf.\ Definition~\ref{LAX_SLICE_DFN}). We leave it to the reader to check that this gives a well-defined functor, that is, that the morphism $w \cdot \overline{w} \rightarrow c$ defined above is a morphism of comonads.
 
 To see that $\theta_{w,c}$ is a 2-natural transformation, we have to check that for any 1-cell $(a,\alpha) \colon v \rightarrow w$ and any morphism of comonads $\phi \colon c \rightarrow c^{\prime}$, the equation
 \[
  \phi \cdot \theta_{w,c}\bigl((s,\sigma)\bigr) \cdot L\bigl((a,\alpha)\bigr)=\theta_{v,c^{\prime}}\bigl((\Rep(\phi),\lambda_{\phi}) \cdot (s,\sigma) \cdot (a,\alpha)\bigr)
 \]
 holds. The key observation for this is that the equation
 \[
  \vcenter{\hbox{


}}
 \]
 holds, which follows from the fact that $\lambda_{\phi}$ is a component of the natural transformation $TT^{-1} \Rightarrow \id$ (see Proposition~\ref{REP_2_FUNCTOR_PROP}). We leave it to the reader to check the details.
 
 It remains to show that $\theta$ is fully faithful and essentially surjective. It is obviously full, because there are no nonidentity 2-cells in the codomain. To see that it is faithful, we need to check that there is at most one 2-cell between two 1-cells $(s,\sigma)$ and $(t,\tau)$ from $w$ to $v_c$. But whiskering with $v_c$ is the 2-cell part of the equivalence $T$ from Remark~\ref{TK_DOUBLE_LIMIT_RMK}, so it suffices to check that $v_c \alpha=v_c \beta$ for any two 2-cells $(s,\sigma) \Rightarrow (t,\tau)$. This follows immediately from the definition of 2-cells in $\Map(\ca{M},B)$, see Definition~\ref{LAX_SLICE_DFN}.
 
 To see that $\theta_{w,c}$ is essentially surjective, first note that for any morphism of comonads $\phi \colon L(w) \rightarrow c$, the 2-cell
 \[
 \rho\defl \vcenter{\hbox{

\begin{tikzpicture}[y=0.80pt, x=0.8pt,yscale=-1, inner sep=0pt, outer sep=0pt, every text node part/.style={font=\scriptsize} ]
\path[draw=black,line join=miter,line cap=butt,line width=0.500pt]
  (0.0000,1140.9449) .. controls (0.0000,1155.1181) and (-7.0866,1165.7480) ..
  (-17.7165,1176.3779);
\path[draw=black,line join=miter,line cap=butt,line width=0.500pt]
  (-53.1496,1204.7244) .. controls (-52.9790,1192.1229) and (-53.1496,1186.9661)
  .. (-53.1496,1176.3779) .. controls (-53.1496,1165.7480) and
  (-46.0630,1151.5748) .. (-35.4331,1151.5748) .. controls (-24.8031,1151.5748)
  and (-17.7165,1165.7480) .. (-17.7165,1176.3779);
\path[draw=black,line join=miter,line cap=butt,line width=0.500pt]
  (-17.7165,1204.7244) .. controls (-17.7165,1187.0079) and (-17.7165,1194.0945)
  .. (-17.7165,1176.3779);
\path[fill=black] (-17.5991,1176.4987) node[circle, draw, line width=0.500pt,
  minimum width=5mm, fill=white, inner sep=0.25mm] (text31673) {$\phi$    };
\path[fill=black] (-21.370337,1211.3826) node[above right] (text31677) {$c$
  };
\path[fill=black] (-56.882809,1212.0112) node[above right] (text31681) {$w$
  };
\path[fill=black] (-21.259842,1158.6614) node[above right] (text31685)
  {$\overline{w}$     };
\path[fill=black] (-3.7712359,1137.5293) node[above right] (text31689) {$w$   };

\end{tikzpicture}
}}
 \]
 is a map coaction. The axioms for a morphism of comonads correspond precisely to the axioms for a coaction. Let $s=T^{-1}(\rho)$ and let $\sigma \colon v_c \cdot s \Rightarrow w$ be the $\rho$-component of the natural transformation $TT^{-1} \Rightarrow \id $. It is now easy to see that $\theta_{w,c}\bigl((s,\sigma)\bigr)=\phi$, which shows that $\theta_{w,c}$ is indeed an equivalence of categories.
 \end{proof}
 
\section{Details for the Tannakian biadjunction in \texorpdfstring{$\Mod(\ca{V})$}{Mod(V)}}\label{TANNAKIAN_BIADJUNCTION_MODV_SECTION}
 
 Let $\ca{V}$ be a cosmos. In order to see that Theorem~\ref{TANNAKIAN_BIADJUNCTION_VMOD_THM} is a consequence of Theorem~\ref{TANNAKIAN_BIADJUNCTION_2CAT_THM}, we first have to describe a 2-category $\ca{M}$ which is biequivalent to $\Mod(\ca{V})$. In a second step, we have to show that $\ca{M}$ has Tannaka-Krein objects. Lastly, we have to show that the pseudofunctor $\Rep(-)$ from Proposition~\ref{REP_2_FUNCTOR_PROP} is equivalent to the 2-functor $R$ described in Section~\ref{RIGHT_BIADJOINT_SUBSECTION}. In order to do all this we need the notions of weighted colimits and cocontinuous $\ca{V}$-functors. This is a place where the theory of enriched categories differs considerably from the unenriched theory. We will mention all the technical details about weighted colimits that are used later in the proof of our recognition result. 
 
 \subsection{Recollections about weighted colimits}\label{WEIGHTED_COLIMIT_SECTION}
 When we enrich the notion of colimits, we naturally arrive at the concept of a \emph{weighted colimit}\footnote{Weighted colimits are called \emph{indexed colimits} in \cite{KELLY_BASIC}.}: an object $K$ of a $\ca{V}$-category $\ca{E}$ is said to be the colimit of $G\colon \ca{D} \rightarrow \ca{E}$ weighted by  $H \colon \ca{D}^{\op} \rightarrow \ca{V}$ if there is an isomorphism
\[
\xymatrix{\ca{E}(K,E) \ar[r]^-{\phi_E} & \Prs{D}\bigl(H,\ca{E}(G-,E)\bigr)}
\]
 of $\ca{V}$-functors which is $\ca{V}$-natural in $E$. The object $K$ is usually denoted by $H \star G$. For any small $\ca{V}$-category $\ca{B}$, the category $\Prs{B}$ of enriched presheaves on $\ca{B}$ has all weighted colimits (see \cite[\S~3.3]{KELLY_BASIC}). The identity of $H \star G$ corresponds under $\phi$ to the \emph{unit}
\[
 \xymatrix{H \ar[r]^-{\lambda} & \ca{E}(G-,H\star G)}
\]
 of $H \star G$, which has the property that for any $\ca{V}$-natural transformation $\alpha \colon H \Rightarrow \ca{E}(G-,E)$, there is a unique morphism $a \colon H\star G \rightarrow E$ such that $\alpha = \ca{E}(G-,a)$. For $\ca{V}=\Mod_R$, the existence of a $\lambda$ with this property is equivalent to the existence of the natural isomorphism $\phi$ (see \cite[\S~3.1]{KELLY_BASIC}). In particular, if $L \colon \ca{E} \rightarrow \ca{E}^{\prime}$ is a $\ca{V}$-functor such that the colimit $H \star LG$ exists, there is a unique morphism $\widehat{L} \colon H \star LG \rightarrow L(H\star G)$ for which the diagram
\[
 \xymatrix{H  \ar[rr]^-{\lambda} \ar[d]_-{\lambda^{\prime}} && \ca{E}(G-,H\star G) \ar[d]^-{L} \\
 \ca{E}^{\prime}(LG-,H\star LG) \ar[rr]_-{\ca{E}(LG-,\widehat{L})} && \ca{E}^{\prime}\bigl(LG-,L(H\star G)\bigr)}
\]
 is commutative, where $\lambda^{\prime}$ denotes the unit of $H \star LG$. The morphism $\widehat{L} \colon H\star LG \rightarrow L(H\star G)$ is called the \emph{comparison morphism}, and we say that $L$ \emph{preserves} the colimit $H \star G$ if $H\star LG$ exists and $\widehat{L}$ is an isomorphism. A $\ca{V}$-functor is said to be \emph{cocontinuous} if it preserves all small weighted colimits that exist. 
 
 A $\ca{V}$-functor $L\colon \ca{E} \rightarrow \ca{E}^{\prime}$ is called a \emph{left $\ca{V}$-adjoint} or simply \emph{left adjoint} if there is a $\ca{V}$-functor $R\colon \ca{E}^{\prime} \rightarrow \ca{E}$ and $\ca{V}$-natural transformations $\eta \colon \id\Rightarrow RL$ and $\varepsilon \colon LR \Rightarrow \id$ satisfying the usual triangle identities. In other words, a left $\ca{V}$-adjoint is precisely a map in $\VCAT$. Recall that we get underlying ordinary categories, functors and natural transformations if we apply the forgetful functor $V=\ca{V}(I,-)\colon \ca{V} \rightarrow \Set$ to the hom-objects of a $\ca{V}$-category. The condition that $L$ is a left $\ca{V}$-adjoint is in general stronger than saying that the underlying ordinary functor $L_0$ is a left adjoint, but if $\ca{V}=\Mod_R$, the two notions agree (see \cite[\S~1.11]{KELLY_BASIC}).
 
 As one would expect, if $L$ is a left adjoint, then it is cocontinuous (see \cite[\S~3.2]{KELLY_BASIC}). The category of cocontinuous $\ca{V}$-functors $\ca{A} \rightarrow \ca{B}$ will be denoted by $\Cocts[\ca{A},\ca{B}]$.

 Let $V\in \ca{V}$. If the $\ca{V}$-functor $[V,\ca{E}(E,-)] \colon \ca{E} \rightarrow \ca{V}$ is representable, we denote the representing object by $V \odot E$ and we call it the \emph{tensor product} or simply \emph{tensor} of $V$ and $E$. This concept is a special case of a weighted colimit: for $\ca{D}=\ca{I}$, the unit $\ca{V}$-category, giving a weight amounts to giving an object $V \in \ca{V}$, giving a $\ca{V}$-functor $\ca{I} \rightarrow \ca{E}$ amounts to giving an object $E \in \ca{E}$, and the colimit of $E$ weighted by $V$ is precisely the tensor $V\odot E$. For a subcategory $\ca{X} \subseteq \ca{V}$ we say that $\ca{E}$ is \emph{$\ca{X}$-tensored} if the tensor $V\odot E$ exists for all $E\in \ca{E}$ and all $V\in \ca{X}$. If $\ca{E}$ is $\ca{X}$-tensored for $\ca{X}=\ca{V}$ we simply say that $\ca{E}$ is \emph{tensored}. 

If a $\ca{V}$-category has all small weighted colimits, then the colimit of $G\colon \ca{D} \rightarrow \ca{E}$ weighted by  $H \colon \ca{D}^{\op} \rightarrow \ca{V}$ is given by the coend
\[
H \star G = \int^{D\in \ca{D}} HD \odot GD
\]
 (see \cite[\S~3.10]{KELLY_BASIC}). There are also weights corresponding to ordinary diagrams in the underlying category. To distinguish them from general weights the corresponding colimits are called \emph{conical colimits}. For $\ca{V}=\Mod_R$, a conical colimit exists if and only if the corresponding ordinary colimit exists in the underlying category (see \cite[\S~3.8]{KELLY_BASIC}), so in this case there is no need to distinguish the two notions.
 
 \subsection{Free cocompletions} 
 In order to talk about free cocompletions we will use the concept of a left Kan extension of $\ca{V}$-functors. These are discussed in \cite[\S~4]{KELLY_BASIC}. We only need the special case of left Kan extension along the Yoneda embedding, for which we use the following notation. All the facts we need about them are implicit in Theorem~\ref{FREE_COCOMPLETION_THM}.

\begin{notation}\label{FREE_COCOMPLETION_NOTATION}
 We write $L_K$ for the left Kan extension $\Lan_Y K$ of $K \colon \ca{A} \rightarrow \ca{C}$ along the Yoneda embedding $Y\colon \ca{A} \rightarrow \Prs{A}$, and we denote the unit of this Kan extension by $\alpha_K \colon K \Rightarrow L_K Y$. We let $\widetilde{K} \colon \ca{C} \rightarrow \Prs{A}$ be the right adjoint of $L_K$, that is, $\widetilde{K}(C)=\ca{C}(K-,C)$. The unit and counit of the adjunction $L_K \dashv \widetilde{K}$ are denoted by $\eta^K \colon \id \Rightarrow \widetilde{K} L_K$ and $\varepsilon^K \colon L_K \widetilde{K} \Rightarrow \id$ respectively.
\end{notation}

 \begin{thm}\label{FREE_COCOMPLETION_THM}
 Let $Y \colon \ca{A} \rightarrow \Prs{A}$ be the Yoneda embedding of the small $\ca{V}$-category $\ca{A}$, and let $\ca{C}$ be a cocomplete $\ca{V}$-category. Then for any cocontinuous $\ca{V}$-functor $S \colon \Prs{A} \rightarrow \ca{C}$ we have
\[
S\cong L_K \cong -\star K \colon \Prs{A} \rightarrow \ca{C}
\]
 where $K=SY \colon \ca{A} \rightarrow \ca{C}$. The assignment $S \mapsto SY$ is an equivalence of $\ca{V}$-categories
\[
[Y,\ca{C}] \colon \Cocts[\Prs{A},\ca{C}] \rightarrow [\ca{A},\ca{C}]\rlap{.}
\]
 The inverse to this equivalence sends $K$ to (a choice of) $\Lan_Y K$.
\end{thm}

\begin{proof}
 This is (part of) \cite[Theorem~4.51]{KELLY_BASIC}.
\end{proof}

 \subsection{The 2-category \texorpdfstring{$\ca{M}$}{M} and maps in \texorpdfstring{$\Mod(\ca{V})$}{Mod(V)}} The goal of this section is to establish a relationship between maps in the bicategory of modules on the one hand, and $\ca{V}$-functors on the other. 
 
  \begin{rmk}\label{VMOD_MAPS_RMK}
 A $\ca{V}$-functor $f \colon \ca{A} \rightarrow \ca{B}$ induces a module $f_\ast=\ca{B}(1,f) \colon \ca{A} \xslashedrightarrow{} \ca{B}$ which sends $(b,a)$ to $\ca{B}(b,fa)$, and a module $f^\ast=\ca{B}(f,1) \colon \ca{B} \xslashedrightarrow{} \ca{A}$ which sends $(a,b)$ to $\ca{B}(fa,b)$. The module $f_\ast$ is left adjoint to $f^\ast$, with unit given by
 \[
 \xymatrix{ \ca{A}(a,a^\prime) \ar[r]^-{f} & \ca{B}(fa,fa^{\prime}) \ar[r]^-{\cong} & \int^{b \in \ca{B}} \ca{B}(fa,b)\otimes \ca{B}(b,fa^\prime) }
 \]
 where the isomorphism is a consequence of the enriched Yoneda lemma (see \cite[Formula~(3.71)]{KELLY_BASIC}).
 \end{rmk}
 
 Using the fact that enriched presheaf categories are free cocompletions, we can now prove the following proposition.

\begin{prop}\label{MODULES_2CATEGORY_PROP}
 The bicategory $\Mod(\ca{V})$ is biequivalent to the 2-category with objects the small $\ca{V}$-categories, with 1-cells from $\ca{A}$ to $\ca{B}$ given by the left adjoint $\ca{V}$-functors $\Prs{A} \rightarrow \Prs{B}$, and 2-cells the $\ca{V}$-natural transformations between those.
\end{prop}

\begin{proof}
 This follows from the hom-tensor adjunction for $\ca{V}$-categories and from the fact that presheaf categories are free cocompletions. More precisely, giving a $\ca{V}$-functor
 \[
  \ca{A}\otimes \ca{B}^{\op} \rightarrow \ca{V}
 \]
 is the same as giving a $\ca{V}$-functor $\ca{A} \rightarrow \Prs{B}$, which corresponds to a unique left adjoint $\Prs{A} \rightarrow \Prs{B}$ by Theorem~\ref{FREE_COCOMPLETION_THM}.
\end{proof}

\begin{lemma}\label{MAPS_LEMMA}
 The category of maps $\ca{A} \xslashedrightarrow{} \ca{B}$ in $\Mod(\ca{V})$ is equivalent to the category of $\ca{V}$-functors $\ca{A} \rightarrow \overline{\ca{B}}$ from $\ca{A}$ to the Cauchy completion $\overline{\ca{B}}$ of $\ca{B}$. More concretely, under the biequivalence from Proposition~\ref{MODULES_2CATEGORY_PROP}, this means that whiskering with the Yoneda embedding gives a bijection
\[
 \Map\bigl( \Mod(\ca{V})\bigr)(\ca{A},\ca{B}) \rightarrow [\ca{A},\overline{\ca{B}}]\rlap{\smash{.}}
\]
\end{lemma}

\begin{proof}
 By Proposition~\ref{MODULES_2CATEGORY_PROP}, a module is a left adjoint in $\Mod(\ca{V})$ if and only if the corresponding left adjoint $L_F$ has a right adjoint which is cocontinuous. This right adjoint is given by $X \mapsto \Prs{B}(F-,X)$. Since colimits in presheaf categories are computed pointwise, it is cocontinuous if and only if for every $A\in \ca{A}$, the $V$-functor $\Prs{B}(FA,-)$ preserves all $\ca{V}$-colimits. Objects in $\Prs{B}$ with the property that the corresponding representable functor preserves $V$-colimits are by definition the objects of the  Cauchy completion $\overline{\ca{B}}$ of $\ca{B}$ (see Section~\ref{CAUCHY_COMPLETION_SECTION}). In other words: a module $M \colon \ca{A} \xslashedrightarrow{} \ca{B}$ is a left adjoint if and only if the corresponding functor $F \colon \ca{A} \rightarrow \Prs{B}$ factors through the Cauchy completion $\overline{\ca{B}}$.
\end{proof}

\subsection{The existence of Tannaka-Krein objects}\label{TK_OBJECTS_IN_VMOD_SECTION} 
In order to construct Tannaka-Krein objects in $\Mod(\ca{V})$, we will use the fact that the 2-category $\VCAT$ has Eilenberg-Moore objects.

\begin{prop}\label{VCAT_EM_OBJECTS_PROP}
Let $C\colon \ca{B} \rightarrow \ca{B} $ be a comonad in $\VCAT$. Then the forgetful functor  $V \colon \ca{B}_C \rightarrow \ca{B}$ from the $\ca{V}$-category of $C$-comodules to $\ca{B}$ is an Eilenberg-Moore object of $C$. The component of the coaction $\rho \colon V \Rightarrow C \cdot V$ at a comodule $(M,r)$ is given by $\rho_{(M,r)}=r$.
\end{prop}

\begin{proof}
 The statement that every coaction has a unique lift is dual to \cite[Proposition~II.1.1]{DUBUC}. The statement about $\ca{V}$-natural transformations between lifts follows from the fact that for any two $C$-comodules $M$ and $M^{\prime}$, the component $V_{M,M^{\prime}} \colon \ca{B}_C (M,M^{\prime}) \rightarrow \ca{C}(VM, VM^{\prime})$ of $V$ is by definition an equalizer (cf.\ \cite[p.~64]{DUBUC}).
 \end{proof}

\begin{rmk}\label{NEW_EM_ADJUNCTION_RMK}
 Let $C \colon \ca{B} \rightarrow \ca{B}$ be a comonad in $\VCAT$. Then the forgetful functor $\ca{B}_C \rightarrow \ca{B}$ has a right adjoint, which sends an object $B \in\ca{B}$ to the comodule $(CB,\delta_B)$ where $\delta \colon C \Rightarrow C^2$ is the comultiplication of $C$.
\end{rmk}

 We can now describe the object part of a Tannaka-Krein object in $\Mod(\ca{V})$.

\begin{dfn}\label{CAUCHY_COMODULE_DFN}
 Let $\ca{V}$ be a cosmos, $\ca{B}$ a small $\ca{V}$-category, and let $C$ be a comonad on $\ca{B}$ in $\Mod(\ca{V})$ (equivalently, $C$ is a cocontinuous comonad on $\Prs{B}$). A \emph{Cauchy comodule} of $C$ is a comodule whose underlying object lies in the Cauchy completion $\overline{\ca{B}}$ of $\ca{B}$ (see Section~\ref{CAUCHY_COMPLETION_SECTION}). The category of Cauchy comodules of $C$ is denoted by $\Rep(C)$.
\end{dfn}

The following lemma will be relevant for showing the 2-cell part of the universal property of a Tannaka-Krein object in $\Mod(\ca{V})$.

\begin{lemma}\label{VMOD_TK_2CELL_LEMMA}
 Let $K \colon \ca{B} \rightarrow \ca{C}$ be a fully faithful $\ca{V}$-functor where $\ca{B}$ is small and $\ca{C}$ is cocomplete. Let $F \colon \Prs{A} \rightarrow \Prs{B}$ and $G \colon \Prs{A}^{\prime} \rightarrow \Prs{B}$ be cocontinuous $\ca{V}$-functors with cocontinuous right adjoints, and let $M \colon \Prs{A} \rightarrow \Prs{A}^{\prime}$ be any cocontinuous $\ca{V}$-functor. Then whiskering with $L_K$ induces a bijection
 \[
 L_K(-) \colon  \VNat(GM,F) \rightarrow \VNat(L_K G M, L_K F)\rlap{\smash{.}}
 \]
\end{lemma}

\begin{proof}
 Since $F$, $G$, $M$ and $L_K$ are cocontinuous, it follows from the fact that whiskering with the Yoneda embedding is fully faithful (see Theorem~\ref{FREE_COCOMPLETION_THM}) that it suffices to show that
 \[
 L_K(-) \colon  \VNat(GMY,FY) \rightarrow \VNat(L_K G MY, L_K FY)
 \]
 is invertible. The assignment which sends a 2-cell $\phi \colon GMY \Rightarrow FY $ to
 \[
  \vcenter{\hbox{


}}
 \]
 gives a bijection $\VNat(GMY,FY)\cong \VNat(MY,\overline{G}FY)$. The inverse is given by whiskering with the counit of the adjunction $G \dashv \overline{G}$. Similarly, the assignment which sends a 2-cell $\psi \colon L_K GMY \Rightarrow L_K F Y$ to
 \[
  \vcenter{\hbox{

%
}}
 \]
 gives a bijection $\VNat(L_K GMY, L_K FY) \cong \VNat(MY,\overline{G} \widetilde{K} L_K FY)$. It is immediate from the definition of these bijections that the diagram
 \[
  \xymatrix{ \VNat(GMY,FY) \ar[d]_{\cong} \ar[rrr]^-{L_K(-)} &&& \VNat(L_K GM, L_K  FY) \ar[d]^{\cong} \\
   \VNat(MY,\overline{G}FY) \ar[rrr]_-{\VNat(M,\overline{G} \eta^K FY)} &&&  \VNat(MY,\overline{G} \widetilde{K} L_K FY)}
 \]
 is commutative. Therefore it suffices to check that $\overline{G} \eta^K FY$ is invertible. This follows if we can show that $\eta^K FY$ is invertible.
 
 From Theorem~\ref{FREE_COCOMPLETION_THM} and Lemma~\ref{MAPS_LEMMA} we know that $FY$ factors through the Cauchy completion $\overline{\ca{B}}$ of $\ca{B}$. We claim that $\eta^K_X$ is invertible for any object $X \in \overline{\ca{B}}$. From \cite{STREET_ABSOLUTE} we know that $\overline{\ca{B}}$ is also the free cocompletion of $\ca{B}$ under absolute colimits, that is, colimits that are preserved by any functor. For example, in the case $\ca{V}=\Mod_R$, these absolute colimits are given by finite direct sums and splittings of idempotents, which are clearly preserved by any $R$-linear functor. Thus the problem can be reduced to showing that $\eta^K_X$ is invertible for any representable presheaf $X$. Thus we have to show that for any object $B \in \ca{B}$, the $\ca{V}$-natural transformation
 \[
  \eta^K_{\ca{B}(-,B)} \colon \ca{B}(-,B) \rightarrow \ca{C}\bigl(K-, L_K \ca{B}(-,B)\bigr)
 \]
 is invertible. From the Yoneda lemma it follows that 
 \[
  K_{-,B} \colon \ca{B}(-,B) \rightarrow \ca{C}(K-,KB)
 \]
 has the same universal property as $\eta^K_{\ca{B}(-,B)}$. This is not hard to see in the unenriched case, and the same proof works for $\ca{V}=\Mod_R$. For the general case see \cite[Formula (3.10)]{KELLY_BASIC}. Thus $\eta^K_{\ca{B}(-,B)}$ is up to isomorphism given by $K_{-,B}$, which is invertible by assumption.
\end{proof}

\begin{thm}\label{VMOD_TK_OBJECTS_THM}
 Let $\ca{V}$ be a cosmos and let $C \colon \ca{B} \xslashedrightarrow{} \ca{B}$ be a comonad in $\Mod(\ca{V})$, that is, $C$ is a cocontinuous comonad on $\Prs{B}$. Let
\[
V \dashv W \colon (\Prs{B})_C \rightarrow \Prs{B}
\]
 be the adjunction between the $C$-comodules and $\Prs{B}$ (cf.\ Proposition~\ref{VCAT_EM_OBJECTS_PROP}). We write $K$ for the canonical inclusion $\Rep(C) \rightarrow \Prs{B}_C$  and we choose a left Kan extension
 \[
 L_K \colon \mathcal{P}{\Rep(C)} \rightarrow (\Prs{B})_C
 \]
 of $K$ along the Yoneda embedding. Then $V \cdot L_K$, together with the coaction
 \[
\rho_C \defl \vcenter{\hbox{


}}
 \]
 is a Tannaka-Krein object for $C$.
\end{thm}

\begin{proof}
We first have to show that the functor $T$ from Remark~\ref{TK_DOUBLE_LIMIT_RMK} is essentially surjective. Thus let $(F,\rho) \colon \Prs{A} \rightarrow \Prs{B}$ be a map coaction of $C$. By Proposition~\ref{VCAT_EM_OBJECTS_PROP} and Definition~\ref{EM_OBJECTS_DFN}, there exists a unique $\ca{V}$-functor $\widehat{F}$ such that the equation
 \[
 \vcenter{\hbox{

%
}}
 \]
 holds. From Lemma~\ref{MAPS_LEMMA} we know that the composite $FY$ of the Yoneda embedding of $\ca{A}$ and $F$ factors through $\overline{\ca{B}}$. But this implies that $\widehat{F}Y=KG$ for some $\ca{V}$-functor $G \colon \ca{A} \rightarrow \Rep(C)$. Let $L_{YG}$ be a left Kan extension of $YG \colon \ca{A} \rightarrow \mathcal{P}{\Rep(C)}$ along the Yoneda embedding of $\ca{A}$. We have $\ca{V}$-natural isomorphisms
 \[
 L_K L_{YG} Y \cong L_K YG\cong KG =\widehat{F}Y
 \]
 between cocontinuous functors. Theorem~\ref{FREE_COCOMPLETION_THM} implies that the above composite of isomorphisms comes from an isomorphism $\sigma \colon L_K L_{YG} \Rightarrow \widehat{F}$. It is now easy to see that $V \sigma$ gives an isomorphism between $T(L_{YG})$ and $(F,\rho)$ in $\Coact^m_{\ca{A}}(C)$. This shows that $T$ is essentially surjective. 
 
 It remains to show that $VL_K$ satisfies the 2-cell part of the universal property of a Tannaka-Krein object. This follows immediately from Lemma~\ref{VMOD_TK_2CELL_LEMMA} and from the fact that the $\ca{V}$-functor $V \colon \Prs{B}_C \rightarrow \Prs{B}$ is an Eilenberg-Moore object in $\VCAT$ (see Proposition~\ref{VCAT_EM_OBJECTS_PROP} and Definition~\ref{EM_OBJECTS_DFN}).
\end{proof}

\begin{rmk}\label{REP_EQUIVALENT_R_RMK}
 Let $\ca{M}$ be the 2-category from Proposition~\ref{MODULES_2CATEGORY_PROP}. The composite of the pseudofunctor $\Rep(-)$ with the biequivalence
 \[
  \Map\bigl(\Mod(\ca{V}),\ca{B}\bigr) \rightarrow \Vcat \slash \overline{\ca{B}}
 \]
 (cf.\ Lemma~\ref{MAPS_LEMMA}) is equivalent to the strict 2-functor obtained by restricting the 2-functor $\Prs{B}_{(-)}$ to Cauchy comodules. This follows from the fact that $L_K Y$ is naturally isomorphic to $K \colon \Rep(C) \rightarrow \Prs{B}_C$ (see Theorem~\ref{FREE_COCOMPLETION_THM}).
\end{rmk}

\begin{proof}[Proof of Theorem~\ref{TANNAKIAN_BIADJUNCTION_VMOD_THM}] 
Let $\ca{M}$ be the 2-category from Proposition~\ref{MODULES_2CATEGORY_PROP}. It has Tannaka-Krein objects by Theorem~\ref{VMOD_TK_OBJECTS_THM}. From Theorem~\ref{TANNAKIAN_BIADJUNCTION_2CAT_THM} we know that the functor $\Rep(-)$ has a left biadjoint $L$, and in Remark~\ref{REP_EQUIVALENT_R_RMK} we saw that $\Rep(-)$ is equivalent to the 2-functor $R$ defined in Section~\ref{RIGHT_BIADJOINT_SUBSECTION}.
\end{proof}

\subsection{The counit of the Tannakian biadjunction} 
Since we now know what Tannaka-Krein objects in $\Mod(\ca{V})$ look like, we can give an explicit description of the counit of the Tannakian adjunction for $\Mod(\ca{V})$.

\begin{prop}\label{COUNIT_PROP}
 Let $\ca{V}$ be a cosmos, and let $C$ be a comonad in $\Mod(\ca{V})$. With the notation from Theorem~\ref{VMOD_TK_OBJECTS_THM}, the $C$-component of the counit of the Tannakian adjunction is given by
 \[
  V\varepsilon^K W
 \]
 which is invertible if and only if $\varepsilon^K W$ is invertible.
\end{prop}

\begin{proof}
 This is an immediate consequence of the explicit description of Tannaka-Krein objects in $\Mod(\ca{V})$ (see Theorem~\ref{VMOD_TK_OBJECTS_THM}) and the construction of the pseudonatural equivalence from Theorem~\ref{TANNAKIAN_BIADJUNCTION_2CAT_THM}. The fact about invertibility follows because the forgetful functor $V$ reflects isomorphisms.
\end{proof}

\subsection{The semantics-structure adjunction}\label{SEMANTICS_STRUCTURE_ADJUNCTION}
 In any 2-category with Eilenberg-Moore objects, there exists an analogous adjunction between comonads on an object and maps into that object, called the semantics-structure adjunction (see \cite{STREET_FTM} for the general case and \cite{DUBUC} for the case of $\VCAT$). This adjunction is closely related to the Tannakian adjunction. Given an adjunction $L \dashv R \colon \ca{A} \rightarrow \ca{B}$ in $\VCAT$, we get a comonad $(LR,L\eta R, \varepsilon)$ on $\ca{B}$. The category $\ca{B}_{LR}$ of $LR$-comodules (also known as coalgebras) is the Eilenberg-Moore object of $LR$ in $\VCAT$ (see Proposition~\ref{VCAT_EM_OBJECTS_PROP}).
 
 \begin{dfn}\label{COMPARISON_FUNCTOR_DFN}
 Let $L \dashv R \colon \ca{A} \rightarrow \ca{B}$ be an adjunction in $\VCAT$. The \emph{comparison functor}
  \[
   J \colon \ca{A} \rightarrow \ca{B}_{LR}
  \]
 is the functor which sends an object $C \in \ca{C}$ to the $LR$-comodule $(LC,L\eta_C)$, that is, $J$ is the lifted functor corresponding to the $LR$-coaction $L\eta$ under the equivalence $T$ from Definition~\ref{EM_OBJECTS_DFN} and Proposition~\ref{VCAT_EM_OBJECTS_PROP}.
 \end{dfn}

 The functor $J$ is in fact the unit of the semantics-structure adjunction (see \cite{DUBUC}), which has been studied quite extensively. For example, Beck's monadicity theorem gives necessary and sufficient conditions for $J$ to be an equivalence. The next proposition shows that the unit of the Tannakian biadjunction is a composite of a Yoneda embedding with a comparison functor. Thus we can use the existing body of knowledge about $J$ to prove facts about the unit of the Tannakian adjunction.

\subsection{The unit of the Tannakian biadjunction} The unit of the Tannakian biadjunction for $\Mod(\ca{V})$ can be written in terms of the unit of the semantics-structure adjunction.

\begin{prop}\label{UNIT_PROP}
 Let $w \colon \ca{A} \rightarrow \overline{\ca{B}}$ be a $\ca{V}$-functor, and let $L_w$ be the left Kan extension of $w$ along the Yoneda embedding. Let $C$ be the comonad $L_w \widetilde{w}$ and let $N$ be the corestriction of the composite
 \[
 \xymatrix{ \ca{A} \ar[r]^-{Y} & \Prs{A} \ar[r]^-{J} & (\Prs{B})_C }
 \]
 to $\Rep(C)$, where $J$ is the comparison functor. Then the $w$-component of the unit of the Tannakian adjunction is $L_{YN}$, the left Kan extension of the composite
 \[
 \xymatrix{ \ca{A} \ar[r]^-{N} & \Rep(L_w \cdot \widetilde{w}) \ar[r]^-{Y} & \mathcal{P}{\Rep(L_w \cdot \widetilde{w}\bigr)}}
 \]
 along the Yoneda embedding of $\ca{A}$. In particular, if $\ca{A}$ is Cauchy complete, then the unit of the Tannakian adjunction is an equivalence in $\Mod(\ca{V})$ if and only if $N$ is an equivalence of $\ca{V}$-categories.
\end{prop}

\begin{proof}
 Since precomposing with the Yoneda embedding gives a fully faithful functor
 \[
  \Cocts[\Prs{A},(\Prs{\ca{B}})_C ] \rightarrow [\ca{A},(\Prs{B})_C ] 
 \]
 there is a unique $\ca{V}$-natural isomorphism $\phi \colon L_K \cdot L_{YN} \Rightarrow J$ such that the equation
\[
 \vcenter{\hbox{


}}
\]
 holds. We let $\sigma \colon (V \cdot L_K) \cdot L_{YN} \Rightarrow L_w$ be the 2-cell
 \[
 \sigma \defl \vcenter{\hbox{

%
}}
 \]
 The definition of $\theta_{L_w,C}$ (see the proof of Theorem~\ref{TANNAKIAN_BIADJUNCTION_2CAT_THM}) and the definition of $J$ (see Definition~\ref{COMPARISON_FUNCTOR_DFN}) imply that the equation
 \[
  \theta_{L_w,C}(L_{YN},\sigma)=
  \vcenter{\hbox{

%
}}=\id_{L_w \widetilde{w}}
 \]
 holds. This shows that the unit of the Tannakian adjunction is indeed $L_{YN}$. If $\ca{A}$ is Cauchy complete it follows from Lemma~\ref{MAPS_LEMMA} that $L_{YN}$ is an equivalence if and only if $N$ is.
\end{proof}

\section{The recognition theorem in \texorpdfstring{$\Mod(\ca{V})$}{Mod(V)}}\label{RECOGNITION_SECTION}
\subsection{Statement of the theorem}
 In this section we study the unit of the Tannakian biadjunction. We fix a cosmos $\ca{V}$ with the property that Cauchy completions of small categories are again small. All cosmoi of (differential graded or graded) $R$-modules have this property. More generally, all cosmoi which are locally presentable as closed categories have this property (see Definition~\ref{LOCALLY_FINITELY_PRESENTABLE_DFN} for the notion and \cite{JOHNSEN} for a proof).
 
\begin{thm}\label{RECOGNITION_THM}
 Let $\ca{A}$ be a small $\ca{V}$-category, and let $w \colon \ca{A} \rightarrow \overline{\ca{B}}$ be a $\ca{V}$-functor. The $(\ca{A},w)$-component of the unit of the Tannakian biadjunction is an equivalence of $\ca{V}$-categories if the following hold:
\begin{enumerate}
\item[i)]
 The functor $w \colon \ca{A} \rightarrow \Prs{B}$ reflects isomorphisms;
\item[ii)]
 The left adjoint $L_w \colon \Prs{A} \rightarrow \Prs{B}$ preserves equalizers of $L_w$-cosplit pairs (see Definition~\ref{COSPLIT_EQUALIZER_DFN}); and
\item[iii)]
 For all small $\ca{V}$-categories $\ca{D}$, all weights $H \colon \ca{D}^{\op} \rightarrow \ca{V}$ and all $\ca{V}$-functors $G \colon \ca{D} \rightarrow \ca{A}$ such that the weighted colimit $H \star w G$ lies in the subcategory $\overline{\ca{B}} \subseteq \Prs{B}$, the weighted colimit $H\star G$ exists and is preserved by $w \colon \ca{A} \rightarrow \Prs{B}$.
\end{enumerate}
 Moreover, if $\mathbf{\Phi}$ is a class of weights with the property that for each $X \in \Prs{A}$ there is a weight $H \colon \ca{D}^{\op} \rightarrow \ca{V}$ in $\mathbf{\Phi}$ and a $\ca{V}$-functor $G \colon \ca{D} \rightarrow \ca{A}$ such that $X \cong H \star YG$, then $N \colon \ca{A} \rightarrow \Rep\bigl(L(w)\bigr)$ is an equivalence if i) and ii) hold and iii) holds for all weights in the class $\mathbf{\Phi}$.
\end{thm}
 
 Note that condition iii) is a necessary condition: the forgetful functor from the category of all comodules to $\Prs{B}$ creates colimits. Similarly, condition i) is a necessary condition because the forgetful functor from comodules to $\Prs{B}$ reflects isomorphisms. The question whether or not ii) is a necessary condition for the unit to be an equivalence is open.
 
 \subsection{The enriched Beck monadicity theorem}
 As we saw in Proposition~\ref{UNIT_PROP}, the unit of the Tannakian biadjunction is closely related to the comparison functor $J$ (see Definition~\ref{COMPARISON_FUNCTOR_DFN}). The following proposition summarizes the facts about $J$ that we'll need in order to prove both the general and the specialized recognition theorems. In order to state it we need to introduce the following concept.
 
\begin{dfn}\label{COSPLIT_EQUALIZER_DFN}
 A \emph{cosplit equalizer} in an unenriched category $\ca{C}_0$ is a diagram of the form
\[
\turnradius{5pt}
\xymatrix{
E \ar[r]^{s} & A \ar@<0.5ex>[r]^u \ar@<-0.5ex>[r]_v \ar `u[l] `[l]+<0.4pt,8pt>_p [l]+<0pt,5pt> &  B \ar `u[l] `[l]+<2.4pt,8pt>_{q} [l]+<2pt,5pt>\\
}
\]
 such that the equalities $us=vs$, $ps=\id_E$, $qu=\id_A$ and $qv=sp$ hold. These identities imply that $s$ exhibits $E$ as equalizer of $u$ and $v$. If $F \colon \ca{A}_0 \rightarrow \ca{C}_0$ is a functor, then we say that a pair $f,g\colon A \rightarrow A^\prime$ in $\ca{A}_0$ is \emph{$F$-cosplit} if there is a cosplit equalizer in $\ca{C}_0$ as above with $u=Ff$, $v=Fg$.
 \end{dfn}

\begin{prop}\label{BECK_PROP}
 Let $L \dashv R \colon \ca{A} \rightarrow \ca{B}$ be an adjunction in $\VCAT$. Then the following hold:
 \begin{enumerate}
 \item[i)] If $\ca{A}$ has equalizers of $L$-cosplit pairs, then the comparison functor
 \[
 J \colon \ca{A} \rightarrow \ca{B}_{LR}
 \]
 (see Definition~\ref{COMPARISON_FUNCTOR_DFN}) has a right adjoint $E$.
 \item[ii)] If, in addition, $L$ preserves equalizers of $L$-cosplit pairs, then $E$ is fully faithful.
 \item[iii)] \emph{(Beck's monadicity theorem)} The comparison functor is an equivalence if, in addition, $L$ reflects isomorphisms.
 \end{enumerate}
\end{prop}

\begin{proof}
 In the case where $\ca{V}=\Set$ and $\ca{B}$ is a presheaf category, this can be found in \cite[\S~3]{APPLEGATE_TIERNEY}. For arbitrary $\ca{V}$ a proof of iii) can be found in \cite[Theorem~2.II.1]{DUBUC}, and parts i) and ii) are almost implicit there. We say that a diagram consisting of $\ca{V}$-functors has a certain property \emph{pointwise} if the diagram evaluated in any object of the domain has the corresponding property. Let $V \colon \ca{B}_{LR} \rightarrow \ca{B}$ be the forgetful functor. The pair
\[
\xymatrix{RV \ar@<0.5ex>[r]^-{\eta RV} \ar@<-0.5ex>[r]_-{R\rho} & RLRV}
\]
 in the $\ca{V}$-functor category $[\ca{B}_{LR},\ca{A}]$ is pointwise $L$-cosplit since the diagram
 \[
\turnradius{7pt}
\xymatrix{
V \ar[r]^-{\rho} & LRV \ar@<0.5ex>[r]^-{L\eta RV} \ar@<-0.5ex>[r]_-{LR \rho} \ar `u[l] `[l]+<0.4pt,8pt>_-{\varepsilon V} [l]+<0pt,6pt> &  LRLRV \ar `u[l] `[l]+<2.4pt,8pt>_-{\varepsilon LRV} [l]+<2pt,6pt>\\
}  
 \]
 is a pointwise cosplit equalizer in $[\ca{B}_{LR},\ca{B}]$. Let $e \colon E \rightarrow RV$ be the pointwise equalizer of this pointwise $L$-cosplit pair, which exists by assumption i). The $\ca{V}$-natural transformation $\eta$ equalizes $\eta RVJ=\eta RL$ and $R \rho J=RL \eta$. Thus there is a unique natural transformation $\pi \colon \id \rightarrow EJ$ such that $eJ \cdot \pi=\eta$. Similarly we get a unique $\ca{V}$-natural transformation $\xi \colon VJE=LE \rightarrow V$ with $\rho \cdot \xi=Le$. Using the fact that $\rho$ is a split monomorphism one can show that $\xi$ is a morphism of coactions, and that $\pi$ and the lift $\overline{\xi} \colon JE \rightarrow \id$ satisfy the triangle identities. Thus $E$ is right adjoint to $J$.
 
 To see ii) note that a right adjoint is fully faithful if and only if the counit is an isomorphism. But $\overline{\xi}$ is an isomorphism if and only if $\xi$ is (because $V$ reflects isomorphisms), and $\xi$ is an isomorphism if and only if $L$ preserves this particular equalizer of a pointwise $L$-cosplit pair.

 This also gives us an alternative proof of Beck's monadicity theorem: from the triangle identities we find that $J\pi$ is an isomorphism, and from the fact that $L=VJ$ reflects isomorphisms we conclude that both the unit and the counit of the adjunction $J \dashv E$ are isomorphisms. This shows that $J$ is an equivalence if iii) holds. 
\end{proof}

\subsection{The proof of the recognition theorem}\label{PROOF_STRATEGY_SECTION}
We first give a brief outline of our strategy of proof. From Proposition~\ref{BECK_PROP} we know that the unit of the Tannakian adjunction is essentially given by the composite $JY$ of the Yoneda embedding and the comparison functor $J$ associated to the adjunction $L_w \dashv \widetilde{w}$ (see Definition~\ref{COMPARISON_FUNCTOR_DFN}). From assumption ii) in Theorem~\ref{RECOGNITION_THM} and from Proposition~\ref{BECK_PROP} we know that $J$ is left adjoint to a fully faithful functor $E$. It is a purely formal consequence that the restriction of $J$ to the image of $E$ is fully faithful. Thus, in order to show that the unit of the Tannakian biadjunction is fully faithful, it suffices to check that the Yoneda embedding factors through the image of $E$.

Afterwards we will use assumption iii) of Theorem~\ref{RECOGNITION_THM} to show that the unit is essentially surjective. The following proposition and its corollary summarize some of the facts we will need to show that the Yoneda embedding factors through the image of $E$.

\begin{prop}\label{REPLETE_IMAGE_PROP}
 Let $L \dashv R \colon \ca{A} \rightarrow \ca{B}$ be an adjunction in $\VCAT$, and assume that $\ca{A}$ has equalizers of $L$-cosplit pairs. Let $J \dashv E \colon \ca{A} \rightarrow \ca{B}_{LR}$ be as in Proposition~\ref{BECK_PROP}, and let $\pi$ be the unit of the adjunction $J\dashv E$. Let $\ca{C}$ be the full subcategory of $\ca{A}$ consisting of objects $X$ for which $\pi_X$ is an isomorphism.
 
 If $L$ preserves equalizers of $L$-cosplit pairs, then $L\pi$ is an isomorphism, $\ca{C}$ is the replete image of $E$ (the full subcategory consisting of objects isomorphic to objects in the image of $E$), and $J$ restricts to an equivalence $\ca{C} \rightarrow \ca{B}_{LR}$. Moreover, the category $\ca{C}$ is closed under retracts.
\end{prop}

\begin{proof}
 From one of the triangle identities and part ii) of Proposition~\ref{BECK_PROP} we know that $J\pi$ is invertible. Thus $L\pi=VJ\pi$ is invertible. The other triangle identity shows that $\pi E$ is invertible. This shows that $E$ factors through $\ca{C}$. The unit of $J \dashv E$ restricted to $\ca{C}$ is invertible by definition, the counit is invertible by part ii) of Proposition~\ref{BECK_PROP}. Closure under retracts follows from naturality of $\pi$ and the fact that isomorphisms are closed under retracts in the arrow category.
\end{proof}

\begin{cor}\label{FULLY_FAITHFUL_COR}
 Let $\ca{A}$ be Cauchy complete, let $w \colon \ca{A} \rightarrow \overline{\ca{B}}$ be a $\ca{V}$-functor, and let $J\colon \Prs{A} \rightarrow \Prs{B}_{L_w \widetilde{w}}$ be the comparison functor corresponding to the adjunction $L_w \dashv \widetilde{w} \colon \Prs{A} \rightarrow \Prs{B}$. Let $\ca{C}$ be as in Proposition~\ref{REPLETE_IMAGE_PROP}. 
 
 If the Yoneda embedding factors through $\ca{C}$, and if $L_w$ preserves equalizers of $L_w$-cosplit pairs, then the $w$-component of the unit of the Tannakian adjunction is fully faithful. In that case it is essentially surjective if and only if every object $X$ in $\ca{C}$ with the property that $L_w(X)$ lies in $\overline{\ca{B}}$ is in the essential image of the Yoneda embedding.
\end{cor}

\begin{proof}
 The unit of the Tannakian adjunction is fully faithful if and only if the composite $JY$ of the comparison functor with the Yoneda embedding is (see Proposition~\ref{UNIT_PROP}), and the latter follows from Proposition~\ref{REPLETE_IMAGE_PROP} if $Y$ factors through $\ca{C}$.
 
 The statement about essential surjectivity follows from the fact that the restriction of $J$ to $\ca{C}$ is an equivalence and the description of the unit in Proposition~\ref{UNIT_PROP}.
\end{proof}

\begin{proof}[Proof of Theorem~\ref{RECOGNITION_THM}]
 Note that it suffices to prove the statement where assumption iii) involves a class $\mathbf{\Phi}$ of weights. Indeed, we can consider the Yoneda embedding $Y \colon \ca{A} \rightarrow \Prs{A}$ as a diagram on $\ca{A}$, and any $H \in \Prs{A}$ as a weight $\ca{A}^{\op} \rightarrow \ca{V}$. In \cite[Formula~(3.17)]{KELLY_BASIC} it is shown that the weighted colimit $H \star Y$ is isomorphic to $H$. Thus we can always let $\mathbf{\Phi}$ be the class of \emph{all} weights $\ca{A}^{\op} \rightarrow \ca{V}$.

 We have to check that the conditions from Corollary~\ref{FULLY_FAITHFUL_COR} hold. Cauchy completeness of $\ca{A}$ follows immediately from condition iii) and the characterization of Cauchy complete categories in terms of absolute colimits (see \cite{STREET_ABSOLUTE}). From ii) we know that $L_w$ preserves equalizers of $L_w$-cosplit pairs, and condition iii) implies that all objects $X$ of $\ca{C}$ with $L_w(X) \in \overline{\ca{B}}$ lie in the essential image of the Yoneda embedding. It only remains to check that the Yoneda embedding of $\ca{A}$ factors through the full subcategory $\ca{C}$ of $\Prs{A}$ from Proposition~\ref{REPLETE_IMAGE_PROP}. 
 
 Fix an object $A \in \ca{A}$, and let $f \colon \ca{A}(-,A) \rightarrow F$ be the $\ca{A}(-,A)$-component of the unit of the adjunction $J\dashv E$ defined in Proposition~\ref{BECK_PROP}. From Proposition~\ref{REPLETE_IMAGE_PROP} we know that $F$ lies in $\ca{C}$ and that $L_w f$ is invertible. We will use this fact to show that $f$ is a split monomorphism. This implies that $\ca{A}(A,-)$ is a retract of $F \in \ca{C}$, which completes our proof because $\ca{C}$ is closed under retracts.

 By assumption, there is a weight $H \colon \ca{D}^{\op} \rightarrow \ca{V}$ in the class $\mathbf{\Phi}$ and a functor $G \colon \ca{D} \rightarrow \ca{A}$ such that $F \cong H\star YG$, so we may as well assume that $f$ is a morphism from $\ca{A}(-,A)$ to $H\star YG$. 

 By the remarks in Section~\ref{WEIGHTED_COLIMIT_SECTION} we have a chain of isomorphisms
\[
 \xymatrix@C=25pt{w(A) \ar[r]^-{\alpha_w} & L_w YA \ar[r]^-{L_w f} & L_w (H\star YG) \ar[r]^-{\widehat{L_w}^{-1}} & H\star L_w YG \ar[r]^-{H\star \alpha_w G}  & H\star w G},
\]
 which shows that $H\star w G$ lies in $\overline{\ca{B}}$. By iii) it follows that the colimit $H \star G \in \ca{A}$ exists, and that $\widehat{w} \colon H\star w G \rightarrow w(H\star G)$ is an isomorphism. From $ w \cong L_w Y$ we conclude that $\widehat{L_w Y} \colon H\star L_w Y G \rightarrow L_w Y(H\star G)$ is an isomorphism. An easy application of the Yoneda Lemma and the definition of the comparison morphism (see Section~\ref{WEIGHTED_COLIMIT_SECTION}) show that $\widehat{L_w Y}=L_w (\widehat{Y}) \cdot \widehat{L_w}$, hence $L_w (\widehat{Y})$ is an isomorphism. Since the Yoneda embedding is full, the composite $\widehat{Y} \cdot f \colon \ca{A}(-,A) \rightarrow \ca{A}(-,H\star G)$ is of the form $Yh=\ca{A}(-,h)$ for a unique morphism $h \colon A \rightarrow H\star G$. Note that
 \[
  L_w Y(h)=L_w(\widehat{Y}) \cdot L_w (f)
 \]
 is an isomorphism, by the above argument and by Proposition~\ref{REPLETE_IMAGE_PROP} (recall that $f$ is a component of the unit of $J \dashv E$). Since $w \cong L_w Y$ it follows that $w(h)$ is an isomorphism, and condition i) implies that $h$ itself is an isomorphism. Thus
\[
 \xymatrix{\ca{A}(-,A) \ar@{=}[r] \ar[d]_{f} & \ca{A}(-,A) \\
H\star YG \ar[d]_{\widehat Y} \\
\ca{A}(-,H\star G) \ar[ruu]_{\ca{A}(-,h^{-1})}}
\]
 is commutative, which shows that $\ca{A}(-,A)$ is indeed a retract of an object in $\ca{C}$, hence it lies in $\ca{C}$.
\end{proof}
\section{Cosmoi with dense autonomous generator}\label{DAG_RECOGNITION_SECTION}

 When working with additive $R$-linear categories for some commutative ring $R$, the notion of weighted colimits is generally not needed. For example, any $R$-linear presheaf is a conical colimit of representable functors. The reason for this is that the finitely generated free modules form a dense autonomous generator of the cosmos of $R$-modules. We explain what this means and then show how the questions of when the unit and counit of the Tannakian biadjunction are equivalences simplify for cosmoi with dense autonomous generator.
 
 \subsection{Dense \texorpdfstring{$\ca{V}$}{V}-functors and dense \texorpdfstring{$\ca{V}$}{V}-categories}\label{DENSE_FUNCTORS_SECTION}
 The notion of a dense $\ca{V}$-functor is motivated as follows (see \cite[Chapter~5]{KELLY_BASIC}). A continuous map $f\colon X\rightarrow Y$ between Hausdorff topological spaces has dense image if and only if a continuous map $g\colon Y\rightarrow Z$ into another Hausdorff space is uniquely determined by the composite $gf$. A dense functor has an analogous property, where ``continuous map'' is replaced by ``cocontinuous functor.''
 
 \begin{dfn}
 A $\ca{V}$-functor $K \colon \ca{A} \rightarrow \ca{B}$ is \emph{dense} if precomposing with $K$ induces a fully faithful functor $\ca{A}\mbox{-}\Cocts[\ca{B},\ca{C}] \rightarrow [\ca{A},\ca{C}]$ for every $\ca{V}$-category $\ca{C}$, where $\ca{A}\mbox{-}\Cocts[\ca{B},\ca{C}]$ stands for the full subcategory of those $\ca{V}$-functors which preserve those weighted colimits whose weights have domain $\ca{A}^{\op}$.
 
 A full subcategory $\ca{A}$ of a $\ca{V}$-category $\ca{C}$ is called \emph{$\ca{V}$-dense} if the inclusion $\ca{A} \rightarrow \ca{C}$ is a dense $\ca{V}$-functor. It is called \emph{$\Set$-dense} if the inclusion $\ca{A}_0 \rightarrow \ca{C}_0$ of the underlying unenriched categories is a $\Set$-dense functor.
 \end{dfn}
 
 Note that $\ca{A}\mbox{-}\Cocts[\ca{B},\ca{C}]$ contains $\Cocts[\ca{B},\ca{C}]$ if $\ca{A}$ is small. It follows that for any dense $\ca{V}$-functor $K \colon \ca{A} \rightarrow \ca{B}$ with small domain, for any two cocontinuous $\ca{V}$-functors $F,G \colon \ca{B} \rightarrow \ca{C}$ and for any $\ca{V}$-natural transformation $\alpha \colon FK \Rightarrow GK$ there is a unique $\ca{V}$-natural transformation $\beta \colon F \Rightarrow G$ such that for all $A\in \ca{A}$, $\beta_{KA}=\alpha_A$. 
 
 \begin{rmk}
 The notions of $\ca{V}$-density and $\Set$-density are generally quite different. For example, for any commutative ring $R$, the free $R$-module on one generator is $\Mod_R$-dense in $\Mod_R$, but not $\Set$-dense. The free $R$-module on two generators on the other hand is $\Set$-dense in $\Mod_R$.
 \end{rmk}

 One of the most important examples of a dense functor is the Yoneda embedding $Y\colon \ca{A} \rightarrow \Prs{A}$ of a small $\ca{V}$-category $\ca{A}$ (see \cite[Proposition~5.16]{KELLY_BASIC}). A $\ca{V}$-functor $K \colon \ca{A} \rightarrow \ca{C}$ is $\ca{V}$-dense if and only if any object $C$ in $\ca{C}$ is a canonical weighted colimit of $K$ (see \cite[Theorem~5.1]{KELLY_BASIC}).

\subsection{Dense autonomous generators}
For certain cosmoi $\ca{V}$ and $\ca{V}$-categories $\ca{A}$, $\ca{C}$, a $\ca{V}$-functor $K \colon \ca{A} \rightarrow \ca{C}$ is $\ca{V}$-dense if and only if $K_0 \colon \ca{A}_0 \rightarrow \ca{C}_0$ is $\Set$-dense. If this is the case for the Yoneda embedding (which is always $\ca{V}$-dense), we find that every presheaf is a \emph{conical} colimit of representables. This allows us to drastically reduce the class of colimits that have to be considered in condition iii) of Theorem~\ref{RECOGNITION_THM}. The notion of a dense autonomous generator allows us to identify a large class of $\ca{V}$-categories for which $\Set$-density and $\ca{V}$-density coincide.

\begin{dfn}\label{DAG_DFN}
 An essentially small full subcategory $\ca{X} \subseteq \ca{V}$ of a cosmos $\ca{V}$ is called a \emph{dense autonomous generator} if $\ca{X}$ consists of objects with duals, is closed under the tensor product in $\ca{V}$ and under the formation of duals, and is $\Set$-dense in $\ca{V}$.
\end{dfn}

\begin{example}
 Let $R$ be a commutative ring. The full subcategory of finitely generated free $R$-modules is a $\Set$-dense subcategory which is closed under the tensor product in $\Mod_R$ and under the formation of duals.
\end{example}
 
 In the $R$-linear context it is not hard to prove the following result directly. We defer the proof in general to Appendix~\ref{DENSITY_IN_COSMOI_WITH_DAG_APPENDIX}.

\begin{cit}[Theorem~\ref{DENSITY_THM}]
 Let $\ca{V}$ be a cosmos which has a dense autonomous generator $\ca{X}$. Let $\ca{A}$ be an $\ca{X}$-tensored $\ca{V}$-category, and let $\ca{C}$ be a $\ca{V}$-category which is cotensored\footnote{Cotensors are the dual notion of tensors.}. A $\ca{V}$-functor $K \colon \ca{A} \rightarrow \ca{C}$ is $\ca{V}$-dense if and only if the underlying ordinary functor $K_0 \colon \ca{A}_0 \rightarrow \ca{C}_0$ is $\Set$-dense.
\end{cit}

\begin{example}
 Any additive $R$-linear category has in particular tensors with finitely generated free $R$-modules: we have $R^n\odot A \cong \oplus_{i=1}^n A$. Thus a $\ca{V}$-functor from an additive $R$-linear category to a complete $R$-linear category is $\Ab$-dense if and only if it is $\Set$-dense.
\end{example}

\begin{thm}\label{CONICAL_COLIMIT_THM}
 Let $\ca{V}$ be a cosmos which has a dense autonomous generator $\ca{X}$. Let $\ca{A}$ be a small $\ca{X}$-tensored $\ca{V}$-category. Fix a presheaf $F \in \Prs{A}$, and let $\ca{A}\slash F$ be the category of representable functors over $F$. Then $F$ is the conical colimit of the domain functor $D \colon \ca{A} \slash F \rightarrow \Prs{A}$.
\end{thm}

\begin{proof}
 The fact that the Yoneda embedding is always $\ca{V}$-dense and the completeness of $\Prs{A}$ imply that the conditions of Proposition~\ref{DENSITY_THM} are satisfied. It follows that $F$ is the ordinary colimit of the tautological cocone on $D \colon \ca{A}\slash F \rightarrow \Prs{A}_0$. But $\Prs{A}$ is cotensored, hence the notion of conical colimit and ordinary colimit coincide (see \cite[\S~3.8]{KELLY_BASIC}).
\end{proof}

\subsection{Locally presentable categories}
 The notion of a dense autonomous generator allows us to simplify condition~iii) of Theorem~\ref{RECOGNITION_THM}. In order to simplify condition~ii), we will use the notion of \emph{finite} limits in the enriched context introduced by  Max Kelly in \cite{KELLY_FINLIM}. This only makes sense if our cosmos is locally finitely presentable as a closed category. Recall that a subset $\ca{G} \subseteq \ca{C}_0$ of an unenriched category $\ca{C}_0$ is called a \emph{generator} if the representable functors $\ca{C}_0(G,-)$ are jointly faithful. A generator is called \emph{strong} if for each $C \in \ca{C}_0$ and each proper subobject $A$ of $C$ there exists a $G \in \ca{G}$ and a morphism $G \rightarrow C$ which does not factor through $A$.
 
 \begin{dfn}\label{LOCALLY_FINITELY_PRESENTABLE_DFN}
 An object $C$ of an (unenriched) category $\ca{C}_0$ is called \emph{finitely presentable} if the functor $\ca{C}_0(C,-) \colon \ca{C}_0 \rightarrow \Set$ preserves filtered colimits. A category $\ca{C}_0$ is called \emph{locally finitely presentable} if it is cocomplete and if there exists a strong generator consisting of finitely presentable objects.
  
 A cosmos $\ca{V}$ is \emph{locally finitely presentable as a closed category} if $\ca{V}_0$ is locally finitely presentable, the unit $I$ is finitely presentable and finitely presentable objects are closed under the tensor product in $\ca{V}$. 
 \end{dfn}

 \begin{prop}\label{COSMOI_WITH_DAG_ARE_LFP_PROP}
 A cosmos $\ca{V}$ with a dense autonomous generator $\ca{X}$ is locally finitely presentable as a closed category if and only if the unit $I$ is finitely presentable.
 \end{prop}

\begin{proof}
 Necessity is obvious from the definition. For $X\in \ca{X}$ we have
\[
 \ca{V}_0(X,-)\cong \ca{V}_0(I,[X,-])\cong\ca{V}_0(I,X^\vee \otimes -),
\]
 which preserves filtered colimits by our assumption on the unit object $I$. Thus all the objects of $\ca{X}$ are finitely presentable. But $\ca{X}_0$ is $\Set$-dense, so it is in particular a strong generator, which shows that $\ca{V}_0$ is indeed locally finitely presentable.
 
 The unit object is finitely presentable by assumption. A slight generalization of the above argument shows that the tensor product $Y\otimes X$ of a finitely presentable object $Y$ with an object $X \in \ca{X}$ is finitely presentable. It remains to show that $Y \otimes Z$ is finitely presentable whenever $Z$ is. 
 
 Since $\ca{X}$ is $\Set$-dense, $Z$ is a colimit of objects in $\ca{X}$. Every colimit is the filtered colimit of its finite subcolimits (see \cite[Proposition~2.13.7]{BORCEUX}), so the identity of $Z$ factors through a finite colimit of objects of $\ca{X}$. Thus $Z$ is a retract of a finite colimit of objects of $\ca{X}$. Since finitely presentable objects are closed under finite colimits and retracts, and $Y\otimes-$ preserves colimits and retracts it follows that $Y\otimes Z$ is finitely presentable.
\end{proof}

\subsection{The recognition theorem when \texorpdfstring{$\ca{V}$}{V} has a DAG}
The proof of the recognition theorem for cosmoi with dense autonomous generators follows the strategy outlined in Section~\ref{PROOF_STRATEGY_SECTION}.

\begin{dfn}\label{OMEGA_RIGID_DFN}
 Let $\ca{A}$, $\ca{B}$ be small $\ca{V}$-categories, and let $w \colon \ca{A} \rightarrow \Prs{B}$ be a $\ca{V}$-functor, with underlying ordinary functor $w_0 \colon \ca{A}_0 \rightarrow \Prs{B}_0$. Let $\ca{D}$ be a small ordinary category. A diagram $D \colon \ca{D} \rightarrow \ca{A}_0$ is called \emph{$w$-rigid} if the colimit of $w_0 D \colon \ca{D} \rightarrow \Prs{B}_0$ lies in the Cauchy completion $\overline{\ca{B}}$ of $\ca{B}$.
\end{dfn}

\begin{thm}\label{DENSE_RECOGNITION_THM}
 Let $\ca{V}$ be a cosmos which has a dense autonomous generator $\ca{X}$, and assume that $I$ is finitely presentable. Let $\ca{A}$ be an $\ca{X}$-tensored small $\ca{V}$-category, let $\ca{B}$ be a small $\ca{V}$-category, and let $w \colon \ca{A} \rightarrow \Prs{B}$ be a $\ca{V}$-functor whose image is contained in $\overline{\ca{B}}$. Then the $w$-component of the unit of the Tannakian adjunction is an equivalence if:
\begin{enumerate}
 \item[i)]
 The functor $w_0$ reflects isomorphisms;
 \item[ii)]
 For each object $B\in\ca{B}$, the category $\el \bigl(V (\ev_B w)_0\bigr)$ of elements of the functor $V(\ev_B w)_0 \colon \ca{A}_0 \rightarrow \Set$ is cofiltered, where $\ev_B \colon \Prs{B} \rightarrow \ca{V}$ denotes the evaluation functor; and
 \item[iii)]
 The category $\ca{A}_0$ has colimits of $w$-rigid diagrams (see Definition~\ref{OMEGA_RIGID_DFN}), and $w_0$ preserves them.
\end{enumerate}
If these conditions are satisfied, then the comonad $L(w) \colon \Prs{B} \rightarrow \Prs{B}$ preserves finite limits.
\end{thm}

\begin{proof}
 We apply Theorem~\ref{RECOGNITION_THM}. Condition i) coincides with condition i) in Theorem~\ref{RECOGNITION_THM}. To see that iii) holds, we first note that $\ca{A}$ has cotensors with objects in $\ca{X}$: since any $X\in \ca{X}$ has a dual $X^\vee$, any $\ca{V}$-functor preserves tensors with $X$ (see \cite{STREET_ABSOLUTE}). In particular we have natural isomorphisms
\[
 \ca{A}(B,X^\vee \odot A) \cong X^\vee \otimes \ca{A}(B,A) \cong[X,\ca{A}(B,A)],
\]
 which shows that $\ca{A}$ is $\ca{X}$-cotensored. Since $\ca{X}\subseteq \ca{V}$ is $\Set$-dense, this implies that the notion of conical colimit in $\ca{A}$ coincides with the notion of ordinary colimit in $\ca{A}_0$ (cf.\ \cite[\S~3.8]{KELLY_BASIC}). If we let $\mathbf{\Phi}$ be the class of conical weights, the above observation, Theorem~\ref{CONICAL_COLIMIT_THM} and iii) imply that condition iii) of Theorem~\ref{RECOGNITION_THM} is satisfied. Therefore it only remains to check that ii) implies that $L_w \colon \Prs{A} \rightarrow \Prs{B}$ preserves the necessary equalizers. 
 
 Since $\ca{V}$ is locally finitely presentable as a closed category (see Proposition~\ref{COSMOI_WITH_DAG_ARE_LFP_PROP}) we can use the theory of flat functors developed by Kelly \cite{KELLY_FINLIM}. We will in fact show that $L_w$ preserves all \emph{finite} weighted limits (see \cite[\S~4]{KELLY_FINLIM} for a definition of finite weighted limits).

 To see this, we first note that the functor
\[
(\ev_B)_{B\in \ca{B}} \colon \Prs{B} \rightarrow \prod_{B\in \ca{B}} \ca{V} 
\]
 preserves and reflects all limits and all colimits because both limits and colimits in $\Prs{B}$ are computed pointwise. Hence, it suffices to check that for a fixed $B\in \ca{B}$, the functor $\ev_B L_w \colon \Prs{A} \rightarrow \ca{V}$ preserves finite limits. Using the terminology of \cite{KELLY_FINLIM}, we need to show that $\ev_B L_w$ is \emph{left exact}. Since the functor $(\ev_B)_{B\in \ca{B}}$ preserves all colimits, it also preserves left Kan extensions (see \cite[Proposition~4.14]{KELLY_BASIC}). Thus, $\ev_B L_w$ is naturally isomorphic to the left Kan extension of $\ev_B w$ along the Yoneda embedding $Y \colon \ca{A} \rightarrow \Prs{A}$. This reduces the problem to showing that $\ev_B w$ is \emph{flat}. Because $\ca{V}$ is locally finitely presentable we can apply \cite[\S~6.3]{KELLY_FINLIM}, that is, it suffices to check that $\ev_B w \in [\ca{A},\ca{V}]$ is a filtered (conical) colimit of representable functors.

 We have seen that $\ca{A}$ is $\ca{X}$-cotensored, which means that $\ca{A}^{\op}$ is $\ca{X}$-tensored. Theorem~\ref{CONICAL_COLIMIT_THM}, applied to the $\ca{V}$-category $\ca{A}^{\op}$ and the contravariant Yoneda embedding $Y^\prime \colon \ca{A}^{\op} \rightarrow [\ca{A}, \ca{V}]$, shows that $\ev_B w$ is isomorphic to the conical colimit of the domain functor $\ca{A}^{\op} \slash \ev_B w \rightarrow [\ca{A},\ca{V}]$. Using the weak Yoneda lemma (see \cite[\S~1.9]{KELLY_BASIC}) one can show that the category $\ca{A}^{\op} \slash \ev_B w$ is isomorphic to the opposite of the category of elements of $V(\ev_B w)_0$. The latter is filtered by assumption ii), hence the remarks in \cite[\S~6.3]{KELLY_FINLIM} show that $\ev_B w$ is indeed flat. 
 
 The functor $\widetilde{w}$ is right adjoint, so it preserves all limits, and we have just shown that $L_w$ preserves finite limits. Thus $L(w)=L_w \widetilde{w}$ preserves finite limits.
\end{proof}

\subsection{The counit of the biadjunction when \texorpdfstring{$\ca{V}$}{V} has a DAG}
The reconstruction problem also simplifies in the presence of a dense autonomous generator.

\begin{thm}\label{DENSE_RECONSTRUCTION_THM}
 Let $\ca{V}$ be a cosmos which has a dense autonomous generator $\ca{X}$. Let $\ca{B}$ be a $\ca{V}$-category with small Cauchy completion. Then the $C$-component of the counit of the Tannakian adjunction is an isomorphism if and only if for each $B \in \ca{B}$, the tautological cocone on the domain functor $D \colon \Rep(C) \slash (CB,\delta_B) \rightarrow \Prs{B}_C$ exhibits $(CB,\delta_B)$ as the conical colimit of $D$.
\end{thm}

\begin{proof}
 As in Proposition~\ref{COUNIT_PROP} we write $K$ for the inclusion $\Rep(C) \rightarrow \Prs{B}_C$. Since the forgetful functor from the category of comodules creates colimits it follows that $\ca{A}=\Rep(C)$ is $\ca{X}$-tensored. Indeed, for $B \in \overline{\ca{B}}$, $X \in \ca{X}$, we have $X \odot B \in \overline{\ca{B}}$ because $\ca{X}$ consists of objects with duals.
 
 From Proposition~\ref{COUNIT_PROP} we know that it suffices to check that $V\varepsilon^K W$, a $\ca{V}$-natural transformation between cocontinuous $\ca{V}$-functors, is invertible. Since the Yoneda embedding is dense, it suffices to check this for representable functors. The forgetful functor $V$ reflects isomorphisms, so we only need to show that $\varepsilon^K_{(CB,\delta_B)}$ is invertible for every representable functor $B$.
 
 By Theorem~\ref{CONICAL_COLIMIT_THM}, the presheaf $F = \Prs{B}_C\bigl(K-,(CB,\delta_B)\bigr)$ on $\ca{A}$ is the colimit of the diagram of representable functors over $F$. We have $L_K(\ca{A}(-,A)) \cong KA$ and $L_K$ preserves colimits, so we find that $L_K \widetilde{K}(CB,\delta_B)$ is the colimit of the diagram $D \colon \ca{A}\slash (CB,\delta_B) \rightarrow \Prs{B}_C$ of Cauchy comodules over $(CB,\delta_B)$. It is not hard to check that the comparison morphism induced by the tautological cocone is precisely the $(CB,\delta_B)$-component of $\varepsilon^K$.
\end{proof}

\begin{cor}\label{ENOUGH_CAUCHY_COMODULES_COR}
 Let $R$ be a commutative ring and let $C$ be a flat $R$-coalgebra such that the subcategory of Cauchy $C$-comodules is a generator of the category of all $C$-comodules. Then the $C$-component of the counit of the Tannakian adjunction is an isomorphism.
\end{cor}

\begin{proof}
 Note that the subcategory of Cauchy comodules is a generator if and only if $C$ has \emph{enough} Cauchy comodules, in the sense that for every $C$-comodule $M$ and every element $m\in M$ there is a Cauchy comodule $M^\prime$ and a morphism $M^\prime \rightarrow M$ of $C$-comodules whose image contains $m$ (the  proof of \cite[Proposition~1.4.1]{HOVEY} for Hopf algebroids works equally well for coalgebras). It follows immediately that the comparison morphism $\alpha \colon \colim D \rightarrow (C,\delta)$ is surjective. It remains to show that $\alpha$ is injective.
 
 Any element of $\colim D$ lies in the image of one of its structure maps $\kappa_\phi$ for a Cauchy comodule $\phi\colon M\rightarrow (C,\delta)$ over $(C,\delta)$. Now let $x_1,x_2 \in \colim D$ with $\alpha(x_1)=\alpha(x_2)$. For $i=1,2$ let $\phi_i \colon M_i \rightarrow (C,\delta)$ be finite dimensional comodules over $(C,\delta)$, with elements $x_i^\prime \in M$ satisfying $\kappa_{\phi_i}(x_i^\prime)=x_i$. By definition of the comparison map we have $\alpha \kappa_{\phi_i}=\phi_i$, hence $\phi_1(x_1^\prime)=\phi_2(x_2^\prime)$. It follows that there is an element $y$ in the pullback $N$ of $\phi_1$ and $\phi_2$ which gets sent to $x_1^\prime$ and $x_2^\prime$ under the composite
 \[
 \xymatrix{N \ar[r]^{\mathrm{pr}_i} & M_i \ar[r]^-{\phi_i} & (C,\delta) \rlap{\smash{.}}}  
 \]
 Since $C$ is flat, the pullback $N$ is itself again a comodule. The only problem is that $N$ is not necessarily finitely generated and projective. However, $C$ has enough Cauchy comodules, so there is a Cauchy comodule $M$ with an element $z\in M$ and a morphism of comodules $\phi \colon M \rightarrow N$ with $\phi(z)=y$. The structure map corresponding to the composite $M \rightarrow N \rightarrow (C,\delta)$ sends $z$ to $x_1$ and to $x_2$, hence we must have $x_1=x_2$.
\end{proof}

 \begin{cor}[{\cite[\S~5.13]{WEDHORN}}]
 Let $C$ be a flat coalgebra over a hereditary Noetherian ring $R$. Then the $C$-component of the counit of the Tannakian adjunction is an isomorphism. 
 \end{cor}
 
 \begin{proof}
 It suffices to show that the subcategory of Cauchy comodules of $C$ is a generator. First note that every element $m$ of a comodule $M$ is contained in a finitely generated subcomodule. Indeed, let $\rho(m)=\sum c_i \otimes m_i$, and let $N$ be the submodule generated by the $m_i$. The pullback diagram
 \[
  \newdir{ >}{{}*!/-7pt/@{>}}
  \xymatrix{ 
   N^{\prime} \ar@{{ >}->}[d] \ar@{{ >}->}[r] & C\otimes N \ar@{{ >}->}[d] \\ 
   M \ar@{{ >}->}[r]_-{\rho} &  C\otimes M 
  }
 \]
 in $\Mod_R$ is also a pullback diagram in the category of $C$-comodules because $C$ is flat. By definition of $N$ it follows that $N^{\prime}$ contains $m$. Commutativity of the diagram
 \[
  \newdir{ >}{{}*!/-7pt/@{>}}
  \xymatrix{ 
   N^{\prime} \ar@{{ >}->}[d] \ar@{{ >}->}[r] & C\otimes N \ar@{{ >}->}[d] \ar[r]^-{\varepsilon \otimes N} & N \ar@{{ >}->}[d] \\ 
   M \ar@{{ >}->}[r]_-{\rho} & C\otimes M \ar[r]_-{\varepsilon \otimes M} & M
  }  
 \]
 and the fact that $\varepsilon \otimes M \cdot \rho=\id_M$ imply that $N^\prime$ is a submodule of $N$, so it is finitely generated.
 
 It remains to show that for every $m \in M$, there exists a Cauchy comodule $M^{\prime}$ and a homomorphism of comodules $M^\prime \rightarrow M$ whose image contains $m$ (cf.\ \cite[Proposition~1.4.1]{HOVEY}). Choose a homomorphism of $R$-modules $\phi \colon R^n \rightarrow M$ whose image contains all the $m_i$, and let
 \[
  \newdir{ >}{{}*!/-7pt/@{>}}
  \xymatrix{
   E \ar@{{ >}->}[r] \ar[d]_{\psi} & C\otimes R^n \ar[d]^{C\otimes \phi} \\
   M \ar@{{ >}->}[r]_-{\rho} & C\otimes M
  }
 \]
 be a pullback diagram. Since $C$ is flat, this is in fact a pullback diagram in the category of comodules, and $E$ contains an element $e$ with $\psi(e)=m$ by definition of $\phi$. Let $E^{\prime}$ be a finitely generated subcomodule of $E$ containing $e$. Thus $E^{\prime}$ is a finitely generated submodule of the flat module $C\otimes R^n$. Since $R$ is hereditary and Noetherian it follows that $E^{\prime}$ is projective. Thus $E^\prime$ is the desired Cauchy comodule.
 \end{proof}

\section{Further simplifications when \texorpdfstring{$\ca{V}$}{V} is abelian} \label{ABELIAN_RECOGNITION_SECTION}
In this section we take a closer look at certain special cosmoi, in particular cosmoi of (differential graded) $R$-modules for a commutative ring $R$. In this case we can further simplify the recognition principle. This section contains generalizations of results found in \cite{SZLACHANYI}, adapted to the setting of more general cosmoi.

\subsection{Projective objects and tame \texorpdfstring{$\ca{V}$}{V}-categories}
We have already seen that a cosmos with dense autonomous generator is locally finitely presentable as a closed category if the unit object is locally presentable (see Proposition~\ref{COSMOI_WITH_DAG_ARE_LFP_PROP}). In the abelian context we can further simplify the reconstruction theorem when working with tame $\ca{V}$-categories. If the unit object of our cosmos is projective in the sense of the following definition, then all $\ca{V}$-categories are tame.

\begin{dfn}
Let $\ca{C}_0$ be an unenriched category. An object $C\in \ca{C}$ is called \emph{projective} if the functor $\ca{C}_0(C,-)$ preserves epimorphisms.
\end{dfn}
 
 Examples of cosmoi where the unit object is finitely presentable and projective are given by categories of $A$-graded $R$-modules, $A$ any abelian group. A nonexample is the cosmos of differential graded $R$-modules, where the unit object is not projective in the above sense. In order to prove the recognition principle for cosmoi of (graded) modules, we need the following technical lemmas.

\begin{lemma}\label{CAUCHY_OBJECTS_FGP_LEMMMA}
 Let $\ca{B}$ be a small $\ca{V}$-category and let $G$ be an object of the Cauchy completion $\overline{\ca{B}}$ of $\ca{B}$. If the unit of $\ca{V}$ is finitely presentable or projective, then $G$ is finitely presentable or projective in the category $\Prs{B}_0$.
\end{lemma}

\begin{proof}
 The ($\Set$-valued) hom-functor $\Hom_{\Prs{B}_0}(G,-)$ of $G$ is  given by the composite $\ca{V}_0(I,-) \circ \Prs{B}(G,-)_0$. The $\ca{V}$-functor $\Prs{B}(G,-)\colon \Prs{B} \rightarrow \ca{V}$ is cocontinuous by definition of Cauchy completions (see Section~\ref{CAUCHY_COMPLETION_SECTION}), so by Theorem~\ref{FREE_COCOMPLETION_THM} it has a right $\ca{V}$-adjoint. Thus $\Prs{B}(G,-)_0$ has an ordinary right adjoint, and it follows that $\Hom_{\Prs{B}_0}(G,-)$ preserves all the colimits that are preserved by $\ca{V}_0(I,-)$. In particular, it preserves filtered colimits (resp.\ epimorphisms) if $I$ is locally presentable (resp.\ projective). 
\end{proof}

 If the cosmos $\ca{V}$ is enriched in abelian groups, then the forgetful functor $V \colon \ca{V} \rightarrow \Set$ factors through $U \colon \Ab \rightarrow \Set$. Therefore the underlying unenriched category of any $\ca{V}$-category is enriched in abelian groups. In particular, it makes sense to talk about kernels and cokernels in a $\ca{V}$-category. 

\begin{dfn}\label{TAME_CATEGORY_DFN}
 Let $\ca{V}$ be an $\Ab$-enriched cosmos. A $\ca{V}$-category $\ca{B}$ is \emph{tame} if for all epimorphisms $p \colon A \rightarrow B$ in $\Prs{B}$ with $A, B\in \overline{\ca{B}}$, the kernel of $p$ lies in $\overline{\ca{B}}$. A cosmos $\ca{V}$ is called tame if all small $\ca{V}$-categories are tame.
\end{dfn}

\begin{rmk}
 We will only apply the definition of tame $\ca{V}$-categories in combination with Lemma~\ref{CAUCHY_OBJECTS_FGP_LEMMMA}, that is, we only care about the fact that the kernel of an epimorphism between objects of the Cauchy completion is finitely generated. The only examples of $\ca{V}$-categories with this property we know of are tame $\ca{V}$-categories in cosmoi with finitely presentable unit object.
\end{rmk}

\begin{prop}\label{TAME_IF_FGP_UNIT_PROP}
 Let $\ca{V}$ be an $\Ab$-enriched cosmos. If the unit of $\ca{V}$ is projective, then $\ca{V}$ is tame.
\end{prop}

\begin{proof}
 Let $\ca{B}$ be a small $\ca{V}$-category. Let $p \colon A \rightarrow B$ be an epimorphism, where $A$ and $B$ are in the Cauchy completion of $\ca{B}$. From Lemma~\ref{CAUCHY_OBJECTS_FGP_LEMMMA} we know that $\ca{B}$ is projective. Thus $p$ is a split epimorphism, and it follows that the kernel of $p$ is a retract of $A$. Since Cauchy complete categories are closed under retracts, we find that $\ca{B}$ is tame.
\end{proof}

\begin{cor}\label{MODULE_CATS_TAME_COR}
The cosmos of (graded) modules over a commutative ring $R$ is tame.
\end{cor}

\begin{proof}
 We already observed that the unit of the cosmos of graded $R$-modules is finitely presentable and projective.
\end{proof}

Note that Proposition~\ref{TAME_IF_FGP_UNIT_PROP} is not applicable to the cosmos of differential graded $R$-modules because its unit is not projective. It turns out that the differential graded categories we care about most are all tame (see Corollary~\ref{DGA_TAME_COR}), but it is open whether or not \emph{all} differential graded categories are tame.

\subsection{The recognition theorem for abelian cosmoi with DAG} We prove the following theorem in the next section. Before doing that we investigate some of its consequences.

\begin{thm}\label{ABELIAN_RECOGNITION_THM}
 Let $\ca{V}$ be an abelian cosmos with a dense autonomous generator $\ca{X}$ and with a finitely presentable unit object. Let $\ca{A}$, $\ca{B}$ be small $\ca{V}$-categories and $w \colon \ca{A} \rightarrow \Prs{B}$ a $\ca{V}$-functor whose image is contained in $\overline{\ca{B}}$. Assume that $\ca{A}$ is Cauchy complete, that $\ca{B}$ is tame, and that $w$ satisfies the conditions:
\begin{enumerate}
\item[i)] 
 The functor $w_0 \colon \ca{A}_0 \rightarrow (\Prs{B})_0$ is faithful and reflects isomorphisms;
\item[ii)] 
 For each object $B\in\ca{B}$, the category $\el \bigl(V (\ev_B w)_0\bigr)$ of elements of the functor $V(\ev_B w)_0 \colon \ca{A}_0 \rightarrow \Set$ is cofiltered (cf.\ Theorem~\ref{DENSE_RECOGNITION_THM}); and
\item[iii)]
 If $f$ is a morphism of $\ca{A}_0$ for which the cokernel of $w_0(f)$ is in $\overline{\ca{B}}$, then the cokernel of $f$ exists and is preserved by $w_0$.
\end{enumerate}
 Then the $w$-component of the unit of the Tannakian adjunction is an equivalence. Moreover, the comonad $L(w)$ preserves finite limits.
\end{thm}

 If we further specialize to the cosmos of differential graded $R$-modules or the cosmos of $R$-modules, we obtain the following recognition results. For the cosmos of abelian groups, a similar result was proved by K.\ Szlach\'anyi (see \cite[Corollary~6.4]{SZLACHANYI}).

\begin{thm}\label{DGA_THM}
 Let $R$ be a commutative ring, $B$ a commutative differential graded $R$-algebra, and let $\ca{V}$ be the cosmos of differential graded $R$-modules. Let $\ca{A}$ be a small Cauchy complete differential graded $R$-linear category, equipped with a differential graded $R$-linear functor $w \colon \ca{A} \rightarrow \Mod_B$ into the category of differential graded left $B$-modules such that $w(A)$ is dualizable for all $A \in \ca{A}$. Suppose the following conditions are satisfied, where $Z_0 \colon \Mod_B \rightarrow \Set$ denotes the functor which sends a differential graded $B$-module to its set of cycles of degree $0$:
\begin{enumerate}
\item[i)] 
 The functor $w_0 \colon \ca{A}_0 \rightarrow (\Mod_B)_0$ is faithful and reflects isomorphisms;
\item[ii)] 
 The category $\el\bigr(Z_0(w_0)\bigl)$ of elements of $Z_0(w_0)$ is cofiltered; and
\item[iii)]
 If the cokernel of $w_0(f)$ lies in the Cauchy completion of $B$, then the cokernel of $f$ exists and is preserved by $w_0$.
\end{enumerate}
 Then the $w$-component of the unit of the Tannakian adjunction is an equivalence of categories. Moreover, the comonad $L(w) \colon \Mod_{B} \rightarrow \Mod_{B}$ preserves finite limits, that is, the corresponding differential graded $B$-$B$-bimodule is flat as a differential graded right $B$-module.
\end{thm}

\begin{thm}\label{COMMUTATIVE_RING_THM}
 Let $R$ be a commutative ring, $B$ an $R$-algebra, and let $\ca{V}$ be the category $\Mod_R$ of $R$-modules. Let $\ca{A}$ be a small additive $R$-linear category, equipped with an $R$-linear functor $w \colon \ca{A} \rightarrow \Mod_B$ into the category of left $B$-modules such that $w(A)$ is finitely generated and projective for all $A \in \ca{A}$. Suppose the following conditions are satisfied:
\begin{enumerate}
\item[i)] 
 The functor $w_0$ is faithful and reflects isomorphisms;
\item[ii)] 
 The category $\el(w_0)$ of elements of $w_0$ is cofiltered; and
\item[iii)]
 If the cokernel of $w_0(f)$ is finitely generated and projective, then the cokernel of $f$ exists and is preserved by $w_0$.
\end{enumerate}
 Then the $w$-component of the unit of the Tannakian adjunction is an equivalence of categories. Moreover, the comonad $L(w) \colon \Mod_{B} \rightarrow \Mod_{B}$ preserves finite limits, that is, the corresponding $B$-$B$-bimodule is flat as a right $B$-module.
\end{thm}

 We conclude with a few remarks about the necessity of the conditions in Theorem~\ref{COMMUTATIVE_RING_THM}. It is clear that conditions i) and iii) are necessary conditions. Condition ii) implies that $L(w)$ is a flat coalgebroid. The question whether or not the converse is true is open: is the forgetful functor from Cauchy $L$-comodules to $B$-modules flat (in the sense that its category of elements is cofiltered) if $L$ is a flat coalgebroid acting on $B$? This is for example the case if the Cauchy $L$-comodules generate the category of all $L$-comodules. A lot of the examples of Hopf algebroids studied in algebraic topology are \emph{Adams Hopf algebroids}, and their categories of Cauchy comodules form generators (see \cite[Proposition~2.3.3]{HOVEY}). As far as the author knows it is an open question wether or not there are flat Hopf algebroids for which the Cauchy comodules don't form a generator.

\subsection{Proof of Theorem~\ref{ABELIAN_RECOGNITION_THM}} We prove Theorem~\ref{ABELIAN_RECOGNITION_THM} using a series of lemmas.

\begin{lemma}\label{COCONE_EPI_LEMMA}
Let $\ca{C}$ be a cocomplete unenriched category and let $\ca{D}$ be small unenriched category with finite coproducts. If the colimit $C$ of a diagram $D \colon \ca{D} \rightarrow \ca{C}$ is finitely presentable, then some morphism in the colimiting cocone is an epimorphism.
\end{lemma}

\begin{proof}
 Let $\bigl(C,(\kappa_i)_{i \in \ca{D}}\bigr)$ be the colimit of $D$. Since $\ca{C}$ is cocomplete, this colimit can be computed as filtered colimit of the colimits of the finitely generated subcategories of $\ca{D}$ (see \cite[Proposition~2.13.7]{BORCEUX}). The identity of $C$ factors through one of the structure maps of this filtered diagram because $C$ is finitely presentable. But these structure maps are induced by a finite family $(\kappa_{i_k})_{k=1,\ldots, n}$ of structure maps of the original diagram $D$. This finite family is in particular collectively epimorphic. Let $j=\sum_{k=0}^n i_k \in \ca{D}$. If $f\kappa_j=g\kappa_j$, then we also have $f\kappa_{i_k}=g\kappa_{i_k}$ for all $k$. Thus the structure map $\kappa_j$ is an epimorphism. 
\end{proof}

\begin{lemma}\label{QUOTIENT_OF_REPRESENTABLE_LEMMA}
 Let $\ca{V}$ be an $\Ab$-enriched cosmos with dense autonomous generator $\ca{X}$ and with finitely presentable unit object. Let $w \colon \ca{A} \rightarrow \Prs{B}$ be a $\ca{V}$-functor with image contained in $\overline{\ca{B}}$.

 Assume further that $\ca{A}$ is $\ca{X}$-tensored and that $\ca{A}$ has finite coproducts. If $G \in \Prs{A}$ is a presheaf such that $L_w(G)$ is finitely presentable, then there exists an object $C \in \ca{A}$ and a morphism $y \colon \ca{A}(-,C) \rightarrow G$ such that $L_w(y)$ is an epimorphism.
\end{lemma}

\begin{proof}
 From Theorem~\ref{CONICAL_COLIMIT_THM} we know that the morphisms $y \colon \ca{A}(-,C) \rightarrow G$ form a colimiting cocone on the domain functor $\ca{A}\slash G \rightarrow \Prs{A}_0$. Since $L_w$ is a left adjoint, it follows that the $L_w(y)$ form a colimiting cocone in $\Prs{B}_0$. We will apply Lemma~\ref{COCONE_EPI_LEMMA} to this particular colimit. Note that $\ca{A}\slash G$ has finite coproducts: they exist in $\ca{A}$, and the Yoneda embedding preserves them because our category is enriched in abelian groups, so that finite coproducts are absolute colimits. The conclusion follows from Lemma~\ref{CAUCHY_OBJECTS_FGP_LEMMMA}.
\end{proof}

\begin{prop}
 If $\ca{V}$ is an abelian cosmos, then $\Prs{A}_0$ is abelian for every small $\ca{V}$-category $\ca{A}$.
\end{prop}

\begin{proof}
 Since $\ca{V}$ is abelian, the forgetful functor $\ca{V}_0(I,-)$ naturally factors through the category $\Ab$ of abelian groups. It follows that the underlying categories of $\ca{V}$-categories are $\Ab$-enriched, and that the underlying functors of $\ca{V}$-functors are additive. Moreover, limits and colimits in the underlying category $\Prs{A}_0$ of a presheaf category $\Prs{A}$ are computed pointwise, so it has a zero object, every monomorphism is the kernel of its cokernel and every epimorphism is the cokernel of its kernel. 
\end{proof}

 Using this fact we can prove the following lemma, which is a generalization of \cite[Proposition~5.10]{SZLACHANYI}.

\begin{lemma}\label{ADDITIVE_FULLY_FAITHFUL_LEMMA}
 Let $\ca{V}$ be an abelian cosmos with a dense autonomous generator $\ca{X}$ and with a finitely presentable unit object. Let $\ca{A}$ be an $\ca{X}$-tensored category with finite direct sums, and let $w \colon \ca{A}\rightarrow \Prs{B}$ be a $\ca{V}$-functor whose image is contained in $\overline{\ca{B}}$, and such that $L_w$ preserves finite limits (cf.\ \cite{KELLY_FINLIM}). Assume that $w$ satisfies the conditions:
\begin{enumerate}
\item[i)]
 The functor $w_0 \colon \ca{A}_0 \rightarrow \Prs{B}_0$ is faithful.
\item[ii)]
 A morphism $f \colon w(B) \rightarrow w(A)$ is in the image of the functor $w_0$ if (and only if) there is an object $C \in \ca{A}$ together with morphisms $h\colon C \rightarrow B$, $g \colon C \rightarrow A$ such that $w(h)$ is an epimorphism and $f w(h)=w(g)$.
\end{enumerate}
 Then the Yoneda embedding of $\ca{A}$ factors through the category $\ca{C}$ from Proposition~\ref{REPLETE_IMAGE_PROP}. Consequently, the $w$-component of the unit of the Tannakian adjunction is fully faithful.
\end{lemma}

\begin{proof}
 From Corollary~\ref{FULLY_FAITHFUL_COR} we know that we only have to check that the Yoneda embedding factors through the full subcategory $\ca{C}$ defined in Proposition~\ref{REPLETE_IMAGE_PROP}. 

 To do this it suffices to show that the $\ca{A}(-,A)$-component of the unit of the adjunction $J \dashv E \colon \Prs{A} \rightarrow \Prs{B}_{L_w \widetilde{w}}$ is a split monomorphism. For the sake of brevity we denote this morphism by $f \colon \ca{A}(-,A) \rightarrow F$. From Proposition~\ref{REPLETE_IMAGE_PROP} we know that $L_w(f)$ is an isomorphism. 

 Let $x \colon \ca{A}(-,B) \rightarrow F$ be an arbitrary morphism in $\Prs{A}$ whose domain is representable, that is, an object of $\ca{A}\slash F$. Consider the pullback diagram
\[
\xymatrix{ G \ar[r]^-{x^\prime} \ar[d]_{f^\prime} & \ca{A}(-,A) \ar[d]^f \\
\ca{A}(-,B) \ar[r]_-{x} & F}
\]
 and note that $L_w(f^\prime)$ is an isomorphism because $L_w$ preserves pullbacks. Since $L_w \bigl(\ca{A}(-,B)\bigr) \cong w(B)$, we know that $L_w G$ lies in the Cauchy completion of $\ca{B}$. It is in particular finitely presentable. From Lemma~\ref{QUOTIENT_OF_REPRESENTABLE_LEMMA} it follows that there exist $C \in \ca{A}$ and $y \colon \ca{A}(-,C) \rightarrow G$ such that $L_w(y)$ is an epimorphism. Since the Yoneda embedding is fully faithful, there exist morphisms $h\colon C\rightarrow B$ and $g \colon C \rightarrow A$ such that $\ca{A}(-,h)=f^\prime y$ and $\ca{A}(-,g)=x^\prime y$. The above observation that $L_w(f^\prime)$ is an isomorphism implies that $L_w \bigl( \ca{A}(-,h)\bigr)=L_w Y(h)$ is an epimorphism. Commutativity of the above diagram shows that the equality
\[
L_w(f)^{-1} L_w(x) L_w Y(h)=L_w Y(g)
\]
 holds. Using the fact that $L_w Y \cong w$ and condition ii) we find that there is a morphism $k_x \colon B \rightarrow A$ such that $L_w Y(k_x)=L_w(f)^{-1} L_w (x)$. Since $w_0$ is faithful, so is $(L_w Y)_0$. Hence the morphisms $\ca{A}(-,k_x) \colon \ca{A}(-,B) \rightarrow \ca{A}(-,A)$ form a cocone on the domain functor $\ca{A} \slash F \rightarrow \Prs{A}_0$. From Theorem~\ref{CONICAL_COLIMIT_THM} it follows that there exists a unique morphism $k \colon F \rightarrow \ca{A}(-,A)$ such that $k \circ x=\ca{A}(-,k_x)$ for every $x \colon \ca{A}(-,B) \rightarrow F$. In particular, for $x=f\colon \ca{A}(-,A) \rightarrow F$ we get $k f = \ca{A}(-,k_f)$, and thus
\[
L_w (k) L_w(f)=L_w Y(k_f)=L_w(f)^{-1} L_w (f)=\id.
\]
 The morphism $kf$ is in the image of the Yoneda embedding, so from the fact that $(L_w Y)_0 \cong w_0$ is faithful we conclude that $kf=\id$. This shows that $\ca{A}(-,A)$ is a retract of $F$, hence that $\ca{A}(-,A)$ lies in $\ca{C}$.
\end{proof}

\begin{proof}[Proof of Theorem~\ref{ABELIAN_RECOGNITION_THM}]
 We check that the conditions from Corollary~\ref{FULLY_FAITHFUL_COR} are satisfied.  Since the objects in $\ca{X}$ are dualizable, tensors with $X$ are absolute colimits. Thus $\ca{A}$ is $\ca{X}$-tensored (cf.\ \cite{STREET_ABSOLUTE}).
 
 As in the proof of Theorem~\ref{DENSE_RECOGNITION_THM}, we conclude from condition ii) that $L_w$ preserves finite limits in the sense of \cite{KELLY_FINLIM}, so it is in particular exact. To show that the Yoneda embedding of $\ca{A}$ factors through $\ca{C}$ we have to check that condition ii) of Lemma~\ref{ADDITIVE_FULLY_FAITHFUL_LEMMA} holds. We prove this with an argument inspired by \cite[Lemma 6.2]{SZLACHANYI}. 

 Let $f\colon w(B) \rightarrow w(A)$ be a morphism such that there exists an object $C \in \ca{A}$ together with morphisms $h\colon C \rightarrow B$ and $g \colon C \rightarrow A$ such that $w(h)$ is an epimorphism and $f w(h)=w(g)$. We have to show that $f$ is in the image of $w$.

 Let $k\colon K \rightarrow YC$ be the kernel of $Yh$ in $\Prs{A}$. This kernel is preserved by $L_w$, and since $\ca{B}$ is tame we know that $L_w K$ is finitely presentable. By Lemma~\ref{QUOTIENT_OF_REPRESENTABLE_LEMMA} it follows that there is an object $C^\prime \in \ca{A}$ together with a morphism $l \colon YC^{\prime} \rightarrow K$ such that $L_w(l)$ is an epimorphism. Let $r\colon C^{\prime} \rightarrow C$ be the unique morphism in $\ca{A}$ with $Y(r)=kl$. The fact that $w\cong L_w Y$ implies that the sequence
\[
 \xymatrix{wC^{\prime} \ar[r]^-{w(r)} & wC \ar[r]^-{w(h)} & wB \ar[r] &0}
\]
 is exact. By iii) it follows that the cokernel $K^{\prime}$ of $r$ exists in $\ca{A}$, and is preserved by $w$. Thus the comparison morphism $K^{\prime} \rightarrow B$ gets sent to an isomorphism by $w$, and from i) it follows that $h$ is a cokernel of $r$ in $\ca{A}$. From the definition of $r$ it follows easily that $w(g) \cdot w(r)=0$, thus that $gr=0$. Hence there exists a morphism $g^\prime \colon B \rightarrow A$ such that $g^\prime h=g$. We have $w(g^\prime) \cdot w(h)=w(g)=f \cdot w(h)$, and $w(h)$ is an epimorphism by assumption. Thus $f=w(g^\prime)$. This shows that the conditions of Lemma~\ref{ADDITIVE_FULLY_FAITHFUL_LEMMA} are indeed satisfied, hence that every representable presheaf on $\ca{A}$ lies in $\ca{C}$. 

 It remains to show that for every presheaf $X \in\ca{C}$ with $L_w(X) \in \overline{\ca{B}}$, there exists an object $A \in \ca{A}$ and an isomorphism $X \cong \ca{A}(-,A)$. Fix such a presheaf $X$. By Lemma~\ref{QUOTIENT_OF_REPRESENTABLE_LEMMA} there exists an object $B\in \ca{A}$ and a morphism $x \colon \ca{A}(-,B) \rightarrow X$ such that $L_w(x)$ is an epimorphism. Let $y \colon G \rightarrow \ca{A}(-,B)$ be its kernel in $\Prs{A}$. Since $L_w$ preserves finite limits, $L_w(G)$ is the kernel of the epimorphism $L_w(x)$. But $\ca{B}$ is tame, hence $L_w G$ lies in the Cauchy completion of $\ca{B}$. It is therefore finitely presentable by Lemma~\ref{CAUCHY_OBJECTS_FGP_LEMMMA}. From Lemma~\ref{QUOTIENT_OF_REPRESENTABLE_LEMMA} we conclude that there is an $A \in \ca{A}$ and $z \colon \ca{A}(-,A) \rightarrow G$ such that $L_w(z)$ is an epimorphism. The composite $yz$ is of the form $\ca{A}(-,f)$ for a unique $f\colon A \rightarrow B$. Therefore the sequence
\[
\xymatrix{L_w Y(A) \ar[r]^-{L_w Y(f)} & L_w Y(B) \ar[r]^-{L_w(x)} & L_w(X) \ar[r]  &0 \rlap{.}}
\]
 in $\Prs{B}$ is exact.

 We claim that $x\colon \ca{A}(-,B) \rightarrow X$ is the cokernel of $Y(f)$ in $\ca{C}$. Indeed, we have $L_w=VJ$ by definition of the comparison functor (see Definition~\ref{COMPARISON_FUNCTOR_DFN}), $V$ creates colimits and the restriction of $J$ to $\ca{C}$ is an equivalence (see Proposition~\ref{REPLETE_IMAGE_PROP}).
 
 Moreover, by iii) the cokernel $C$ of $f$ in $\ca{A}_0$ exists and is preserved by $w$. We get a comparison morphism $X \rightarrow \ca{A}(-,C)$, and both $L_w (X)$ and $L_w Y(C) \cong w(C)$ are cokernels of $w(f)$, so this comparison morphism is sent to an isomorphism by $L_w$. The fact that $L_w=VJ$ implies that the restriction of $L_w$ to $\ca{C}$ reflects isomorphisms. Thus the comparison morphism in question is itself an isomorphism, showing that $X \cong \ca{A}(-,C)$. This concludes the proof of essential surjectivity.
\end{proof}

\begin{prop}\label{CAUCHY_CHARACTERIZATION_FGP_PROP}
 Let $\ca{V}$ be an $\Ab$-enriched cosmos with dense autonomous generator $\ca{X}$ and finitely presentable projective unit object. Then a small $\ca{V}$-category $\ca{B}$ is Cauchy complete if and only if $\ca{B}$ is $\ca{X}$-tensored, has finite direct sums, and all idempotents in $\ca{B}$ split.
\end{prop}

\begin{proof}
 First note that any Cauchy complete category is $\ca{X}$-tensored, because the objects in $\ca{X}$ are dualizable. Since $\ca{V}$ is $\Ab$-enriched, finite coproducts (direct sums) are preserved by any $\ca{V}$-functor. Thus they are absolute colimits, and any Cauchy complete category has absolute colimits (see \cite{STREET_ABSOLUTE}). Idempotents in a Cauchy complete category are always split (cf.\ \cite[Proposition~5.25]{KELLY_BASIC}).
 
 Conversely, assume that $\ca{B}$ is $\ca{X}$-tensored, has finite direct sums, and that idempotents in $\ca{B}$ split. Let $X \in \Prs{B}$ be an object of the Cauchy completion. We have to show that $X$ is representable. Since idempotents in $\ca{B}$ split, it suffices to show that $X$ is a retract of a representable functor. From Theorem~\ref{CONICAL_COLIMIT_THM} we know that $X$ is the conical colimit of the diagram $\ca{B} \slash X$ of representable functors over $X$. This diagram has finite coproducts because $\ca{B}$ has finite direct sums. By Lemma~\ref{COCONE_EPI_LEMMA} it follows that there is an epimorphism $p \colon \ca{B}(-,B) \rightarrow X$. But $X$ is projective by Lemma~\ref{CAUCHY_OBJECTS_FGP_LEMMMA}, so this epimorphism is split. Thus $X$ is a retract of $\ca{B}(-,B)$.
\end{proof}

\begin{proof}[Proof of Theorem~\ref{COMMUTATIVE_RING_THM}]
 We let $\ca{X}$ be the full monoidal subcategory of $\Mod_R$ consisting of finitely generated free modules, which is clearly a dense autonomous generator. Since $\ca{A}$ is additive, it is $\ca{X}$-tensored: the tensor product of $A\in \ca{A}$ with $R^n$ is simply the $n$-fold direct sum $\oplus_{i=1}^n A$. Condition~iii) implies that idempotents in $\ca{A}$ split. Thus $\ca{A}$ is Cauchy complete by Proposition~\ref{CAUCHY_CHARACTERIZATION_FGP_PROP}. 
 
 All $R$-linear categories are tame by Remark~\ref{MODULE_CATS_TAME_COR}. The Cauchy completion of $B^{\op}$, considered as a one-object $\ca{V}$-category, is the full subcategory of $\Mod_B$ consisting of finitely generated projective modules. Thus the image of $w(A)$ is an object of the Cauchy completion of $B^{\op}$, and it makes sense to speak of the $w$-component of the unit of the Tannakian adjunction. The remaining conditions are precisely the conditions in Theorem~\ref{ABELIAN_RECOGNITION_THM}.
\end{proof}

\subsection{Proof of Theorem~\ref{DGA_THM}}\label{PROOF_OF_DGA_THM_SECTION}
 Dealing with categories enriched in differential graded $R$-modules is more complicated, because the unit object is not projective. In particular, not every epimorphism between objects in the Cauchy completion of a small differential graded $R$-linear category is split. Therefore we need a different method to show that a kernel of such an epimorphism again lies in the Cauchy completion. Our proof relies on the existence of additional categorical structure on the category, which is always present in one of the main cases of interest. 
 
\begin{prop}\label{AUTONOMOUS_CAUCHY_PROP}
 Let $\ca{B}$ be an autonomous monoidal $\ca{V}$-category. Then $X \in \Prs{B}$ lies in the Cauchy completion of $\ca{B}$ if and only if it has a left and a right dual under the Day convolution tensor product on $\Prs{B}$. Moreover, $X$ has a left dual if and only if it has a right dual.
\end{prop}

\begin{proof}
 This follows from a very general result about autonomous pseudomonoids in a monoidal bicategory (see \cite[Proposition~4.6]{LOPEZ_FRANCO}). We provide a more elementary proof that only works for the monoidal bicategory $\Mod(\ca{V})$.
 
 We first prove that objects in the Cauchy completion have both duals. Since $\ca{B}$ is autonomous and the Yoneda embedding is strong monoidal for the Day convolution tensor product, it follows that all representable presheaves have both duals. The Day convolution monoidal structure is closed, that is, there exist left and right internal hom objects characterized by
 \[
  \Prs{B}(A,[B,C]_\ell) \cong \Prs{B}(A\otimes B, C) \quad \text{and} \quad 
  \Prs{B}(A,[B,C]_r) \cong  \Prs{B}(B\otimes A, C) \rlap{\smash{.}}
 \]
Therefore, an object $B$ has a right dual $B^\vee$ if and only if the $\ca{V}$-natural transformation
 \[
  [B,I]_r \otimes A \rightarrow [B,A]_r
 \]
 which corresponds to
 \[
  B\otimes [B,I]_r \otimes A \rightarrow A
 \]
 under the adjunction $B\otimes- \dashv [B,-]_r$ is an isomorphism. Both sides depend contravariantly on $B$, so it follows that the subcategory of right dualizable objects is closed under absolute colimits. Similarly we find that the subcategory of left dualizable objects is closed under absolute colimits. Since the Cauchy completion is the closure of the representable functors under absolute colimits (see \cite{STREET_ABSOLUTE}), all objects in the Cauchy completion have a left and a right dual.

 Conversely, suppose that $X$ has a right dual. We have to show that the functor $\Prs{B}(X,-)$ is cocontinuous. Note that the unit of the Day convolution tensor product is represented by the unit $U$ of $\ca{B}$. Using the definition of right duals and the Yoneda lemma we find that
\[
 \Prs{B}(X,Y) \cong \Prs{B}\bigl(X\otimes \ca{B}(-,U),Y\bigr) \cong \Prs{B}\bigl(\ca{B}(U,-), X^\vee \otimes Y\bigr) \cong \ev_U(X^\vee \otimes Y)\smash{\rlap{.}}
\]
 Evaluation in $U$ is certainly cocontinuous, and tensoring with $X^\vee$ is a left adjoint because $\Prs{B}$ is closed. Thus $X \in \overline{\ca{B}}$, and therefore also has a left dual because all objects of the Cauchy completion do. The case of an object with a left dual is proved similarly, using the isomorphism $\ca{B}(-,U)\otimes X \cong X$.
\end{proof}

 In the proof of Proposition~\ref{AUTONOMOUS_CAUCHY_PROP} we have seen that an object in a monoidal closed category is dualizable if and only if a map between two objects built out of tensor products and internal hom-objects is invertible. Thus, if we have a functor between two monoidal closed categories which preserves tensor products and internal hom-objects (strictly or up to coherent isomorphism), and reflects isomorphisms, we find that an object in the domain is dualizable if and only if its image is.
 
 \begin{example}\label{DUALIZABLE_DGM_EXAMPLE}
 Fix a commutative ring Let $B$ be a commutative differential graded $R$-algebra, and let $U_\ast B$ be its underlying graded algebra. The forgetful functor from differential graded $B$-modules to graded $U_\ast B$-modules reflects isomorphisms and preserves tensor products and internal hom-objects strictly. Thus an object $X$ in the category of differential graded $B$-modules is dualizable if and only if its underlying graded module $U_\ast X$ is dualizable.
 \end{example}

\begin{cor}\label{DGA_TAME_COR}
 All commutative differential graded $R$-algebras $B$ are tame.
\end{cor}

\begin{proof}
 We use the notation introduced in Example~\ref{DUALIZABLE_DGM_EXAMPLE}. Note that any commutative algebra in a cosmos $\ca{V}$ is an autonomous symmetric monoidal $\ca{V}$-category. Indeed, the single object is the unit, and the unit of a monoidal category is always dualizable. Therefore, Proposition~\ref{AUTONOMOUS_CAUCHY_PROP} tells us that an object lies in the Cauchy completion of $B$ (or $U_\ast B$) if and only if it is dualizable. From Example~\ref{DUALIZABLE_DGM_EXAMPLE} it follows that $X$ lies in the Cauchy completion of $B$ if and only if $U_\ast X$ lies in the Cauchy completion of $U_\ast B$.

 The functor $U_\ast$ from differential graded $B$-modules to graded $U_\ast B$-modules is exact. Thus $U_\ast$ sends epimorphisms to epimorphisms and kernels to kernels. The conclusion follows from the fact that $U_\ast B$ is tame (see Proposition~\ref{TAME_IF_FGP_UNIT_PROP}). 
\end{proof}

\begin{proof}[Proof of Theorem~\ref{DGA_THM}]
 The subcategory of differential graded $R$-modules $X$ with the property that $X_i$ is finitely generated free for all $i$ and $X_i=0$ for all but finitely many $i \in \mathbb{Z}$ forms a dense autonomous generator of the cosmos $\ca{V}$ of differential graded $R$-modules, and the unit object of $\ca{V}$ is finitely presentable. This is not hard to see directly, but it will also follow from Propositions~\ref{DAG_FOR_MODULES_PROP} and \ref{HOPF_ALGEBRA_FOR_DGM_PROP}. Moreover, $B$ is a tame $\ca{V}$-category by Corollary~\ref{DGA_TAME_COR}.
 
 The Cauchy completion of $B^{\op}$, considered as a one-object $\ca{V}$-category, is the full subcategory of $\Mod_B$ consisting of dualizable modules (see Proposition~\ref{AUTONOMOUS_CAUCHY_PROP}). Thus the image of $w(A)$ is an object of the Cauchy completion of $B^{\op}$, and it makes sense to speak of the $w$-component of the unit of the Tannakian adjunction. Note that $Z_0(w_0)$ is precisely the composite of $(\ev_\ast w)_0 \colon (\Mod_B)_0 \rightarrow \ca{V}_0$ with the forgetful functor $V \colon \ca{V}_0 \rightarrow \Set$, so the conditions i)-iii) are precisely the conditions i)-iii) in Theorem~\ref{ABELIAN_RECOGNITION_THM}.
\end{proof}

\subsection{Examples of tame categories}
 In this section we will generalize the arguments used in Section~\ref{PROOF_OF_DGA_THM_SECTION} to give a large class of examples of tame categories. Specifically, we shall prove the following result.
 
 \begin{thm}\label{AUTONOMOUS_MONOIDAL_TAME_THM}
 Let $\ca{V}$ be an $\Ab$-enriched cosmos such that all small autonomous monoidal $\ca{V}$-categories are tame. Let $H$ be a cocommutative Hopf algebra in $\ca{V}$. Then all small autonomous monoidal $\Mod_H$-categories are tame.
\end{thm}

 In order to prove this, we will use a base change functor between bicategories of modules induced by a monoidal functor between cosmoi. Namely, we consider base change along the forgetful functor $\Mod(H) \rightarrow \ca{V}$. This functor preserves tensor products, internal hom-objects, limits and colimits, and reflects isomorphisms.
 
 In general, if $F \colon \ca{V} \rightarrow \ca{V}^{\prime}$ is a strong monoidal functor between cosmoi, we can use it to turn $\ca{V}$-categories into $\ca{V}^{\prime}$-categories. Note that this procedure may change the associated underlying unenriched category: this happens whenever the triangle consisting of $F$ and the two canonical forgetful functors fails to be commutative. Thus our perspective that a $\ca{V}$-category is an abstract structure which can be used to \emph{construct} an underlying unenriched category is very relevant when considering base change functors. Note that one of the main examples we consider---the base change functor from differential graded modules to graded modules which forgets the differential---does not commute with the canonical forgetful functors. 
 
 Base change of $\ca{V}$-categories is functorial both for $\ca{V}$-functors (see \cite{EILENBERG_KELLY}) and for modules (see \cite{VERITY} and \cite{CRUTTWELL}).

\begin{dfn}[Base change for $\ca{V}$-modules]\label{BASE_CHANGE_DFN}
 Let $F \colon \ca{V} \rightarrow \ca{V}^{\prime}$ be a symmetric strong monoidal functor. Let $\ca{A}$ be a $\ca{V}$-category. The $\ca{V}^{\prime}$-category $F_\ast \ca{A}$ has the same objects as $\ca{A}$, with hom-objects given by $F\bigl(\ca{A}(a,b)\bigr)$. For a module $M \colon \ca{A} \xslashedrightarrow{} \ca{B}$, we define a module $F_\ast M \colon F_\ast \ca{A} \xslashedrightarrow{} F_\ast \ca{B}$ by $(b,a) \mapsto F\bigl(M(b,a)\bigr)$, with evident action of $F_\ast \ca{A}$ and $F_\ast \ca{B}$.
\end{dfn}

\begin{prop}\label{BASE_CHANGE_PROP}
 If $F \colon \ca{V} \rightarrow \ca{V}^{\prime}$ is symmetric strong monoidal and cocontinuous, then the assignments in Definition~\ref{BASE_CHANGE_DFN} extend to a strong monoidal pseu\-do\-func\-tor $F_\ast \colon \Mod(\ca{V}) \rightarrow \Mod(\ca{V}^{\prime})$. If, in addition, $F$ is continuous and strong closed (that is, it preserves internal hom-objects), then $F_\ast$ preserves right liftings.
\end{prop}

\begin{proof}
 Since composition of modules is defined in terms of colimits, it follows that $F_\ast$ is a pseudofunctor. It is strong monoidal because $F$ is (cf.\ \cite[\S~7.3.3]{CRUTTWELL}). The right lifting of $M \colon \ca{A} \xslashedrightarrow{} \ca{C}$ along $J \colon \ca{B} \xslashedrightarrow{} \ca{C}$ in $\Mod(\ca{V})$ is given by the coend
 \[
 J\slash M (b,a) = \int_{\ca{C}}[J(b,c),M(c,a)]
 \]
 in $\ca{V}$. Since $F$ preserves limits and internal hom-objects, it also preserves ends (which are defined in terms of limits and internal hom-objects). In particular, it preserves the above end which is again defined in terms of internal hom-objects.
\end{proof}

Let $\ca{B}$ be a small $\ca{V}$-category. Note that the category $\Prs{B}$ can be thought of as the hom-category $\Mod(\ca{V})(\ca{I},\ca{B})$. Thus $F_\ast$ induces a functor $(F_\ast)_{\ca{I},\ca{B}} \colon \Prs{B}_0 \rightarrow \mathcal{P}{F_\ast \ca{B}}_0$.

\begin{cor}\label{DAY_CONVOLUTION_PRESERVED_COR}
 Let $F \colon \ca{V} \rightarrow \ca{V}^{\prime}$ be symmetric strong monoidal, continuous, cocontinuous, and strong closed. Let $\ca{B}$ be a monoidal $\ca{V}$-category. Then the induced functor $\Prs{B}_0 \rightarrow \mathcal{P}{F_\ast \ca{B}}_0$ is strong monoidal and strong closed for the Day convolution tensor product (that is, it preserves both the left and right internal hom-objects). Furthermore, it is continuous and cocontinuous, and it reflects isomorphisms if $F$ does.
\end{cor}

\begin{proof}
 The internal hom-objects for the Day convolution tensor are defined in terms of right liftings in the bicategory of modules (see \cite[Proposition~6]{DAY_STREET}), and the Day convolution tensor product of $A$ and $B$ is given by the composite
 \[
 I \cong I \otimes I \xslashedrightarrow{A\otimes B} \ca{B} \otimes \ca{B} \xslashedrightarrow{m_\ast} \ca{B}
 \]
 where $m$ denotes the monoidal structure on $\ca{B}$. This formula is given in terms of the monoidal structure on $\Mod(\ca{V})$, so it is preserved up to isomorphism by $F_\ast$. 
 
 It remains to check that $F_\ast$ is continuous and cocontinuous, which follows immediately from the fact that limits and colimits in functor categories are computed pointwise. 
 
 A morphism of modules is an isomorphism if and only if all its components are. By definition of $F_\ast$ it follows that $(F_\ast)_{\ca{I},\ca{B}}$ reflects isomorphisms if $F$ does.
\end{proof}

\begin{prop}\label{CAUCHY_COMPLETION_BASE_CHANGE_PROP}
 Let $F \colon \ca{V} \rightarrow \ca{V}^{\prime}$ be symmetric strong monoidal, continuous, cocontinuous, strong closed, and conservative (that is, it reflects isomorphisms). Let $\ca{B}$ be an autonomous $\ca{V}$-category. Then an object $X \in \Prs{B}$ lies in the Cauchy completion of $\ca{B}$ if and only if $F_\ast X$ lies in the Cauchy completion of $F_\ast \ca{B}$.
\end{prop}

\begin{proof}
 By Corollary~\ref{DAY_CONVOLUTION_PRESERVED_COR}, the induced functor $F_\ast \colon \Prs{B}_0 \rightarrow \mathcal{P}{F_\ast \ca{B}}_0$ is strong monoidal, continuous and cocontinuous, strong closed, and conservative. Since left and right dualizable objects can be characterized in terms of an isomorphism between left and right internal hom-objects, it follows that $X\in \Prs{B}$ is left (right) dualizable if $F_\ast X$ is. Conversely, if $X$ is left (right) dualizable, then so is $F_\ast X$, because $F_\ast$ is strong monoidal. From Proposition~\ref{AUTONOMOUS_CAUCHY_PROP} it follows that $X$ lies in the Cauchy completion of $\ca{B}$ if and only if $F_\ast X$ lies in the Cauchy completion of the autonomous $\ca{V}^{\prime}$-category $F_\ast \ca{B}$.
\end{proof}

\begin{proof}[Proof of Theorem~\ref{AUTONOMOUS_MONOIDAL_TAME_THM}]
 Let $\ca{B}$ be a small autonomous monoidal $\Mod_H$-category. The forgetful functor $U \colon \Mod_H \rightarrow \ca{V}$ is symmetric strong monoidal, continuous, cocontinuous, strong closed, and it reflects isomorphisms. 
 
 We now want to show that $\ca{B}$ is tame. Thus let $p \colon X \rightarrow Y$ be an epimorphism between objects of the Cauchy completion of $\ca{B}$, and let $k \colon Z \rightarrow X$ be its kernel. Since $U_\ast$ is continuous and cocontinuous, $U_\ast k$ is the kernel of $U_\ast p$ and $U_\ast p$ is an epimorphism. By Proposition~\ref{CAUCHY_OBJECTS_FGP_LEMMMA}, both $U_\ast X$ and $U_\ast Y$ lie in the Cauchy completion of $U_\ast \ca{B}$. But $U_\ast \ca{B}$ is a tame $\ca{V}$-category (it is autonomous, so tame by assumption). Thus $U_\ast Z$ lies in the Cauchy completion of $U_\ast \ca{B}$. Proposition~\ref{CAUCHY_COMPLETION_BASE_CHANGE_PROP} implies that $Z$ lies in the Cauchy completion of $\ca{B}$, which shows that $\ca{B}$ is indeed tame.
\end{proof}

\begin{rmk}
 Theorem~\ref{AUTONOMOUS_MONOIDAL_TAME_THM} applies in particular to commutative algebras in $\Mod_H$, thought of as one-object symmetric monoidal $\Mod_H$-categories. The unique object of such a category is the unit, which is always dualizable.
\end{rmk}

The following proposition shows that our recognition theorem for ablian cosmoi with dense autonomous generator is applicable to the cosmos $\Mod_H$ if it is applicable to $\ca{V}$ and $H$ is dualizable.

\begin{prop}\label{DAG_FOR_MODULES_PROP}
 If $\ca{V}$ is an abelian cosmos with a dense autonomous generator $\ca{X}$ and $H \in \ca{V}$ is a dualizable Hopf algebra, then $\Mod_H$ is an abelian cosmos with dense autonomous generator. If the unit of $\ca{V}$ is finitely presentable, then so is the unit of $\Mod_H$.
\end{prop}

\begin{proof}
 Clearly $\Mod_H$ is abelian. The category $\Mod_H$ is the category of algebras for the cocontinuous monad $H\otimes -$. Any cocontinuous monad is strictly linear in the sense of \cite[Definition~1.8]{DAY_LINEAR}. Thus the objects $H\otimes X$, $X \in \ca{X}$ form a $\Set$-dense subcategory of $\Mod_H$ (see \cite[Theorem~1.3]{DAY_LINEAR}). They are dualizable because their underlying objects are, hence their closure under tensor products and duals forms a dense autonomous generator.
 
 The hom-sets of $\Mod_H$ are given by the equalizer
 \[
 \xymatrix{\Mod_H(A,B) \ar[r] & \ca{V}(A,B) \ar@<0.5ex>[r] \ar@<-0.5ex>[r] & \ca{V}(H\otimes A,B)}
 \]
 of sets. If $I \in \ca{V}$ is locally presentable, then so is $H$ (see Lemma~\ref{CAUCHY_OBJECTS_FGP_LEMMMA}). Since finite limits of sets commute with filtered colimits it follows that $A \in \Mod_H$ is finitely presentable if its underlying object is (recall from Proposition~\ref{COSMOI_WITH_DAG_ARE_LFP_PROP} that $\ca{V}$ is locally presentable as a closed category).
\end{proof}

The following result shows that Corollary~\ref{DGA_TAME_COR} is a special case of Theorem~\ref{AUTONOMOUS_MONOIDAL_TAME_THM}. A related result can be found in \cite[Theorem~18]{PAREIGIS_DG}, where it is shown that there exists a (non-commutative) Hopf algebra in the category of \emph{ungraded} $R$-modules whose comodules are differential graded $R$-modules. This is not quite sufficient for our purposes, because we really want to get a \emph{symmetric} monoidal category of modules over a Hopf algebra.

\begin{prop}\label{HOPF_ALGEBRA_FOR_DGM_PROP}
 There is a dualizable cocommutative Hopf algebra $H$ in the cosmos of graded $R$-modules such that $\Mod_H$ is the cosmos of differential graded $R$-modules.
\end{prop}

\begin{proof}
 Here we think of graded modules in the topologists sense, that is, we don't allow summation of elements of different degrees. The symmetry on the cosmos of graded $R$-modules is given by $a\otimes b \mapsto (-1)^{\lvert a \rvert \lvert b \rvert} b\otimes a$.
 
 The Hopf algebra $H$ is the exterior algebra $E(t)$ on a single generator of degree $-1$ (or $1$ if one wants the differentials to increase the degree by one). The underlying graded module of $H$ is given by $R \oplus R\cdot t$, with $\lvert t \rvert =-1$, so it is indeed dualizable. The multiplication satisfies $t^2=0$. The comultiplication is given by $\Delta(t)=t\otimes 1 + 1\otimes t$, and the antipode is given by $\nu(t)=-t$. In order to check that the multiplication and the comultiplication are compatible one has to use the fact that $(t\otimes 1)(1\otimes t)=-(1\otimes t)(t\otimes 1)$, which follows from the grading conventions of the symmetry in the category of graded $R$-modules. It is not hard to check that left modules of $H$ are differential graded $R$-modules, and that the tensor product of modules corresponds to the tensor product of differential graded $R$-modules with the usual differential $d(x\otimes y)=d(x)\otimes y + (-1)^{\lvert x \rvert} x\otimes y$. The induced symmetry inherits the usual sign conventions from the symmetry in the cosmos of graded $R$-modules. Since there is at most one closed structure for any monoidal category, it follows that $\Mod_H$ is equivalent to the cosmos of differential graded $R$-modules.
\end{proof}
\section{Tannakian duality for bialgebras and Hopf algebras}\label{BIALGEBRAS_HOPF_ALGEBRAS_SECTION}

\subsection{Multiplicative structures on \texorpdfstring{$\Mod(\ca{V})$}{Mod(V)}}
 The symmetric monoidal structure of $\ca{V}$ induces the structure of a symmetric monoidal bicategory on $\Mod(\ca{V})$. The tensor product of two $\ca{V}$-categories $\ca{A}$ and $\ca{A}^{\prime}$ has objects the pairs $(a,a^{\prime})$ of objects of $\ca{A}$ and of $\ca{A}^{\prime}$, with hom-objects given by the tensor product in $\ca{V}$. The tensor product of modules is given by $M\otimes N \bigl((b,b^{\prime}),(a,a^{\prime})\bigr)=M(b,a)\otimes N(b^{\prime},a^{\prime})$.
 
 Implicitly the interaction between the Tannakian adjunction and this monoidal structure has been studied extensively for certain cosmoi $\ca{V}$, in particular for the cosmos of vector spaces. For example, it is well-known that the coalgebra $L(w)$ of a fiber functor $w$ is a bialgebra if $w$ is a strong monoidal functor (tensor functor). A bialgebra is simply a monoid in the category of monoidal functors, and there is a monoidal structure on the category of fiber functors for which a suitably weakened notion of monoids gives the strong monoidal functors. These weak monoids are called \emph{pseudomonoids}. 

 The result about bialgebras and monoidal fiber functors therefore tells us that the Tannakian adjunction lifts to the categories of pseudomonoids. We will see that this is a consequence of the fact that the Tannakian adjunction is compatible with the tensor product on either side. The bicategorical interpretation of Tannaka duality makes this compatibility between the monoidal structure and the Tannakian biadjunction evident.

 Let $\ca{M}$ be the bicategory $\Mod(\ca{V})$ of $\ca{V}$-modules. To simplify the notation we will use capital letters $A, B \ldots$ for $\ca{V}$-categories and lowercase letters $f,g, \ldots$ (and ordinary arrows) for modules in the remainder of this section. Let $B$ be a monoidal $\ca{V}$-category. The tensor product and unit of $B$ induce modules $m \colon B\otimes B \rightarrow B$ and $u \colon I \rightarrow B$. Note that $m$ and $u$ are maps because they are modules induced by $\ca{V}$-functors (see Remark~\ref{VMOD_MAPS_RMK}). Then the bicategory $\Map(\ca{M},B)$ is a monoidal bicategory, with tensor product $w \bullet w^{\prime}$ of $w$ and $w^{\prime}$ given by
\[
 \xymatrix{ A\otimes A^\prime \ar[r]^-{w\otimes w^{\prime}}  & B \otimes B \ar[r]^-{m} & B}\smash{\rlap{.}}
\]
 On the other hand, for any map pseudomonoid $B$ (that is, a pseudomonoid whose multiplication and unit are given by left adjoint 1-cells), the category of 1-cells $B \rightarrow B$ comes equipped with a convolution monoidal structure. The convolution product $f \star g$ of two 1-cells $f,g \colon B \rightarrow B$ is given by
 \[
  \xymatrix{ B \ar[r]^-{\overline{m}} & B\otimes B \ar[r]^-{f\otimes g} & B\otimes B \ar[r]^-{m} & B}\smash{\rlap{.}}
 \]
 The convolution product lifts to the category $\Comon(B)$ of comonads on $B$. Recall that the left biadjoint of the Tannakian adjunction is given by the 2-functor
 \[
  L \colon \Map(\ca{M}, B) \rightarrow \Comon(B)  
 \]
 which sends an object $w \colon A \rightarrow B$ to the comonad $L(w)=w \cdot \overline{w}$, with comultiplication and counit induced by the unit and counit of the adjunction $w \dashv \overline{w}$.

 Since the tensor product $\otimes \colon \ca{M} \times \ca{M} \rightarrow \ca{M}$ is a pseudofunctor, we have an invertible 2-cell
 \[
  \def\objectstyle{\scriptstyle}
  \def\labelstyle{\scriptscriptstyle}
\def\twocellstyle{\scriptscriptstyle}
  \xymatrix{
  && A \otimes A \ar[rd]^{w \otimes w^{\prime}} \\
 B \ar[r]^-{\overline{m}} & B\otimes B 
 \ar[ru]^{\overline{w}\otimes \overline{w^{\prime}}} 
 \ar@/_/[rr]_{w \cdot \overline {w} \otimes w^{\prime} \cdot \overline{w^{\prime}} }
 \xtwocell[0,2]{}\omit{<-2>\cong}
 && B\otimes B \ar[r]^-{m} & B 
  } \smash{\rlap{.}}
 \]
 A pseudofunctor sends maps to maps, so the composite on the top is $L(m.w\otimes w^{\prime})$. The composite on the bottom is $L(w) \star L(w^{\prime})$ by definition of the convolution monoidal structure. In other words, we have shown that the left adjoint of the Tannakian adjunction preserves tensor products up to isomorphism: $L(w \bullet w^{\prime}) \cong L(w) \star L(w^{\prime})$.

 Suppose that $L$ is left biadjoint to $R$, and that $L$ is compatible with the monoidal structure up to equivalence $\phi \colon LA \otimes LB \rightarrow L(A\otimes B)$. Then the composite
 \[
  \xymatrix@!C=75pt{RA\otimes RB \ar[r]^-{\eta_{RA\otimes RB}} & RL(RA\otimes RB)  \ar[r]^{R(\phi^{-1})} & R(LRA \otimes LRB) \ar[r]^-{R(\varepsilon_A \otimes \varepsilon_B)} & R(A\otimes B) }
 \]
 endows the right biadjoint $R$ with a weak monoidal structure, where $\eta$ and $\varepsilon$ denote the unit and the counit of the biadjunction. Thus the biadjunction lifts to a biadjunction between the respective bicategories of pseudomonoids (this is a generalization of results from \cite{DAY_STREET} and \cite{MCCRUDDEN_BALANCED}, see Appendix~\ref{MONOIDAL_BIADJUNCTION_APPENDIX}). Once we check that our functor $L$ is indeed compatible with the tensor product, we get the desired Tannakian theorem for bialgebras.
 
\begin{example}
 For $B=I$, the unit $\ca{V}$-category, we have $\ca{M}(I,I)=\ca{V}$ and both the convolution monoidal structure and the monoidal structure given by composition coincide with the original monoidal structure on $\ca{V}$. A comonad is precisely a coalgebra, and a pseudomonoid in $\Comon(I)$ is precisely a bialgebra (recall that $\Comon(I)$ only has identity 2-cells, so a pseudomonoid is simply a monoid). On the other hand, a pseudomonoid in $\Map(\ca{M},B)$ is precisely a monoidal $\ca{V}$-category $A$ equipped with a strong monoidal $\ca{V}$-functor $w \colon A \rightarrow \overline{I}$. The Cauchy completion $\overline{I}$ of $I$ is the full subcategory $\ca{V}^d$ of $\ca{V}$ consisting of objects with duals. Thus we do get the desired relationship between monoidal categories with a monoidal fiber functor on the one hand and bialgebras in $\ca{V}$ on the other.
\end{example}

 The problem is that the notion of ``compatibility'' with the monoidal structure is quite complicated. The definition of a strong monoidal homomorphism between monoidal bicategories can be found in \cite[pp.\ 15-18]{GORDON_POWER_STREET}. In addition to a pseudonatural equivalence $\star \cdot L\times L \Rightarrow L \cdot \bullet$, we need to define three invertible modifications which are subject to two axioms. In our case, the situation is a bit simpler because the target bicategory only has identity 2-cells. This means in particular that all the modifications must be identity modifications, and defining them boils down to checking that their domain and codomain are actually equal. 
 
 Even this is a daunting task in the case $\ca{M}=\Mod(\ca{V})$, because the composition of 1-cells is given by a colimit formula. Instead of proving this directly for $\Mod(\ca{V})$, we use a theorem from \cite{GORDON_POWER_STREET} which tells us that $\Mod(\ca{V})$ is biequivalent to a much stricter structure called a \emph{Gray monoid} (see Definition~\ref{GRAY_MONOID_DFN}). The tensor product in a Gray monoid is strictly associative on objects, which will greatly simplify our calculations. In other words: we work with the seemingly more complicated notion of a Gray monoid instead of the particular monoidal bicategory $\Mod(\ca{V})$ to make our life simpler, not because we want to give a more general result. In the following sections we state our main theorems on lifting Tannakian biadjunctions. The terminology will be made precise in later sections.
 
\subsection{Monoidal 2-categories}\label{MONOIDAL_OUTLINE_SECTION}
 In general, if we start with a monoidal 2-category $\ca{M}$, then any map pseudomonoid induces a monoidal structure on the domain and target of the Tannakian biadjunction (see Propositions~\ref{CONVOLUTION_COMONAD_PROP} and \ref{SLICE_MONOIDAL_2CAT_PROP} respectively). In Section~\ref{TANNAKIAN_BIADJUNCTION_GRAY_SECTION}, we will prove that the Tannakian biadjunction lifts to the categories of pseudomonoids.

 This is a formal consequence of the fact that the left adjoint is endowed with a strong monoidal structure. Day and Street showed in \cite[Proposition~2]{DAY_STREET} that the right biadjoint of a strong monoidal pseudofunctor inherits the structure of a weak monoidal pseudofunctor. In Appendix~\ref{MONOIDAL_BIADJUNCTION_APPENDIX} we prove the following refinement of their result.
 
  \begin{cit}[Proposition~\ref{MONOIDAL_BIADJUNCTION_PROP}]
 Let $T \colon \ca{M} \rightarrow \ca{N}$ be a strong monoidal left biadjoint between  monoidal 2-categories, with right biadjoint $H$. Then $H$ can be endowed with the structure of a weak monoidal pseudofunctor, and the unit and counit with the structure of weak monoidal pseudonatural transformations, in such a way that the invertible modifications $\alpha$ and $\beta$ that replace the triangle identities become monoidal modifications.
 \end{cit}

 This result makes the following corollary plausible. It is a categorification of the well-known result that monoidal adjunctions between monoidal categories lift to the respective categories of monoids.

 \begin{cit}[Corollary~\ref{LIFT_TO_PSEUDOMONOIDS_COR}]
 Let $T \colon \ca{M} \rightarrow \ca{N}$ be a strong monoidal left biadjoint between two monoidal 2-categories, with right biadjoint $H$. If both $H$ and $T$ are normal, that is, they preserve identities strictly, then the biadjunction lifts to a biadjunction
 \[
 \xymatrix{ \PsMon(\ca{N}) \ar[dd]_{U} \rrtwocell<5>^{\PsMon(T)}_{\PsMon(H)}{`\perp} && \PsMon(\ca{M}) \ar[dd]^{U} \\ \\
           \ca{N} \rrtwocell<5>^{T}_{H}{`\perp} && \ca{M}}
 \]
 between the 2-categories of pseudomonoids. The underlying morphisms of the unit and the counit are given by the unit and the counit of the biadjunction $T \dashv H$.
 \end{cit}

Thus, in order to prove the following theorem we only need to show that the left biadjoint of the Tannakian adjunction is strong monoidal.

\begin{thm}\label{TANNAKIAN_ADJUNCTION_MONOIDAL_THM}
 Let $\ca{M}$ be a Gray monoid with Tannaka-Krein objects, and let $B$ be a map pseudomonoid in $\ca{M}$. Then the Tannakian biadjunction lifts to a biadjunction
 \[
 \xymatrix{ \mathbf{MonComon}(B) \ar[dd]_{U} \rrtwocell<5>{`\perp} && \PsMon\bigl(\Map(\ca{M},B) \bigr) \ar[dd]^{U} \\ \\
 \Comon(B) \rrtwocell<5>^{L(-)}_{*!<0pt,2pt>+{\Rep(-)}}{`\perp} && \Map(\ca{M},B) }
 \]
 between the category of monoidal comonads and the 2-category of pseudomonoids in $\Map(\ca{M},B)$. The latter is precisely the 2-category of map pseudomonoids in $\ca{M}$ equipped with a strong monoidal map to $B$.

 The underlying morphisms of the unit and the counit of the lifted adjunction are equal to the unit and the counit of the Tannakian biadjunction.
\end{thm}

These results provide a conceptual explanation for the fact that the coalgebra associated to a strong monoidal fiber functor inherits the structure of a (not necessarily commutative) bialgebra. A similar fact holds for braided and symmetric monoidal fiber functors.

\subsection{Braiding}\label{BRAIDING_OUTLINE_SECTION}
If $\ca{M}$ is also endowed with a braiding, and $B$ is a braided map pseudomonoid, then the monoidal structure on $\Comon(B)$ and $\Map(\ca{M},B)$ are braided, and we get the following results, whose proofs follow the same pattern as the proofs for the compatibility with the monoidal structure described in Section~\ref{MONOIDAL_OUTLINE_SECTION}.

\begin{cit}[Corollary~\ref{LIFT_TO_BRAIDED_PSEUDOMONOIDS_COR}]
 Let $T \colon \ca{M} \rightarrow \ca{N}$ be a braided strong monoidal left biadjoint between braided monoidal 2-categories, with right biadjoint $H$. If both $H$ and $T$ are normal, then the biadjunction lifts to a biadjunction
 \[
 \xymatrix{ \BrPsMon(\ca{N}) \ar[dd]_{U} \rrtwocell<5>^{\BrPsMon(T)}_{\BrPsMon(H)}{`\perp} && \BrPsMon(\ca{M}) \ar[dd]^{U} \\ \\
           \ca{N} \rrtwocell<5>^{T}_{H}{`\perp} && \ca{M}}
 \]
 between the 2-categories of braided pseudomonoids. The underlying morphisms of the unit and the counit are given by the unit and the counit of the biadjunction $T \dashv H$.
 \end{cit}

\begin{thm}\label{TANNAKIAN_ADJUNCTION_BRAIDED_THM}
 Let $\ca{M}$ be a braided Gray monoid with Tannaka-Krein objects, and let $B$ be a braided map pseuodomonoid in $\ca{M}$. Then the Tannakian biadjunction lifts to a biadjunction
 \[
 \xymatrix{ \mathbf{BrMonComon}(B) \ar[dd]_{U} \rrtwocell<5>{`\perp} && \BrPsMon\bigl(\Map(\ca{M},B) \bigr) \ar[dd]^{U} \\ \\
 \Comon(B) \rrtwocell<5>^{L(-)}_{*!<0pt,2pt>+{\Rep(-)}}{`\perp} && \Map(\ca{M},B) }
 \]
 between the category of braided monoidal comonads and the 2-category of braided pseudomonoids in $\Map(\ca{M},B)$. The latter is precisely the 2-category of braided map pseudomonoids in $\ca{M}$ equipped with a braided strong monoidal map to $B$.

 The underlying morphisms of the unit and the counit of the lifted adjunction are equal to the unit and the counit of the Tannakian biadjunction.
\end{thm}

\subsection{Syllepsis and symmetry}\label{SYLLEPSIS_OUTLINE_SECTION}
In the world of monoidal 2-categories, there is a notion lying between a braiding and a symmetry called a \emph{syllepsis}. Similarly to how a symmetric monoidal category is a braided monoidal category subject to an additional equation between certain 1-cells, a symmetric monoidal 2-category is a sylleptic monoidal 2-category subject to one additional equation between certain 2-cells. We are mainly interested in the case of symmetric monoidal 2-categories $\ca{M}$, but it is clear from what we just said that we need a good understanding of the sylleptic monoidal case first.

We will show that if $\ca{M}$ is sylleptic (symmetric) and $B$ is a symmetric map pseudomonoid, then $\Comon(B)$ is symmetric (it is, after all, a 1-category) and $\Map(\ca{M},B)$ is a sylleptic (symmetric) 2-category.

\begin{cit}[Corollary~\ref{LIFT_TO_SYMMETRIC_PSEUDOMONOIDS_COR}]
 Let $T \colon \ca{M} \rightarrow \ca{N}$ be a sylleptic strong monoidal left biadjoint between symmetric monoidal 2-categories, with right biadjoint $H$. If both $H$ and $T$ are normal, then the biadjunction lifts to a biadjunction
 \[
 \xymatrix{ \SymPsMon(\ca{N}) \ar[dd]_{U} \rrtwocell<5>^{\SymPsMon(T)}_{\SymPsMon(H)}{`\perp} && \SymPsMon(\ca{M}) \ar[dd]^{U} \\ \\
           \ca{N} \rrtwocell<5>^{T}_{H}{`\perp} && \ca{M}}
 \]
 between the 2-categories of symmetric pseudomonoids. The underlying morphisms of the unit and the counit are given by the unit and the counit of the biadjunction $T \dashv H$. 
 \end{cit}

\begin{thm}\label{TANNAKIAN_ADJUNCTION_SYLLEPTIC_THM}
 Let $\ca{M}$ be a sylleptic (or symmetric) Gray monoid with Tannaka-Krein objects, and let $B$ be a symmetric map pseudomonoid in $\ca{M}$. Then the Tannakian biadjunction lifts to a biadjunction
 \[
 \xymatrix{ \mathbf{SymMonComon}(B) \ar[dd]_{U} \rrtwocell<5>{`\perp} && \SymPsMon\bigl(\Map(\ca{M},B) \bigr) \ar[dd]^{U} \\ \\
 \Comon(B) \rrtwocell<5>^{L(-)}_{*!<0pt,2pt>+{\Rep(-)}}{`\perp} && \Map(\ca{M},B) }
 \]
 between the category of symmetric monoidal comonads and the 2-category of symmetric pseudomonoids in $\Map(\ca{M},B)$. The latter is precisely the 2-category of symmetric map pseudomonoids in $\ca{M}$ equipped with a symmetric strong monoidal map to $B$.

 The underlying morphisms of the unit and the counit of the lifted adjunction are equal to the unit and the counit of the Tannakian biadjunction.
\end{thm}

\subsection{Autonomous pseudomonoids and Hopf monoidal comonads}
 We want to show that the left adjoint of the Tannakian adjunction on $\Mod(\ca{V})$ sends autonomous categories to Hopf algebroids if the fiber functor lands in modules for a commutative algebra in $\ca{V}$. In order to deal with the more general fiber functors landing in the category of presheaves of a small $\ca{V}$-category, we need to use the concept of a Hopf monoidal comonad (see Definition~\ref{HOPF_COMONAD_DFN}).
 
 There is a structural way to characterize autonomous monoidal $\ca{V}$-categories in terms of the bicategory $\Mod(\ca{V})$: they correspond to autonomous pseudomonoids in $\Mod(\ca{V})$ (see Definition~\ref{AUTONOMOUS_DEFINITION} and Proposition~\ref{VCAT_AUTONOMOUS_PROP}). The notion of an autonomous pseudomonoid in a monoidal bicategory involves 1-cells which are not maps, so it can't be detected directly in the 2-category of maps. Therefore the analysis of duals requires more work than the analysis of the various monoidal structures discussed above. In Section~\ref{HOPF_MONOIDAL_COMONADS_SECTION} we will prove the following result.
 
 \begin{thm}\label{HOPF_COMONAD_THM}
 Let $A$ and $B$ be autonomous map pseudomonoids in a Gray monoid $\ca{M}$, and let $(w,\psi,\psi_0)\colon A \rightarrow B$ be a strong monoidal map. Then the induced comonad $L(w)=w \cdot \overline{w}$ is a Hopf monoidal comonad.
\end{thm}

\begin{cor}
 Let $\ca{V}$ be a cosmos. The left adjoint of the neutral Tannakian biadjunction
 \[
  L \colon \Vcat \slash \ca{V}^d \rightarrow  \Coalg(\ca{V})
 \]
 (where $\ca{V}^d=\overline{\ca{I}}$ denotes the sub-$\ca{V}$-category of dualizable objects in $\ca{V}$) sends (symmetric) autonomous monoidal categories to (commutative) Hopf algebras.
\end{cor}

\begin{proof}
 By Theorem~\ref{HOPF_COMONAD_THM}, the comonad $L(w) \colon \ca{I} \xslashedrightarrow{} \ca{I}$ is a Hopf monoidal if the domain of $w$ is autonomous monoidal. Moreover, a module $M \colon \ca{I} \xslashedrightarrow{} \ca{I}$ can be identified with the cocontinuous $\ca{V}$-functor $\ca{V} \rightarrow \ca{V}$. Thus $L(w)\otimes -$ is a Hopf monoidal comonad, and the conclusion follows from \cite[Remark~5.6]{BRUGUIERES_LACK_VIRELIZIER} (the fusion operators for the monoidal comonad $L(w)\otimes -$ are precisely the fusion operators of the bialgebra $L(w)$).
\end{proof}

 In Section~\ref{AFFINE_GROUPOIDS} we will show that the Hopf monoidal comonads on a  symmetric monoidal $\ca{V}$-category with one object (that is, on a commutative algebra in $\ca{V}$) correspond to Hopf algebroids in the case where $\ca{M}=\Mod(\ca{V})$. We obtain the desired relation between autonomous symmetric monoidal categories and Hopf algebroids as a corollary. Since it is our main case of interest, we will discuss this consequence first.

\section{Affine groupoids over commutative rings}\label{AFFINE_GROUPOIDS}
 We relate the results about Hopf monoidal comonads from Section~\ref{HOPF_MONOIDAL_COMONADS_SECTION} to the concrete situation where $\ca{M}$ is a Gray monoid equivalent to $\Mod(\ca{V})$ for some cosmos $\ca{V}$. The base pseudomonoid $B$ will be a monoidal $\ca{V}$-category with a single object, that is, a commutative algebra (or commutative monoid) in $\ca{V}$. Further specializing this to the case where $\ca{V}$ is the cosmos of $R$-modules for some commutative ring $R$ we will get the desired result about Tannaka duality for affine groupoids over $R$.

\subsection{Monoidal morphisms in \texorpdfstring{$\Mod(\ca{V})$}{V} and cospans in \texorpdfstring{$\CommAlg(\ca{V})$}{CommAlg(V)}}
Once we introduce the notion of Hopf monoidal comonads (see Definition~\ref{HOPF_COMONAD_DFN}), it will be easy to see that Hopf monoidal comonads in the category of cospans in a finitely cocomplete category $\ca{E}$ are exactly the groupoids internal to $\ca{E}^{\op}$ (see Example~\ref{GROUPOIDS_ARE_HOPF_COMONADS_EXAMPLE}). Taking $\ca{E}=\CommAlg(\ca{V})$, the category of commutative algebras in $\ca{V}$ we find that Hopf algebroids in the usual sense correspond to Hopf monoidal comonads in the category of cospans of $\ca{E}$. Thus, in order to relate the Hopf monoidal comonads in $\Mod(\ca{V})$ to groupoids, it suffices to relate $\Mod(\ca{V})$ to the bicategory of cospans in $\ca{E}$.

\begin{prop}\label{PROFUNCTOR_COSPAN_PROP}
Let $\ca{V}$ be a cosmos, and let $\ca{N}$ be the bicategory with objects the commutative algebras in $\ca{V}$, morphisms the symmetric monoidal morphisms in $\Mod(\ca{V})$ between these algebras, thought of as symmetric pseudomonoids in $\Mod(\ca{V})$, and 2-cells the monoidal 2-cells between them. Then $\ca{N}$ is biequivalent to the bicategory of cospans in the category of commutative algebras in $\ca{V}$ and algebra homomorphisms between them.
\end{prop}

\begin{lemma}\label{COMMUTATIVE_ALGEBRA_LEMMA}
Let $\ca{V}$ be a cosmos. Then $\CommAlg(\ca{V})$ has pushouts given by the tensor product of commutative algebras. The initial object is given by the unit object $I$ of $\ca{V}$. In particular, $\ca{E}=\CommAlg(\ca{V})$ is finitely cocomplete.
\end{lemma}

\begin{proof}
 This is a well-known fact in the case $\ca{V}=\Mod_R$ for some commutative ring $R$. The general case follows from a straightforward diagram chase.
\end{proof}

\begin{proof}[Proof of Proposition~\ref{PROFUNCTOR_COSPAN_PROP}.]
 Let $A$ and $B$ be commutative algebras in $\ca{V}$. The above lemma can be applied to the monoidal category $\ca{N}(A,B)=\Mod(\ca{V})(A,B)$ with convolution tensor product. This is simply the category of $A\otimes B$-modules, and the convolution tensor product is the coequalizer of the two $A$-actions and the two $B$-actions on the ordinary tensor product, that is, it is the tensor product over $A\otimes B$. Moreover, an application of the calculus of mates shows that a module is a symmetric monoidal morphism from $A$ to $B$ if and only if it is a commutative algebra for the convolution monoidal structure (cf.\ Example~\ref{BRAIDED_PSEUDOMONOIDS_IN_HOM_EXAMPLE}). 

 In other words, the bicategory $\ca{N}$ has objects the commutative algebras, and 1-cells between $A$ and $B$ are precisely the commutative $A\otimes B$-algebras. It is well known that an $A$-algebra $C$ can equivalently be described as an algebra in $\ca{V}$ equipped with a morphism of algebras $A \rightarrow C$. This shows that the category $\ca{N}(A,B)$ can be identified with the category of cospans between $A$ and $B$. 

 From Lemma~\ref{COMMUTATIVE_ALGEBRA_LEMMA} we know that the pushout in the category of commutative algebras is given by tensoring over the common domain of the two homomorphisms. On the other hand, the coend that computes the composition of an $A \otimes B$-module $M$ with a $B\otimes C$-module $N$ is given by the coequalizer of the two maps
\[
\xymatrix{M \otimes B \otimes N \ar@<0.5ex>[r] \ar@<-0.5ex>[r] & M\otimes N } 
\]
 which are given by the two different $B$-actions, that is, composition of modules is also given by tensoring over $B$. Thus the composition of two modules has the same universal property as the composition between the corresponding cospans. This shows that the correspondence between symmetric monoidal modules $A \xslashedrightarrow{} B$ and cospans of commutative algebras between $A$ and $B$ is compatible with compositions, at least up to isomorphism. Since this isomorphism is induced by a universal property we do in fact get a pseudofunctor, which shows that the two bicategories are indeed biequivalent.
\end{proof}

\subsection{Tannaka duality for Hopf algebroids and affine groupoids}
As a corollary, we obtain the desired fact that the left adjoint of the Tannakian adjunction sends autonomous $\ca{V}$-categories to Hopf algebroids. Further specializing to the case of cosmoi of $R$-modules, we obtain our recognition results for categories of representations of affine groupoids.

\begin{cor}\label{TANNAKIAN_ADJUNCTION_GROUPOID_COR}
 Let $\ca{V}$ be a cosmos, and let $B$ be a commutative algebra in $B$. We write $\Cat_B$ for the category of categories internal to $\CommAlg(\ca{V})^{\op}$ whose object of objects is $B$. The morphisms in $\Cat_B$ are the internal functors which are the identity on the object of objects. The Tannakian adjunction gives a biadjunction
 \[
 \xymatrix{
  \Cat_B^{\op} \xtwocell[0,3]{}<5>^{L(-)}_{\Rep(-)}{`\perp} &&&
  {\mathbf{SymMon}\mbox{-}\Vcat \slash\Mod^{d}_B}
 }
 \]
between $\Cat_B^{\op}$ and the category of small symmetric monoidal $\ca{V}$-categories equipped with a symmetric monoidal functor to the $\ca{V}$-category $\Mod_B^d$ of dualizable $B$-modules.
 
 Moreover, the left biadjoint sends symmetric monoidal $\ca{V}$-categories with duals to groupoids internal to $\CommAlg(\ca{V})^{\op}$.
\end{cor}

\begin{proof}
 Proposition~\ref{PROFUNCTOR_COSPAN_PROP} implies that the domain of the left biadjoint of the Tannakian adjunction is given by the comonads on $B$ in the category of cospans in $\CommAlg(\ca{V})$. This is the opposite of the category of monads in the category of spans in $\CommAlg(\ca{V})^{\op}$, and it is well-known that an internal category is precisely a monad in the category of spans. We get the desired adjunction if we combine this with Theorem~\ref{TANNAKIAN_ADJUNCTION_SYLLEPTIC_THM} (recall from Proposition~\ref{AUTONOMOUS_CAUCHY_PROP} that $\Mod_B^d$ is the Cauchy completion of $B$, considered as an autonomous monoidal $\ca{V}$-category). The fact about the left adjoint landing in internal groupoids follows from Theorem~\ref{HOPF_COMONAD_THM} and Example~\ref{GROUPOIDS_ARE_HOPF_COMONADS_EXAMPLE}.
\end{proof}

\begin{thm}\label{TANNAKA_AFFINE_GROUPOID_THM}
 Let $B$ be a commutative $R$-algebra, let $\ca{A}$ be an additive autonomous symmetric monoidal $R$-linear category, and let $w \colon \ca{A} \rightarrow \Mod_B$ be a symmetric strong monoidal $R$-linear functor. Suppose that
 \begin{enumerate}
\item[i)] 
 the functor $w_0$ is faithful and reflects isomorphisms;
\item[ii)] 
 the category $\el(w_0)$ of elements of $w_0$ is cofiltered; and
\item[iii)]
 if the cokernel of $w_0(f)$ is finitely generated and projective, then the cokernel of $f$ exists and is preserved by $w_0$.
 \end{enumerate}
Then there exists an affine groupoid $G=\Spec(H)$ acting on $\Spec(B)$ and a symmetric strong monoidal equivalence $\ca{A} \simeq \Rep(G)$. This equivalence is compatible with $w$ and the forgetful functor. Moreover, the Hopf algebroid $H$ is given by the coend
\[
 H = \int^{A \in \ca{A}} w(A) \otimes_B w(A)^\vee\rlap{,}
\]
 where the right action on $H$ is induced by the $B$-actions on $w(A)^\vee$, and the left action is induced by the $B$-actions on $w(A)$, and $H$ is flat as a left and as a right $B$-module.
\end{thm}

\begin{proof}
 Since $w$ is strong monoidal, $w(A)$ is dualizable for every $A \in \ca{A}$, that is, $w$ factors through the full subcategory of finitely generated projective $B$-modules. Thus we can apply Corollary~\ref{TANNAKIAN_ADJUNCTION_GROUPOID_COR}. The statement about flatness follows from Theorem~\ref{COMMUTATIVE_RING_THM}, Corollary~\ref{TANNAKIAN_ADJUNCTION_GROUPOID_COR}, and the fact that the source and the target morphism of a groupoid are isomorphic, so that one of them is flat if and only if the other is. The explicit description of $H$ in terms of the above coend is a consequence of the definition of $L(w)$ (see Proposition~\ref{LEFT_BIADJOINT_2FUNCTOR_PROP}) and the definition of composition in the bicategory of modules (see Definition~\ref{MODULE_DFN}). 
\end{proof}

\section{The Tannakian biadjunction for Gray monoids}\label{TANNAKIAN_BIADJUNCTION_GRAY_SECTION}
\subsection{String diagrams for Gray monoids}\label{GRAY_MONOIDS_SECTION}
 Mac Lane's coherence theorem tells us that every monoidal category is equivalent to a strictly associative and unital monoidal category. The analogue for monoidal 2-categories is not true: not every monoidal bicategory is equivalent to a strict monoidal 2-category. It is, however, still true that every monoidal bicategory is equivalent to a much stricter structure, called a Gray monoid. The category of (small) 2-categories admits a symmetric monoidal closed structure with internal hom given by the 2-category of 2-functors, pseudonatural transformations and modifications (see\cite[\S~4.8]{GORDON_POWER_STREET}). Its tensor product is called the Gray tensor product. Certain squares that commute strictly in the cartesian product of 2-categories only commute up to invertible 2-cell in the Gray tensor product. A Gray monoid is exactly a monoid in for the Gray tensor product. In order to work with Gray monoids we will need to give a more explicit definition (see \cite[Definition~1]{DAY_STREET}).

\begin{dfn}\label{GRAY_MONOID_DFN}
A \emph{Gray monoid} is a 2-category $\ca{M}$ endowed with the following data:
\begin{enumerate}
 \item[(a)] an object $I$;
 \item[(b)] for every object $A$, two 2-functors $L_A, R_A \colon \ca{M} \rightarrow \ca{M}$ (giving left and right multiplication by the object $A$ respectively) such that for all objects $A, B \in \ca{M}$, we have $L_A(B)=R_B(A)$, and the equations 
 \[
L_I=R_I=\id_{\ca{M}},\quad L_{AB}=L_A L_B,\quad R_{AB}=R_B R_A, \quad \text{and}\quad R_B L_A = L_A R_B  
 \]
 hold, where $AB\defl L_A(B)$. For a morphism $f$ and a 2-cell $\alpha$ we use the abbreviations $fB \defl R_B(f)$, $A^\prime \alpha \defl L_{A^\prime}(\alpha)$, and so on; and
 \item[(c)] for any two morphisms $f \colon A \rightarrow A^{\prime}$ and $g \colon B \rightarrow B^\prime$ an invertible 2-cell
\[
\def\objectstyle{\scriptstyle}
\def\labelstyle{\scriptscriptstyle}
\def\twocellstyle{\scriptscriptstyle}
\xymatrix{AB \ar[r]^-{fB} \ar[d]_{Ag} \xtwocell[1,1]{}\omit{*!<-5pt,0pt>+{c_{f,g}}} & A^\prime B \ar[d]^{A^\prime g} \\
AB^\prime \ar[r]_-{fB^\prime} & A^\prime B^\prime}
\]
\end{enumerate}
 subject to the axioms:
\begin{enumerate}
 \item[(i)] if both $f$ and $g$ are identity arrows, then $c_{f,g}$ is an identity 2-cell;
 \item[(ii)] for all morphisms $f \colon A \rightarrow A^\prime$ and $g \colon B \rightarrow B^\prime$ and all objects $X$, the equalities $Xc_{f,g}=c_{Xf,g}$, $c_{fX,g}=c_{f,Xg}$ and $c_{f,g}X=c_{f,gX}$ hold; 
 \item[(iii)] for all 2-cells $\alpha \colon f \Rightarrow h \colon A \rightarrow A^\prime$ and $\beta \colon g \Rightarrow k \colon B \rightarrow B^\prime$,
\[
\vcenter{\hbox{
\def\objectstyle{\scriptstyle}
\def\labelstyle{\scriptscriptstyle}
\def\twocellstyle{\scriptscriptstyle}
\xymatrix{
& A^\prime B \ar[dr]^-{A^\prime g} \\
AB \drtwocell^{Ag}_{Ak}{*!<-1pt,-1pt>+{A \beta}} \ar[ur]^-{fB} \rrtwocell\omit{<-2>*!<-5pt,0pt>+{c_{f,g}}} & & A^\prime B^\prime  \\
& AB^\prime \urtwocell^{f B^\prime }_{h B^\prime}{*!<-2pt,-2pt>+{\alpha B^\prime}}
}}}
=
\def\objectstyle{\scriptstyle}
\def\labelstyle{\scriptscriptstyle}
\def\twocellstyle{\scriptscriptstyle}
\vcenter{\hbox{\xymatrix{
& A^\prime B \drtwocell^{A^\prime g}_{A^\prime k}{*!<-1pt,1pt>+{A^\prime \beta}} \\
AB \urtwocell^{fB}_{hB}{*!<-1pt,-1pt>+{\alpha B}} \ar[rd]_{Ak} \rrtwocell\omit{<1>*!<-5pt,0pt>+{c_{h,k}}} & & A^\prime B^\prime  \\
& AB^\prime \ar[ur]_-{hB^\prime}
}}}
\]
 \item[(iv)] for all morphisms $f \colon A \rightarrow A^\prime$, $f^\prime \colon A^\prime \rightarrow A^{\prime \prime}$, $g \colon B \rightarrow B^\prime$, and $g^\prime \colon B^\prime \rightarrow B^{\prime \prime}$, the equality
\[
\vcenter{\hbox{
\def\objectstyle{\scriptstyle}
\def\labelstyle{\scriptscriptstyle}
\def\twocellstyle{\scriptscriptstyle}
\xymatrix{
AB \ar[r]^-{fB} \ar[d]_{Ag} \xtwocell[1,1]{}\omit{*!<-5pt,0pt>+{c_{f,g}}} & A^\prime B \ar[d]^{A^\prime g} \xtwocell[1,1]{}\omit{*!<-6pt,0pt>+{c_{f^\prime,g}}} \ar[r]^-{f^\prime B} & A^{\prime\prime} B \ar[d]^{A^{\prime\prime} g} \\
AB^\prime \ar[r]|-{fB^\prime} \ar[d]_{A g^\prime} \xtwocell[1,1]{}\omit{*!<-6pt,0pt>+{c_{f,g^\prime}}} & A^\prime B^\prime \ar[r]|-{f^\prime B^{\prime}} \ar[d]^{A^\prime g^\prime} \xtwocell[1,1]{}\omit{*!<-7pt,0pt>+{c_{f^\prime,g^\prime}}} & A^{\prime \prime} B^\prime \ar[d]^{A^{\prime\prime} g^\prime} \\
AB^{\prime \prime} \ar[r]_-{fB^{\prime\prime}} & A^\prime B^{\prime\prime} \ar[r]^-{f^\prime B^{\prime \prime}} & A^{\prime\prime} B^{\prime\prime} \\
}}}
=
\vcenter{\hbox{
\def\objectstyle{\scriptstyle}
\def\labelstyle{\scriptscriptstyle}
\def\twocellstyle{\scriptscriptstyle}
\xymatrix@=1cm{
AB \ar[r]^-{f^\prime fB} \ar[d]_{A g^\prime g} \xtwocell[1,1]{}\omit{*!<-2pt,5pt>+{c_{f^\prime f,g^\prime g}}} & A^\prime B \ar[d]^{A^\prime g^\prime g} \\
AB^\prime \ar[r]_-{f^\prime f B^\prime} & A^\prime B^\prime
}}}
\]
holds.
\end{enumerate}
\end{dfn}

 The tensor product on a Gray monoid is defined on objects by $A\otimes B \defl AB$, on 1-cells $f \colon A \rightarrow A^{\prime}$, $g \colon B \rightarrow B^{\prime}$ by $f\otimes g \defl A^{\prime} g \cdot fB$, and similarly for 2-cells. Note that this convention is opposite to the one from \cite{DAY_STREET}, but the interchange morphism gives a pseudonatural equivalence between the two. We will frequently use the fact that $c_{f,1}$ and $c_{1,g}$ are identity 2-cells, which is an immediate consequence of axioms i) and iv). 
 
 These axioms become more intuitive once we pass to a reasonable string diagram notation. Let $\ca{M}$ be a Gray monoid, and let $f \colon A \rightarrow A^{\prime}$, $g \colon B \rightarrow B^{\prime}$ be 1-cells in $\ca{M}$. We introduce the notation
\[
c_{f,g}=\vcenter{\hbox{


}}
\]
 for the interchange morphisms. This is justified by Axioms (ii) and (iv), and by the following lemma.

\begin{lemma}
 The two Reidemeister moves are valid operations for string diagrams in a Gray monoid, that is, the two equalities
\[
\vcenter{\hbox{

%
}}
\]
hold for all morphisms $f \colon A \rightarrow A^\prime$, $g \colon B \rightarrow B^\prime$, and $h \colon C \rightarrow C^\prime$.
\end{lemma}

\begin{proof}
The first equation follows from Axiom (iv) and Axiom (iii) applied to a Gray interchange 2-cell, and the second equation is an immediate consequence from the definition.
\end{proof}

\begin{rmk}
We can interpret this notation as a 2-dimensional projection of surface diagrams for Gray monoids, a tentative notion which is not yet fully developed. In our notation, the information about the different layers of the surface diagram is stored only in the labels. Axiom (iii) for a Gray monoid says that strings (morphisms) and 2-cells in different layers can be moved past each other:
\[
\vcenter{\hbox{


}}
\]
Generally, one has to be careful not to apply topological manipulations unless the objects in question are located in different layers, so the labels are relevant for deciding whether or not a certain move of strings is allowed.
\end{rmk}

The reason for our interest in Gray monoids is the following theorem. It allows us to prove theorems about monoidal bicategories as long as the statements are invariant under biequivalence.

\begin{thm}\label{MON_BICAT_COHERENCE_THM}
 Every monoidal bicategory is biequivalent to a Gray monoid.
\end{thm}

\begin{proof}
 See \cite[Theorem~8.1]{GORDON_POWER_STREET}.
\end{proof}

\begin{dfn}\label{PSEUDOMONOID_DFN}
 A \emph{pseudomonoid} in a Gray monoid $\ca{M}$ is a monoid for which the associativity law and the unit laws only hold up to coherent invertible 2-cells. More precisely, a pseudomonoid $(B,m,u,\alpha,\lambda,\rho)$ is an object $B \in \ca{M}$, equipped with two morphisms $m \colon BB \rightarrow B$ and $u \colon I \rightarrow B$ and invertible 2-cells
 \[
  \vcenter{\hbox{


}}
 \]
 subject to the axioms
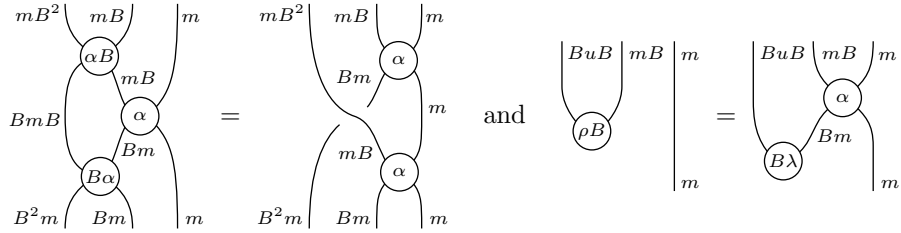
\[
\vcenter{\hbox{

%
}}
\]
 We call $B$ a \emph{map pseudomonoid} if $m$ and $u$ are maps.
\end{dfn}

 A pseudomonoid in a Gray monoid equivalent to the 2-category of categories is simply a monoidal category. The axiom involving the natural isomorphism $\alpha$ is precisely the famous pentagon axiom for a monoidal category.

 \subsection{Compatibility with the monoidal structure} \label{LEFT_ADJOINT_MONOIDAL_SECTION} 
 The goal of this section is to show that the left adjoint of the Tannakian biadjunction can be endowed with the structure of a strong monoidal pseudofunctor if $\ca{M}$ is a Gray monoid and $B$ is a map pseudomonoid. As already mentioned in Section~\ref{MONOIDAL_OUTLINE_SECTION}, this allows us to lift the biadjunction to the categories of pseudomonoids on either side. At the end of this section we can therefore give a proof of Theorem~\ref{TANNAKIAN_ADJUNCTION_MONOIDAL_THM}.
 
 The unit and counit of the resulting biadjunction between pseudomonoids has the same underlying 1-cells as the Tannakian biadjunction, so we can solve the reconstruction and recognition problems for bialgebras (more generally, monoidal comonads) as long as we understand the reconstruction and recognition problem for coalgebras (comonads). In particular, the recognition results proved in Sections~\ref{RECOGNITION_SECTION}, \ref{DAG_RECOGNITION_SECTION}, and \ref{ABELIAN_RECOGNITION_SECTION} extend to the biadjunction between map pseudomonoids.
 
 In order to turn the Tannakian biadjunction into a monoidal biadjunction we first have to endow source and target 2-categories with the structure of a monoidal 2-category.

\begin{prop}\label{SLICE_TENSOR_PRODUCT_PROP}
 Let $\ca{M}$ be a Gray monoid, and let $(B,m,u)$ be a pseudomonoid in $\ca{M}$. Let $v \colon X \rightarrow B$, $v^{\prime} \colon X^{\prime} \rightarrow B$, $w \colon Y \rightarrow B$ and $w^{\prime} \colon Y^{\prime} \rightarrow B$ be objects of the lax slice category $\ca{M}\slash_\ell B$. Let $(a,\alpha)$ and $(a^{\prime},\alpha^{\prime})$ (resp.\ $(b,\beta)$ and $(b^{\prime},\beta^{\prime})$) be morphisms from $v$ to $v^{\prime}$ (resp.\ from $w$ to $w^{\prime}$). Let $\phi \colon (a,\alpha) \Rightarrow (a^{\prime},\alpha^{\prime})$ and $\psi \colon (b,\beta) \Rightarrow (b^{\prime},\beta^{\prime})$ be 2-cells in $\ca{M}\slash_\ell B$.
 Then the assignments $v\bullet w \defl m \cdot Bw \cdot vY$,
\[
 (a,\alpha) \bullet (b,\beta) \defl \vcenter{\hbox{


}}
\]
 define a normal pseudofunctor $\bullet \colon \ca{M}\slash_{\ell} B \times \ca{M}\slash_{\ell} B \rightarrow \ca{M}\slash_{\ell} B$. If $(B,m,u)$ is a map pseudomonoid, then $\bullet$ restricts to a pseudofunctor on $\Map(\ca{M},B)$.
\end{prop}

\begin{proof}
To show that $\bullet$ gives a pseudofunctor we need to define the interchange 2-cell. Let $(c,\gamma) \colon v^{\prime} \rightarrow v^{\prime\prime}$ and $(d,\delta) \colon w^{\prime} \rightarrow w^{\prime \prime}$ be morphisms in $\ca{M}\slash_{\ell} B$. The 2-cell
\[
\vcenter{\hbox{

%
}}
\]
endows $\bullet$ with the desired structure of a normal pseudofunctor. If $m$ and $u$ are maps this restricts to a pseudofunctor on $\Map(\ca{M},B)$ because any 2-functor preserves maps and invertible 2-cells.
\end{proof}

\begin{prop}\label{SLICE_MONOIDAL_2CAT_PROP}
Let $\ca{M}$ be a Gray monoid and let $(B,m,u)$ be a pseudomonoid in $\ca{M}$. Then $\ca{M}\slash_\ell B$ is a monoidal 2-category with tensor product $\bullet$ and unit $u \colon I \rightarrow B$. The component of the associator at the objects $w \colon X \rightarrow B$, $w^{\prime} \colon X^{\prime} \rightarrow B$, $w^{\prime\prime} \colon X^{\prime\prime} \rightarrow B$ is given by the 2-cell $(\id_{X X^{\prime} X^{\prime\prime}}, \alpha_\bullet)$ where
\[
\alpha_\bullet=\vcenter{\hbox{


}}
\]
and the components of the left and right unit 2-cells are given by $(\id_X, \rho_\bullet)$ and $(\id_X, \lambda_\bullet)$ where
\[
\rho_\bullet= 
\vcenter{\hbox{

%
}}
\]
 respectively. Each of these gives \emph{strict} natural transformations, and all the modifications in the definition of a monoidal 2-category are identity modifications. If $(B,m,u)$ is a map pseudomonoid, then the structure of monoidal 2-category on $\ca{M} \slash_\ell B$ restricts to $\Map(\ca{M},B)$.
\end{prop}

\begin{proof}
We only need to check is that the domain and codomain 2-cells of the modifications in the definition of a monoidal 2-category (see \cite[Section~2.6]{GORDON_POWER_STREET}) coincide. Then the modifications can indeed be chosen to be identities. We leave the routine calculations to the reader.
\end{proof}

\begin{dfn}
Let $\ca{M}$ be a Gray monoid, and let $(A,p,j)$ and $(B,m,u)$ be map pseudomonoids in $\ca{M}$. The \emph{convolution product} is the functor
\[
\star \colon \ca{M}(A,B)\times \ca{M}(A,B) \rightarrow \ca{M}(A,B)
\]
which is given on objects (1-cells of $\ca{M}$) by $(f,g) \mapsto m \cdot Bg \cdot fA \cdot \overline{p}$ and on morphisms (2-cells of $\ca{M}$) by
\[
\vcenter{\hbox{


}}
\]
\end{dfn}

\begin{prop}
Let $\ca{M}$ be a Gray monoid, and let $(A,p,j)$ and $(B,m,u)$ be map pseudomonoids in $\ca{M}$. Then the convolution product defines a monoidal structure on $\ca{M}(A,B)$ with unit $u \cdot \overline{j} \colon A \rightarrow B$. The associator is given by
\[
\vcenter{\hbox{

%
}}
\]
and the unit isomorphisms are given by
\[
\vcenter{\hbox{

%
}}
\]
 respectively.
\end{prop}

\begin{proof}
See \cite[Proposition~4]{DAY_STREET}.
\end{proof}

\begin{prop}\label{CONVOLUTION_COMONAD_PROP}
Let $\ca{M}$ be a Gray monoid and let $(B,m,u)$ be a map pseudomonoid in $\ca{M}$. Then the convolution monoidal structure on $\ca{M}(B,B)$ lifts to a monoidal structure on the category $\Comon(B)$ of comonads on $B$. The comultiplication and the counit of $c \star c^{\prime}$ are given by
\[
\vcenter{\hbox{

%
}}
\]
respectively. The unit $u \cdot \overline{u}$ is endowed with the comonad structure induced by the adjunction $u \dashv \overline{u}$.
\end{prop}

\begin{proof}
 We have to check that the associator and the unit morphisms are compatible with the comultiplications and the counits of their domain and codomain. Using the fact that the mate of a natural transformation can be moved past the unit of the adjunction we find that the 2-cells
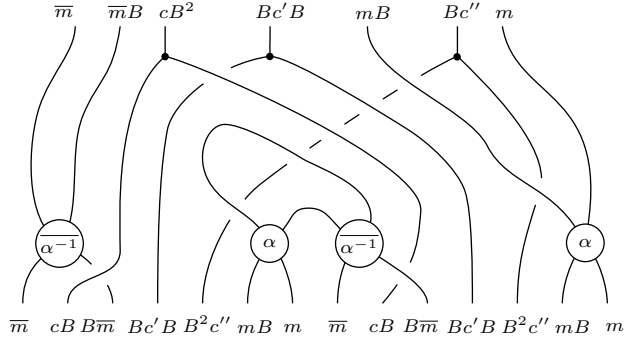
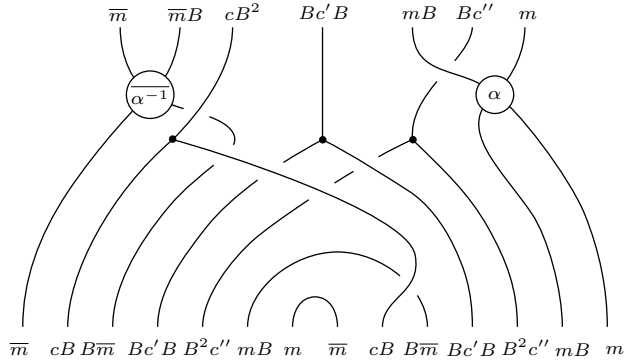
\[
\vcenter{\hbox{
%
}}
\]
 are equal. This shows that the associator is compatible with the comultiplication. Compatibility with the counit can be shown in a similar fashion: one uses the fact that a mate can be moved past the counit of the adjunction. 
 
 Checking that the unit isomorphisms for the tensor product $\star$ are compatible with the comultiplication and the counit is left to the reader.
\end{proof}

\begin{example}\label{MONOIDAL_MORPHISM_EXAMPLE}
 A pseudomonoid in $\ca{M}(A,B)$ is simply a commutative monoid in $\ca{M}(A,B)$, because there are no nonidentity 2-cells. Moreover, such a monoid is precisely a monoidal morphism $A\rightarrow B$. This correspondence is obtained by taking mates under the adjunctions $p \colon A^2 \rightarrow A$ and $j \colon I \rightarrow A$. In particular, monoids in $\Comon(B)$ are precisely monoidal comonads.
\end{example}

\begin{lemma}\label{INTERCHANGE_MATE_LEMMA}
 Let $f \colon A \rightarrow A^{\prime}$ and $g \colon B \rightarrow B^{\prime}$ be maps in a Gray monoid $\ca{M}$. Then the 2-cells
 \[
 \vcenter{\hbox{


}}
 \]
are mates under adjunction.
\end{lemma}

\begin{proof}
 The mate in question is given by
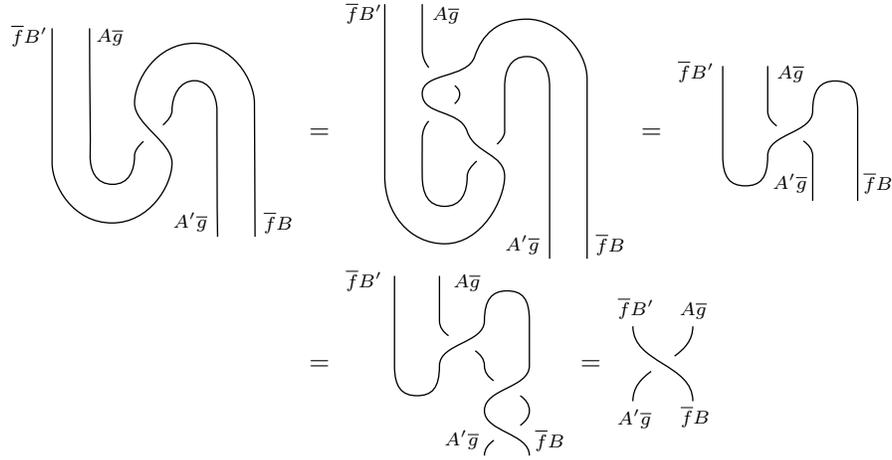
\begin{align*}
\vcenter{\hbox{
%
}}
\end{align*}
 where the equalities follow from two applications of Axiom iii) for Gray monoids and triangle identities. 
\end{proof}

\begin{prop} \label{LEFT_ADJ_MONOIDAL_PROP}
 Let $\ca{M}$ be a Gray monoid and let $(B,m,u)$ be a map pseudomonoid in $\ca{M}$.  Let $v\colon X \rightarrow B$ and $w \colon Y \rightarrow B$ be objects in $\Map(\ca{M},B)$. Then the 2-cells
\[
\chi_{v,w} \defl \vcenter{\hbox{

%
}} \quad \text{and} \quad \iota=\id \colon u  \cdot  \overline{u} \rightarrow L(u)
\]
endow the 2-functor $L$ from Proposition~\ref{LEFT_BIADJOINT_2FUNCTOR_PROP} with the structure of a strong monoidal 2-functor where all the necessary modifications are identity modifications.
\end{prop}

\begin{proof}
 We first have to check that $\chi$ is a well-defined 2-natural transformation, that is, that it is an isomorphism of comonads and that it is 2-natural. The latter boils down to showing that the naturality square commutes on the nose because there are no nonidentity 2-cells in the target 2-category. We leave both of these computations to the reader.

 To see that $L$ defines a strong monoidal 2-functor we need to check that the domain and the codomain of the necessary modifications (see \cite[Section~2]{MCCRUDDEN_BALANCED}) coincide. The key observation for this is the fact about mates of the Gray interchange from Lemma~\ref{INTERCHANGE_MATE_LEMMA}. The remaining calculations are fairly straightforward with the string diagram notations introduced in Section~\ref{GRAY_MONOIDS_SECTION} and are left to the reader.
\end{proof}

\begin{proof}[Proof of Theorem~\ref{TANNAKIAN_ADJUNCTION_MONOIDAL_THM}]
 This is an immediate consequence of Proposition~\ref{LEFT_ADJ_MONOIDAL_PROP} and Corollary~\ref{LIFT_TO_PSEUDOMONOIDS_COR}. Recall from Proposition~\ref{REP_2_FUNCTOR_PROP} that the right biadjoint of the Tannakian adjunction is normal.
\end{proof}
\subsection{Braiding}\label{BRAIDING_SECTION}
 The goal of this section is to to show that the left adjoint of the Tannakian biadjunction can be endowed with the structure of a \emph{braided} strong monoidal pseudofunctor if $\ca{M}$ is a braided Gray monoid and $B$ is a braided map pseudomonoid. As already mentioned in Section~\ref{BRAIDING_OUTLINE_SECTION}, this allows us to lift the biadjunction to the categories of braided pseudomonoids on either side. At the end of this section we can therefore give a proof of Theorem~\ref{TANNAKIAN_ADJUNCTION_BRAIDED_THM}.

 We start by recalling the definition of a braiding on a Gray monoid from \cite[Definition~12]{DAY_STREET}. For every braided map pseudomonoid $B$ we endow $\Comon(B)$ and $\Map(\ca{M},B)$ with a braiding in the sense of \cite[\S~3]{MCCRUDDEN_BALANCED}. 

\begin{dfn}
 Let $\ca{M}$ be a Gray monoid. A \emph{braiding} for $\ca{M}$ is a pseudonatural equivalence
 \[
 \def\objectstyle{\scriptstyle}
 \def\labelstyle{\scriptscriptstyle}
 \def\twocellstyle{\scriptscriptstyle}
  \xymatrix{\ca{M}^2 \ar[rr]^{\sigma} \ar[rd]_{\otimes} & \dtwocell\omit{^\rho}& \ca{M}^2 \ar[ld]^{\otimes}\\
  & \ca{M}}
 \]
 (where $\sigma$ denotes the switch 2-functor), together with two modifications with components
 \[
\vcenter{\hbox{
 \def\objectstyle{\scriptstyle}
 \def\labelstyle{\scriptscriptstyle}
 \def\twocellstyle{\scriptscriptstyle}
\xymatrix@!R=5pt{ & A(BC) \ar[r]^{\rho_{A,BC}} & (BC)A \ar@{=}[rd] \\ A(BC) \ar@{=}[ru] \ar[rd]_{\rho_{A,B} C} & \rtwocell\omit{R} & & B(CA)\\
& (BA)C \ar@{=}[r] & B(AC) \ar[ru]_{B \rho_{A,C}}} 
}}
 \]
 and
 \[
\vcenter{\hbox{
 \def\objectstyle{\scriptstyle}
 \def\labelstyle{\scriptscriptstyle}
 \def\twocellstyle{\scriptscriptstyle}
\xymatrix@!R=5pt{ & (AB)C \ar[r]^{\rho_{AB,C}} & C(AB) \ar@{=}[rd] \\ A(BC) \ar@{=}[ru] \ar[rd]_{A \rho_{B,C}} & \rtwocell\omit{S} & & (CA)B\\
& A(CB) \ar@{=}[r] & (AC)B \ar[ru]_{\rho_{A,C} B}} 
}} 
 \]
subject to coherence conditions (see \cite[Appendix~A]{MCCRUDDEN_BALANCED} for details).
\end{dfn}

\begin{rmk}\label{BRAIDING_RMK}
 Recall that the tensor product of $f \colon A \rightarrow A^{\prime}$ and $g \colon B \rightarrow B^{\prime}$ on a Gray monoid is defined by $f\otimes g \defl A^{\prime}g \cdot fB$. The pseudonatural transformation $\rho$ thus consists of invertible 2-cells
 \[
\vcenter{\hbox{


}}
 \]
 and the fact that $R$ and $S$ are modifications means that for any additional 1-cell $h \colon C \rightarrow C^{\prime}$, the equation
 \[
 \vcenter{\hbox{

%
}}
 \]
 and the corresponding equation for $S$ hold.
\end{rmk}

\begin{dfn}\label{BRAIDED_PSEUDOMONOID_DFN}
 Let $\ca{M}$ be a braided Gray monoid. A pseudomonoid $(B,m,u)$ in $\ca{M}$ is \emph{braided} if there is an invertible 2-cell $\gamma \colon m \Rightarrow m \cdot \rho_{B,B}$ (called the \emph{braiding}) subject to two coherence equations (see \cite[p.\ 87]{MCCRUDDEN_BALANCED}).
\end{dfn}

\begin{example}
 Let $\ca{V}$ be a symmetric monoidal category. The monoidal 2-category $\Vcat$ is braided, and the modifications $R$ and $S$ can be taken to be identities (see \cite[p.\ 85]{MCCRUDDEN_BALANCED}). A braided pseudomonoid in $\Vcat$ is precisely a braided monoidal $\ca{V}$-category.
\end{example}

\begin{prop}\label{SLICE_BRAIDED_PROP}
 Let $\ca{M}$ be a braided Gray monoid, and let $B$ be a braided map pseudomonoid. For two objects $v\colon X \rightarrow B$ and $w\colon Y \rightarrow B$ in $\Map(\ca{M},B)$ the 1-cell
 \[
\vcenter{\hbox{


}}
 \]
 together with the 2-cells $\rho$, $R$ and $S$ endow $\Map(\ca{M},B)$ with the structure of a braided monoidal 2-category (see \cite[Appendix~A]{MCCRUDDEN_BALANCED}).
\end{prop}

\begin{proof}
 This is a consequence of Remark~\ref{BRAIDING_RMK}, Definition~\ref{BRAIDED_PSEUDOMONOID_DFN} and the fact that the forgetful 2-functor $\Map(\ca{M},B) \rightarrow \ca{M}$ is injective on 2-cells, that is, all the axioms $R$ and $S$ must satisfy follow from the fact that they are satisfied in $\ca{M}$.
\end{proof}

\begin{prop}\label{HOM_CATEGORY_BRAIDED_PROP}
 Let $\ca{M}$ be a braided Gray monoid, and let $A$ and $B$ be braided map pseudomonoids in $\ca{M}$. Then the 2-cell
 \[
  \vcenter{\hbox{


}}
 \]
 endows $\ca{M}(A,B)$ with the structure of a braided category. This structure lifts to a braiding on the monoidal category $\Comon(B)$.
\end{prop}

\begin{proof}
 The fact that the convolution tensor product is braided if $A$ and $B$ are is a simple generalization of \cite[Example~5]{DAY_STREET}. One can also show commutativity of the required hexagons directly by using Remark~\ref{BRAIDING_RMK}, Definition~\ref{BRAIDED_PSEUDOMONOID_DFN} and Lemma~\ref{PSEUDOFUNCTOR_MATE_LEMMA} below. Using the string diagram calculus it is straightforward to check that the 2-cell is compatible with the comultiplication and counit of its domain and target, that is, that the braiding lifts to the category of comonads.
\end{proof}

\begin{example}\label{BRAIDED_PSEUDOMONOIDS_IN_HOM_EXAMPLE}
 Let $\ca{M}$, $A$ and $B$ be as in Proposition~\ref{HOM_CATEGORY_BRAIDED_PROP}. A braided pseudomonoid in $\ca{M}(A,B)$ is simply a commutative monoid in $\ca{M}(A,B)$, because there are no nonidentity 2-cells. Moreover, such a commutative monoid is precisely a braided monoidal morphism $A\rightarrow B$. This correspondence is obtained by taking mates under the adjunctions $p \colon A^2 \rightarrow A$ and $j \colon I \rightarrow A$.
\end{example}

\begin{lemma}\label{PSEUDOFUNCTOR_MATE_LEMMA}
 Let $F, G \colon \ca{M} \rightarrow \ca{N}$ be pseudofunctors between 2-categories, and let
\[
 \def\objectstyle{\scriptstyle}
 \def\labelstyle{\scriptscriptstyle}
 \def\twocellstyle{\scriptscriptstyle}
\xymatrix{FX \drtwocell\omit{^\theta_f} \ar[d]_{Ff} \ar[r]^{\theta_X} & GX \ar[d]^{Gf} \\
FY \ar[r]_{\theta_Y} & GY}
\]
 be a pseudonatural transformation from $F$ to $G$. Let $f$ be a map in $\ca{M}$. Then the mate of $\theta^{-1}_{f}$ is $\theta_{\overline{f}}$. In particular, if $v \colon X \rightarrow B$ and $w \colon Y \rightarrow B$ are maps in a braided Gray monoid, then the equation
 \[
  \vcenter{\hbox{


}}
 \]
 holds. 
\end{lemma}

\begin{proof}
 One way to prove this is as follows. One can first compose the mate in question with $\theta_{\overline{f}}^{-1}$ and then use pseudonaturality of $\theta$, pseudofunctoriality of $F$, and one of the triangle identities to show that this composite is the identity. 
 
 The second part follows from an application of the first to $F=\otimes$, $G=\otimes \cdot \tau$ and $\theta=\rho$, the braiding.
\end{proof}

\begin{prop}\label{LEFT_ADJ_BRAIDED_PROP}
 Let $\ca{M}$ be a braided Gray monoid and let $B$ be a braided map pseudomonoid in $\ca{M}$. Then the strong monoidal 2-functor
 \[
L \colon \Map(\ca{M},B) \rightarrow \Comon(B) 
 \]
 from Proposition~\ref{LEFT_ADJ_MONOIDAL_PROP} is braided.
\end{prop}

\begin{proof}
 We again need to check that the modifications can be chosen to be identities, that is, that two pasting composites involving the 2-natural transformation $\chi$ from Proposition~\ref{LEFT_ADJ_MONOIDAL_PROP} and $\rho$ coincide (see \cite[Appendix A, (BHD1)]{MCCRUDDEN_BALANCED}). The key observation for checking this is the equation from Lemma~\ref{PSEUDOFUNCTOR_MATE_LEMMA}.
\end{proof}

\begin{proof}[Proof of Theorem~\ref{TANNAKIAN_ADJUNCTION_BRAIDED_THM}]
 The proof follows the same pattern as the proof of Theorem~\ref{TANNAKIAN_ADJUNCTION_MONOIDAL_THM}. The result follows from Proposition~\ref{LEFT_ADJ_BRAIDED_PROP} and Corollary~\ref{LIFT_TO_BRAIDED_PSEUDOMONOIDS_COR}.
\end{proof}

\subsection{Syllepsis and symmetry}\label{SYLLEPSIS_SECTION}
 The goal of this section is to to show that the left adjoint of the Tannakian biadjunction can be endowed with the structure of a \emph{sylleptic} strong monoidal pseudofunctor if $\ca{M}$ is a sylleptic Gray monoid and $B$ is a symmetric map pseudomonoid. As already mentioned in Section~\ref{SYLLEPSIS_OUTLINE_SECTION}, this allows us to lift the biadjunction to the categories of symmetric pseudomonoids on either side. At the end of this section we can therefore give a proof of Theorem~\ref{TANNAKIAN_ADJUNCTION_SYLLEPTIC_THM}.

 We start by recalling the definition of a syllepsis on a braided Gray monoid from \cite[Definition~15]{DAY_STREET}. For every symmetric map pseudomonoid $B$ we endow $\Comon(B)$ with a symmetry and $\Map(\ca{M},B)$ with a syllepsis in the sense of \cite[\S~4]{MCCRUDDEN_BALANCED}.
 
 A \emph{symmetric} Gray monoid is a Gray monoid equipped with a syllepsis subject to one additional axiom. Thus being symmetric is a \emph{property} of a sylleptic Gray monoid. Therefore there are no further compatibility requirements for morphisms between symmetric Gray monoids, similarly to how a symmetric monoidal functor between symmetric monoidal categories is the same as a braided monoidal functor. Therefore we can discuss the sylleptic and the symmetric case together.

\begin{dfn}
 A \emph{syllepsis} for a braided Gray monoid $\ca{M}$ is an invertible modification $\nu$ from the identity 2-natural transformation on $\otimes$ to
 \[
 \def\objectstyle{\scriptstyle}
 \def\labelstyle{\scriptscriptstyle}
 \def\twocellstyle{\scriptscriptstyle}
 \xymatrix{ && \ca{M}^2 \ar[dd]^{\otimes}\\
 \ca{M}^2 \ar[r]^{\sigma} \ar@/^1pc/[rru]^{\id} \ar@/_1pc/[rrd]_{\otimes} \xtwocell[1,2]{}\omit{\rho} \xtwocell[-1,2]{}\omit{=} & \ca{M} \ar[rd]_{\otimes} \ar[ru]^{\sigma} \rtwocell\omit{\rho} & \\
 && \ca{M}}
 \]
 subject to two equations (see \cite[Appendix A, (SA1) and (SA2)]{MCCRUDDEN_BALANCED}). A braided Gray monoid equipped with a syllepsis is called \emph{sylleptic}. A sylleptic Gray monoid is \emph{symmetric} if the equation
 \[
  \vcenter{\hbox{


}}
 \]
holds for all objects $X$ and $Y$.
\end{dfn}

\begin{dfn}
 Let $\ca{M}$ be a sylleptic Gray monoid. A braided pseudomonoid $(B,m,u)$ is called \emph{symmetric} if the equation
 \[
  \vcenter{\hbox{

%
}}
 \]
 holds.
\end{dfn}

\begin{prop}\label{SLICE_SYLLEPTIC_PROP}
 Let $\ca{M}$ be a sylleptic Gray monoid and let $(B,m,u)$ be a symmetric map pseudomonoid in $\ca{M}$. Then the syllepsis of $\ca{M}$ defines a syllepsis on $\Map(\ca{M},B)$, with braiding defined as in Proposition~\ref{SLICE_BRAIDED_PROP}. If $\ca{M}$ is symmetric, then the syllepsis on $\Map(\ca{M},B)$ is a symmetry.
\end{prop}

\begin{proof}
 From the definition of symmetric pseudomonoid and from the fact that $\nu$ is a modification we get the equation
 \[
  \vcenter{\hbox{

%
}}
 \]
 which shows that the syllepsis of $\ca{M}$ lifts to a 2-cell in $\Map(\ca{M},B)$. It is immediate that it defines a modification and that the desired equations hold because the 2-cell part of the braiding on $\Map(\ca{M},B)$ is identical to the one on $\ca{M}$ (see Proposition~\ref{SLICE_BRAIDED_PROP}). Moreover, it is immediate that $\nu$ is a symmetry on $\Map(\ca{M},B)$ if and only if it is one considered as a syllepsis on $\ca{M}$.
\end{proof}

\begin{prop}\label{HOM_CATEGORY_SYLLEPTIC_PROP}
 Let $\ca{M}$ be a sylleptic Gray monoid. Let $(A,p,j)$ and $(B,m,u)$ be symmetric map pseudomonoids in $\ca{M}$. Then the braiding on $\ca{M}(A,B)$ from Proposition~\ref{HOM_CATEGORY_BRAIDED_PROP} is a symmetry. In particular, the monoidal category $\Comon(B)$ is symmetric. It is also symmetric when thought of as a 2-category with no nonidentity 2-cells.
\end{prop}

\begin{proof}
 From the definition of symmetric pseudomonoids we find that the equation
 \[
  \vcenter{\hbox{


}}
 \]
 holds. This, together with the facts that $B$ is symmetric and that $\nu$ is a modification, implies that the braiding from Proposition~\ref{HOM_CATEGORY_BRAIDED_PROP} is a symmetry.
\end{proof}

\begin{prop}\label{LEFT_ADJ_SYLLEPTIC_PROP}
 Let $\ca{M}$ be a sylleptic Gray monoid and let $B$ be a symmetric map pseudomonoid in $\ca{M}$. Then the braided strong monoidal 2-functor
 \[
L \colon \Map(\ca{M},B) \rightarrow \Comon(B)  
 \]
 from Proposition~\ref{LEFT_ADJ_BRAIDED_PROP} is sylleptic. 
\end{prop}

\begin{proof}
 By definition of a sylleptic monoidal 2-functor we only have to check that an equality between two 2-cells in $\Comon(B)$ holds (see \cite[Appendix A, (SHA1)]{MCCRUDDEN_BALANCED}). This is obviously the case because $\Comon(B)$ doesn't have nonidentity 2-cells.
\end{proof}

\begin{proof}[Proof of Theorem~\ref{TANNAKIAN_ADJUNCTION_SYLLEPTIC_THM}]
 The proof again follows the same pattern as the proof of Theorem~\ref{TANNAKIAN_ADJUNCTION_MONOIDAL_THM}. Specifically, the result is a consequence of Proposition~\ref{LEFT_ADJ_SYLLEPTIC_PROP} and Corollary~\ref{LIFT_TO_SYMMETRIC_PSEUDOMONOIDS_COR}.
\end{proof}

\subsection{Biduality and autonomous Gray monoids}
 So far we have not talked about the relationship between antipodes on a bialgebra and the existence of duals in the category of representations. More precisely, we would like to show that the left adjoint of the Tannakian adjunction sends autonomous categories to Hopf algebroids.
 
 In \cite{BRUGUIERES_VIRELIZIER}, the notion of a Hopf monad was introduced and in \cite{BRUGUIERES_LACK_VIRELIZIER} an equivalent characterization of Hopf monads that doesn't reference antipodes was provided. It was shown that a comonoidal monad is Hopf if two associated natural transformations (called the \emph{fusion operators}) are invertible. This definition never refers to the actual objects of the category, so it is more suitable for our purposes. In this section we will show that $L(w)$ is a Hopf monoidal comonad if the domain $A$ of $w$ is autonomous. In Section~\ref{AFFINE_GROUPOIDS} we used this fact to show that Hopf monoidal comonads on a monoidal $\ca{V}$-category with one object (that is, on an algebra in $\ca{V}$) correspond to Hopf algebroids in the case where $\ca{M}=\Mod(\ca{V})$.
 
 In order to give a ``formal,'' $2$-categorical proof of the fact that $L(w)$ is Hopf monoidal whenever the domain of $w$ is autonomous we need to give a formal definition of autonomous pseudomonoid in a $2$-category, that is, we want to talk about the fact that objects have duals without actually referring to any objects. Such a definition was given in \cite{DUALIZATIONS_ANTIPODES}. 
 
 There is a guiding principle in higher category theory due to Baez and Dolan \cite[p.~12]{BAEZ_DOLAN}, called the \emph{microcosm principle}, which says that usually an algebraic structure can be put on an object of an $n$-category if the $n$-category in question has the corresponding categorified structure. For example, to talk about a monoid in a $1$-category $\ca{M}$ we need a monoidal structure on $\ca{M}$. To define a monoidal category, we secretly use the fact that the cartesian product equips the bicategory of categories with the structure of a monoidal bicategory. From this point of view it should not be surprising that we need to talk about duals in a monoidal bicategory before we can give a formal definition of autonomous pseudomonoids.
 
 \begin{dfn}\label{BIDUAL_DFN}
 Let $\ca{M}$ be a Gray monoid, and let $A$, $B$ be two objects of $\ca{M}$. We say that $B$ is a right bidual of $A$ and $A$ is a left bidual of $B$ if there exist morphisms $n \colon I \rightarrow BA$ and $e \colon AB \rightarrow I$ and invertible 2-cells $\eta \colon 1_B \rightarrow Be \cdot nB$ and $\varepsilon \colon eA \cdot An \rightarrow 1_A$ such that the equations
 \[
\vcenter{\hbox{


}} =1_e
 \]
 hold. The morphisms $n$ and $e$ are called the unit and counit of the bidual situation.
 \end{dfn}
 
 This is a categorification of the notion of a right dual object in a monoidal category: the triangle identities only hold up to invertible 2-cell, and these 2-cells satisfy certain coherence conditions. It turns out that the coherence conditions are automatically satisfied in the following sense: in the situation of Definition~\ref{BIDUAL_DFN}, if $\eta$ and $\varepsilon$ are invertible but don't satisfy the desired equations, then we can replace $\varepsilon$ by a different invertible 2-cell $\varepsilon^\prime \colon eA \cdot An \rightarrow I$ such that $n$, $e$, $\eta$ and $\varepsilon^\prime$ do exhibit $B$ as a right bidual of $A$.

\begin{prop}\label{BIDUAL_BIADJUNCTION_PROP}
 For any bidual situation $(A,B,n,e,\eta,\varepsilon)$, the functor 
\[
 \xymatrix{\ca{M}(A\otimes X,Y) \ar[r] & \ca{M}(X,B \otimes Y) }
\]
 given by $f \mapsto Bf \cdot nX$ is an equivalence of categories, with inverse $g \mapsto eY \cdot Ag$.
\end{prop}

\begin{proof}
 The natural isomorphisms which exhibit these functors as mutually inverse equivalences are given by $\eta$ and $\varepsilon$.
\end{proof}

\begin{prop}\label{UNIQUENESS_OF_DUALS_PROP}
Right biduals are unique up to equivalence. More precisely, let $(A,B,n,e,\eta,\varepsilon)$ and $(A,B^\prime,n^\prime,e^\prime,\eta^\prime,\varepsilon^\prime)$ be bidual situations in $\ca{M}$. Then the morphisms $B^\prime e \cdot n^\prime B \colon B \rightarrow B^\prime$ and $B e^\prime \cdot nB^\prime \colon B^\prime \rightarrow B$ are mutually inverse equivalences.
\end{prop}

\begin{proof}
 The two invertible 2-cells
\[
\vcenter{\hbox{


}} 
\]
 give the desired isomorphisms between $(B e^\prime \cdot nB^\prime) \cdot (B^\prime e \cdot n^\prime B)$ and $\id_B$, and between $(B^\prime e \cdot n^\prime B) \cdot (B e^\prime \cdot nB^\prime) $ and $\id_{B^\prime}$.
\end{proof}

\begin{example}\label{DUALS_MODV_EXAMPLE}
 Let $\ca{V}$ be a cosmos. In $\Mod(\ca{V})$, every object $\ca{A}$ has a right bidual, given by $\ca{A}^{\op}$. The unit and counit are given by $\ca{I} \xslashedrightarrow{} \ca{A}^{\op} \otimes \ca{A}$, $(a,b,\ast) \mapsto \ca{A}(a,b)$ and $\ca{A} \otimes \ca{A}^{\op} \xslashedrightarrow{} \ca{I}$, $(\ast, a,b) \mapsto \ca{A}(b,a)$.
\end{example}

 A Gray monoid is called \emph{autonomous} if every object $A$ has a right bidual $\dual{A}$ and a left bidual $\ldual{A}$. If $\ca{M}$ is symmetric, then the left and right dual are equivalent.

 Let $\ca{A}$ be an autonomous symmetric monoidal $\ca{V}$-category. The $\ca{V}$-functor which sends an object to (a choice of) its dual gives an equivalence $d \colon \ca{A}^{\op} \rightarrow \ca{A}$ of categories. This functor is central for the formal definition of an autonomous pseudomonoid in a Gray monoid $\ca{M}$.
 
 Let $(B,m,u)$ be a pseudomonoid in the Gray monoid $\ca{M}$. Even though the lax slice $\ca{M}\slash_{\ell} B$ is not a Gray monoid we can talk about biduals in this monoidal 2-category; we just have to be careful to insert structural isomorphisms and equivalences when necessary. If $(A,\dual{A},n,e,\eta,\varepsilon)$ is a bidual situation in $\ca{M}$, we can ask if it lifts to a bidual situation in $\ca{M}\slash_\ell B$. More precisely, given $g \colon A \rightarrow B$ we can ask for a right bidual $f$ whose domain is $\dual{A}$ and whose structural morphisms extend $n$ and $e$. The data of such a right bidual consists of 2-cells $\pi \colon f\bullet g \circ n \Rightarrow u$ and $\xi \colon u \circ e \Rightarrow g \bullet f$. In order to give a bidual situation, we require that $\eta$ and $\varepsilon$ lift to 2-cells in the lax slice $\ca{M} \slash_\ell B$. We can prove the following result by unraveling the definition of a bidual situation in the lax slice.
 
\begin{prop}\label{DUALIZATIONS_PROP}
 Let $(A,\dual{A},n,e,\eta,\varepsilon)$ be a bidual situation in the Gray monoid $\ca{M}$, let $(B,m,u)$ be a pseudomonoid in $\ca{M}$, and let $\pi \colon f\bullet g \cdot n \Rightarrow u$ and $\xi \colon u \cdot e \Rightarrow g \bullet f$ be 2-cells in $\ca{M}$. Then the 1-cells $(n,\pi)$ and $(e,\xi)$ exhibit $f$ as right bidual of $g$ in the lax slice $\ca{M} \slash_{\ell} B$ if and only if the 2-cells
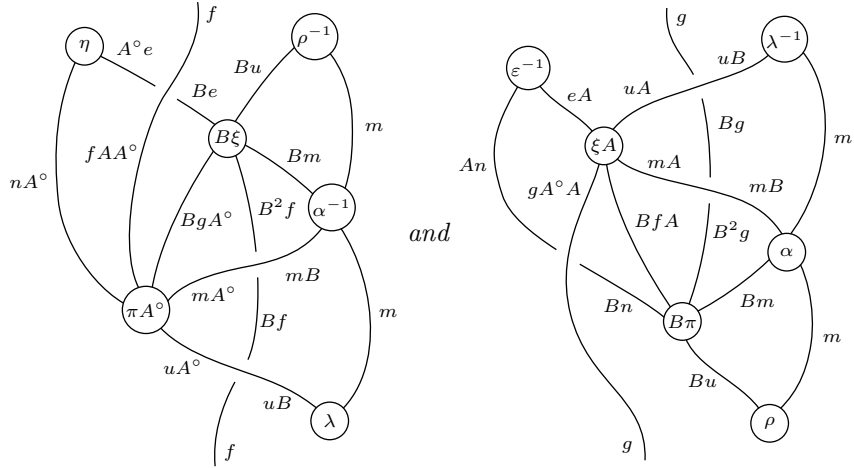
\[
\vcenter{\hbox{


}}
 \]
 are equal to the identity 2-cell on $f$ and $g$ respectively. 
\end{prop}
 
 Note that the above equations were already present in \cite{DUALIZATIONS_ANTIPODES}, at least for the special case $g=\id_A$. Steve Lack later realized that these give precisely a bidual situation in the lax slice (cf.\ \cite[Section~2.1]{DUALS_INVERT}). There is also a different terminology in the literature: what we call a right bidual of $g$ in $\ca{M} \slash_\ell B$ is called a \emph{left dualization} of $g$ in \cite{DUALIZATIONS_ANTIPODES} and \cite{DUALS_INVERT}.

\begin{dfn}\label{AUTONOMOUS_DEFINITION}
 Let $(B,m,u)$ be a pseudomonoid in a Gray monoid $\ca{M}$. We call $B$ \emph{left autonomous} if the identity $\id_B \colon B \rightarrow B$ has a right bidual $d \colon \dual{B} \rightarrow B$ in $\ca{M} \slash_\ell B$. We call $B$ \emph{right autonomous} if it is left autonomous in $\ca{M}^{\rev}$, the Gray monoid with reversed tensor product, and simply \emph{autonomous} if it is both left and right autonomous.
\end{dfn}

The following proposition shows that this is a sensible definition.

\begin{prop}\label{VCAT_AUTONOMOUS_PROP}
 A monoidal $\ca{V}$-category $\ca{B}$ is (left) autonomous if and only if the corresponding pseudomonoid in $\Mod(\ca{V})$ is (left) autonomous.
\end{prop}

\begin{proof}
 This is \cite[Proposition~1.6]{DUALIZATIONS_ANTIPODES}.
\end{proof}

 In \cite[Proposition~1.4]{DUALIZATIONS_ANTIPODES} it was shown that for any autonomous map pseudomonoid $(A,p,j)$, the right bidual $d \colon \dual{A} \rightarrow A$ of the identity is an equivalence. Since the right bidual $\dual{A}$ of $A$ is only well-defined up to equivalence, we could simply choose $\dual{A}=A$. By doing this we can find simpler conditions for when a map pseudomonoid is autonomous.

\begin{dfn}\label{NATURALLY_FROBENIUS_DFN}
 Let $\ca{M}$ be a Gray monoid. A map pseudomonoid $(A,p,j)$ is called \emph{naturally Frobenius} if the two mates
 \[
  \phi=\vcenter{\hbox{


}}
 \]
 of the associator $\alpha$ are invertible.
\end{dfn}

\begin{prop}\label{FROBENIUS_PROP}
 If $(A,p,j)$ is naturally Frobenius, then the morphisms $n=\overline{p} \cdot j$, $e=\overline{j} \cdot p$ and the 2-cells
 \[
 \eta=\vcenter{\hbox{

%
}}
 \]
 exhibit $A$ as a right bidual of itself. Moreover, for
 \[
  \pi=\vcenter{\hbox{

%
}}
 \]
 the morphisms $(n, \pi)$ and $(e,\xi)$ exhibit $1_A \colon A \rightarrow A$ as a right bidual of itself in the lax slice $\ca{M}/_{\ell} A$. A map pseudomonoid $A$ is autonomous if and only if it is naturally Frobenius. 
\end{prop}

\begin{proof}
 The first statement is \cite[Theorem~6.8]{DUALS_INVERT}. It shows in particular that a naturally Frobenius pseudomonoid is left autonomous. Moreover, the right bidual $d \colon \dual{A} \rightarrow A$ of the identity is equal to the identity, so in particular an equivalence. It follows from \cite[Proposition~1.5]{DUALIZATIONS_ANTIPODES} that $A$ is autonomous. The converse can be found in \cite[Proposition~3.1]{STREET_FROBENIUS} in the case where $\ca{M}$ is autonomous; the general case follows from an application of \cite[Corollary~4.4]{DUALS_INVERT} applied to $\ca{M}$ and $\ca{M}^{\rev}$, the Gray monoid with the same underlying 2-category and reversed tensor product. 
\end{proof}

\subsection{Hopf monoidal comonads}\label{HOPF_MONOIDAL_COMONADS_SECTION}
 The second ingredient we need to deal with Hopf algebroids is the notion of a Hopf monoidal comonad. A monoidal comonad on a pseudomonoid $(B,m,u)$ in a Gray monoid $\ca{M}$ is a comonad $c$ on $B$ with a monoidal structure $(c,\phi,\phi_0)$ such that the counit and the comultiplication are monoidal 2-cells. Equivalently, it is a monoid in the monoidal category $\Comon(B)$ under convolution product. We get $\phi$ and $\phi_0$ from the multiplication and unit maps by taking mates under the adjunctions $m \dashv \overline{m}$ and $u \dashv \overline{u}$ respectively (see Example~\ref{MONOIDAL_MORPHISM_EXAMPLE}).

\begin{dfn}\label{HOPF_COMONAD_DFN}
 A monoidal comonad $(c,\phi,\phi_0)$ on a pseudomonoid $(B,m,u)$ is called \emph{left Hopf}, respectively \emph{right Hopf}, if the 2-cells
 \[
 \vcenter{\hbox{


 }}
 \]
 are invertible.
\end{dfn}

 The following example shows that Hopf monoidal comonads in the bicategory of cospans in a finitely cocomplete category $\ca{E}$ are precisely the groupoids internal to $\ca{E}^{\op}$. We have used this fact in Section~\ref{AFFINE_GROUPOIDS} to prove our recognition results for affine groupoids.

\begin{example}\label{GROUPOIDS_ARE_HOPF_COMONADS_EXAMPLE}
 Let $\ca{E}$ be a finitely cocomplete category (for example, the category of $R$-algebras for some commutative ring $R$), and let $\ca{M}=\cospan{\ca{E}}$ be the symmetric monoidal bicategory of cospans in $\ca{E}$. Every object in $\ca{M}$ is a pseudomonoid with multiplication
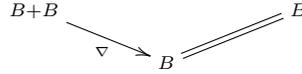
\[
\def\objectstyle{\scriptstyle}
\def\labelstyle{\scriptscriptstyle}
\def\twocellstyle{\scriptscriptstyle}
\xymatrix@!0@R=20pt@C=50pt{B+B \ar[rd]_-{\nabla} && \ar@{=}[ld] B \\
&B}
\]
and the 2-cell induced by
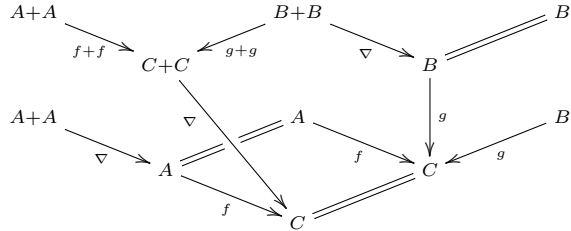
\[
\def\objectstyle{\scriptstyle}
\def\labelstyle{\scriptscriptstyle}
\def\twocellstyle{\scriptscriptstyle}
\xymatrix@!0@R=20pt@C=50pt{A+A \ar[rd]_-{f+f} && \ar[ld]^-{g+g} B+B \ar[rd]_-{\nabla} && B \ar@{=}[ld] \\
& C+C \ar[rddd]_(0.3){\nabla} && B \ar[dd]^-{g}\\
A+A \ar[rd]_-{\nabla} && \ar@{=}[ld]|{\phantom{\bullet\bullet}} A \ar[rd]_-{f} && \ar[ld]^-{g} B\\
& A \ar[rd]_-{f} && \ar@{=}[ld] C \\
&& C}
\]
 endows every morphism $C \colon A \rightarrow B$ in $\ca{M}$ with a monoidal structure. Note that a comonad 
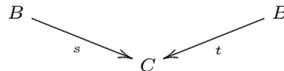
\[
\def\objectstyle{\scriptstyle}
\def\labelstyle{\scriptscriptstyle}
\def\twocellstyle{\scriptscriptstyle}
\xymatrix@!0@R=20pt@C=50pt{B \ar[rd]_-{s} && \ar[ld]^-{t} B \\
&C}
\] 
 in the category of cospans is precisely a category internal to $\ca{E}^{\op}$. That is, the represented functors $\ca{E}(C,-)$ and $\ca{E}(B,-)$ come equipped with natural maps giving a category with objects $\ca{E}(B,X)$ and morphisms $\ca{E}(C,X)$ for every object $X \in \ca{E}$. For example, source and target of a morphism $f \colon C \rightarrow X$ are given by precomposition with $s$ and $t$ respectively, and the comultiplication of $C$ gives the composition map. Moreover, the monoidal structure discussed above is compatible with the comonad structure. The domain and codomain of the 2-cell that determines if a monoidal comonad in $\ca{M}$ is right Hopf are given by the objects which represent the functors
\[
X \mapsto \{(f,g) \in \ca{E}(C,X)^2 \vert ft=gs \}
\]
 and
\[
X \mapsto \{(f,g) \in \ca{E}(C,X)^2 \vert ft=gt \}
\]
 respectively. A careful analysis of the pushouts involved shows that the 2-cell from Definition~\ref{HOPF_COMONAD_DFN} represents the natural transformation which sends $(f,g)$ to $(g\circ f,g)$. If this map is surjective, then $(\id_{gt},g)$ is in its image, so every morphism has a section. But this only happens if the category represented by $(B,C)$ is in fact a groupoid. Conversely, if every morphism is invertible, the natural transformation above is evidently invertible. Thus right Hopf comonads in $\ca{M}$ are precisely the groupoids internal to $\ca{E}^{\op}$. In particular, for $\ca{E}$ the category of commutative $R$-algebras, we find that Hopf comonads in $\cospan{\ca{E}}$ are precisely affine groupoids acting on a commutative $R$-algebra $B$.
\end{example}

 In order to prove Theorem~\ref{HOPF_COMONAD_THM}, we will show that a strong monoidal morphism $w$ between autonomous pseudomonoids always satisfies the following definition. We can therefore apply the proposition below.

\begin{dfn}
 Let $(w,\psi,\psi_0) \colon A \rightarrow B$ be a strong monoidal map between two pseudomonoids in a Gray monoid $\ca{M}$. Then $w$ is \emph{strong left coclosed}, respectively \emph{strong right coclosed}, if the mates
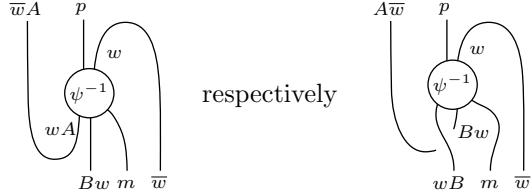
\[
\vcenter{\hbox{


}}
\]
 of the 2-cell $\psi^{-1} \colon w \cdot p \rightarrow m \cdot w\otimes w$ are invertible.
\end{dfn}

\begin{prop}\label{COCLOSED_IMPLIES_HOPF_PROP}
 If $w\colon A \rightarrow B$ is a strong monoidal map that is both strong right and strong left coclosed, then the induced comonad $L(w)=w \cdot \overline{w}$ is a Hopf monoidal comonad.
\end{prop}

\begin{proof}
 This follows from \cite[Proposition~4.4]{CHIKHLADZE_LACK_STREET} applied to $\ca{M}$ and $\ca{M}^{\rev}$.
\end{proof}

In order to prove that the conditions of the above proposition hold for all strong monoidal maps between autonomous map pseudomonoids, we will use the following lemma whose proof we defer to Appendix~\ref{TECHNICAL_LEMMA_APPENDIX}. It follows from a generalization of the fact that any strong monoidal functor preserves duals.

\begin{cit}[Lemma~\ref{HOPF_LEMMA}]
 Let $A$ and $B$ be autonomous map pseudomonoids in a Gray monoid $\ca{M}$, and let $(w,\psi,\psi_0)\colon A \rightarrow B$ be a strong monoidal map. Then the 2-cell
 \[
  \tau=\vcenter{\hbox{


}}
 \]
 is invertible.
\end{cit}

\begin{proof}[Proof of Theorem~\ref{HOPF_COMONAD_THM}.]
 Let $(w,\psi,\psi_0)$ be a strong monoidal map between two autonomous map pseudomonoids $(A,p,j)$ and $(B,m,u)$. We have to show that the induced comonad $L(w)=w \cdot \overline{w}$ is a Hopf monoidal comonad. We first show that every strong monoidal map $(w,\psi,\psi_0) \colon A \rightarrow B$ is strong right and left coclosed, and then we apply Proposition~\ref{COCLOSED_IMPLIES_HOPF_PROP} to conclude that $L(w)=w \cdot \overline{w}$ is indeed Hopf monoidal. In other words, we only have to show that the two 2-cells
 \[
 \vcenter{\hbox{


}}
 \]
 are invertible. We focus on the left one of these; invertibility of the right one follows by the same reasoning applied to the Gray monoid $\ca{M}^{\rev}$ with reversed tensor product. Indeed, if $(A,p,j,\alpha,\lambda,\rho)$ is an autonomous map pseudomonoid in $\ca{M}$, then $(A,p,j,\alpha^{-1},\rho,\lambda)$ is an autonomous map pseudomonoid in $\ca{M}^{\rev}$, and if $(w,\psi,\psi_0)$ is monoidal map, then $(w,\psi \cdot c_{w,w}^{-1},\psi_0)$ is a monoidal map in $\ca{M}^{\rev}$. 
 
 The equivalence from Proposition~\ref{BIDUAL_BIADJUNCTION_PROP}, applied to the case $X=BA$, $Y=I$ in $\ca{M}^{\rev}$ sends the left one of the above 2-cells to
 \[
  \lambda=\vcenter{\hbox{


  }}
 \]
 so it suffices to show that $\lambda$ is invertible (recall from Proposition~\ref{NATURALLY_FROBENIUS_DFN} that $e_A=\overline{j} \cdot p$). From one of the triangle identities and from one of the axioms for a monoidal morphism we get the equation
 \[
 \vcenter{\hbox{

%
 }}
 \]
 where the rightmost 2-cell is invertible by Lemma~\ref{HOPF_LEMMA}. The same lemma implies that the 2-cell below the dashed line is also invertible, and it follows that $\lambda$ is invertible. This shows that $w$ is strong left coclosed.
\end{proof}
\section{Base change}\label{BASE_CHANGE_SECTION}

\subsection{Base change for 2-categories}\label{2CAT_BASE_CHANGE_SECTION}
We investigate the question how the Tannakian biadjunction interacts with base change functors. If $F \colon \ca{M} \rightarrow \ca{M}^{\prime}$ is a pseudofunctor between two 2-categories, then it sends maps to maps and preserves comonads. Thus it induces a functor $F_\ast \colon \Comon(B) \rightarrow \Comon(FB)$ and a pseudofunctor $F_\ast \colon \Map(\ca{M},B) \rightarrow \Map(\ca{M}^{\prime},FB)$. The specified adjoint $\overline{F(w)}$ of $F(w)$ is chosen to be $F(\overline{w})$. 

\begin{prop}\label{BASE_CHANGE_NATURAL_ISO_PROP}
The diagram
\[
\xymatrix{
\Map(\ca{M},B) \ar[d]_{F_\ast} \ar[r]^{L} & \Comon(B) \ar[d]^{F_\ast}\\
\Map(\ca{M}^{\prime},FB) \ar[r]_{L^\prime} & \Comon(FB) 
}
\]
is commutative up to natural isomorphism, given by the structure 2-cell
\[
 F(w.\overline w) \cong F(w) \cdot \overline{F(w)}
\]
of the pseudofunctor $F$. If $F$ is strict, then the above diagram is commutative.
\end{prop}

\begin{proof}
This is clear from the definition of the left biadjoint of the Tannakian adjunction.
\end{proof}

On the other hand, if $F$ is only lax or oplax, then it does not preserve maps in general, so it doesn't induce any kind of functor on $\Map(\ca{M},B)$.

A different kind of base change involves a map $f \colon B \rightarrow B^{\prime}$ in $\ca{M}$. Composition with $f$ clearly induces a 2-functor $\Map(\ca{M},f) \colon \Map(\ca{M},B) \rightarrow \Map(\ca{M},B^{\prime})$.

\begin{prop}
 The assignment which sends a comonad $c$ on $B$ to
 \[
 fc\overline{f} \colon B^{\prime} \rightarrow B^\prime \smash{\rlap{,}} 
 \]
 with comultiplication and counit given by
 \[
\vcenter{\hbox{


}}
 \]
 respectively defines a functor $f_\ast \colon \Comon(B) \rightarrow \Comon(B^{\prime})$.
\end{prop}

\begin{proof}
 The comonad axioms are easily proved using string diagrams, and functoriality follows from the fact that whiskering with a 1-cell preserves vertical composition of 2-cells.
\end{proof}

\subsection{Base change and monoidal structures}
If $F \colon \ca{M} \rightarrow \ca{M}^{\prime}$ is a strong monoidal pseudofunctor between two Gray monoids, then it lifts to the categories of map pseudomonoids (see \cite[Proposition~5]{DAY_STREET}). In particular, it preserves monoidal morphisms between map pseudomonoids, which tells us that it induces a pseudofunctor
\[
 F_\ast \colon \MonComon(B) \rightarrow \MonComon(FB)
\]
for any map pseudomonoid $B$. Similarly, it induces a pseudofunctor
\[
 F_\ast \colon \PsMon\bigl(\Map(\ca{M},B)\bigr) \rightarrow \PsMon \bigl( \Map(\ca{M}^{\prime},FB) \bigr)
\]
since objects of $\PsMon\bigl(\Map(\ca{M},B)\bigr)$ are precisely map pseudomonoids equipped with a strong monoidal map to $B$. Clearly these pseudofunctors are compatible with the pseudofunctors of the same name introduced in Section~\ref{2CAT_BASE_CHANGE_SECTION}, in the sense that the diagrams
\[
\vcenter{\hbox{
 \xymatrix{ \PsMon\bigr(\Map(\ca{M},B)\bigl) \ar[d]_U \ar[r]^{F_\ast} & \PsMon\bigr(\Map(\ca{M}^{\prime},FB)\bigl) \ar[d]^{U} \\
 \Map(\ca{M},B) \ar[r]_{F_\ast} & \Map(\ca{M}^{\prime}, FB)
 }}}
\]
and
\[
\vcenter{\hbox{
 \xymatrix{ \MonComon(B) \ar[d]_U \ar[r]^{F_\ast} & \MonComon(FB) \ar[d]^{U} \\
 \Comon(B) \ar[r]_{F_\ast} & \Comon(FB)
 }}} 
\]
commute. Thus the natural isomorphism from Proposition~\ref{BASE_CHANGE_NATURAL_ISO_PROP} induces a natural isomorphism $UF_\ast L \rightarrow ULF_\ast$. The goal of this section is to show that there exists a lift of this natural isomorphism to the category of monoidal comonads.

\begin{prop}\label{MONOIDAL_BASE_CHANGE_PROP}
 Let $F \colon \ca{M} \rightarrow \ca{M}^{\prime}$ be a strong monoidal pseudofunctor between Gray monoids, and let $B\in \ca{M}$ be a map pseudomonoid. Then the natural isomorphism $F_\ast L \Rightarrow L F_\ast$ from Proposition~\ref{BASE_CHANGE_NATURAL_ISO_PROP} lifts to the category of monoidal comonads.
\end{prop}

\begin{proof}
 From \cite[Theorem~11.3.1]{GURSKI} we know that any strong monoidal functor between monoidal bicategories can be replaced by an equivalent Gray functor. Thus, if $(B,m,u)$ is a map pseudomonoid, so is $(FB,Fm,Fu)$. Recall that we chose the adjoints of $Fm$ and $Fu$ to be equal to $F \overline{m}$ and $F \overline{u}$ respectively. Any Gray-functor is in particular a strict 2-functor, so the diagram
 \[
\vcenter{\hbox{
 \xymatrix{ \PsMon\bigr(\Map(\ca{M},B)\bigl) \ar[d]_U \ar[r]^{F_\ast} & \PsMon\bigr(\Map(\ca{M}^{\prime},FB)\bigl) \ar[d]^{U} \\
 \Map(\ca{M},B) \ar[r]_{F_\ast} & \Map(\ca{M}^{\prime}, FB)
 }}}  
 \]
 is commutative. Moreover, any Gray functor preserves the interchange 2-cells. It follows immediately from the definition of $\chi_{v,w}$ (see Proposition~\ref{LEFT_ADJ_MONOIDAL_PROP}) that the identity natural transformation commutes strictly with $\chi$, and from our choice of adjoint of $Fu$ it also clear that it preserves $\iota$ strictly. Thus the identity natural transformation is monoidal, and therefore lifts to the categories of pseudomonoids.
\end{proof}

If $F \colon \ca{M} \rightarrow \ca{M}^{\prime}$ is a braided or sylleptic strong monoidal pseudofunctor, then we get similar lifts to the 2-categories of braided or symmetric pseudomonoids on the one hand and to the categories of braided or symmetric monoidal comonads.

\begin{prop}\label{BRAIDED_BASE_CHANGE_PROP}
 Let $F \colon \ca{M} \rightarrow \ca{M}^{\prime}$ be a braided or sylleptic strong monoidal pseudofunctor, and let $B \in \ca{M}$ be a braided or symmetric map pseudomonoid. Then the natural isomorphism $F_\ast L \Rightarrow L F_\ast$ from Proposition~\ref{BASE_CHANGE_NATURAL_ISO_PROP} lifts to the category of braided or symmetric monoidal comonads.
\end{prop}

\begin{proof}
 To the author's knowledge there are currently no strictification results for braided or sylleptic strong monoidal pseudofunctors, so we can't prove this in the same way we proved Proposition~\ref{MONOIDAL_BASE_CHANGE_PROP}.
 
 Luckily, the target category of the Tannakian biadjunction is fairly degenerate: braided monoidal comonads are precisely the commutative monoids in the braided monoidal category $\Comon(B)$ under the convolution tensor product. This is a \emph{full} subcategory of the category of monoids. That is, a monoidal natural transformation between braided strong monoidal morphisms is automatically braided, there is no additional coherence condition required. Thus the forgetful functor from braided monoidal comonads to comonads factors as
 \[
  \xymatrix{ \BrMonComon(B) \ar[r]^-{U_2} & \MonComon(B) \ar[r]^-{U_1} & \Comon(B)}
 \]
 where $U_2$ is fully faithful. In Proposition~\ref{MONOIDAL_BASE_CHANGE_PROP} we have shown that the desired natural transformation can be lifted along $U_1$, which concludes the proof in the braided case.

 The case of a sylleptic or symmetric strong monoidal pseudofunctor $F$ is even easier: in a symmetric monoidal bicategory with no nonidentity 2-cells, there is no distinction between braided and symmetric pseudomonoids. In our case this means that a monoidal morphism between symmetric map pseudomonoids is symmetric if and only if it is braided. Therefore the forgetful functor
\[
 U_3 \colon \SymMonComon(B) \rightarrow \BrMonComon(B)
\]
is an equality categories.
\end{proof}

Next we investigate the base change along a map $f \colon B \rightarrow B^{\prime}$ in $\ca{M}$.

\begin{prop}
 Let $(f,\psi,\psi_0) \colon B \rightarrow B^{\prime}$ be a strong monoidal map between map pseudomonoids in $\ca{M}$. Then $\Map(\ca{M},f) \colon \Map(\ca{M},B) \rightarrow \Map(\ca{M},B^{\prime})$ is a strong monoidal 2-functor, with structural 1-cells in $\Map(\ca{M},B^{\prime})$ given by the identity on objects and 2-cell part
 \[
 \vcenter{\hbox{


}}
\]
 respectively.
\end{prop}

\begin{proof}
 We use the notation for monoidal morphisms between monoidal 2-categories introduced in \cite[\S~2]{MCCRUDDEN_BALANCED}. From the definition of the monoidal structure on $\Map(\ca{M},B)$ and the definition of strong monoidal maps it follows that $\chi$ is a strict 2-natural transformation, and that the modifications $\omega$, $\zeta$, $\kappa$ can be chosen to be identities.
\end{proof}

\begin{prop}
 Let $(f,\psi,\psi_0) \colon B \rightarrow B^{\prime}$ be a strong monoidal map between map pseudomonoids in $\ca{M}$. Then the natural transformations
 \[
  \vcenter{\hbox{


}}
 \]
 endow $f_\ast \colon \Comon(B) \rightarrow \Comon(B^{\prime})$ with the structure of a strong monoidal functor.
\end{prop}

\begin{proof}
 We leave the lengthy computation involving string diagrams to the reader. The desired equalities all follow from the definition of a strong monoidal map.
\end{proof}

\begin{prop}
 Let $f \colon B \rightarrow B^{\prime}$ be a strong monoidal map between map pseudomonoids in $\ca{M}$. Then the diagram
 \[
  \xymatrix{\Map(\ca{M},B) \ar[r]^{L} \ar[d]_{\Map(\ca{M},f)} & \Comon(B) \ar[d]^{f_\ast} \\
   \Map(\ca{M},B^{\prime}) \ar[r]_{L} & \Comon(B^{\prime})}
 \]
 of monoidal 2-functors commutes. Consequently, the lifts of these functors to the 2-categories of pseudomonoids commute.
\end{prop}

\begin{proof}
 It is clear that the diagram commutes on the level of 2-functors. Thus we only need to check that the monoidal structure of the two composites coincides. This is not hard to see from the definition of the monoidal structure of $L$, $\Map(\ca{M},f)$ and $f_\ast$ respectively.
\end{proof}

\begin{prop}
 If $\ca{M}$ is a braided Gray monoid, and $f \colon B \rightarrow B^{\prime}$ is a braided strong monoidal map between braided map pseudomonoids, then the identity modification endows $\Map(\ca{M},f)$ with a braiding, and $f_\ast \colon \Comon(B) \rightarrow \Comon(B^{\prime})$ is a braided strong monoidal functor.
 
 If, in addition, $\ca{M}$ is sylleptic and $B$, $B^{\prime}$ are symmetric, then $\Map(\ca{M},f)$ is sylleptic (and $f_\ast$ is symmetric as a braided functor between symmetric monoidal categories). 
 
 The diagram
 \[
  \xymatrix{\Map(\ca{M},B) \ar[r]^{L} \ar[d]_{\Map(\ca{M},f)} & \Comon(B) \ar[d]^{f_\ast} \\
   \Map(\ca{M},B^{\prime}) \ar[r]_{L} & \Comon(B^{\prime})}
 \]
 commutes in the category of braided (resp.\ sylleptic) strong monoidal 2-functors. Consequently, the lifts of these functors to the 2-categories of braided (resp.\ symmetric) pseudomonoids commute.
\end{prop}

\begin{proof}
 In order to check that the identity modification gives a braiding on the 2-functor $\Mod(\ca{M},f)$, we only need to check that the domain and the codomain of the modification coincide. The defining diagram can be found in \cite[p.~86]{MCCRUDDEN_BALANCED}. The desired equality follows from the defining equation of a braided strong monoidal map. We leave the details to the reader.
 
 A similar computation (using pseudonaturality of $\rho$ and Lemma~\ref{PSEUDOFUNCTOR_MATE_LEMMA}) shows that $f_\ast \colon \Comon(B) \rightarrow \Comon(B^{\prime})$ is symmetric.
 
 Now suppose that $\ca{M}$ is sylleptic. Since $B$ and $B^{\prime}$ are symmetric, the domain and codomain of $\Map(\ca{M},f)$ inherit a syllepsis. Moreover the braiding on $\Map(\ca{M},f)$ is the identity, so it suffices to show that $\Map(\ca{M},f)$ preserves the syllepsis. This follows immediately from the fact that $\Map(\ca{M},f)$ is the identity on 2-cells.
\end{proof}

\begin{rmk}
 If we only want to show that the diagram
 \[
  \xymatrix{\Map(\ca{M},B) \ar[r]^{L} \ar[d]_{\Map(\ca{M},f)} & \Comon(B) \ar[d]^{f_\ast} \\
   \Map(\ca{M},B^{\prime}) \ar[r]_{L} & \Comon(B^{\prime})}  
 \]
 lifts to various categories of pseudomonoids, we could use the fact that the 2-functor $\Map(\ca{M},f)$ has an evident lift to pseudomonoids if $f$ is a morphism of pseudomonoids, namely the 2-functor given by composition with $f$. The advantage of showing that $\Map(\ca{M},f)$ is a strong monoidal 2-functor is that it allows us to apply it to all kinds of structures defined using only the language of monoidal 2-categories.
\end{rmk}

\subsection{Base change for cosmoi}
We can further specialize this to the case where $\ca{M}$ is a Gray monoid equivalent to $\Mod(\ca{V})$ for some cosmos $\ca{V}$. In that case, we have given a characterization of the 2-category $\Map(\ca{M},B)$ in terms of the 2-category of $\ca{V}$-categories (see Lemma~\ref{MAPS_LEMMA}). Thus, in order to transfer the above results to the case of $\ca{V}$-categories, we first need to show that the equivalence described in Lemma~\ref{MAPS_LEMMA} is compatible with the symmetric monoidal structure on $\Vcat$ and $\Mod(\ca{V})$. 

\begin{prop}\label{COMPANION_FUNCTOR_MONOIDAL_PROP}
 Let $\ca{V}$ be a cosmos and let $G \colon \Vcat \rightarrow \Mod(\ca{V})$ be the pseudofunctor which sends a $f \colon \ca{A} \rightarrow \ca{B}$ to $\ca{B}(-,f-) \colon \ca{A} \xslashedrightarrow{} \ca{B}$. Then $G$ is strict symmetric monoidal.
\end{prop}

\begin{proof}
 On way to construct the symmetric monoidal structure on $\Mod(\ca{V})$ is to define coherence constraints on $\Mod(\ca{V})$ to be the images of the coherence constraints in $\Vcat$ under the pseudofunctor $G$ (see \cite[Theorem~5.1]{SHULMAN}). From this construction it is obvious that $G$ is strict symmetric monoidal.
\end{proof}

 The following definition first appeared in \cite{EILENBERG_KELLY}. A detailed exposition can also be found in \cite[\S~4]{CRUTTWELL}.

\begin{dfn}[Base change for $\ca{V}$-functors]
 Let $F \colon \ca{V} \rightarrow \ca{V}^{\prime}$ be a cocontinuous symmetric strong monoidal functor, and let $w \colon \ca{A} \rightarrow \ca{B}$ be a $\ca{V}$-functor. We define the $\ca{V}^{\prime}$-functor $F_\ast w$ between the $\ca{V}^{\prime}$-categories $F_\ast \ca{A}$ and $F_\ast \ca{B}$ (see Definition~\ref{BASE_CHANGE_DFN}) by $F_\ast w(a)=w(a)$ and $(F_\ast w)_{a,a^{\prime}}=F(w_{a,a^{\prime}})$.
\end{dfn}

 This base change functor is obviously compatible with the symmetric monoidal structure on $\Vcat$. Thus it lifts to the category of small symmetric monoidal $\ca{V}$-categories. If the base change functor is cocontinuous, it can be extended to a base change pseudofunctor $\Mod(\ca{V}) \rightarrow \Mod(\ca{V})$ (see Definition~\ref{BASE_CHANGE_DFN}). Since autonomous monoidal $\ca{V}$-categories can be detected by the fact that they are autonomous map pseudomonoids in $\Mod(\ca{V})$ (see Proposition~\ref{VCAT_AUTONOMOUS_PROP}), it follows that cocontinuous base change functors lift to the 2-category of autonomous symmetric monoidal $\ca{V}$-categories.

\begin{prop}\label{CAUCHY_COMPARISON_FUNCTOR_PROP}
 Let $\ca{V}$ and $\ca{V}^{\prime}$ be cosmoi, let $F \colon \ca{V} \rightarrow \ca{V}^{\prime}$ be a cocontinuous symmetric strong monoidal functor, and let $B$ be a commutative monoid in $\ca{V}$. Write $\Mod_B^d$ for the category of dualizable $B$-modules. Then there is a canonical fully faithful symmetric strong monoidal $\ca{V}^{\prime}$-functor $i \colon F_\ast \Mod^d_B \rightarrow \Mod^d_{FB}$, which sends an object $M$ to $FM$ with the evident $FB$-action.
\end{prop}

\begin{proof}
 The hom-object in $\ca{V}$ between two dualizable $B$-modules $M$ and $N$ is given by the underlying object in $\ca{V}$ of the $B$-module $M^{\vee}\otimes_B N$, and the composition morphisms are given by coevaluation maps. Since $F$ preserves colimits and tensor products, we get an isomorphism
 \[
 i_{M,N} \colon  F(M^{\vee} \otimes_B N) \rightarrow (FM)^{\vee} \otimes_{FB} FN
 \]
 in $\ca{V}^{\prime}$. It is not hard to check that these isomorphisms give the desired $\ca{V}^{\prime}$-functor $i \colon F_\ast \Mod^d_B \rightarrow \Mod^d_{FB}$, and that it is symmetric strong monoidal.
\end{proof}

\begin{thm}
 Let $F \colon \ca{V} \rightarrow \ca{V}^{\prime}$ be a cocontinuous symmetric strong monoidal functor, and let $B$ be a commutative monoid in $\ca{V}$. Then the diagram
 \[
 \xymatrix{ \mathbf{SymMon}\mbox{-}\Vcat \slash \Mod^d_{B} \ar[r]^-{L} \ar[d]_{iF_\ast} & \mathbf{Bialg}(\ca{V},B) \ar[d]^{F_\ast} \\
 \mathbf{SymMon}\mbox{-}\ca{V}^{\prime}\mbox{-}\Cat \slash \Mod^d_{FB} \ar[r]_-{L} & \mathbf{Bialg}(\ca{V}^{\prime},FB)
  }
 \]
 commutes up to natural isomorphism, where $\mathbf{SymMon}\mbox{-}\Vcat$ denotes the category of small symmetric monoidal $\ca{V}$-categories and $\mathbf{Bialg}(\ca{V},B)$ denotes the category of bialgebroids acting on $B$. Similarly, the diagram
 \[
  \xymatrix{ \mathbf{AutSymMon}\mbox{-}\Vcat \slash \Mod^d_B \ar[d]_{iF_\ast} \ar[r]^-{L} & \mathbf{Hopf}(\ca{V},B) \ar[d]^{iF_\ast} \\  
  \mathbf{AutSymMon}\mbox{-}\ca{V}^{\prime}\mbox{-}\Cat \slash \Mod^d_{FB} \ar[r]_-{L} & \mathbf{Hopf}(\ca{V}^{\prime},FB)
  }
 \]
 commutes up to natural isomorphism, where $\mathbf{AutSymMon}\mbox{-}\Vcat$ denotes the category of autonomous symmetric monoidal $\ca{V}$-categories and $\mathbf{Hopf}(\ca{V},B)$ denotes the category of Hopf algebroids acting on $B$.
\end{thm}

\begin{proof}
 Since $\mathbf{Hopf}(\ca{V}^{\prime},FB)$ is a full subcategory of $\mathbf{Bialg}(\ca{V}^{\prime},FB)$, it clearly suffices to show that the first of the two diagrams commutes up to natural isomorphism. Let $\ca{M}$ and $\ca{M}^{\prime}$ be symmetric Gray monoids equivalent to $\Mod(\ca{V})$ and $\Mod(\ca{V}^{\prime})$ respectively. Let $G$ and $G^{\prime}$ denote the strict symmetric monoidal pseudofunctors from Proposition~\ref{COMPANION_FUNCTOR_MONOIDAL_PROP}. 
 
 Recall that $\Mod^d_B=\overline{B}$ and $\Mod^d_{FB}=\overline{FB}$ (see Proposition~\ref{AUTONOMOUS_CAUCHY_PROP}). Moreover, the inclusion $ j \colon B \rightarrow \overline{B}$ induces a $\ca{V}^{\prime}$-functor $F_\ast j \colon FB=F_\ast B \rightarrow F_\ast \overline{B}$. It is clear from the construction of $i$ (see Proposition~\ref{CAUCHY_COMPARISON_FUNCTOR_PROP}) that the diagram
 \[
  \xymatrix{FB \ar[rr]^{k} \ar[rd]_{F_\ast j} && \overline{FB} \\
  & F_\ast\overline{B} \ar[ru]_{i}}
 \]
 is commutative, where $k\colon FB \rightarrow \overline{FB}$ denotes the natural inclusion. From \cite[\S~5]{KELLY_BASIC} we know that an inclusion of $\ca{V}$-categories induces an equivalence in $\Mod(\ca{V})$ if and only if the two categories have the same Cauchy completion. Thus $Gj$, $G^{\prime} i$ $G^{\prime}F_\ast j$ and $G^{\prime}k$ are equivalences in $\ca{M}$ and $\ca{M}^{\prime}$ respectively. Moreover, the equivalence $\Vcat \slash \overline{B} \simeq \Map(\ca{M},B)$ from Lemma~\ref{MAPS_LEMMA} is precisely the composite
 \[
  \xymatrix@C=50pt{\Vcat \slash \overline{B} \ar[r]^-G & \Map(\ca{M},\overline{B}) \ar[r]^-{\Map(\ca{M},Gj^{-1})} & \Map(\ca{M},B)}\smash{\rlap{,}}
 \]
 and similarly for $\ca{V}^{\prime}\mbox{-}\Cat \slash \overline{FB} \simeq \Map(\ca{M}^{\prime},FB)$. It remains to show that the lifts to symmetric pseudomonoids of the pseudofunctors and 2-functors in the diagram
 \[
  \xymatrix{
   \Vcat \slash \overline{B} \ar[r]^-{G} \ar[d]_{F_\ast} \ar@{}[rd]|{(1)} &
   \Map(\ca{M},\overline{B}) \ar[d]^{F_\ast} \ar[r]^-{Gj^{-1}} &
   \Map(\ca{M},B) \ar@{}[rdd]|{(3)} \ar[r]^-{L} \ar[d]^{F_\ast} &
   \Comon(B) \ar[dd]^{F_\ast} \\
   \ca{V}^{\prime}\mbox{-}\Cat \ar[d]_{i} \slash F_\ast \overline{B} \ar[r]^-{G^{\prime}} &
   \Map(\ca{M}^{\prime},F_\ast \overline{B}) \ar[d]^{G^{\prime}i} \ar[r]^-{G^{\prime}F_\ast j^{-1}} \ar@{}[rd]|{(2)} & 
   \Map(\ca{M}^{\prime},FB) \ar@{=}[d] \\
   \ca{V}^{\prime}\mbox{-}\Cat \slash \overline{FB} \ar[r]^-{G^{\prime}} & 
   \Map(\ca{M}^{\prime},\overline{FB}) \ar[r]^-{G^{\prime} k^{-1}} &
   \Map(\ca{M}^{\prime},FB) \ar[r]^-{L} &
   \Comon(FB) }
 \]
 commute up to pseudonatural or 2-natural equivalence, where we used the abbreviation $f$ for a the functor that is given by composing with $f$. Diagram~(1) is a commutative diagram of symmetric strong monoidal pseudofunctors. Indeed, the symmetric monoidal structure of the base change functor on modules is defined to be the image under $G^{\prime}$ of the symmetric monoidal structure of the base change functor  $\Vcat \rightarrow \ca{V}^{\prime}\mbox{-}\Cat$. (This construction is analogous to how the symmetric monoidal structure on $\Mod(\ca{V})$ is defined by transfer along $G$, cf.\ Proposition~\ref{COMPANION_FUNCTOR_MONOIDAL_PROP} and \cite[Theorem~5.1]{SHULMAN}.) Diagram~(2) commutes up to isomorphism because $i.F_\ast j=k$, and Diagram~(3) commutes by Proposition~\ref{BRAIDED_BASE_CHANGE_PROP}.
 
 To see that the lifts to symmetric pseudomonoids of the unlabeled diagrams commute, first note that $G^{\prime} F_\ast j=F_\ast Gj$. Thus both these diagrams compare the operations of first composing with a morphism and then applying a pseudofunctor to applying a pseudofunctor and then composing with the image of the morphism in question. Therefore they commute up to pseudonatural isomorphism.
\end{proof}

\appendix

\section{Density in cosmoi with dense autonomous generator} \label{DENSITY_IN_COSMOI_WITH_DAG_APPENDIX}

In Section~\ref{DAG_RECOGNITION_SECTION} we frequently used the fact that for a cosmos with dense autonomous generator $\ca{X}$, the notion of $\Set$-density and $\ca{V}$-density coincide in a lot of important cases. Our proof relies on the following concept.

\subsection{Representations of monoidal categories}\label{REPRESENTATIONS_SECTION}
 Let $\ca{V}$ be a cosmos with dense autonomous generator $\ca{X}$ (see Definition~\ref{DAG_DFN}). To each $\ca{X}$-tensored $\ca{V}$-category $\ca{A}$ we can associate an ordinary category endowed with an action of $\ca{X}_0$. Such a category with an action of a monoidal category is called a \emph{$\ca{X}_0$-actegory} (in \cite{MCCRUDDEN_REPR_COALGEBROIDS}) or \emph{$\ca{X}_0$-representation} (in \cite{GORDON_POWER}). An $\ca{X}_0$-representation is an ordinary category $\ca{L}$, together with a functor $-\odot - \colon \ca{X}_0 \times \ca{L} \rightarrow \ca{L}$ and natural isomorphisms $l\colon L \rightarrow I \odot L$ and $a \colon X \odot (X^\prime \odot L) \rightarrow (X\otimes X^\prime) \odot L$ for all $L\in \ca{L}$, subject to certain coherence conditions (details can be found in \cite[\S~2]{GORDON_POWER} or \cite[\S~3]{MCCRUDDEN_REPR_COALGEBROIDS}). Since we assume that $\ca{X}_0$ is $\Set$-dense, the assignment which sends an $\ca{X}$-tensored $\ca{V}$-category $\ca{A}$ to the $\ca{X}_0$-representation $\ca{A}_0$, with action given by the tensor functor $-\odot- \colon \ca{X}_0 \times \ca{A}_0 \rightarrow \ca{A}_0$ is in fact a fully faithful 2-functor (see \cite[Theorem~3.4]{GORDON_POWER}). This means that giving a $\ca{V}$-functor $F \colon \ca{A} \rightarrow \ca{A}^\prime$ between $\ca{X}$-tensored $\ca{V}$-categories is the same as giving an ordinary functor $F_0 \colon \ca{A}_0 \rightarrow \ca{A}^\prime_0$, together with morphisms $\widehat{F} \colon X\odot F_0 A \rightarrow F_0(X\odot A)$ making the diagrams
\[
\vcenter{\xymatrix{
F_0 A \ar[rd]_{F_0 l} \ar[r]^{l^\prime} & I \odot F_0 A \ar[d]^{\widehat{F}}\\
& F_0(I\odot A)
}}
\quad\mathrm{and}\quad
\vcenter{\xymatrix{ X \odot (X^\prime \odot F_0 A) \ar[r]^{a^\prime}  \ar[d]_{X \odot \widehat{F}} & (X\otimes X^\prime) \odot F_0 A \ar[dd]^{\widehat{F}} \\
X\odot F_0(X^\prime\odot A) \ar[d]_{\widehat{F}} &\\
F_0\bigl(X\odot(X^\prime \odot A)\bigr) \ar[r]_{F_0 a} & F_0\bigl((X\otimes X^\prime) \odot A\bigr)
}}
\]
commutative. The arrows $l$ and $a$ correspond to the canonical isomorphisms $\id \cong [I,-]$ and $[X,[X^\prime,-]] \cong [X\otimes X^\prime,-]$ under the $\ca{V}$-natural isomorphisms which define the respective tensors, and $\widehat{F}$ is given by the map of the same name introduced in Section~\ref{WEIGHTED_COLIMIT_SECTION}. Moreover, tensors with objects in $\ca{X}$ are absolute colimits (see \cite{STREET_ABSOLUTE}), so the morphisms $\widehat{F} \colon X\odot FA \rightarrow F(X\odot A)$ are isomorphisms. Still under the assumption that $\ca{X}_0$ is $\Set$-dense and that $\ca{A}$, $\ca{A}^\prime$ are $\ca{X}$-tensored, giving a $\ca{V}$-natural transformation $\alpha \colon F \Rightarrow F^\prime \colon \ca{A} \rightarrow \ca{A}^\prime$ is the same as giving an ordinary natural transformation $\alpha \colon F_0 \Rightarrow F^\prime_0$ such that
\[
 \xymatrix{X\odot F_0A \ar[r]^{X\odot \alpha_A} \ar[d]_{\widehat{F}} & X \odot F^\prime_0 A \ar[d]^{\widehat{F^\prime}}\\
F_0(X\odot A) \ar[r]_{\alpha_{X\odot A}} & F^\prime_0 (X\odot A)}
\]
is commutative.

\begin{thm}\label{DENSITY_THM}
 Let $\ca{V}$ be a cosmos which has a dense autonomous generator $\ca{X}$. Let $\ca{A}$ be an $\ca{X}$-tensored $\ca{V}$-category, and let $\ca{C}$ be a $\ca{V}$-category which is cotensored. A $\ca{V}$-functor $K \colon \ca{A} \rightarrow \ca{C}$ is $\ca{V}$-dense if and only if the underlying ordinary functor $K_0 \colon \ca{A}_0 \rightarrow \ca{C}_0$ is $\Set$-dense.
\end{thm}

\begin{proof}
 The assumption that $\ca{C}$ is cotensored implies that $K$ is $\ca{V}$-dense if and only if the map $\ca{C}_0(C,D) \rightarrow \VNat\bigl(\ca{C}(K-,C),\ca{C}(K-,D)\bigr)$ which sends $g \colon C \rightarrow D$ to the $\ca{V}$-natural transformation $\ca{C}(K-,g)$ is a bijection of sets (see \cite[\S~5.1]{KELLY_BASIC}). For $C\in \ca{C}$, let $K\slash C$ be the category with objects the morphisms $\phi \colon KA \rightarrow C$, $A \in \ca{A}$, and morphisms $\phi \rightarrow \phi^\prime$ the morphisms in $\ca{A}_0$ which make the evident triangle commute. From \cite[Formula~5.4]{KELLY_BASIC} we know that $K_0$ is $\Set$-dense if and only if each object $C$ is the colimit of the tautological cocone on the functor $V_C \colon K\slash C \rightarrow \ca{C}$ which sends $\phi$ to its domain. We write $S_D$ for the set of cocones on $V_C$ with vertex $D$, and we let $V=\ca{V}_0(I,-) \colon \ca{V}_0 \rightarrow \Set$ be the canonical forgetful functor. Let $\chi \colon \VNat\bigl(\ca{C}(K-,C), \ca{C}(K-,D)\bigr) \rightarrow S_D$ be the map which sends $\alpha$ to the cocone $\chi(\alpha)\defl\bigl(V\alpha_A (\phi)\bigr)_{\phi \in K \slash C}$. The composite
\[
 \xymatrix{ \ca{C}_0(C,D) \ar[r] & \VNat \bigl(\ca{C}(K-,C),\ca{C}(K-,D)\bigr)  \ar[r]^-{\chi} & S_D}
\]
 sends $g$ to the cocone $(g \phi)_{\phi \in K\slash C}$. This composite is a bijection if and only if $C$ is the colimit of the tautological cocone, that is, if and only if $K_0$ is $\Set$-dense. If we can show that $\chi$ is a bijection, then $K_0$ is $\Set$-dense if and only if $K$ is $\ca{V}$-dense, as claimed.

 We now construct an inverse for $\chi$, as follows. Given a cocone $\gamma=(\gamma_\phi)_{\phi \in K\slash C}$, we let
\[
\beta_A \colon \ca{C}_0(KA,C) \rightarrow \ca{C}_0(KA,D)
\]
 be the map with $\beta_A(\phi)=\gamma_\phi$. We write $F, G \colon \ca{A} \rightarrow \ca{V}^{\op}$ for the functors $\ca{C}(K-,C)$ and $\ca{C}(K-,D)$ respectively. Note that we have $VF_0=\ca{C}_0(KA,C)$, and $\beta$ is a natural transformation between the $\Set$-valued functors $VF_0$ and $VG_0$. We first use the density assumption to lift this to a natural transformation $\xi(\gamma) \colon F_0 \rightarrow G_0$ between the underlying ordinary $\ca{V}_0$-valued functors of $F$ and $G$, and we then show that $\xi(\gamma)$ is in fact $\ca{V}$-natural. The tensor of $B$ and $X$ in $\ca{V}^{\op}$ is given by $[X,B]$. Since all $\ca{V}$-functors preserve tensors with objects which have duals (see \cite{STREET_ABSOLUTE}), we get isomorphisms
\[
\xymatrix{F(X\odot A) \ar[r]^-{\widehat{F}} & [X,FA]}\quad\mathrm{and}\quad \xymatrix{G(X\odot A) \ar[r]^-{\widehat{G}} & [X,GA] },
\]
 and the composite $V\widehat{G}_0 \circ \beta_{X\odot A} \circ V \widehat{F}_0^{-1} \colon V[X,\ca{C}(KA,C)]_0 \rightarrow V[X,\ca{C}(KA,D)]_0$ is natural in $X$. Since $\ca{V}_0(X,-)$ is naturally isomorphic to $V[X,-]_0$, it follows by $\Set$-density of $\ca{X}$ in $\ca{V}$ that there is a unique morphism $\xi(\gamma)_A\colon \ca{C}(KA,C) \rightarrow \ca{C}(KA,D)$ in $\ca{V}$ such that
\[
 \xymatrix@C=45pt{\ca{V}_0\bigl(X,\ca{C}(KA,C)\bigr) \ar[d]_{\cong}  \ar[r]^-{\ca{V}_0(X,\xi(\gamma)_A)} & \ca{V}_0\bigl(X,\ca{C}(KA,D)\bigr) \ar[d]^{\cong} \\
 V[X,\ca{C}(KA,C)]_0 \ar[d]_{V\widehat{F}_0^{-1}} \ar[r]^-{V[X,\xi(\gamma)_A]_0} & V[X,\ca{C}(KA,D)]_0 \ar[d]^{V\widehat{G}_0^{-1}}\\
 \ca{C}_0\bigl(K(X\odot A),C\bigr) \ar[r]_-{\beta_{X\odot A}} & \ca{C}_0\bigl(K(X\odot A),D\bigr) }
\]
 is commutative for every $X \in \ca{X}$. Hence part (1) and (3) of the diagram
\[
 \xymatrix@!C=45pt{V[X,[X^{\prime},FA]]_0 \ar[rrrr]^{V[X,[X^\prime,\xi(\gamma)_A]]_0} && \ar@{}[d]|(0.4){(0)} && V[X,[X^{\prime},FA]]_0 \\
 & V[X,F(X\odot A)]_0 \ar[lu]_(0.4){V[X,\widehat{F}]_0} \ar[rr]^{V[X,\xi(\gamma)_{X^\prime \odot A}]_0} &\ar@{}[d]|(0.4){(1)}& V[X,G(X\odot A)]_0 \ar[ru]^(0.4){V[X,\widehat{G}]_0}\\
 & VF_0\bigr(X\odot(X^\prime\odot A)\bigl) \ar[u]^{V\widehat{F}_0} \ar[rr]^{\beta_{X\odot(X^\prime \odot A)}} &\ar@{}[d]|{(2)}& VG_0\bigr(X\odot(X^\prime\odot A)\bigl) \ar[u]_{V\widehat{G}_0}\\
 & VF_0\bigr((X\otimes X^\prime)\odot A)\bigl) \ar[u]^{VF_0 a} \ar[ld]^(0.4){V\widehat{F}_0} \ar[rr]_{\beta_{(X\otimes X^\prime)\odot A}} && VG_0\bigr((X\otimes X^\prime)\odot A\bigl) \ar[u]_{VG_0 a} \ar[rd]_(0.4){V\widehat{G}_0}\\
V[X\otimes X^{\prime},FA]_0 \ar[uuuu]^{V a^\prime} \ar[rrrr]_-{V[X\otimes X^\prime,\xi(\gamma)_A]_0}&&\ar@{}[u]|(0.4){(3)}&& V[X\otimes X^{\prime},GA]_0 \ar[uuuu]_{V a^\prime} }
\]
 are commutative. Part (2) is commutative since $\beta$ is natural, and the two pentagons are instances of the coherence diagrams in Section~\ref{REPRESENTATIONS_SECTION}. The outer diagram is commutative because $a^\prime$ is natural, hence it follows that part (0) is commutative. Since $\ca{X}$ is $\Set$-dense we find that $[X^\prime,\xi(\gamma)_A] \circ \widehat{F}=\widehat{G} \circ \xi(\gamma)_{X^\prime \odot A}$. The considerations in Section~\ref{REPRESENTATIONS_SECTION} therefore imply that $\xi(\gamma)$ is a $\ca{V}$-natural transformation. The second coherence diagram of Section~\ref{REPRESENTATIONS_SECTION} implies that $V\xi(\gamma)_A=\beta_A$ for all objects $A \in \ca{A}$, and it follows that $\chi\bigl(\xi(\gamma)\bigr) = \gamma$. Moreover, if we start with a $\ca{V}$-natural transformation $\alpha \colon F \Rightarrow G$ and construct the $\beta_A$ associated to the cocone $\chi(\alpha)$, we clearly get $\beta_{X\odot A}=V\alpha_{X \odot A}$, that is, $V\xi\bigl(\chi(\alpha)\bigr)_{X\odot A}=V\alpha_{X \odot A}$. Both $\alpha$ and $\xi(\gamma)$ are $\ca{V}$-natural, hence we must have $V[X,\xi\bigl(\chi(\alpha)\bigr)_A]_0=V[X,\alpha_A]_0$, and by density of $\ca{X}$ it follows that $\xi\bigl(\chi(\alpha)\bigr)=\alpha$. In other words, the assignment which sends a cocone $\gamma$ to the $\ca{V}$-natural transformation $\xi(\gamma)$ constructed above gives the desired inverse to $\chi$.
\end{proof}
\section{Monoidal biadjunctions}\label{MONOIDAL_BIADJUNCTION_APPENDIX}

\subsection{Overview}
 It is well-known that if a left adjoint between monoidal categories is strong monoidal, then its right adjoint inherits a weak monoidal structure in such a way that the unit and counit become monoidal natural transformations (see \cite{KELLY_DOCTRINAL}). Moreover, the resulting adjunction lifts to the categories of monoids. There exist similar results for the case of braided and symmetric strong monoidal left adjoints. In this appendix we will see that these results generalize to biadjunctions between monoidal 2-categories whose left adjoint is strong monoidal.
 
 \begin{prop}\label{MONOIDAL_BIADJUNCTION_PROP}
 Let $T \colon \ca{M} \rightarrow \ca{N}$ be a strong monoidal left biadjoint between monoidal 2-categories, with right biadjoint $H$. Then $H$ can be endowed with the structure of a weak monoidal pseudofunctor, and the unit and counit with the structure of weak monoidal pseudonatural transformations, in such a way that the invertible modifications $\alpha$ and $\beta$ that replace the triangle identities become monoidal modifications.
 \end{prop}
 
 \subsection{Monoidal biadjunctions and strictification}
 In order to prove this result we apply some strictification theorems. First of all, we can replace our monoidal bicategories by Gray monoids. Moreover, we can make sure that these Gray monoids are cofibrant in the sense of \cite{LACK_GRAY}. This ensures that the 2-categories in question are cofibrant as 2-categories (see \cite[S~9]{LACK_GRAY} and \cite{LACK_2CAT}), which implies that the pseudofunctors $T$ and $H$ are equivalent to 2-functors. By doing this, we are effectively working in the Gray-category $\Gray$ of 2-categories, 2-functors, pseudonatural transformations and modifications. By replacing the modification $\beta$ we can make sure that the biadjunction is in fact a pseudoadjunction in this Gray-category in the sense of 
\cite{LACK_PSEUDOMONADS}.
 
 \begin{lemma}
 Let $\mathbb{G}$ be a Gray category, let $T \colon \ca{M} \rightarrow \ca{N}$ and $H \colon \ca{N} \rightarrow \ca{M}$ be 1-cells, $n \colon \id \Rightarrow HT$ and $e \colon TH \Rightarrow \id$ 2-cells, and let $\alpha \colon eT \cdot Tn \cong T$ and $\beta \colon H \cong He \cdot nH$ be invertible 2-cells. Then there exists an invertible 3-cell $\alpha^{\prime}$ such that $(T,H,n,e,\alpha^{\prime},\beta)$ is a pseudoadjunction in the sense of \cite{LACK_PSEUDOMONADS}.
 \end{lemma}
 
 \begin{proof}
 When working in a hom-2-category of a Gray-category $\mathbb{G}$, we can use a string diagram notation similar to the one introduced for Gray monoids in Section~\ref{GRAY_MONOIDS_SECTION}. The tensor product (composition) of 1-cells will of course only be partially defined, and there is no convenient way to keep track of the name of the 0-cells. We leave it to the reader to check that the 2-cell $\alpha^{\prime}$ given by
 \[
  \alpha^{\prime} \defl \vcenter{\hbox{

\begin{tikzpicture}[y=0.80pt, x=0.8pt,yscale=-1, inner sep=0pt, outer sep=0pt, every text node part/.style={font=\scriptsize} ]
\path[draw=black,line join=miter,line cap=butt,line width=0.500pt]
  (-460.6299,1335.8267) .. controls (-478.8865,1340.3909) and
  (-492.4398,1348.4825) .. (-501.2898,1358.2840)(-509.2536,1369.9194) ..
  controls (-522.6107,1397.0658) and (-506.4028,1430.6827) ..
  (-460.6299,1442.1260)(-513.7795,1318.1102) .. controls (-513.7795,1364.1732)
  and (-496.0630,1388.9763) .. (-460.6299,1406.6929);
\path[draw=black,line join=miter,line cap=butt,line width=0.500pt]
  (-460.6299,1335.8267) .. controls (-425.1969,1353.5433) and
  (-425.1969,1388.9763) .. (-460.6299,1406.6929);
\path[draw=black,line join=miter,line cap=butt,line width=0.500pt]
  (-460.6299,1442.1260) .. controls (-389.7638,1424.4094) and
  (-389.7638,1353.5433) .. (-389.7638,1318.1102);
\path[fill=black] (-461.33331,1335.9177) node[circle, draw, line width=0.500pt,
  minimum width=5mm, fill=white, inner sep=0.25mm] (text3761) {$TH\alpha^{-1}$
  };
\path[fill=black] (-460.62991,1406.6929) node[circle, draw, line width=0.500pt,
  minimum width=5mm, fill=white, inner sep=0.25mm] (text3765) {$T\beta^{-1}T$
  };
\path[fill=black] (-460.62991,1442.126) node[circle, draw, line width=0.500pt,
  minimum width=5mm, fill=white, inner sep=0.25mm] (text3769) {$\alpha$    };
\path[fill=black] (-520.86615,1314.5669) node[above right] (text3773) {$Tn$
  };
\path[fill=black] (-393.3071,1314.5669) node[above right] (text3777) {$eT$
  };
\path[fill=black] (-531.49603,1406.6929) node[above right] (text3781) {$Tn$
  };
\path[fill=black] (-428.74014,1374.8031) node[above right] (text3785) {$THeT$
  };
\path[fill=black] (-485.43307,1385.4331) node[above right] (text4611) {$TnHT$
  };
\path[fill=black] (-492.51968,1364.1732) node[above right] (text4615) {$THTn$
  };

\end{tikzpicture}
}}
 \]
 has the desired  properties.
 \end{proof}

 Thus we can apply the coherence theorem for pseudoadjunctions \cite[Proposition~5.1]{LACK_PSEUDOMONADS}, which implies that any two 3-cells built out of $\alpha$, $\beta$ and the Gray interchange between iterated composites of $T$ and $H$ in the Gray-category $\Gray$ of 2-categories are equal. We thus reduced the problem to proving the following proposition.

 \begin{prop}\label{MONOIDAL_PSEUDOADJUNCTION_PROP}
 Let $(T,H,n,e,\alpha,\beta)$ be a pseudoadjunction in $\Gray$ between two Gray monoids $\ca{M}$ and $\ca{N}$, and let $(\chi,\iota,\omega,\zeta,\rho)$ endow $T$ with the structure of a strong monoidal 2-functor (see \cite[Definition~2]{DAY_STREET}). Then there exists a structure of a weak monoidal 2-functor for $H$ as well as structures of monoidal pseudonatural transformations for $n$ and $e$ in such a way that the modifications $\alpha$ and $\beta$ become monoidal modifications. 
 \end{prop}

 Before proving this, we need to introduce some notation and prove a lemma that will simplify the computations. In $\Gray$, the interchange is given by the pseudonaturality square (see \cite[\S~5.3]{GORDON_POWER_STREET}). Therefore we write
 \[
  \vcenter
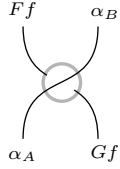
{\hbox{

\begin{tikzpicture}[y=0.80pt, x=0.8pt,yscale=-1, inner sep=0pt, outer sep=0pt, every text node part/.style={font=\scriptsize} ]
\path[draw=black,line join=miter,line cap=butt,line width=0.500pt]
  (-265.7480,1318.1102) .. controls (-265.7480,1353.5433) and
  (-301.1811,1335.8267) .. (-301.1811,1371.2598)(-301.1811,1318.1102) ..
  controls (-301.1811,1331.4440) and (-296.1635,1337.2514) ..
  (-289.9046,1341.1968)(-277.0246,1348.1732) .. controls (-270.7657,1352.1186)
  and (-265.7480,1357.9260) .. (-265.7480,1371.2598);
\path[fill=black] (-269.29135,1314.5669) node[above right] (text6293)
  {$\alpha_B$     };
\path[fill=black] (-308.26773,1381.8898) node[above right] (text6297)
  {$\alpha_A$     };
\path[fill=black] (-269.29135,1381.8898) node[above right] (text6301) {$Gf$
  };
\path[fill=black] (-308.26773,1314.5669) node[above right] (text6305) {$Ff$
  };
\path[shift={(-25.952917,-0.05209923)},draw=black,miter limit=4.00,draw
  opacity=0.294,line width=1.276pt]
  (-248.0315,1344.6851)arc(0.000:180.000:8.858)arc(-180.000:0.000:8.858) --
  cycle;

\end{tikzpicture}
}}
 \]
 for the $f$-component of a pseudonatural transformation $\alpha \colon F \Rightarrow G$ between 2-functors $F$ and $G$. This notation is justified by the fact that 2-categories, 2-functors, pseudonatural transformations and modifications form the Gray-category $\Gray$. At the same time it allows for the distinction between pseudonaturality squares from a Gray interchange cell internal to some monoid in $\Gray$: in the latter case we omit the small disk indicating the 2-cell.
 
\begin{lemma}\label{MATES_ARE_UNIQUE_LEMMA}
 Let $\ca{A}$ be a 2-category. For $i=1,2$ let $l_i \colon A \rightarrow B$ be an adjoint equivalence in $\ca{A}$, with right adjoint inverse $r_i$. Then the inverses of the unit and the counit exhibit $r_i$ as \emph{left} adjoint of $l_i$. If $\alpha \colon l_1 \Rightarrow l_2$ is invertible, then the two mates
 \[
\vcenter{\hbox{


}}
 \]
 of $\alpha$ coincide.
\end{lemma}

\begin{proof}
 It is an immediate consequence of the triangle identities that the composite of the first mate with the inverse of the second is equal to the identity. Thus the two 2-cells must be equal.
\end{proof}

\begin{proof}[Proof of Proposition~\ref{MONOIDAL_PSEUDOADJUNCTION_PROP}.]
 We list the structure cells and leave it to the reader to check that the necessary axioms (see \cite[Definition~2]{DAY_STREET} and \cite[Definition~3]{DAY_STREET}) hold. We also use the notation introduced there. Note that their choice of tensor product of 1-cells and 2-cells in a Gray monoid differs from ours (for example, $f\otimes g=fB^\prime \cdot Ag$ as opposed to $f\otimes g=A^\prime g \cdot fB$), which means that some of the interchanges appearing in their axioms have to be flipped to adapt the axioms to our convention. As usual we write the tensor product of objects in a Gray monoid simply as concatenation.
 
 We can assume that the pseudonatural transformation
 \[
  \chi_{A,A^\prime} \colon (TA)(TA^\prime) \rightarrow T(AA^\prime)
 \]
 is an adjoint equivalence, with right adjoint inverse $\chi^{-1}_{A,A^{\prime}}$. The unit and counit of this adjoint equivalence are thus invertible modifications, from which we deduce that graphically evident simplifications such as
 \[
  \vcenter{\hbox{


}}
 \]
 are valid. Note that the inverses of the unit and the counit exhibit $\chi^{-1}$ as left adjoint of $\chi$, which implies that they satisfy similar laws in the graphical calculus of pseudonatural transformations. Similarly we assume that a choice of right adjoint inverse equivalence $\iota^{-1} \colon TI \rightarrow I$ for $\iota \colon I \rightarrow TI$ has been made.
 
 We now list the structure cells that turn $H$ into a weak monoidal functor and $n$ and $e$ into monoidal pseudonatural transformations. We leave it to the reader to check that these satisfy all the necessary axioms, and that $\alpha$ and $\beta$ become monoidal natural transformations. The graphical calculus introduced above simplifies these computations considerably. To avoid excessive use of parentheses we write $H_B$ and $T_A$ for the evaluation of the functors $H$ and $T$ on objects, and similarly for maps.
 
 We define the pseudonatural transformation $\chi^H$ to be the adjoint of $e\otimes e \cdot \chi^{-1}$, that is, the composite
\[
 \xymatrix@C=40pt{H_B H_{B^{\prime}} \ar[r]^-{n_{H_B H_{B^{\prime}} }} & H_{T_{H_B H_{B^\prime}}} \ar[r]^-{H_{\chi^{-1}_{H_B,H_{B^{\prime}}}}} & H_{T_{H_B} T_{H_{B^\prime}}} \ar[r]^-{H_{e_B\otimes e_{B^\prime}}} & H_{BB^\prime} } \smash{\rlap{.}}
\]
 The pseudonatural transformation $\iota^H$ is given by
\[
\xymatrix{I \ar[r]^-{n_I} & T_{H_I} \ar[r]^-{H \iota^{-1}} & H_I} \smash{\rlap{.}} 
\]
 We define required modifications $\zeta^H$, $\kappa^H$ by
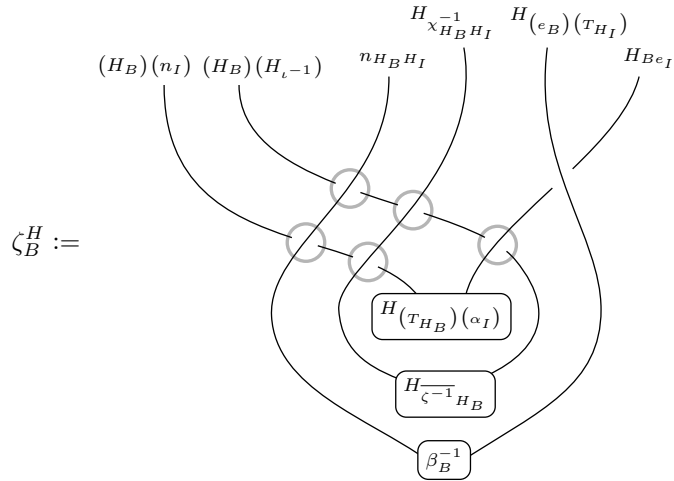
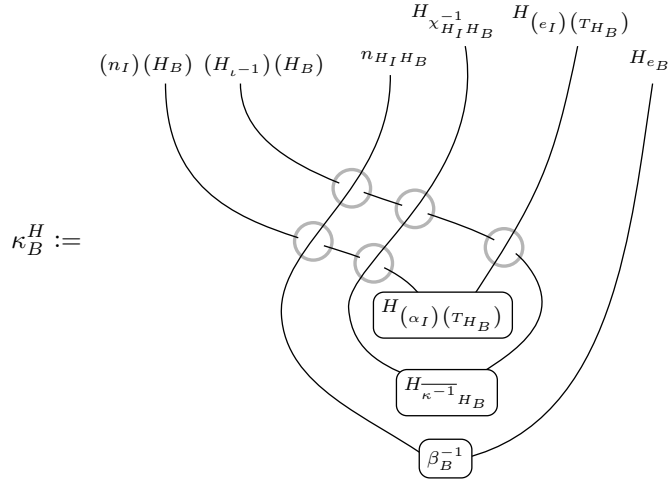
\[
 \zeta^H_{B} \defl \vcenter{\hbox{


}}
\]
 respectively. Note that the mates of $\zeta^{-1}$ and $\kappa^{-1}$ are well-defined by Lemma~\ref{MATES_ARE_UNIQUE_LEMMA}. The components
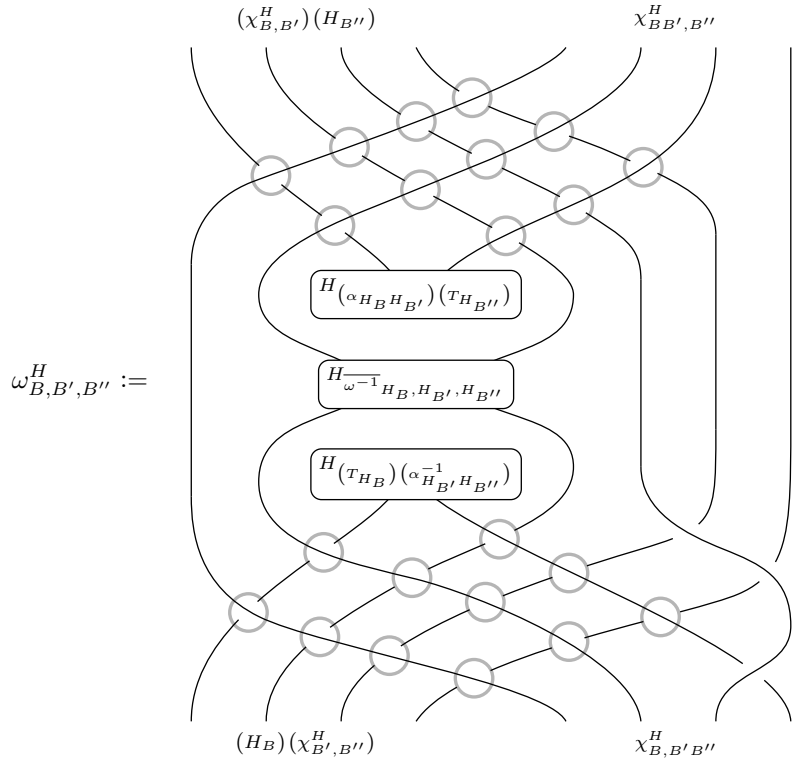
\[
  \omega^H_{B,B^{\prime},B^{\prime\prime}} \defl \quad \vcenter{\hbox{

%
}}
\]
 define the modification $\omega^H$. The mate of $\omega^{-1}$ is again well-defined by Lemma~\ref{MATES_ARE_UNIQUE_LEMMA}.
 
 The invertible modifications
 \[
 \theta_{2;A,A^{\prime}} \defl
 \vcenter{\hbox{

%
 }}
 \]
 endow $n \colon \id \Rightarrow HT$ with the structure of a monoidal pseudonatural transformation (see \cite[Definition~3]{DAY_STREET}), where the monoidal structure of the identity functor is given by actual equalities. Similarly, the invertible modifications
 \[
 \theta_{2;B,B^{\prime}} \defl
 \vcenter{\hbox{

%
 }}  
 \]
 endow $e \colon TH \Rightarrow \id$ with the structure of a monoidal pseudonatural transformation. One can check that with these choices for $\theta$, the modifications $\alpha$ and $\beta$ become monoidal modifications in the sense of \cite[Definition~3]{DAY_STREET}.
\end{proof}

\begin{cor}\label{LIFT_TO_PSEUDOMONOIDS_COR}
 Let $T \colon \ca{M} \rightarrow \ca{N}$ be a strong monoidal left biadjoint between monoidal 2-categories, with right biadjoint $H$. If both $H$ and $T$ are normal, that is, they preserve identities strictly, then the biadjunction lifts to a biadjunction
 \[
 \xymatrix{ \PsMon(\ca{N}) \ar[dd]_{U} \rrtwocell<5>^{\PsMon(T)}_{\PsMon(H)}{`\perp} && \PsMon(\ca{M}) \ar[dd]^{U} \\ \\
           \ca{N} \rrtwocell<5>^{T}_{H}{`\perp} && \ca{M}}
 \]
 between the categories of pseudomonoids. The underlying morphisms of the unit and the counit are given by the unit and the counit of the biadjunction $T \dashv H$.
 \end{cor}
 
 \begin{proof}
 We prove this using the following idea from \cite[Proposition~5]{DAY_STREET} and \cite{MCCRUDDEN_BALANCED}. A pseudomonoid in a monoidal 2-category $\ca{M}$ can be identified with a weak monoidal normal pseudofunctor from the terminal 2-category $1$ to $\ca{M}$ (equivalently, a weak monoidal 2-functor $1 \rightarrow \ca{M}$). Then the lifted biadjunction is simply given by composition with $T$ and $H$.
 
 We can make this argument more precise using the language of tricategories. Let $\mathbb{M}$ be the tricategory of monoidal 2-categories, weak monoidal normal pseudofunctors, monoidal pseudonatural transformations and monoidal modifications. Then the assignment which sends $\ca{M}$ to $\PsMon(\ca{M})$ is the object part of the represented pseudo-3-functor $\mathbb{M}(1,-)$. It is a general fact that pseudo-3-functors preserve biadjunctions; to see this, we notice that the notion of biadjunction is `flexible' in the sense that it only talks about equations between 3-cells, not between 1-cells and 2-cells. Thus we can apply coherence results and reduce the problem to showing that a Gray-functor preserves pseudoadjunctions. Steve Lack showed that there is a Gray category $\Psa$ which is free on a pseudoadjunction, in the sense that pseudoadjunctions in a Gray-category $\mathbb{G}$ correspond bijectively to Gray-functors $\Psa \rightarrow \mathbb{G}$ (see \cite{LACK_PSEUDOMONADS}). A composite of Gray-functors is clearly a Gray-functor, hecne Gray-functors preserve pseudoadjunctions.
 
 Therefore the desired biadjunction between pseudomonoids is simply obtained by applying $\mathbb{M}(1,-)$ to the biadjunction from Proposition~\ref{MONOIDAL_BIADJUNCTION_PROP}.
 \end{proof}

\subsection{Braiding, syllepsis, and symmetry} We will see that a result analogous to Proposition~\ref{MONOIDAL_PSEUDOADJUNCTION_PROP} is true for braided, sylleptic, and symmetric monoidal 2-categories.

\begin{prop}\label{BRAIDED_PSEUDOADJUNCTION_PROP}
In the situation of Proposition~\ref{MONOIDAL_PSEUDOADJUNCTION_PROP}, if $T$ is a braided strong monoidal 2-functor between braided Gray monoids, then the right adjoint inherits a structure of a braided 2-functor in such a way that the pseudonatural transformations $e$ and $n$ and the modifications $\alpha$ and $\beta$ become braided. 

The same is true for biadjunctions between braided monoidal 2-categories whose left biadjoint is braided.
\end{prop}

\begin{proof}
 As before, using the coherence theorem for monoidal 2-categories and cofibrant replacement we can prove the second part from the first. Let $u$ be a braiding for the strong monoidal 2-functor $T$ (see \cite[Definition~14]{DAY_STREET}). In \cite[Proposition~12]{DAY_STREET} it is proved that the modification $u^H$ given by
 \[
 u^H \defl \vcenter{\hbox{


}}
 \]
 gives a braiding for $H$. A monoidal pseudonatural transformation between braided pseudofunctors is braided if it satisfies a compatibility axiom (that is, being braided is a property of a monoidal pseudonatural transformation, not an additional structure; see \cite[Definition~14]{DAY_STREET}). One can check that $n$ and $e$, endowed with the monoidal structures from Proposition~\ref{MONOIDAL_PSEUDOADJUNCTION_PROP} are braided. Any monoidal modification between braided pseudonatural transformations is braided.
\end{proof}

\begin{cor}\label{LIFT_TO_BRAIDED_PSEUDOMONOIDS_COR}
 Let $T \colon \ca{M} \rightarrow \ca{N}$ be a braided strong monoidal left biadjoint between braided monoidal 2-categories, with right biadjoint $H$. If both $H$ and $T$ are normal, then the biadjunction lifts to a biadjunction
 \[
 \xymatrix{ \BrPsMon(\ca{N}) \ar[dd]_{U} \rrtwocell<5>^{\BrPsMon(T)}_{\BrPsMon(H)}{`\perp} && \BrPsMon(\ca{M}) \ar[dd]^{U} \\ \\
           \ca{N} \rrtwocell<5>^{T}_{H}{`\perp} && \ca{M}}
 \]
 between the 2-categories of braided pseudomonoids. The underlying morphisms of the unit and the counit are given by the unit and the counit of the biadjunction $T \dashv H$.
 \end{cor}

\begin{proof}
 The terminal 2-category is braided monoidal in a unique way, and braided normal pseudofunctors $1 \rightarrow \ca{M}$ are precisely braided pseudomonoids in $\ca{M}$ (by definition; see \cite[\S~3]{MCCRUDDEN_BALANCED}). Thus we can prove this result in the same way we proved Corollary~\ref{LIFT_TO_PSEUDOMONOIDS_COR} except that we replace the tricategory $\mathbf{M}$ with the tricategory $\mathbf{B}$ of braided monoidal 2-categories, braided weak monoidal normal pseudofunctors, braided pseudonatural transformations and braided modifications. From Proposition~\ref{BRAIDED_PSEUDOADJUNCTION_PROP} we know that the biadjunction $T \dashv H$ lives in this tricategory.
\end{proof}

\begin{prop}\label{SYLLEPTIC_BIADJUNCTION_PROP}
In the situation of Proposition~\ref{BRAIDED_PSEUDOADJUNCTION_PROP}, if $T$ is a sylleptic strong monoidal 2-functor between braided Gray monoids (see \cite[Definition~16]{DAY_STREET}), then the right adjoint inherits a structure of sylleptic 2-functor. The pseudonatural transformations $n$ and $e$ and the modifications $\alpha$ and $\beta$ are sylleptic.

The same is true for biadjunctions between sylleptic monoidal 2-categories whose left biadjoint is sylleptic.
\end{prop}

\begin{proof}
 The fact that $H$ is sylleptic is proved in \cite[Proposition~15]{DAY_STREET}. Being sylleptic is a property of a braided functor, so any braided pseudonatural transformation between sylleptic 2-functors is sylleptic, and any braided modification between sylleptic pseudonatural transformations is sylleptic.
\end{proof}

\begin{cor}\label{LIFT_TO_SYMMETRIC_PSEUDOMONOIDS_COR}
 Let $T \colon \ca{M} \rightarrow \ca{N}$ be a sylleptic strong monoidal left biadjoint between symmetric monoidal 2-categories, with right biadjoint $H$. If both $H$ and $T$ are normal, then the biadjunction lifts to a biadjunction
 \[
 \xymatrix{ \SymPsMon(\ca{N}) \ar[dd]_{U} \rrtwocell<5>^{\SymPsMon(T)}_{\SymPsMon(H)}{`\perp} && \SymPsMon(\ca{M}) \ar[dd]^{U} \\ \\
           \ca{N} \rrtwocell<5>^{T}_{H}{`\perp} && \ca{M}}
 \]
 between the 2-categories of symmetric pseudomonoids. The underlying morphisms of the unit and the counit are given by the unit and the counit of the biadjunction $T \dashv H$. 
 \end{cor}

\begin{proof}
 The terminal 2-category is sylleptic monoidal in a unique way (it is in fact symmetric), and sylleptic normal pseudofunctors $1 \rightarrow \ca{M}$ are precisely symmetric pseudomonoids in $\ca{M}$ (by definition; see \cite[\S~4]{MCCRUDDEN_BALANCED}). We get the result from the same argument we used in Corollaries~\ref{LIFT_TO_PSEUDOMONOIDS_COR} and \ref{LIFT_TO_BRAIDED_PSEUDOMONOIDS_COR} applied to the tricategory $\mathbf{S}$ of sylleptic monoidal 2-categories, sylleptic weak monoidal normal pseudofunctors, sylleptic pseudonatural transformations and sylleptic modifications. From Proposition~\ref{SYLLEPTIC_BIADJUNCTION_PROP} we know that the biadjunction $T \dashv H$ lives in this tricategory.
\end{proof}

\section{A technical lemma}\label{TECHNICAL_LEMMA_APPENDIX}

\subsection{Statement of the lemma}
 In this section we will prove the following lemma which was a key ingredient in our proof of Theorem~\ref{HOPF_COMONAD_THM}.

\begin{lemma}\label{HOPF_LEMMA}
 Let $A$ and $B$ be autonomous map pseudomonoids in a Gray monoid $\ca{M}$, and let $(w,\psi,\psi_0)\colon A \rightarrow B$ be a strong monoidal map. Then the 2-cell
 \[
  \tau=\vcenter{\hbox{


}}
 \]
 is invertible.
\end{lemma}

\subsection{Duals and strong monoidal maps}
 To do this we will use the fact that a strong monoidal map automatically ``preserves duals'': if $f$ is a right bidual of $g$ in the lax slice $\ca{M}\slash_\ell A$, then $w \cdot f$ is a right bidual of $w \cdot g$ in $\ca{M} \slash_\ell B$. Since we are interested in showing that certain 2-cells are invertible, we will record precisely which 2-cells exhibit $w \cdot f$ as right bidual of $w \cdot g$. We will only need the special case of this result where $g=\id_A$, that is, where $A$ is autonomous and the right bidual is $d\colon \dual{A} \rightarrow A$

\begin{prop}\label{PRESERVATION_OF_DUALS_PROP}
 Let $A$ and $B$ be pseudomonoids in a Gray monoid $\ca{M}$, and let $(w,\psi,\psi_0) \colon A \rightarrow B$ be a strong monoidal map. If $A$ is autonomous, then the 2-cells
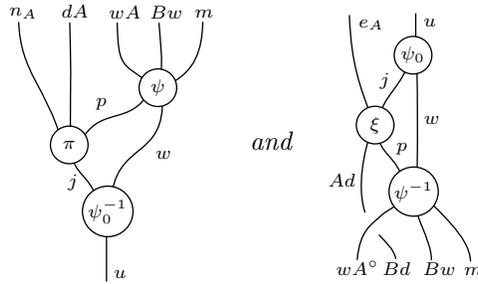
\[
\vcenter{\hbox{


}}
\]
 exhibit $w \cdot d$ as right bidual of $w$ in the sense of Proposition~\ref{DUALIZATIONS_PROP}.
\end{prop}

\begin{proof}
 It is not hard to check the two conditions from Proposition~\ref{DUALIZATIONS_PROP} directly. The proposition is also a special case of \cite[Theorem~3.1]{DUALS_INVERT} (note that any strong monoidal morphism is Frobenius; this is a simple generalization of \cite[Proposition~3]{DAY_PASTRO}).
\end{proof}

\begin{dfn}
 Let $(A,\dual{A},n_A,e_A,\eta,\varepsilon)$ and $(B,\dual{B},n_B,e_B,\eta,\varepsilon)$ be bidual situations in a Gray monoid $\ca{M}$, and let $w \colon A \rightarrow B$ be a morphism. We write $\dual{w}$ for the composite
 \[
\xymatrix{\dual{B} \ar[r]^-{n_A \dual{B}} & \dual{A} A \dual{B} \ar[r]^-{\dual{A} w \dual{B}} & \dual{A} B \dual{B} \ar[r]^-{\dual{A} e_B} & \dual{A}} \smash{\rlap{.}}
 \]
\end{dfn}

 Proposition~\ref{PRESERVATION_OF_DUALS_PROP} shows that postcomposition with a strong monoidal morphism preserves a bidual situation in the lax slice category. The next proposition concerns precomposition of a bidual situation. As before, we only care about the bidual $d$ of the identity morphism, but we are interested in the 2-cells which give the bidual situation.

\begin{prop}\label{MAPS_AND_DUALS_PROP}
 Let $B$ be an autonomous pseudomonoid in a Gray monoid $\ca{M}$, with right bidual $d \colon \dual{B} \rightarrow B$ of the identity. Let $(A,\dual{A},n_A,e_A,\eta,\varepsilon)$ be a bidual situation in $\ca{M}$ and let $w \colon A \rightarrow B$ be a map. Then the 2-cells
 \[
 \vcenter{\hbox{


}}
 \]
 exhibit $d.\overline{w}^{\circ}$ as a right bidual of $w$ in the sense of Proposition~\ref{DUALIZATIONS_PROP}.
\end{prop}

\begin{proof}
 This is a consequence of \cite[Lemma~4.5]{DUALS_INVERT}.
\end{proof}

 We have shown that for a strong monoidal functor $w \colon A \rightarrow B$ between autonomous pseudomonoids, both $w \cdot d$ and $d \cdot \overline{w}^{\circ}$ are right biduals of $w$ in the lax slice $\ca{M} \slash_\ell B$. From Proposition~\ref{UNIQUENESS_OF_DUALS_PROP} we know that biduals are unique up to equivalence. In the particular case of the monoidal 2-category $\ca{M} \slash_\ell B$, the equivalence has a particularly simple form: the morphism $d \cdot \overline{w}^{\circ} \rightarrow w \cdot d$ is of the form $(\id_{\dual{A}},\vartheta)$ for an invertible 2-cell $\vartheta$. The next proposition gives an explicit description of this 2-cell.

\begin{prop}\label{UNIQUENESS_LAX_SLICE_PROP}
 Let $B$ be a pseudomonoid in the Gray monoid $\ca{M}$, and let $(X,\dual{X},n,e,\eta,\varepsilon)$ be a bidual situation in $\ca{M}$. If $\pi$, $\xi$ exhibit $f \colon \dual{X} \rightarrow B$ as a right bidual of $g \colon X \rightarrow B$ in the lax slice $\ca{M} \slash_\ell B$ and $\pi^\prime$, $\xi^\prime$ exhibit $f^\prime \colon \dual{X} \rightarrow B$ as right bidual of $g$, then the 2-cell
 \[
  \vartheta=\vcenter{\hbox{


}}
 \]
 is invertible.
\end{prop}

\begin{proof}
 From Proposition~\ref{UNIQUENESS_OF_DUALS_PROP}, we know that the composite
\[
 \xymatrix{f \ar[r]^-{\cong} & u \bullet f \ar[r]^-{(n,\pi^\prime) \bullet f} & (f^\prime \bullet g) \bullet f \ar[r]^-{\cong} & f^\prime \bullet (g \bullet f) \ar[r]^-{f^\prime \bullet (e,\xi)} & f^\prime \bullet u \ar[r]^-{\cong} & f^\prime}  
\]
 is an equivalence in $\ca{M}\slash_\ell B$. A morphism $(a,\alpha)$ in the lax slice is an equivalence if and only if $a$ is an equivalence and $\alpha$ is invertible. Unraveling the definitions we find that the 2-cell
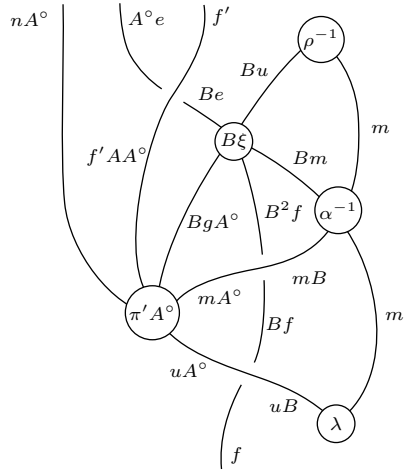
\[
\vcenter{\hbox{


}}
\]
 is invertible. The claim therefore follows from the invertibility of $\eta$.
\end{proof}

\begin{proof}[Proof of Lemma~\ref{HOPF_LEMMA}]
 If we apply Proposition~\ref{UNIQUENESS_LAX_SLICE_PROP} to the right bidual $w \cdot d$ of $w$ (see Proposition~\ref{PRESERVATION_OF_DUALS_PROP}) and the right bidual $d \cdot \overline{w}^{\circ}$ of $w$ (see Proposition~\ref{MAPS_AND_DUALS_PROP}) we find that the 2-cell
 \[
\vcenter{\hbox{

%
}} 
\]
 is invertible. Since $A$ and $B$ are autonomous map pseudomonoids, we can apply Proposition~\ref{FROBENIUS_PROP}, that is, we can assume that $\dual{A}=A$, $\dual{B}=B$, $d_A=\id_A$, $d_B=\id_B$ and the units, counits and all the structural 2-cells are as described in Proposition~\ref{FROBENIUS_PROP}. By inserting this information in the above 2-cell and by applying one of the axioms for a bidual situation we find that the left 2-cell in the equation
 \[
\vcenter{\hbox{
%
}}
 \]
 is invertible. One can see that the above two 2-cells are equal by applying compatibility conditions for mates and units (note that one of the adjunctions in question is the identity map, with unit the identity 2-cell), and a triangle identity for the adjunction $j \dashv \overline{j}$. Moreover, since $B$ is naturally Frobenius (see Definition~\ref{NATURALLY_FROBENIUS_DFN} and Proposition~\ref{FROBENIUS_PROP}) we know that the mate of the associator $\alpha$ is invertible, and we conclude that the 2-cell below the dashed line is invertible. Thus the 2-cell above the dashed line is invertible, too.
 
 The 2-cell above the dashed line is the image of $\tau$ under the equivalence from Proposition~\ref{BIDUAL_BIADJUNCTION_PROP} in the case $X=A$ and $Y=I$ (recall from Proposition~\ref{NATURALLY_FROBENIUS_DFN} that $n_B=\overline{m} \cdot u$). Since equivalences reflect isomorphisms it follows that $\tau$ is invertible. 
\end{proof}

\section{Tannaka duality for pseudomonoidal comonoids}\label{QUANTUM_BIALGEBRAS_SECTION}

\subsection{The 2-category of comonads}
 So far, we considered the category of comonads as a 2-category whose hom-categories are discrete, that is, they have no nonidentity morphisms. But there is a natural definition of 2-cells for comonoids in any monoidal category, in particular for comonads, which are simply comonoids in the monoidal category $\ca{M}(B,B)$ with composition as tensor product.

\begin{dfn}
 Let $\ca{M}$ be a 2-category. Let $f,g \colon c \rightarrow c^{\prime}$ be morphisms of comonads on $B \in \ca{M}$. A 2-cell from $f$ to $g$ is a morphism $\alpha \colon c \rightarrow \id_B$ in $\ca{M}(B,B)$ such that the equation
 \[
  \vcenter{\hbox{


}}
 \]
 holds. If $\beta \colon g \Rightarrow h$ is another 2-cell, their vertical composite is given by
 \[
  \vcenter{\hbox{

%
}}
 \]
 and $\alpha$ whiskered by $k \colon d \rightarrow c$ respectively $l \colon c^{\prime} \rightarrow d^{\prime}$ is given by $\alpha \cdot k$ respectively $\alpha$. Since the resulting 2-category is related to various flavors of quantum bialgebras we denote it by $\Comon_q(B)$.
\end{dfn}

\begin{prop}
 Let $\ca{M}$ be a Gray monoid. Then the assignment which sends $\alpha \colon f \Rightarrow g$ and $\alpha^{\prime} \colon f^{\prime} \Rightarrow g^{\prime}$ to the 2-cell
 \[
  \vcenter{\hbox{

%
}}
 \]
 from $f\star f^{\prime}$ to $g\star g^{\prime}$ extends the convolution tensor product $\star$ on $\Comon(B)$ (see Proposition~\ref{CONVOLUTION_COMONAD_PROP}) to a monoidal 2-category structure on $\Comon_q(B)$ with the same associator and unit isomorphisms. If $\ca{M}$ is braided or sylleptic, then $\Comon_q(B)$ is braided or symmetric, with braiding given as in Proposition~\ref{HOM_CATEGORY_BRAIDED_PROP}.
\end{prop}

\begin{proof}
 This is simply a matter of checking that the assignment described above does indeed give a 2-functor and that the associator, the unit isomorphisms and the braiding are 2-natural transformations. We leave the details to the reader.
\end{proof}

\begin{rmk}
 Let $\ca{M}$ be a Gray monoid equivalent to $\Mod(\Vect_k)$ for a field $k$. Then the monoidal category $\bigl(\Comon(\ca{I}),\star \bigr)$ is equivalent to the category of $k$-coalgebras. A pseudomonoid in the monoidal 2-category $\bigl(\Comon_q(\ca{I}), \star \bigr)$ whose left and right unit isomorphisms are identities is precisely a dual quasi-bialgebra (see \cite{MAJID} for the notion and \cite[Example 2.3]{MCCRUDDEN_BALANCED} for the statement). A dual quasi-triangular quasi-bialgebra is precisely a braided pseudomonoid in $\Comon_q(\ca{I})$ whose left and right unit isomorphisms are identities (see \cite[Example 3.2]{MCCRUDDEN_BALANCED}).
\end{rmk}

\subsection{Monoidal structure}
 Let $\ca{M}$ be a Gray monoid. The 2-category $\Comon(B)$ is contained in the 2-category $\Comon_q(B)$. We also have a corresponding 2-category $\Map_q(\ca{M},B)$ which contains the 2-category $\Map(\ca{M},B)$ in a similar way, that is, it has the same objects and 1-cells but additional 2-cells. The strong monoidal 2-functor
 \[
  L \colon \Map(\ca{M},B) \rightarrow \Comon(B)
 \]
 extends to a strong monoidal 2-functor
 \[
 L \colon \Map_q(\ca{M},B) \rightarrow \Comon_q(B)  
 \]
 which is braided and sylleptic if $\ca{M}$ is. Thus $L$ lifts to the respective categories of (braided or symmetric) pseudomonoids in $\Map_q(\ca{M},B)$ and $\Comon_q(B)$.

\begin{dfn}
 Let $\ca{M}$ be a 2-category. The 2-category $\Map_q(\ca{M}, B)$ has objects the maps with codomain $B$ and morphisms from $f$ to $g$ the pairs $(a,\alpha)$ where $\alpha$ is an invertible 2-cell $g \cdot a \Rightarrow f$. The 2-cells $\gamma$ from $(a,\alpha)$ to $(b,\beta)$ are the 2-cells $\gamma \colon a \Rightarrow b$, subject to no further conditions (that is, there is no compatibility condition between $\alpha$, $\beta$ and $\gamma$).
\end{dfn}

\begin{prop}
 Let $\ca{M}$ be a Gray monoid and let $(B,m,u)$ be a pseudomonoid in $\ca{M}$. Then the pseudofunctor
 \[
\bullet \colon \Map(\ca{M},B) \times \Map(\ca{M},B) \rightarrow \Map(\ca{M},B)
 \]
 from Proposition~\ref{SLICE_MONOIDAL_2CAT_PROP} extends to a pseudofunctor
 \[
\bullet \colon \Map_q(\ca{M},B) \times \Map_q(\ca{M},B) \rightarrow \Map_q(\ca{M},B)\rlap{\smash{.}}
 \]
 This pseudofunctor, together with the associator and the unit isomorphisms from Proposition~\ref{SLICE_MONOIDAL_2CAT_PROP} endows $\Map_q(\ca{M},B)$ with the structure of a monoidal 2-category. If $\ca{M}$ is braided, sylleptic or symmetric, then the braiding from Proposition~\ref{SLICE_BRAIDED_PROP} and the syllepsis or symmetry from Proposition~\ref{SLICE_SYLLEPTIC_PROP} endow $\Map_q(\ca{M},B)$ with the structure of a braided, sylleptic respectively symmetric 2-category. 
\end{prop}

\begin{proof}
 The 1-cell part of these structures do satisfy the necessary conditions (see Proposition~\ref{SLICE_BRAIDED_PROP} and Proposition~\ref{SLICE_SYLLEPTIC_PROP}). The 2-cell part of the braiding and syllepsis coincide with their 2-cell part in $\ca{M}$. But the domain 2-functor
 \[
  \Map_q(\ca{M},B) \rightarrow \ca{M}
 \]
 is locally fully faithful, so the desired 2-naturality and compatibility conditions all follow from the fact that they do hold in $\ca{M}$.
\end{proof}

\begin{rmk}
 Let $\ca{M}$ be a Gray monoid equivalent to $\Mod(\ca{V})$ for some cosmos $\ca{V}$, and let $\ca{B}$ be a monoidal $\ca{V}$-category. A pseudomonoid in $\Map_q(\ca{M},\ca{B})$ whose underlying object is Cauchy complete corresponds to a monoidal $\ca{V}$-category equipped with a $\ca{V}$-functor to $\overline{\ca{B}}$ which is \emph{multiplicative} in the sense of \cite{MAJID}, that is, it is a functor equipped with a $\ca{V}$-natural isomorphism $FA \otimes FB \rightarrow F(A\otimes B)$ and an isomorphism $I \rightarrow FI$, subject to no coherence conditions (cf.\ \cite[Example 2.5]{MCCRUDDEN_BALANCED}). 
\end{rmk}

\begin{prop}\label{Q_LEFT_ADJ_PROP}
 Let $\ca{M}$ be a braided Gray monoid and let $(B,m,u)$ be a map pseudomonoid in $\ca{M}$. Then the assignment that sends a 2-cell $\phi \colon (a,\alpha) \Rightarrow (b,\beta)$ to
 \[
  L(\phi) \defl \vcenter{\hbox{


}}
 \]
 extends the 2-functor $L$ from Proposition~\ref{LEFT_BIADJOINT_2FUNCTOR_PROP} to a 2-functor
 \[
  \Map_q(\ca{M},B) \rightarrow \Comon_q(B)
 \]
 and the 2-cell $\chi$ from Proposition~\ref{LEFT_ADJ_MONOIDAL_PROP} defines a 2-natural isomorphism $\star \cdot L\times L \Rightarrow L \cdot \bullet$ which equips $L$ with the structure of a strong monoidal 2-functor. If $\ca{M}$ and $B$ are braided, then $L$ becomes braided (where the necessary modification is an identity modification, as in Proposition~\ref{LEFT_ADJ_BRAIDED_PROP}). If $\ca{M}$ is sylleptic and $B$ is symmetric, then $L$ is sylleptic.
\end{prop}

\begin{proof}
 We have to check that $L$ is a 2-functor and that $\chi$ is a 2-natural transformation. The former follows from the definition of the 2-category structure on $\Comon_q(B)$, and 2-naturality of $\chi$ follows from the equation
 \[
  \vcenter{\hbox{


}}
 \]
 The fact that $L$ is braided if $\ca{M}$ is follows directly from Proposition~\ref{LEFT_ADJ_BRAIDED_PROP} because the braiding of $\Map_q(\ca{M},B)$ is contained in the subcategory $\Map(\ca{M},B)$. Similarly, the syllepsis is contained in $\Map(\ca{M},B)$, so it gets sent to an identity by $L$, and the claim that $L$ is sylleptic follows from Proposition~\ref{LEFT_ADJ_SYLLEPTIC_PROP}.
\end{proof}

\begin{thm}
 Let $\ca{M}$ be a 2-category with Tannaka-Krein objects, and let $B \in \ca{M}$. Then the 2-functor
 \[
  L \colon  \Map_q(\ca{M},B) \rightarrow \Comon_q(B)
 \]
 from Proposition~\ref{Q_LEFT_ADJ_PROP} has a right biadjoint whose restriction to $\Comon(B)$ is the pseudofunctor $\Rep(-)$ from Proposition~\ref{REP_2_FUNCTOR_PROP}). 
 
 If $\ca{M}$ is a Gray monoid and $B$ is a map pseudomonoid, then the biadjunction lifts to the categories of pseudomonoids in $\Comon_q(B)$ and pseudomonoids in $\Map_q(\ca{M},B)$. Furthermore, if $\ca{M}$ and $B$ are braided the biadjunction lifts to braided pseudomonoids, and if $\ca{M}$ is sylleptic and $B$ is symmetric the biadjunction lifts to symmetric pseudomonoids.
\end{thm}

\begin{proof}
 It suffices to check that $L$ has a right biadjoint; the desired lifts are consequences of Proposition~\ref{Q_LEFT_ADJ_PROP}, Corollary~\ref{LIFT_TO_PSEUDOMONOIDS_COR}, Corollary~\ref{LIFT_TO_BRAIDED_PSEUDOMONOIDS_COR} and Corollary~\ref{LIFT_TO_SYMMETRIC_PSEUDOMONOIDS_COR} respectively.
 
 We thus have to extend $\Rep(-)$ to the new 2-cells in $\Comon_q(B)$ and we have to show that this gives the desired biadjoint. Let $\xi\colon \phi \Rightarrow \phi^{\prime} \colon c \rightarrow c^{\prime}$ be the 2-cell in $\Comon_q(B)$. With the notation from Proposition~\ref{REP_2_FUNCTOR_PROP}, the 2-cell
 \[
\vcenter
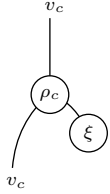
{\hbox{

\begin{tikzpicture}[y=0.80pt, x=0.8pt,yscale=-1, inner sep=0pt, outer sep=0pt, every text node part/.style={font=\scriptsize} ]
\path[draw=black,line join=miter,line cap=butt,line width=0.500pt]
  (1133.8583,999.2126) .. controls (1133.8583,1016.9291) and
  (1133.8583,1016.9291) .. (1133.8583,1034.6456);
\path[draw=black,line join=miter,line cap=butt,line width=0.500pt]
  (1133.8583,1034.6456) .. controls (1123.3781,1044.2293) and
  (1117.4415,1055.0291) .. (1116.1417,1070.0787);
\path[draw=black,line join=miter,line cap=butt,line width=0.500pt]
  (1133.8583,1034.6456) .. controls (1143.3204,1038.5316) and
  (1147.9141,1045.1113) .. (1151.5748,1052.3622);
\path[fill=black] (1133.8849,1035.3916) node[circle, draw, line width=0.500pt,
  minimum width=5mm, fill=white, inner sep=0.25mm] (text5986) {$\rho_c$    };
\path[fill=black] (1152.1125,1053.6193) node[circle, draw, line width=0.500pt,
  minimum width=5mm, fill=white, inner sep=0.25mm] (text5990) {$\xi$    };
\path[fill=black] (1131.3707,996.42218) node[above right] (text5994) {$v_c$
  };
\path[fill=black] (1113.1431,1079.5038) node[above right] (text5998) {$v_c$   };

\end{tikzpicture}
}}
 \]
 defines a morphism of coactions $\rho_\phi \rightarrow \rho_{\phi^{\prime}}$. We define $\Rep(\xi)$ to be the image of that morphism of coactions under $T^{-1}$. We leave it to the reader to check that this defines a pseudofunctor (with constraints as defined in Proposition~\ref{REP_2_FUNCTOR_PROP}).
 
 It remains to check that $\Rep(-)$ is a right biadjoint of $L$. To see this, we extend the strict natural equivalence
 \[
 \theta_{w,c} \colon \Map(\ca{M},B)(w,v_c) \rightarrow \Comon(B)\bigl(L(w),c\bigr)
 \]
 from Theorem~\ref{TANNAKIAN_BIADJUNCTION_2CAT_THM} to a strict natural equivalence
 \[
 \theta_{w,c} \colon \Map_q(\ca{M},B)(w,v_c) \rightarrow \Comon_q(B)\bigl(L(w),c\bigr) \rlap{\smash{,}}
 \]
 that is, on objects and 1-cells we define $\theta_{w,c}$ as in the proof of Theorem~\ref{TANNAKIAN_BIADJUNCTION_2CAT_THM}. Given a 2-cell $\alpha \colon (s,\sigma) \Rightarrow (t,\tau)\colon w \rightarrow v_c$ in $\Map_q(\ca{M},B)$, we define
 \[
 \theta_{w,c}(\alpha) \defl \vcenter{\hbox{


}}
 \]
 From the definition of 2-cells in $\Comon_q(B)$ it follows easily that this is indeed a 2-cell $\theta_{w,c}(s,\sigma) \Rightarrow \theta_{w,c}(t,\tau)$. The naturality square of $\theta_{w,c}$ still commutes strictly (cf.\ the proof of Theorem~\ref{TANNAKIAN_BIADJUNCTION_2CAT_THM}), and we leave it to the reader to check that this defines the desired strict natural transformation.
 
 From Theorem~\ref{TANNAKIAN_BIADJUNCTION_2CAT_THM} we know that $\theta_{w,c}$ is surjective on objects. Hence it is an equivalence if and only if it is fully faithful. Faithfulness follows from the fact that whiskering with $v_c$ is faithful (cf.\ Remark~\ref{TK_DOUBLE_LIMIT_RMK}). If $\xi \colon \theta_{w,c}(s,\sigma) \Rightarrow \theta_{w,c}(t,\tau)$ is a 2-cell in $\Comon_q(B)$, then
 \[
  \vcenter{\hbox{


}}
 \]
 is a morphism of coactions from $\rho_c \cdot s$ to $\rho_c \cdot t$. Since the functor $T$ from Remark~\ref{TK_DOUBLE_LIMIT_RMK} is fully faitfhul, this morphism must be of the form $v_c \cdot \alpha$ for a unique 2-cell $\alpha \colon s \Rightarrow t$. It follows immediately from the definition of $\theta$ that $\theta_{w,c}(\alpha)=\xi$.
\end{proof}

\bibliographystyle{amsalpha}
\bibliography{fttfinal}

\providecommand{\bysame}{\leavevmode\hbox to3em{\hrulefill}\thinspace}
\providecommand{\MR}{\relax\ifhmode\unskip\space\fi MR }
\providecommand{\MRhref}[2]{%
  \href{http://www.ams.org/mathscinet-getitem?mr=#1}{#2}
}
\providecommand{\href}[2]{#2}
\begin{thebibliography}{LFSW11}

\bibitem[AT69]{APPLEGATE_TIERNEY}
H.~Appelgate and M.~Tierney, \emph{Categories with models}, Sem. on {T}riples
  and {C}ategorical {H}omology {T}heory ({ETH}, {Z}{\"u}rich, 1966/67),
  Springer, Berlin, 1969, pp.~156--244. \MR{0242916 (39 \#4243)}

\bibitem[BD97]{BAEZ_DOLAN}
John Baez and James Dolan, \emph{Title: Higher-dimensional algebra {I}{I}{I}:
  n-categories and the algebra of opetopes}, preprint,
  \href{http://arxiv.org/abs/q-alg/9702014v1}{\tt arXiv:\ 9702014v1\
  [math.QA]}, 1997.

\bibitem[B{\'e}n73]{BENABOU}
Jean B{\'e}nabou, \emph{Les distributeurs}, Rapport No.~33, S{\'e}minaires de
  Math.\ Pure, Univ.\ Catholique de Louvain, 1973.

\bibitem[BLV10]{BRUGUIERES_LACK_VIRELIZIER}
Alain Brugui{\`e}res, Stephen Lack, and Alexis Virelizier, \emph{{H}opf monads
  on monoidal categories}, preprint,
  \href{http://arxiv.org/abs/1003.1920v4}{\tt arXiv:\ 1003.1920v4\ [math.QA]},
  2010.

\bibitem[Bor94]{BORCEUX}
Francis Borceux, \emph{Handbook of categorical algebra. 1}, Encyclopedia of
  Mathematics and its Applications, vol.~50, Cambridge University Press,
  Cambridge, 1994, Basic category theory. \MR{MR1291599 (96g:18001a)}

\bibitem[Bru04]{BRUGUIERES}
Alain Brugui{\`e}res, \emph{On a tannakian theorem due to {N}ori}, unpublished,
  available on
  \href{http://www.math.univ-montp2.fr/~bruguieres/}{http://www.math.univ-montp2.fr/~bruguieres/}
  (July 10, 2011), 2004.

\bibitem[BV07]{BRUGUIERES_VIRELIZIER}
Alain Brugui{\`e}res and Alexis Virelizier, \emph{Hopf monads}, Adv. Math.
  \textbf{215} (2007), no.~2, 679--733. \MR{2355605 (2009b:18006)}

\bibitem[CLS10]{CHIKHLADZE_LACK_STREET}
Dimitri Chikhladze, Stephen Lack, and Ross Street, \emph{{H}opf monoidal
  comonads}, preprint, \href{http://arxiv.org/abs/1002.1122v2}{\tt arXiv:\
  1002.1122v2\ [math.CT]}, 2010.

\bibitem[Cru08]{CRUTTWELL}
G.~S.~H. Cruttwell, \emph{{N}ormed spaces and the change of base for enriched
  categories}, Ph.D. thesis, Dalhousie University, 2008, Available on
  \href{http://geoff.reluctantm.com/publications/GThesis.pdf}{http://geoff.reluctantm.com/publications/GThesis.pdf}
  (October 21, 2011).

\bibitem[Day77]{DAY_LINEAR}
B.~J. Day, \emph{Linear monads}, Bull. Austral. Math. Soc. \textbf{17} (1977),
  no.~2, 177--192. \MR{0466260 (57 \#6140)}

\bibitem[Day96]{DAY}
Brian~J. Day, \emph{Enriched {T}annaka reconstruction}, J. Pure Appl. Algebra
  \textbf{108} (1996), no.~1, 17--22. \MR{MR1382240 (97d:18008)}

\bibitem[Del90]{DELIGNE}
P.~Deligne, \emph{Cat\'egories tannakiennes}, The {G}rothendieck {F}estschrift,
  {V}ol.\ {II}, Progr. Math., vol.~87, Birkh\"auser Boston, Boston, MA, 1990,
  pp.~111--195. \MR{MR1106898 (92d:14002)}

\bibitem[DM82]{DELIGNE_MILNE}
Pierre Deligne and James~S. Milne, \emph{Tannakian categories}, Hodge cycles,
  motives, and {S}himura varieties, Lecture Notes in Mathematics, vol. 900,
  Springer-Verlag, Berlin, 1982, pp.~101--228. \MR{654325 (84m:14046)}

\bibitem[DMS03]{DUALIZATIONS_ANTIPODES}
Brian Day, Paddy McCrudden, and Ross Street, \emph{Dualizations and antipodes},
  Appl. Categ. Structures \textbf{11} (2003), no.~3, 229--260. \MR{1990034
  (2004b:18013)}

\bibitem[DP08]{DAY_PASTRO}
Brian Day and Craig Pastro, \emph{Note on {F}robenius monoidal functors}, New
  York J. Math. \textbf{14} (2008), 733--742. \MR{2465800 (2009k:18001)}

\bibitem[DS97]{DAY_STREET}
Brian Day and Ross Street, \emph{Monoidal bicategories and {H}opf algebroids},
  Adv. Math. \textbf{129} (1997), no.~1, 99--157. \MR{MR1458415 (99f:18013)}

\bibitem[Dub68]{DUBUC_TRIANGLE}
Eduardo Dubuc, \emph{Adjoint triangles}, Reports of the {M}idwest {C}ategory
  {S}eminar, {II}, Springer, Berlin, 1968, pp.~69--91. \MR{0233864 (38 \#2185)}

\bibitem[Dub70]{DUBUC}
Eduardo~J. Dubuc, \emph{Kan extensions in enriched category theory}, Lecture
  Notes in Mathematics, Vol. 145, Springer-Verlag, Berlin, 1970. \MR{MR0280560
  (43 \#6280)}

\bibitem[EK66]{EILENBERG_KELLY}
Samuel Eilenberg and G.~Max Kelly, \emph{Closed categories}, Proc. {C}onf.
  {C}ategorical {A}lgebra ({L}a {J}olla, {C}alif., 1965), Springer, New York,
  1966, pp.~421--562. \MR{0225841 (37 \#1432)}

\bibitem[FL82]{FONTAINE_LAFFAILLE}
Jean-Marc Fontaine and Guy Laffaille, \emph{Construction de repr\'esentations
  {$p$}-adiques}, Ann. Sci. \'Ecole Norm. Sup. (4) \textbf{15} (1982), no.~4,
  547--608 (1983). \MR{MR707328 (85c:14028)}

\bibitem[GP97]{GORDON_POWER}
R.~Gordon and A.~J. Power, \emph{Enrichment through variation}, J. Pure Appl.
  Algebra \textbf{120} (1997), no.~2, 167--185. \MR{MR1466618 (98i:18004)}

\bibitem[GPS95]{GORDON_POWER_STREET}
R.~Gordon, A.~J. Power, and Ross Street, \emph{Coherence for tricategories},
  Mem. Amer. Math. Soc. \textbf{117} (1995), no.~558, vi+81. \MR{1261589
  (96j:18002)}

\bibitem[Gur06]{GURSKI}
Nick Gurski, \emph{{A}n algebraic theory of tricategories}, Ph.D. thesis,
  University of Chicago, 2006, Available on
  \href{http://www.math.yale.edu/~mg622/tricats.pdf}{http://www.math.yale.edu/~mg622/tricats.pdf}
  (October 12, 2011).

\bibitem[Hai08]{PHUNG_HO_HAI}
Ph{\`u}ng~H{\^o} Hai, \emph{Tannaka-{K}rein duality for {H}opf algebroids},
  Israel J. Math. \textbf{167} (2008), 193--225. \MR{2448024 (2009g:16059)}

\bibitem[Hov04]{HOVEY}
Mark Hovey, \emph{Homotopy theory of comodules over a {H}opf algebroid},
  Homotopy theory: relations with algebraic geometry, group cohomology, and
  algebraic {$K$}-theory, Contemp. Math., vol. 346, Amer. Math. Soc.,
  Providence, RI, 2004, pp.~261--304. \MR{2066503 (2005f:18011)}

\bibitem[Joh89]{JOHNSEN}
S.~R. Johnsen, \emph{Small {C}auchy completions}, J. Pure Appl. Algebra
  \textbf{62} (1989), no.~1, 35--45. \MR{MR1026873 (90j:18007)}

\bibitem[JS91]{JOYAL_STREET_TENSOR}
Andr{\'e} Joyal and Ross Street, \emph{The geometry of tensor calculus. {I}},
  Adv. Math. \textbf{88} (1991), no.~1, 55--112. \MR{1113284 (92d:18011)}

\bibitem[Kel74]{KELLY_DOCTRINAL}
G.~M. Kelly, \emph{Doctrinal adjunction}, Category {S}eminar ({P}roc. {S}em.,
  {S}ydney, 1972/1973), Springer, Berlin, 1974, pp.~257--280. Lecture Notes in
  Math., Vol. 420. \MR{0360749 (50 \#13196)}

\bibitem[Kel82]{KELLY_FINLIM}
\bysame, \emph{Structures defined by finite limits in the enriched context.
  {I}}, Cahiers Topologie G\'eom. Diff\'erentielle \textbf{23} (1982), no.~1,
  3--42, Third Colloquium on Categories, Part VI (Amiens, 1980). \MR{MR648793
  (83h:18007)}

\bibitem[Kel05a]{KELLY_BASIC}
\bysame, \emph{Basic concepts of enriched category theory}, Repr. Theory Appl.
  Categ. (2005), no.~10, vi+137 pp. (electronic), Reprint of the 1982 original
  [Cambridge Univ. Press, Cambridge; MR0651714]. \MR{MR2177301}

\bibitem[Kel05b]{KELLY_COSMOS}
\bysame, \emph{On the operads of {J}. {P}. {M}ay}, Repr. Theory Appl. Categ.
  (2005), no.~13, 1--13 (electronic). \MR{MR2177746 (2006f:18005)}

\bibitem[KS74]{KELLY_STREET}
G.~M. Kelly and Ross Street, \emph{Review of the elements of {$2$}-categories},
  Category {S}eminar ({P}roc. {S}em., {S}ydney, 1972/1973), Springer, Berlin,
  1974, pp.~75--103. Lecture Notes in Math., Vol. 420. \MR{MR0357542 (50
  \#10010)}

\bibitem[Lac00]{LACK_PSEUDOMONADS}
Stephen Lack, \emph{A coherent approach to pseudomonads}, Adv. Math.
  \textbf{152} (2000), no.~2, 179--202. \MR{1764104 (2001f:18017)}

\bibitem[Lac02]{LACK_2CAT}
\bysame, \emph{A {Q}uillen model structure for 2-categories}, $K$-Theory
  \textbf{26} (2002), no.~2, 171--205. \MR{1931220 (2003m:55028)}

\bibitem[Lac10a]{LACK_COMPANION}
\bysame, \emph{A 2-categories companion}, Towards higher categories, IMA Vol.
  Math. Appl., vol. 152, Springer, New York, 2010, pp.~105--191. \MR{2664622
  (2011d:18012)}

\bibitem[Lac10b]{LACK_GRAY}
\bysame, \emph{A {Q}uillen model structure for {G}ray-categories}, preprint,
  \href{http://arxiv.org/abs/1001.2366v2}{\tt arXiv:\ 1001.2366v2\ [math.CT]},
  2010.

\bibitem[Law73]{LAWVERE_METRIC}
F.~William Lawvere, \emph{Metric spaces, generalized logic, and closed
  categories}, Rend. Sem. Mat. Fis. Milano \textbf{43} (1973), 135--166 (1974).
  \MR{MR0352214 (50 \#4701)}

\bibitem[LF09]{LOPEZ_FRANCO}
Ignacio L{\'o}pez~Franco, \emph{{A}utonomous pseudomonoids}, Ph.D. thesis,
  University of Cambridge, 2009, Available on
  \href{http://www.dspace.cam.ac.uk/handle/1810/219201}{http://www.dspace.cam.ac.uk/handle/1810/219201}
  (October 28, 2011).

\bibitem[LFSW11]{DUALS_INVERT}
Ignacio L{\'o}pez~Franco, Ross Street, and Richard Wood, \emph{Duals invert},
  Applied Categorical Structures \textbf{19} (2011), 321--361,
  10.1007/s10485-009-9210-7.

\bibitem[Maj92]{MAJID}
Shahn Majid, \emph{Tannaka-{K}re\u\i n theorem for quasi-{H}opf algebras and
  other results}, Deformation theory and quantum groups with applications to
  mathematical physics ({A}mherst, {MA}, 1990), Contemp. Math., vol. 134, Amer.
  Math. Soc., Providence, RI, 1992, pp.~219--232. \MR{1187289 (93k:16073)}

\bibitem[McC00a]{MCCRUDDEN_BALANCED}
Paddy McCrudden, \emph{Balanced coalgebroids}, Theory Appl. Categ. \textbf{7}
  (2000), No. 6, 71--147 (electronic). \MR{MR1764504 (2001f:18018)}

\bibitem[McC00b]{MCCRUDDEN_REPR_COALGEBROIDS}
\bysame, \emph{Categories of representations of coalgebroids}, Adv. Math.
  \textbf{154} (2000), no.~2, 299--332. \MR{MR1784678 (2002b:18008)}

\bibitem[McC02]{MCCRUDDEN_MASCHKE}
\bysame, \emph{Tannaka duality for {M}aschkean categories}, J. Pure Appl.
  Algebra \textbf{168} (2002), no.~2-3, 265--307, Category theory 1999
  (Coimbra). \MR{MR1887160 (2003d:18012)}

\bibitem[Par81]{PAREIGIS_DG}
Bodo Pareigis, \emph{A noncommutative noncocommutative {H}opf algebra in
  ``nature''}, J. Algebra \textbf{70} (1981), no.~2, 356--374. \MR{623814
  (83g:16021)}

\bibitem[Par96]{PAREIGIS}
\bysame, \emph{Reconstruction of hidden symmetries}, J. Algebra \textbf{183}
  (1996), no.~1, 90--154. \MR{MR1397390 (98h:18009)}

\bibitem[Pen71]{PENROSE}
Roger Penrose, \emph{Applications of negative dimensional tensors},
  Combinatorial {M}athematics and its {A}pplications ({P}roc. {C}onf.,
  {O}xford, 1969), Academic Press, London, 1971, pp.~221--244. \MR{0281657 (43
  \#7372)}

\bibitem[Saa72]{SAAVEDRA}
Neantro Saavedra{\ }Rivano, \emph{Cat\'egories {T}annakiennes}, Lecture Notes
  in Mathematics, Vol. 265, Springer-Verlag, Berlin, 1972. \MR{MR0338002 (49
  \#2769)}

\bibitem[Ser68]{SERRE}
Jean-Pierre Serre, \emph{Corps locaux}, Hermann, Paris, 1968, Troisi{\`e}me
  {\'e}dition, Publications de l'Universit{\'e} de Nancago, No. VIII.
  \MR{MR0354618 (50 \#7096)}

\bibitem[Shu10]{SHULMAN}
Michael Shulman, \emph{{C}onstructing symmetric monoidal bicategories},
  preprint, \href{http://arxiv.org/abs/1004.0993v1}{\tt arXiv:\ 1004.0993v1\
  [math.CT]}, 2010.

\bibitem[Str72]{STREET_FTM}
Ross Street, \emph{The formal theory of monads}, J. Pure Appl. Algebra
  \textbf{2} (1972), no.~2, 149--168. \MR{MR0299653 (45 \#8701)}

\bibitem[Str83]{STREET_ABSOLUTE}
\bysame, \emph{Absolute colimits in enriched categories}, Cahiers Topologie
  G\'eom. Dif\-f\'eren\-tielle \textbf{24} (1983), no.~4, 377--379.
  \MR{MR749468 (85i:18001)}

\bibitem[Str96]{STREET_STRING_DIAGRAMS}
\bysame, \emph{Categorical structures}, Handbook of algebra, {V}ol.\ 1,
  North-Holland, Amsterdam, 1996, pp.~529--577. \MR{1421811 (97j:18007)}

\bibitem[Str04]{STREET_FROBENIUS}
\bysame, \emph{Frobenius monads and pseudomonoids}, J. Math. Phys. \textbf{45}
  (2004), no.~10, 3930--3948. \MR{2095680 (2005h:18026)}

\bibitem[Str07]{STREET_QUANTUM_GROUPS}
\bysame, \emph{Quantum groups}, Australian Mathematical Society Lecture Series,
  vol.~19, Cambridge University Press, Cambridge, 2007, A path to current
  algebra. \MR{MR2294803 (2008a:16061)}

\bibitem[Szl09]{SZLACHANYI}
Korn{\'e}l Szlach{\'a}nyi, \emph{{F}iber functors, monoidal sites and {T}annaka
  duality for bialgebroids}, preprint,
  \href{http://arxiv.org/abs/0907.1578v1}{\tt arXiv:\ 0907.1578v1\ [math.QA]},
  2009.

\bibitem[Ver11]{VERITY}
Dominic Verity, \emph{{E}nriched categories, internal categories and change of
  base}, Repr. Theory Appl. Categ. (2011), no.~20, 266 pp. (electronic),
  Originally published as: Ph.D. thesis, Cambridge University, 1992.

\bibitem[Wed04]{WEDHORN}
Torsten Wedhorn, \emph{On {T}annakian duality over valuation rings}, J. Algebra
  \textbf{282} (2004), no.~2, 575--609. \MR{MR2101076 (2005j:18007)}

\bibitem[Win84]{WINTENBERGER}
Jean-Pierre Wintenberger, \emph{Un scindage de la filtration de {H}odge pour
  certaines vari\'et\'es alg\'ebriques sur les corps locaux}, Ann. of Math. (2)
  \textbf{119} (1984), no.~3, 511--548. \MR{MR744862 (86k:14015)}

\end{thebibliography}

\end{document}